\documentclass[a4paper,english]{article}
\usepackage[a4paper]{geometry}
\usepackage{amsthm,amsmath,amssymb,mathtools}
\usepackage{enumerate,enumitem,csquotes,verbatim}
\usepackage{setspace}
\usepackage[numbers,sort&compress]{natbib}
\usepackage{hyperref}
\usepackage{color}
\usepackage{tikz}
\usepackage{xcolor}
\definecolor{darkgreen}{rgb}{0.0, 0.5, 0.0}
\definecolor{darkcyan}{rgb}{0.0, 0.55, 0.55}
\usepackage{multicol}

\theoremstyle{plain}
\newtheorem*{theorem*}{Theorem}
\newtheorem{theorem}{Theorem}
\newtheorem{lemma}[theorem]{Lemma}

\newtheorem*{claim*}{Claim}

\newtheorem{conjecture}[theorem]{Conjecture}

\newtheorem{observation}[theorem]{Observation}

\theoremstyle{definition}

\theoremstyle{remark}

\def\PP{\mathcal{P}}
\def\QQ{\mathcal{Q}}
\def\CC{\mathscr{C}}
\def\DD{\mathscr{D}}
\def\EE{\mathscr{E}}
\def\FF{\mathscr{F}}
\def\N{\mathbb{N}}
\def\Z{\mathbb{Z}}
\def\Q{\mathbb{Q}}
\def\R{\mathbb{R}}
\def\P{\mathbb{P}}
\def\E{\mathbb{E}}
\def\ss{\mathbb{S}}
\def\C{\mathcal}
\def\B{\mathbf}
\def\Scr{\mathscr}
\def\cycprec{\prec}

\DeclareMathOperator\Deg{d}
\DeclareMathOperator\Bi{Bin}
\DeclareMathOperator\Po{Po}
\DeclareMathOperator\Geom{Geom}
\DeclareMathOperator\V{Var}
\let\emptyset\varnothing
\let\eps\varepsilon
\newcommand{\sgn}{\operatorname{sgn}}
\newcommand\ex{\ensuremath{\mathrm{ex}}}
\newcommand\st{\ensuremath{\mathrm{st}}}
\newcommand\EX{\ensuremath{\mathrm{EX}}}
\newcommand*{\abs}[1]{\lvert#1\rvert}

\title{The $S$-packing coloring of the infinite diagonal grid \\with $S=(1,6,6,\ldots)$}

\author{Teeradej Kittipassorn\thanks{\,Department of Mathematics and Computer Science, Faculty of Science, Chulalongkorn University, Bangkok 10330, Thailand; \texttt{teeradej.k@chula.ac.th}} \thanks{\,Centre of Excellence in Mathematics, Ministry of Higher Education, Science, Research and Innovation, Thailand}
  \and Peerawit Suriya\thanks{\,Mathematics Institute, University of Warwick,
      Coventry, UK; \texttt{peerawit.suriya@warwick.ac.uk}}}

\begin{document}
\maketitle

\begin{abstract}
For a non-decreasing sequence of positive integers $S = (a_1, a_2,\ldots)$, the \emph{$S$-packing chromatic number} of a graph $G$ is the smallest positive integer $k$ such that the vertices can be colored with $k$ colors, where the distance between any two distinct vertices of color $i$ is greater than $a_i$. In this paper, we show that the $S$-packing chromatic number of the infinite diagonal grid $P_\infty \boxtimes P_\infty$ with $S = (1,6,6,\ldots)$ is $40$. This confirms a conjecture of the first author and Tiyajamorn.

\end{abstract}

\section{Introduction}\label{sec:1}

Graph coloring is one of the most well-known concepts in graph theory with many applications and open problems. A \emph{proper coloring} of a graph $G$ is an assignment of a color to each vertex such that the distance between any two distinct vertices of the same color is greater than $1$. The minimum number of colors is called the \emph{chromatic number} of $G$. A natural question that arises when considering proper coloring is: what is the chromatic number of a given graph $G$?

For a non-decreasing sequence of positive integers $S=(a_1,a_2,\ldots)$, the \emph{$S$-packing coloring} of a graph $G$ is an assignment of a color to each vertex such that the distance between any two distinct vertices of color $i$ is greater than $a_i$. The minimum number of colors is called the \emph{$S$-packing chromatic number}, denoted by $\chi_S(G)$. Note that a proper coloring is precisely an $S$-packing coloring with $S=(1,1,\ldots)$. In the case $S = (1,2,3,\ldots)$, this coloring is also known as the \emph{packing coloring}, with \emph{packing chromatic number} denoted by $\chi_\rho(G)$. This concept was first studied in 2008 by Goddard, S. M. Hedetniemi, S. T. Hedetniemi, Harris and Roll \cite{goddard2008broadcast}.

An $S$-packing coloring has been studied in various types of graphs. Several graphs that have been studied is infinite. The \emph{two-way infinite path} $P_\infty$ is a graph such that $V(P_\infty)=\mathbb{Z}$ and $\{u,v\}\in E(P_\infty)$ if $|u-v|=1$. The \emph{infinite square grid} $P_\infty \Box P_\infty$ is a graph such that $V(P_\infty \Box P_\infty)=\mathbb{Z}\times\mathbb{Z}$ and $\{(x,y),(x',y')\}\in E(P_\infty \Box P_\infty)$ if either $x=x'$ and $|y-y'|=1$, or $y=y'$ and $|x-x'|=1$. The \emph{infinite diagonal grid} $P_\infty \boxtimes P_\infty$ is obtained from $P_\infty \Box P_\infty$ by adding edges $\{(x,y),(x',y')\}$ if $|x-x'|=|y-y'|=1$.

In 2012, Goddard and Xu \cite{path} studied the two-way infinite path $P_\infty$ for many sequences $S$ and obtained exact values and bounds for $\chi_S(P_\infty)$. Goddard and Xu \cite{goddard_grid} also studied the infinite square grid $P_\infty \Box P_\infty$ for many sequences in 2014. In the case $S=(1,2,3,\ldots)$, there have been many results on the infinite square grid since 2008 \cite{ekstein2010packing,Fiala2009,goddard2008broadcast,martin2015packing,martin2017packing,soukal2010packing,subercaseaux2022packing}, improving the lower and upper bounds for $\chi_\rho(P_\infty \Box P_\infty)$. In 2023, Subercaseaux and Heule \cite{subercaseaux2023packing} finally proved that $\chi_\rho(P_\infty \Box P_\infty) = 15$ by using a SAT solver.

In 2017, Korže and Vesel \cite{KORZE} studied the case where $S$ is of the form $a_i = d+ \lfloor \tfrac{i-1}{n} \rfloor$, where $d$ and $n$ are small positive integers, in the infinite diagonal grid $P_\infty \boxtimes P_\infty$. Later, in 2024, the first author and Tiyajamorn~\cite{tiyajamorn2024packing} studied $S$-packing colorings of the infinite diagonal grid with $S = (i,k,k,\ldots)$. One of the incomplete cases is $i=1$ with even $k$. They obtained lower and upper bounds for $\chi_S(P_\infty \boxtimes P_\infty)$ which are
\[
\frac{3k^2}{4}+\frac{3k}{2}+3 \leq \chi_S(P_\infty \boxtimes P_\infty) \leq \left\lceil \frac{3k^2+7k+8}{4} \right\rceil.
\]
For $k=2$ and $k=4$, from the bounds, the exact values of $\chi_S(P_\infty \boxtimes P_\infty)$ are $9$ and $21$, respectively. For $k=6$, their result implies that $39 \leq \chi_S(P_\infty \boxtimes P_\infty) \leq 40$. They conjectured that $\chi_S(P_\infty \boxtimes P_\infty) = 40$.

In this paper, we confirm the conjecture by showing that it is impossible to use only $39$ colors for an $S$-packing coloring of $P_\infty \boxtimes P_\infty$ with $S = (1,6,6,\ldots)$.

\begin{theorem}\label{thm:1}
If $ S = (1, 6, 6, \ldots) $, then $ \chi_S(P_\infty \boxtimes P_\infty) = 40 $.
\end{theorem}

The rest of the paper is organized as follows. In Section~\ref{sec:2}, we make basic observations. We introduce the main Lemma, Lemma \ref{lem:4}, in Section~\ref{sec:3}. In Section~\ref{sec:4}, we prove Theorem~\ref{thm:1} using Observation~\ref{obs:2} and Lemma~\ref{lem:4}. Lemma~\ref{lem:4} is proved in Section~\ref{sec:5}. Finally, in Section~\ref{sec:6}, we conclude the paper with some open problems.

\section{Observations}\label{sec:2}
In this section, we explain how we visualize the graph $P_\infty \boxtimes P_\infty$ and introduce some notations used in this paper. We also provide basic observations along with their proofs.

Let $ S = (1, 6, 6, 6, \ldots) $. Throughout this paper we assume that $ \chi_S(P_\infty \boxtimes P_\infty) \leq 39 $, since this is the assumption in the proof. That is, there exists an $ S $-packing coloring of $ P_\infty \boxtimes P_\infty $ that uses at most 39 colors. We refer to these colors as colors $1,2,\ldots,39$. From the sequence $ S $, any two vertices of color 1 must have distance greater than 1, while any two vertices of color $ i \neq 1 $ must have distance greater than 6.

We represent the graph $ P_\infty \boxtimes P_\infty $ using an infinite square grid, where each square corresponds to a vertex of the graph. The distance between any two vertices is defined as the minimum number of steps between corresponding squares, where each step moves one square horizontally, vertically, or diagonally, like a king's move in chess.

Let $ G $ be a $ P_m \boxtimes P_n $ subgraph of $ P_\infty \boxtimes P_\infty $. We refer to each square in $ G $ as $ (x, y) $, for $ 1 \leq x \leq m $ and $ 1 \leq y \leq n $, where $ (1, 1) $ corresponds to the bottom-left corner of $ G $, and $ (m, n) $ to the top-right corner. Moreover, we refer to squares $ (x, j) $, for $ 1 \leq x \leq m $, as \emph{row $ j $}, and squares $ (i, y) $, for $ 1 \leq y \leq n $, as \emph{column $ i $}. Sometimes, we may allow $x$ and $y$ to be outside the ranges to refer to a square outside $G$. We denote by $ c(G) $ the number of squares of color 1 in the subgraph $ G $.

\begin{observation} \label{obs:2}
\begin{enumerate}[label=(\roman*)]
    \item If $ G $ is a $ P_2 \boxtimes P_2 $ subgraph, then $ c(G) \leq 1 $.
    \item If $ G $ is a $ P_7 \boxtimes P_2 $ subgraph, then $ c(G) \leq 4 $. Moreover, if $ c(G) = 4 $, then the squares of color 1 must lie in columns 1, 3, 5, and 7.
    \item Each $ P_7 \boxtimes P_7 $ subgraph contains at most one square of color $ i $, for $ i \in \{2, 3, \ldots, 39\} $.
    \item If $ G $ is a $ P_7 \boxtimes P_7 $ subgraph, then $ c(G) \geq 11 $. Moreover, if $ c(G) = 11 $, then all colors $ 1, 2, \ldots, 39 $ must appear in $ G $.
\end{enumerate}
\end{observation}

\begin{proof}
(i) Since the distance between any two squares of color $i$ must be at least 2, and the distance between any two squares in $ P_2 \boxtimes P_2 $ is equal to $1$, we have $c(G) \leq 1$.

(ii) We partition $ G $ into four parts, as shown in Figure~\ref{fig:1}. By Observation~\ref{obs:2}\,(i), we have $ c(G) \leq 4 $. Moreover, if $ c(G) = 4 $, it follows that each part must contain exactly one square of color 1. Since the first part, column 1, contains a square of color 1, the second part must have a square of color 1 in column 3. Similarly, the third and fourth parts must have squares of color 1 in columns 5 and 7, respectively.

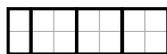
\begin{figure}[ht]
  \centering
  \begin{tikzpicture}[scale=0.30]
\foreach \x in {0,...,7} {
  \draw[gray!70, line width=0.35pt] (\x,0) -- (\x,2);
}

\foreach \y in {0,...,2} {
  \draw[gray!70, line width=0.35pt] (0,\y) -- (7,\y);
}

\draw[very thick] (0,0) rectangle (1,2);
\draw[very thick] (1,0) rectangle (3,2);
\draw[very thick] (3,0) rectangle (5,2);
\draw[very thick] (5,0) rectangle (7,2);
\end{tikzpicture}
  \caption{Partition of $ G $ in the proof of Observation~\ref{obs:2}\,(ii).}
  \label{fig:1}
\end{figure}

(iii) Since the distance between any two squares of color $i$ must be at least 7, and the distance between any two squares in $ P_7 \boxtimes P_7 $ is at most $6$, there is at most one square of color $i$ in $ P_7 \boxtimes P_7 $.

(iv) By Observation~\ref{obs:2}\,(iii), $ G $ has at most 38 squares that are not of color 1. Since $ G $ consists of 49 squares, it follows that $ c(G) \geq 49 - 38 = 11 $. Moreover, if $ c(G) = 11 $, then all colors $ 1, 2, \ldots, 39 $ must appear in $ G $.
\end{proof}

Let $ G $ be a subgraph of $ P_\infty \boxtimes P_\infty $. We write $ G(i,j) $ for the subgraph viewed relative to $ G $, where $ i $ denotes the number of columns (to the right if positive, to the left if negative) and $ j $ denotes the number of rows (upward if positive, downward if negative). For example, $ G(-1,2) $ refers to the $ P_7 \boxtimes P_7 $ subgraph viewed one column to the left and two rows above $ G $.

We refer to a $ P_7 \boxtimes P_7 $ subgraph $ G $ with $ c(G) = 12 $ and $ c(G(1,0)) = c(G(-1,0)) = c(G(0,1)) = c(G(0,-1)) = 11 $ as a \emph{critical subgraph}.

\begin{observation} \label{obs:3}
Let $ G $ be a $ P_7 \boxtimes P_7 $ subgraph with $ c(G) = 12 $.
\begin{enumerate}[label=(\roman*)]
    \item If $ c(G(1, 0)) = 11 $, then there exists a unique color $ k \in \{2, 3, \ldots, 39\} $ which does not appear in $ G $. However, color $k$ appears in column 7 of $ G(1, 0) $.
    \item If $G$ is a critical subgraph, then there exists a unique color $ k \in \{2, 3, \ldots, 39\} $ which does not appear in $ G $. However, color $k$ appears in each of $ c(G(1, 0)) $, $ c(G(-1, 0)) $, $ c(G(0, 1)) $, and $ c(G(0, -1)) $. Moreover, there are only two possible positions in which color $ k $ can appear in each of these subgraphs, referred to as type-A and type-B positions (see Figure~\ref{fig:2}).
    \item If $G$ is a critical subgraph, there exists a color $m \in \{2,3,\ldots,39\}$ such that square $(7,1)$ of $G$ is of color $m$ and color $k$ appears as type-A, then square $(7,7)$ of $G(0,1)$ must also be of color $m$; while if color $k$ appears as type-B, then square $(1,1)$ of $G(-1,0)$ must be of color $m$ (see Figure~\ref{fig:2}).

\end{enumerate}
\end{observation}

\begin{figure}[ht]
  \centering
  \begin{tikzpicture}[scale=0.30]
  
\foreach \x in {0,...,9} {
  \draw[gray!60, line width=0.35pt] (\x,0) -- (\x,9);
}
\foreach \y in {0,...,9} {
  \draw[gray!60, line width=0.35pt] (0,\y) -- (9,\y);
}

\draw[very thick] (0,1) rectangle (1,8);
\draw[very thick] (8,1) rectangle (9,8);
\draw[very thick] (1,8) rectangle (8,9);
\draw[very thick] (1,0) rectangle (8,1);

\draw[thick] (1,1) rectangle (8,8);
\node at (4.5,4.5) {\LARGE $ G $};

\node at (7.5,0.5) {\normalsize $ k $};
\node at (8.5,7.5) {\normalsize $ k $};
\node at (1.5,8.5) {\normalsize $ k $};
\node at (0.5,1.5) {\normalsize $ k $};

\node at (7.5,1.5) {\textcolor{blue}{$m$}};
\node at (7.5,8.5) {\textcolor{blue}{$m$}};

\node at (4.5,-1.2) {\large Type-A};

\foreach \x in {0,...,9} {
  \draw[gray!60, line width=0.2pt] (\x+11,0) -- (\x+11,9);
}
\foreach \y in {0,...,9} {
  \draw[gray!60, line width=0.2pt] (11,\y) -- (20,\y);
}

\draw[very thick] (11,1) rectangle (12,8);
\draw[very thick] (19,1) rectangle (20,8);
\draw[very thick] (12,8) rectangle (19,9);
\draw[very thick] (12,0) rectangle (19,1);

\draw[thick] (12,1) rectangle (19,8);
\node at (15.5,4.5) {\LARGE $ G $};

\node at (12.5,0.5) {\normalsize $ k $};
\node at (11.5,7.5) {\normalsize $ k $};
\node at (18.5,8.5) {\normalsize $ k $};
\node at (19.5,1.5) {\normalsize $ k $};

\node at (18.5,1.5) {\textcolor{blue}{$m$}};
\node at (11.5,1.5) {\textcolor{blue}{$m$}};

\node at (15.5,-1.2) {\large Type-B};
\end{tikzpicture}
  \caption{Two possible positions of color $ k $ in each shifted subgraph of $ G $, referred to as type-A and type-B positions, as described in Observation~\ref{obs:3}\,(ii).}
  \label{fig:2}
\end{figure}
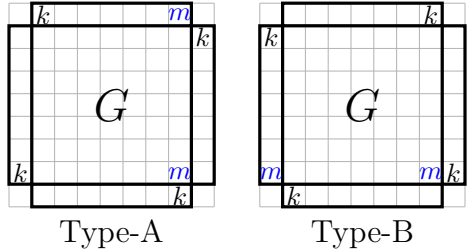

\begin{proof}
(i) Since $ c(G) = 12 $, there are $ 49 - 12 = 37 $ squares that are not of color 1. It follows that there exists a unique color $ k \in \{2, 3, \ldots, 39\} $ which does not appear in $ G $. By Observation~\ref{obs:2} (iv), all colors $ 1, 2, \ldots, 39 $ must appear in $ G(1, 0) $. Therefore, color $ k $ must be in column 7, since columns 1 to 6 of $ G(1, 0) $ are also in $ G $.

(ii) By Observation~\ref{obs:3}\,(i), we obtain that there exists a unique color $ k \in \{2, 3, \ldots, 39\} $ which does not appear in $ G $. However, color $ k $ appears in each of $ c(G(1, 0)) $, $ c(G(-1, 0)) $, $ c(G(0, 1)) $, and $ c(G(0, -1)) $. Assume that color $ k $ does not lie in $ (7,1) $ or $ (7,7) $ of $ G(1, 0) $. Then the possible positions of color $ k $ in $ G(1, 0) $ are $ (7,2), (7,3), (7,4), (7,5), $ and $ (7,6) $. Due to the distance constraint, the square of color $ k $ in $ G(0,1) $ and $ G(0, -1) $ must be $ (1,7) $ and $ (1,1) $, respectively. It follows that column 1 of $ G(-1, 0) $ is too close to those two squares to have color $ k $, which is a contradiction. Hence, the only two possible positions of $ k $ in $ G(1, 0) $ are $ (7,1) $ and $ (7,7) $.

Next, we assume that color $ k $ lies in $ (7, 7) $ of $ G(1, 0) $. It follows that $ (1, 7) $ of $ G(0, 1) $ must be assigned color $ k $. Then, color $ k $ must lie in $ (1, 1) $ of $ G(-1, 0) $. Consequently, $ (7, 1) $ of $ G(0, -1) $ must also be a square of color $ k $. This position of color $ k $ is what we refer to as type-A. Similarly, if color $ k $ lies in $ (7, 1) $ of $ G(1, 0) $, then the position of $ k $ corresponds to type-B (see Figure~\ref{fig:2}).

(iii) Without loss of generality, assume that color $k$ appears as type-A. By Observation~\ref{obs:2}\,(iv) and the distance constraint, color $m$ must appear both in column~1 of $G(-1,0)$ and in row~7 of $G(0,1)$. Since square $(1,1)$ of $G(-1,0)$ is already assigned color~$k$, one of the squares $(1,2), (1,3), (1,4), (1,5), (1,6)$, or $(1,7)$ must be assigned color~$m$. By the distance constraint, square $(7,7)$ of $G(0,1)$ must be assigned color $m$.
\end{proof}

\section{Lemma}\label{sec:3}
In this section, we introduce the main Lemma, Lemma~\ref{lem:4}. The proof of this lemma will be given in Section~\ref{sec:5}.

In the proof of Theorem \ref{thm:1}, we use only Lemma \ref{lem:4}\,(ii), (viii), and (ix), while the other parts of the lemma are used to support the proof of these parts.

We represent a pattern of squares of color 1 using a matrix, where each entry corresponds to a square in the grid. An entry of $ 1 $ indicates a square of color 1, and $ x $ denotes a square that may be assigned any color. An \emph{L-shaped subgraph} is a subgraph in Figure~\ref{fig:28} or its reflections or rotations.

\begin{figure}[ht]
  \centering
  \begin{tikzpicture}[scale=0.3]

  \foreach \x in {0,...,3} {
    \draw[gray!50, thin] (\x,0) -- (\x,3);
  }
  \foreach \y in {0,...,3} {
    \draw[gray!50, thin] (0,\y) -- (3,\y);
  }

  \foreach \pos in {(0,0),(1,0),(2,0),(0,1),(0,2)} {
    \draw[very thick] \pos rectangle ++(1,1);
  }

  \end{tikzpicture}
  \caption{L-shaped subgraph.}
  \label{fig:28}
\end{figure}
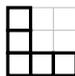

\begin{lemma}\label{lem:4}
The following patterns cannot occur in an $ S $-packing coloring of $ P_\infty \boxtimes P_\infty $ that uses at most 39 colors.
\begin{multicols}{3} 
\begin{enumerate}[label=(\roman*)]
    \item 
    \[
    \begin{bmatrix}
    1 & x & 1 & x & 1
    \end{bmatrix}
    \]
    Consequently, each $ P_5 \boxtimes P_1 $ subgraph contains at most two squares of color~1.
    \item 
    \[
    \begin{bmatrix}
    1 & x & x & 1
    \end{bmatrix}
    \]
    \item 
    \[
    \begin{bmatrix}
    1 & x & 1\\
    x & x & x\\
    1 & x & 1
    \end{bmatrix}
    \]
    \item
    \[
    \begin{bmatrix}
    x & x & x & 1\\
    1 & x & x & x\\
    x & x & x & x\\
    1 & x & 1 & x
    \end{bmatrix}
    \]
    \item 
    \[
    \begin{bmatrix}
    1 & x & x\\
    x & x & x\\
    1 & x & 1
    \end{bmatrix}
    \]
    Consequently, each L-shaped subgraph contains at most two squares of color 1.
    \item 
     \[
    \begin{bmatrix}
    1 & x & 1 & x\\
    x & x & x & x\\
    x & 1 & x & 1
    \end{bmatrix}
    \]
    \item 
    \[
    \begin{bmatrix}
    x & 1 & x\\
    x & x & x\\
    1 & x & 1
    \end{bmatrix}
    \]
    \item 
    \[
    \begin{bmatrix}
    1 & x & 1
    \end{bmatrix}
    \]
    Consequently, each $ P_3 \boxtimes P_1 $ subgraph contains at most one square of color~1.
    \item 
    \[
    \begin{bmatrix}
    x & x & 1\\
    1 & x & x
    \end{bmatrix}
    \]
\end{enumerate}
\end{multicols} 
\end{lemma}

\section{Proof of Theorem~\ref{thm:1}}\label{sec:4}
In this section, we prove Theorem~\ref{thm:1} using Observation \ref{obs:2} and Lemma~\ref{lem:4}.

\begin{proof}
Let $ S = (1, 6, 6, 6, \ldots) $. Suppose, for contradiction, that $ \chi_S(P_\infty \boxtimes P_\infty) \leq 39 $. That is, there exists an $ S $-packing coloring of $ P_\infty \boxtimes P_\infty $ that uses at most 39 colors. By Observation~\ref{obs:2}\,(iv), color 1 appears in the coloring. Let $ G $ be a $ P_7 \boxtimes P_7 $ subgraph in which color 1 is assigned to square $(4,4)$. 

By the distance constraint and Lemma \ref{lem:4}\,(ii), (viii) and (ix), the squares that cannot be assigned color~1 are shown in Figure~\ref{fig:33}. We partition the remaining squares into eight parts as shown in Figure~\ref{fig:33}. By Observation \ref{obs:2}\,(i) and Lemma \ref{lem:4}\,(xiii), each part contains at most one square of color 1. We have that $c(G) \leq 1+8 = 9$ which contradicts Observation \ref{obs:2}\,(iv).
\end{proof}
    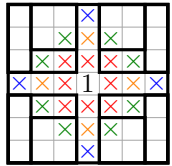
\begin{figure}[ht]
\centering
\begin{tikzpicture}[scale=0.30]

\begin{scope}[shift={(0,0)}]
  \foreach \x in {0,...,7} {
    \draw[gray!70, line width=0.35pt] (\x,0) -- (\x,7);
  }
  \foreach \y in {0,...,7} {
    \draw[gray!70, line width=0.35pt] (0,\y) -- (7,\y);
  }
  \draw[very thick] (0,0) rectangle (7,7);
  \draw[very thick] (0,0) rectangle (3,3);
  \draw[very thick] (0,4) rectangle (3,7);
  \draw[very thick] (4,0) rectangle (7,3);
  \draw[very thick] (4,4) rectangle (7,7);

  \draw[very thick] (0,0) rectangle (1,3);
  \draw[very thick] (0,4) rectangle (1,7);
  \draw[very thick] (6,0) rectangle (7,3);
  \draw[very thick] (6,4) rectangle (7,7);
  \draw[very thick] (1,0) rectangle (3,2);
  \draw[very thick] (1,5) rectangle (3,7);
  \draw[very thick] (4,0) rectangle (6,2);
  \draw[very thick] (4,5) rectangle (6,7);
  \node at (3.5,3.5) {\normalsize 1};

\node at (2.5,2.5) {\textcolor{red}{$\times$}};
\node at (3.5,2.5) {\textcolor{red}{$\times$}};
\node at (4.5,2.5) {\textcolor{red}{$\times$}};

\node at (2.5,3.5) {\textcolor{red}{$\times$}};
\node at (4.5,3.5) {\textcolor{red}{$\times$}};

\node at (2.5,4.5) {\textcolor{red}{$\times$}};
\node at (3.5,4.5) {\textcolor{red}{$\times$}};
\node at (4.5,4.5) {\textcolor{red}{$\times$}};

\node at (3.5,0.5) {\textcolor{blue}{$\times$}};
\node at (3.5,6.5) {\textcolor{blue}{$\times$}};
\node at (0.5,3.5) {\textcolor{blue}{$\times$}};
\node at (6.5,3.5) {\textcolor{blue}{$\times$}};

\node at (1.5,2.5) {\textcolor{darkgreen}{$\times$}};
\node at (1.5,4.5) {\textcolor{darkgreen}{$\times$}};
\node at (5.5,2.5) {\textcolor{darkgreen}{$\times$}};
\node at (5.5,4.5) {\textcolor{darkgreen}{$\times$}};
\node at (2.5,1.5) {\textcolor{darkgreen}{$\times$}};
\node at (4.5,1.5) {\textcolor{darkgreen}{$\times$}};
\node at (2.5,5.5) {\textcolor{darkgreen}{$\times$}};
\node at (4.5,5.5) {\textcolor{darkgreen}{$\times$}};

\node at (3.5,1.5) {\textcolor{orange}{$\times$}};
\node at (3.5,5.5) {\textcolor{orange}{$\times$}};
\node at (1.5,3.5) {\textcolor{orange}{$\times$}};
\node at (5.5,3.5) {\textcolor{orange}{$\times$}};
\end{scope}

\end{tikzpicture}

\caption{Positions of squares of color 1 and squares that cannot be assigned color 1 (marked with $ \times $'s) in $ G $. The colors indicate the order in which the squares that cannot be assigned color 1 are identified in the proof: red (first), blue (second), green (third), and orange (fourth).}
\label{fig:33}
\end{figure}

\section{Proof of Lemma \ref{lem:4}}\label{sec:5}

In this section, we prove Lemma \ref{lem:4}. The main difficulty lies in part (iii).

\subsection{Proof of Lemma \ref{lem:4}\,(i)}
\begin{proof}
Suppose for contradiction that the pattern 
\[
    \begin{bmatrix}
    1 & x & 1 & x & 1
    \end{bmatrix}
    \]
appears in the coloring. Let $ G $ be a $ P_7 \boxtimes P_7 $ subgraph that contains this pattern centered in row 2 (see Figure \ref{fig:3}). We will show that squares of color 1 must appear as shown in Figure \ref{fig:3}. Observe that rows 1, 2, and 3 cannot contain any additional squares of color 1 otherwise there are two squares of color 1 within distance 1. By Observation~\ref{obs:2}\,(ii) and Observation~\ref{obs:2}\,(iv), the combined rows 4 and 5 must contain exactly four squares of color 1, as must rows 6 and 7. Moreover, these squares must lie in columns 1, 3, 5, and 7. That is, both columns 1 and 7 of $ G $ contain a square of color 1 in either row 4 or 5, and a square of color 1 in either row 6 or 7.

Next, we consider $G(-1,0)$. By Observation~\ref{obs:2}\,(iv), column 1 must contain at least two squares of color 1. Due to the distance constraint, these squares cannot lie in rows 4, 5, 6, or 7, and thus must be located in rows 1 and 3. It follows that the square $ (1,4) $ in $ G $ cannot be assigned color 1. Consequently, the squares $ (1,5) $ and $ (1,7) $ in $ G $ must be of color 1. Similarly, by considering $G(1,0)$, the squares $ (7,1) $ and $ (7,3) $ in $G(1,0)$ must be of color 1. It follows that the squares $ (7,5) $ and $ (7,7) $ in $ G $ must also be of color 1.

We now consider $G(0, -4)$. Due to the symmetry between $ G $ and $ G(0,-4) $, rows 1 and 2 together must contain four squares of color 1, and the same holds for rows 3 and 4. Moreover, the squares $ (1,1), (1,3), (7,1) $, and $ (7,3) $ are of color 1 (see Figure~\ref{fig:3}).

Consider $ G(0, -1) $. By Observation~\ref{obs:2} (iv), the squares at $ (3,1) $, $ (3,7) $, $ (5,1) $, and $ (5,7) $ must be of color 1. Likewise, by considering $ G(0, -3) $, it follows that squares of color 1 must also be located at $ (3,1) $, $ (3,7) $, $ (5,1) $, and $ (5,7) $ (see Figure~\ref{fig:3}).

\begin{figure}[ht]
\centering
\begin{tikzpicture}[scale=0.30]

\begin{scope}[shift={(-33,0)}]
  \foreach \x in {0,...,9} {
    \draw[gray!70, line width=0.35pt] (\x,0) -- (\x,11);
  }
  \foreach \y in {0,...,11} {
    \draw[gray!70, line width=0.35pt] (0,\y) -- (9,\y);
  }

  \draw[very thick] (1,4) rectangle (8,11);

  \node at (2.5,5.5) {\normalsize 1};
  \node at (4.5,5.5) {\normalsize 1};
  \node at (6.5,5.5) {\normalsize 1};
  \node at (1.5,6.5) {\textcolor{red}{$\times$}};
  \node at (2.5,6.5) {\textcolor{red}{$\times$}};
  \node at (3.5,6.5) {\textcolor{red}{$\times$}};
  \node at (4.5,6.5) {\textcolor{red}{$\times$}};
  \node at (5.5,6.5) {\textcolor{red}{$\times$}};
  \node at (6.5,6.5) {\textcolor{red}{$\times$}};
  \node at (7.5,6.5) {\textcolor{red}{$\times$}};
  \node at (1.5,5.5) {\textcolor{red}{$\times$}};
  \node at (3.5,5.5) {\textcolor{red}{$\times$}};
  \node at (5.5,5.5) {\textcolor{red}{$\times$}};
  \node at (7.5,5.5) {\textcolor{red}{$\times$}};
  \node at (1.5,4.5) {\textcolor{red}{$\times$}};
  \node at (2.5,4.5) {\textcolor{red}{$\times$}};
  \node at (3.5,4.5) {\textcolor{red}{$\times$}};
  \node at (4.5,4.5) {\textcolor{red}{$\times$}};
  \node at (5.5,4.5) {\textcolor{red}{$\times$}};
  \node at (6.5,4.5) {\textcolor{red}{$\times$}};
  \node at (7.5,4.5) {\textcolor{red}{$\times$}};
  
\end{scope}

\begin{scope}[shift={(-22,0)}]
  \foreach \x in {0,...,9} {
    \draw[gray!70, line width=0.35pt] (\x,0) -- (\x,11);
  }
  \foreach \y in {0,...,11} {
    \draw[gray!70, line width=0.35pt] (0,\y) -- (9,\y);
  }

  \draw[very thick] (1,4) rectangle (8,11);

  \node at (2.5,5.5) {\normalsize 1};
  \node at (4.5,5.5) {\normalsize 1};
  \node at (6.5,5.5) {\normalsize 1};
  \node at (1.5,6.5) {\textcolor{red}{$\times$}};
  \node at (2.5,6.5) {\textcolor{red}{$\times$}};
  \node at (3.5,6.5) {\textcolor{red}{$\times$}};
  \node at (4.5,6.5) {\textcolor{red}{$\times$}};
  \node at (5.5,6.5) {\textcolor{red}{$\times$}};
  \node at (6.5,6.5) {\textcolor{red}{$\times$}};
  \node at (7.5,6.5) {\textcolor{red}{$\times$}};
  \node at (1.5,5.5) {\textcolor{red}{$\times$}};
  \node at (3.5,5.5) {\textcolor{red}{$\times$}};
  \node at (5.5,5.5) {\textcolor{red}{$\times$}};
  \node at (7.5,5.5) {\textcolor{red}{$\times$}};
  \node at (1.5,4.5) {\textcolor{red}{$\times$}};
  \node at (2.5,4.5) {\textcolor{red}{$\times$}};
  \node at (3.5,4.5) {\textcolor{red}{$\times$}};
  \node at (4.5,4.5) {\textcolor{red}{$\times$}};
  \node at (5.5,4.5) {\textcolor{red}{$\times$}};
  \node at (6.5,4.5) {\textcolor{red}{$\times$}};
  \node at (7.5,4.5) {\textcolor{red}{$\times$}};

  \node at (0.5,7.5) {\textcolor{blue}{$\times$}};
  \node at (0.5,8.5) {\textcolor{blue}{$\times$}};
  \node at (0.5,9.5) {\textcolor{blue}{$\times$}};
  \node at (0.5,10.5) {\textcolor{blue}{$\times$}};
  \node at (8.5,7.5) {\textcolor{blue}{$\times$}};
  \node at (8.5,8.5) {\textcolor{blue}{$\times$}};
  \node at (8.5,9.5) {\textcolor{blue}{$\times$}};
  \node at (8.5,10.5) {\textcolor{blue}{$\times$}};
  
\end{scope}

\begin{scope}[shift={(-11,0)}]
  \foreach \x in {0,...,9} {
    \draw[gray!70, line width=0.35pt] (\x,0) -- (\x,11);
  }
  \foreach \y in {0,...,11} {
    \draw[gray!70, line width=0.35pt] (0,\y) -- (9,\y);
  }

  \draw[very thick] (1,4) rectangle (8,11);
  \draw[very thick, dashed] (1,0) rectangle (8,7);

  \node at (2.5,5.5) {\normalsize 1};
  \node at (4.5,5.5) {\normalsize 1};
  \node at (6.5,5.5) {\normalsize 1};
  \node at (0.5,6.5) {\textcolor{blue} 1};
  \node at (0.5,4.5) {\textcolor{blue} 1};
  \node at (8.5,6.5) {\textcolor{blue} 1};
  \node at (8.5,4.5) {\textcolor{blue} 1};
  \node at (1.5,8.5) {\textcolor{darkgreen} 1};
  \node at (7.5,8.5) {\textcolor{darkgreen} 1};
  \node at (1.5,10.5) {\textcolor{darkgreen} 1};
  \node at (7.5,10.5) {\textcolor{darkgreen} 1};
  \node at (1.5,2.5) {\textcolor{darkgreen} 1};
  \node at (7.5,2.5) {\textcolor{darkgreen} 1};
  \node at (1.5,0.5) {\textcolor{darkgreen} 1};
  \node at (7.5,0.5) {\textcolor{darkgreen} 1};
\end{scope}

\begin{scope}[shift={(0,0)}]
  \foreach \x in {0,...,9} {
    \draw[gray!70, line width=0.35pt] (\x,0) -- (\x,11);
  }
  \foreach \y in {0,...,11} {
    \draw[gray!70, line width=0.35pt] (0,\y) -- (9,\y);
  }

  \draw[very thick] (1,4) rectangle (8,11);
  \draw[very thick, dashed] (1,0) rectangle (8,7);

  \node at (2.5,5.5) {\normalsize 1};
  \node at (4.5,5.5) {\normalsize 1};
  \node at (6.5,5.5) {\normalsize 1};
  \node at (0.5,6.5) {\textcolor{blue} 1};
  \node at (0.5,4.5) {\textcolor{blue} 1};
  \node at (8.5,6.5) {\textcolor{blue} 1};
  \node at (8.5,4.5) {\textcolor{blue} 1};

  \node at (3.5,7.5) {\textcolor{orange} 1};
  \node at (5.5,7.5) {\textcolor{orange} 1};
  \node at (1.5,8.5) {\textcolor{darkgreen} 1};
  \node at (7.5,8.5) {\textcolor{darkgreen} 1};
  \node at (3.5,9.5) {\textcolor{orange} 1};
  \node at (5.5,9.5) {\textcolor{orange} 1};
  \node at (1.5,10.5) {\textcolor{darkgreen} 1};
  \node at (7.5,10.5) {\textcolor{darkgreen} 1};

  \node at (3.5,3.5) {\textcolor{orange} 1};
  \node at (5.5,3.5) {\textcolor{orange} 1};
  \node at (1.5,2.5) {\textcolor{darkgreen} 1};
  \node at (7.5,2.5) {\textcolor{darkgreen} 1};
  \node at (3.5,1.5) {\textcolor{orange} 1};
  \node at (5.5,1.5) {\textcolor{orange} 1};
  \node at (1.5,0.5) {\textcolor{darkgreen} 1};
  \node at (7.5,0.5) {\textcolor{darkgreen} 1};
\end{scope}

\end{tikzpicture}
\caption{Positions of squares of color 1 and squares that cannot be assigned color 1 (marked with red $ \times $'s) in $ G $ (black border) and $ G(0, -4) $ (dashed border). The colors indicate the order in which the squares of color 1 are identified in the proof: black (first), blue (second), green (third), and orange (fourth).}
\label{fig:3}
\end{figure}
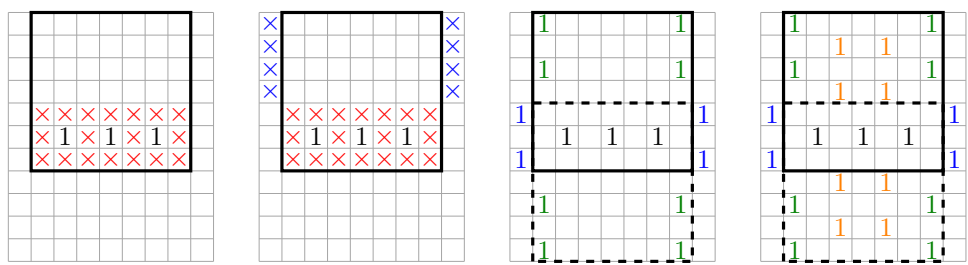

Next, we will show that $ G(0,1) $ is a critical subgraph. Consider $ G(-1,2) $ (See Figure \ref{fig:4}). We observe that rows 1 to 5 cannot contain any additional square of color 1. By Observation~\ref{obs:2}\,(ii) and Observation~\ref{obs:2}\,(iv), the combined rows 6 and 7 of this subgraph must contain exactly four squares of color 1, located in columns 1, 3, 5, and 7. Due to the distance constraint, the squares $ (1,7) $, $ (3,7) $, and $ (7,7) $ in this subgraph must be squares of color 1. Moreover, either $ (5,6) $ or $ (5,7) $ must also be a square of color 1. Since $c(G(-1,1)) \geq 11$ by Observation \ref{obs:2}\,(iv), $(5,6)$ is assigned color 1. Similarly, by considering $ G(1, 2) $, $ G(-1, 6) $ and $ G(1, 6) $, the squares of color 1 must appear as shown in Figure \ref{fig:4}. Since $c(G(0,1)) = 12$ and $c(G(1,1)) = c(G(-1,1)) = c(G(0,2)) = c(G) = 11$, we have that $ G(0,1) $ is a critical subgraph.

By Observation~\ref{obs:3} (ii), let $ k \in \{2, 3, \ldots, 39\} $ be a color that does not appear in $ G(0,1) $ and there are only two possible positions for color $ k $, referred to as type-A and type-B.

Without loss of generality, suppose $ k $ appears in a type-A position. Consider $ G(-1, -6) $. By Observation~\ref{obs:2}\,(iii), since this subgraph contains 11 squares of color 1, it must include all colors $ 1, 2, \ldots, 39 $. Due to the distance constraint, the only possible position for color $ k $ is $ (1,1) $, but this square is already assigned color 1, which is a contradiction (see Figure~\ref{fig:4}).

Moreover, we can show that each $P_5 \boxtimes P_1$ contains at most two squares of color 1. Assume for contradiction that there exists a $P_5 \boxtimes P_1$ subgraph that contains three squares of color 1. Then the subgraph has the pattern  
\[
\begin{bmatrix}
1 & x & 1 & x & 1
\end{bmatrix}
\]  
which is a contradiction.
\end{proof}

\begin{figure}[ht]
\centering
\begin{tikzpicture}[scale=0.30]

\begin{scope}[shift={(-22,0)}]
  \foreach \x in {0,...,9} {
    \draw[gray!70, line width=0.35pt] (\x,0) -- (\x,15);
  }
  \foreach \y in {0,...,15} {
    \draw[gray!70, line width=0.35pt] (0,\y) -- (9,\y);
  }
  \draw[very thick] (1,6) rectangle (8,13);
  \draw[very thick, dashed] (0,8) rectangle (7,15);

  \node at (2.5,7.5) {\normalsize 1};
  \node at (4.5,7.5) {\normalsize 1};
  \node at (6.5,7.5) {\normalsize 1};
  \node at (0.5,8.5) {\normalsize 1};
  \node at (0.5,6.5) {\normalsize 1};
  \node at (8.5,8.5) {\normalsize 1};
  \node at (8.5,6.5) {\normalsize 1};

  \node at (3.5,9.5) {\normalsize 1};
  \node at (5.5,9.5) {\normalsize 1};
  \node at (1.5,10.5) {\normalsize 1};
  \node at (7.5,10.5) {\normalsize 1};
  \node at (3.5,11.5) {\normalsize 1};
  \node at (5.5,11.5) {\normalsize 1};
  \node at (1.5,12.5) {\normalsize 1};
  \node at (7.5,12.5) {\normalsize 1};

  \node at (3.5,5.5) {\normalsize 1};
  \node at (5.5,5.5) {\normalsize 1};
  \node at (1.5,4.5) {\normalsize 1};
  \node at (7.5,4.5) {\normalsize 1};
  \node at (3.5,3.5) {\normalsize 1};
  \node at (5.5,3.5) {\normalsize 1};
  \node at (1.5,2.5) {\normalsize 1};
  \node at (7.5,2.5) {\normalsize 1};
  \node at (1.5,8.5) {\textcolor{red}{$\times$}};
  \node at (2.5,8.5) {\textcolor{red}{$\times$}};
  \node at (3.5,8.5) {\textcolor{red}{$\times$}};
  \node at (4.5,8.5) {\textcolor{red}{$\times$}};
  \node at (5.5,8.5) {\textcolor{red}{$\times$}};
  \node at (6.5,8.5) {\textcolor{red}{$\times$}};

  \node at (0.5,9.5) {\textcolor{red}{$\times$}};
  \node at (1.5,9.5) {\textcolor{red}{$\times$}};
  \node at (2.5,9.5) {\textcolor{red}{$\times$}};
  \node at (4.5,9.5) {\textcolor{red}{$\times$}};
  \node at (6.5,9.5) {\textcolor{red}{$\times$}};

  \node at (0.5,10.5) {\textcolor{red}{$\times$}};
  \node at (2.5,10.5) {\textcolor{red}{$\times$}};
  \node at (3.5,10.5) {\textcolor{red}{$\times$}};
  \node at (4.5,10.5) {\textcolor{red}{$\times$}};
  \node at (5.5,10.5) {\textcolor{red}{$\times$}};
  \node at (6.5,10.5) {\textcolor{red}{$\times$}};

  \node at (0.5,11.5) {\textcolor{red}{$\times$}};
  \node at (1.5,11.5) {\textcolor{red}{$\times$}};
  \node at (2.5,11.5) {\textcolor{red}{$\times$}};
  \node at (4.5,11.5) {\textcolor{red}{$\times$}};
  \node at (6.5,11.5) {\textcolor{red}{$\times$}};

  \node at (0.5,12.5) {\textcolor{red}{$\times$}};
  \node at (2.5,12.5) {\textcolor{red}{$\times$}};
  \node at (3.5,12.5) {\textcolor{red}{$\times$}};
  \node at (4.5,12.5) {\textcolor{red}{$\times$}};
  \node at (5.5,12.5) {\textcolor{red}{$\times$}};
  \node at (6.5,12.5) {\textcolor{red}{$\times$}};

  \node at (0.5,13.5) {\textcolor{red}{$\times$}};
  \node at (1.5,13.5) {\textcolor{red}{$\times$}};
  \node at (2.5,13.5) {\textcolor{red}{$\times$}};
  \node at (6.5,13.5) {\textcolor{red}{$\times$}};

\end{scope}

\begin{scope}[shift={(-11,0)}]
  \foreach \x in {0,...,9} {
    \draw[gray!70, line width=0.35pt] (\x,0) -- (\x,15);
  }
  \foreach \y in {0,...,15} {
    \draw[gray!70, line width=0.35pt] (0,\y) -- (9,\y);
  }
  \draw[very thick] (1,6) rectangle (8,13);
  \draw[very thick, dashed] (0,8) rectangle (7,15);

  \node at (2.5,7.5) {\normalsize 1};
  \node at (4.5,7.5) {\normalsize 1};
  \node at (6.5,7.5) {\normalsize 1};
  \node at (0.5,8.5) {\normalsize 1};
  \node at (0.5,6.5) {\normalsize 1};
  \node at (8.5,8.5) {\normalsize 1};
  \node at (8.5,6.5) {\normalsize 1};

  \node at (3.5,9.5) {\normalsize 1};
  \node at (5.5,9.5) {\normalsize 1};
  \node at (1.5,10.5) {\normalsize 1};
  \node at (7.5,10.5) {\normalsize 1};
  \node at (3.5,11.5) {\normalsize 1};
  \node at (5.5,11.5) {\normalsize 1};
  \node at (1.5,12.5) {\normalsize 1};
  \node at (7.5,12.5) {\normalsize 1};

  \node at (3.5,5.5) {\normalsize 1};
  \node at (5.5,5.5) {\normalsize 1};
  \node at (1.5,4.5) {\normalsize 1};
  \node at (7.5,4.5) {\normalsize 1};
  \node at (3.5,3.5) {\normalsize 1};
  \node at (5.5,3.5) {\normalsize 1};
  \node at (1.5,2.5) {\normalsize 1};
  \node at (7.5,2.5) {\normalsize 1};

  \node at (4.5,13.5) {\textcolor{darkgreen} 1};
  \node at (2.5,14.5) {\textcolor{blue} 1};
  \node at (0.5,14.5) {\textcolor{blue} 1};
  \node at (6.5,14.5) {\textcolor{blue} 1};
  \node at (8.5,14.5) {\textcolor{blue} 1};

  \node at (4.5,1.5) {\textcolor{darkgreen} 1};
  \node at (2.5,0.5) {\textcolor{blue} 1};
  \node at (0.5,0.5) {\textcolor{blue} 1};
  \node at (6.5,0.5) {\textcolor{blue} 1};
  \node at (8.5,0.5) {\textcolor{blue} 1};
\end{scope}

\begin{scope}[shift={(0,0)}]
  \foreach \x in {0,...,9} {
    \draw[gray!70, line width=0.35pt] (\x,0) -- (\x,15);
  }
  \foreach \y in {0,...,15} {
    \draw[gray!70, line width=0.35pt] (0,\y) -- (9,\y);
  }

  \draw[very thick, red] (1,7) rectangle (8,14);
  \draw[very thick, blue] (0,0) rectangle (7,7);

  \node at (2.5,7.5) {\normalsize 1};
  \node at (4.5,7.5) {\normalsize 1};
  \node at (6.5,7.5) {\normalsize 1};
  \node at (0.5,8.5) {\normalsize 1};
  \node at (0.5,6.5) {\normalsize 1};
  \node at (8.5,8.5) {\normalsize 1};
  \node at (8.5,6.5) {\normalsize 1};

  \node at (3.5,9.5) {\normalsize 1};
  \node at (5.5,9.5) {\normalsize 1};
  \node at (1.5,10.5) {\normalsize 1};
  \node at (7.5,10.5) {\normalsize 1};
  \node at (3.5,11.5) {\normalsize 1};
  \node at (5.5,11.5) {\normalsize 1};
  \node at (1.5,12.5) {\normalsize 1};
  \node at (7.5,12.5) {\normalsize 1};

  \node at (3.5,5.5) {\normalsize 1};
  \node at (5.5,5.5) {\normalsize 1};
  \node at (1.5,4.5) {\normalsize 1};
  \node at (7.5,4.5) {\normalsize 1};
  \node at (3.5,3.5) {\normalsize 1};
  \node at (5.5,3.5) {\normalsize 1};
  \node at (1.5,2.5) {\normalsize 1};
  \node at (7.5,2.5) {\normalsize 1};

  \node at (4.5,13.5) {\textcolor{darkgreen} 1};
  \node at (2.5,14.5) {\textcolor{blue} 1};
  \node at (0.5,14.5) {\textcolor{blue} 1};
  \node at (6.5,14.5) {\textcolor{blue} 1};
  \node at (8.5,14.5) {\textcolor{blue} 1};

  \node at (4.5,1.5) {\textcolor{darkgreen} 1};
  \node at (2.5,0.5) {\textcolor{blue} 1};
  \node at (0.5,0.5) {\textcolor{blue} 1};
  \node at (6.5,0.5) {\textcolor{blue} 1};
  \node at (8.5,0.5) {\textcolor{blue} 1};

  \node at (1.5,14.5) {\textcolor{red}{$k$}};
  \node at (0.5,7.5)  {\textcolor{red}{$k$}};
  \node at (7.5,6.5)  {\textcolor{red}{$k$}};
  \node at (8.5,13.5) {\textcolor{red}{$k$}};
  \fill[yellow, opacity=0.3] (0,0) rectangle (1,1);
\end{scope}

\end{tikzpicture}
\caption{$G$ (black), $G(-1,2)$ (dashed), $G(0,1)$ (red), and $G(-1,6)$ (blue).}
\label{fig:4}
\end{figure}
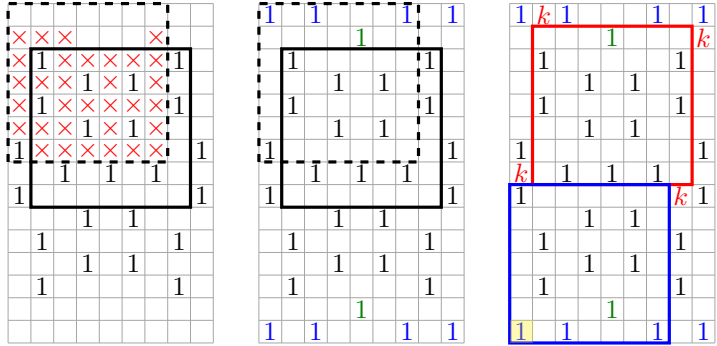

Note that this lemma also applies to any pattern that is a reflection or rotation of the pattern we mentioned. The same holds for all following part of this lemma.

\subsection{Proof of Lemma \ref{lem:4}\,(ii)}
\begin{proof}
Suppose for contradiction that the pattern
\[
    \begin{bmatrix}
    1 & x & x & 1\\
    \end{bmatrix}
    \]
    appears in the coloring. Let $ G $ be a $ P_7 \boxtimes P_7 $ subgraph in which the pattern appears in row 2, where the squares at $ (2,2) $ and $ (5,2) $ are of color 1 (see Figure \ref{fig:5}). Observe that the other squares in rows 1, 2, and 3 cannot be assigned color 1 in columns 1 to 6. We now partition the remaining squares into five parts, as shown in Figure~\ref{fig:5}.

By Observation~\ref{obs:2}\,(i), (ii), (iv) and Lemma \ref{lem:4}\,(i), the $ P_1 \boxtimes P_5 $ part must contain exactly two squares of color 1, each of the $ P_2 \boxtimes P_2 $ parts must contain exactly one square of color 1 and the $ P_7 \boxtimes P_2 $ part must contain exactly four squares of color 1, located in columns 1, 3, 5, and 7.

Next, consider the subgraph $ G(-1, 0) $. By Observation~\ref{obs:2}\,(iv), column 1 of this subgraph must contain at least three squares of color 1. However, since either $ (2,6) $ or $ (2,7) $ is a square of color 1, both squares $ (1,6) $ and $ (1,7) $ cannot be assigned color 1 due to the distance constraint (see Figure~\ref{fig:3}). Therefore, all three squares of color 1 in column 1 must lie within the $ P_1 \boxtimes P_5 $ subgraph from $ (1,1) $ to $ (1,5) $, which contradicts Lemma \ref{lem:4}\,(i).
\end{proof}

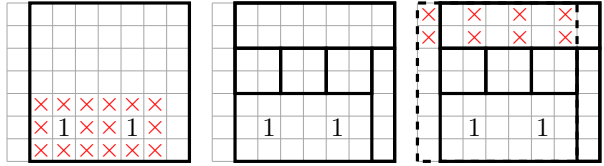
\begin{figure}[ht]
\centering
\begin{tikzpicture}[scale=0.30]

\begin{scope}[shift={(-18,0)}]
  \foreach \x in {0,...,8} {
    \draw[gray!70, line width=0.35pt] (\x,0) -- (\x,7);
  }
  \foreach \y in {0,...,7} {
    \draw[gray!70, line width=0.35pt] (0,\y) -- (8,\y);
  }
  \draw[very thick] (1,0) rectangle (8,7);

  \node at (2.5,1.5) {\normalsize 1};
  \node at (5.5,1.5) {\normalsize 1};
  \node at (1.5,0.5) {\textcolor{red}{$\times$}};
  \node at (2.5,0.5) {\textcolor{red}{$\times$}};
  \node at (3.5,0.5) {\textcolor{red}{$\times$}};
  \node at (4.5,0.5) {\textcolor{red}{$\times$}};
  \node at (5.5,0.5) {\textcolor{red}{$\times$}};
  \node at (6.5,0.5) {\textcolor{red}{$\times$}};
  \node at (1.5,1.5) {\textcolor{red}{$\times$}};
  \node at (3.5,1.5) {\textcolor{red}{$\times$}};
  \node at (4.5,1.5) {\textcolor{red}{$\times$}};
  \node at (6.5,1.5) {\textcolor{red}{$\times$}};
  \node at (1.5,2.5) {\textcolor{red}{$\times$}};
  \node at (2.5,2.5) {\textcolor{red}{$\times$}};
  \node at (3.5,2.5) {\textcolor{red}{$\times$}};
  \node at (4.5,2.5) {\textcolor{red}{$\times$}};
  \node at (5.5,2.5) {\textcolor{red}{$\times$}};
  \node at (6.5,2.5) {\textcolor{red}{$\times$}};
  
\end{scope}

\begin{scope}[shift={(-9,0)}]
  \foreach \x in {0,...,8} {
    \draw[gray!70, line width=0.35pt] (\x,0) -- (\x,7);
  }
  \foreach \y in {0,...,7} {
    \draw[gray!70, line width=0.35pt] (0,\y) -- (8,\y);
  }
  \draw[very thick] (1,0) rectangle (8,7);
  \draw[very thick] (7,0) rectangle (8,5);
  \draw[very thick] (1,5) rectangle (8,7);
  \draw[very thick] (1,3) rectangle (3,5);
  \draw[very thick] (3,3) rectangle (5,5);
  \draw[very thick] (5,3) rectangle (7,5);

  \node at (2.5,1.5) {\normalsize 1};
  \node at (5.5,1.5) {\normalsize 1};
\end{scope}

\begin{scope}[shift={(0,0)}]
  \foreach \x in {0,...,8} {
    \draw[gray!70, line width=0.35pt] (\x,0) -- (\x,7);
  }
  \foreach \y in {0,...,7} {
    \draw[gray!70, line width=0.35pt] (0,\y) -- (8,\y);
  }

  \draw[very thick] (1,0) rectangle (8,7);
  \draw[very thick] (7,0) rectangle (8,5);
  \draw[very thick] (1,5) rectangle (8,7);
  \draw[very thick] (1,3) rectangle (3,5);
  \draw[very thick] (3,3) rectangle (5,5);
  \draw[very thick] (5,3) rectangle (7,5);
  \draw[very thick, dashed] (0,0) rectangle (7,7);

  \node at (2.5,1.5) {\normalsize 1};
  \node at (5.5,1.5) {\normalsize 1};
  \node at (6.5,6.5) {\textcolor{red}{$\times$}};
  \node at (6.5,5.5) {\textcolor{red}{$\times$}};
  \node at (4.5,6.5) {\textcolor{red}{$\times$}};
  \node at (4.5,5.5) {\textcolor{red}{$\times$}};
  \node at (2.5,6.5) {\textcolor{red}{$\times$}};
  \node at (2.5,5.5) {\textcolor{red}{$\times$}};
  \node at (0.5,6.5) {\textcolor{red}{$\times$}};
  \node at (0.5,5.5) {\textcolor{red}{$\times$}};
\end{scope}

\end{tikzpicture}
\caption{$G$ (black) and $G(-1,0)$ (dashed) in the proof of Lemma~4\,(ii).}

\label{fig:5}
\end{figure}

Let $G$ be a subgraph of $P_\infty \boxtimes P_\infty$. Sometimes, we would like to rotate the $x$- and $y$-axes clockwise by $90k^\circ$ for $k \in \{1,2,3\}$. We write $G^{(k)}$ for the same subgraph $G$ when viewed with respect to the new axes.

\subsection{Proof of Lemma \ref{lem:4}\,(iii)}
\begin{proof}
Suppose for contradiction that the pattern
\[
    \begin{bmatrix}
    1 & x & 1\\
    x & x & x\\
    1 & x & 1
    \end{bmatrix}
    \]
appears in the coloring. Let $ G $ be a $ P_7 \boxtimes P_7 $ subgraph in which the pattern appears in the middle of $ G $, where the squares at $ (3,3) $, $ (3,5) $, $ (5,3) $, and $ (5,5) $ are of color~1 (see Figure~\ref{fig:6}). We will show that the positions of squares that cannot be assigned color~1 and squares of color~1 must be as shown in Figure~\ref{fig:7}.
 Consider $ G(-1,-1) $, and partition the remaining squares that can be assigned color~1 in this subgraph into seven parts, as illustrated in Figure~\ref{fig:6}. By Observation~\ref{obs:2}\,(i) and (iv), each part must contain exactly one square of color~1. Due to the distance constraint, color~1 cannot be assigned to the squares marked with a red $ \times $ in the third diagram of Figure~\ref{fig:6}. Therefore, $ (1,1) $ must be a square of color~1.

\begin{figure}[ht]
\centering
\begin{tikzpicture}[scale=0.30]

\begin{scope}[shift={(-18,0)}]
  \foreach \x in {0,...,8} {
    \draw[gray!70, line width=0.35pt] (\x,0) -- (\x,8);
  }
  \foreach \y in {0,...,8} {
    \draw[gray!70, line width=0.35pt] (0,\y) -- (8,\y);
  }

  \draw[very thick] (1,1) rectangle (8,8);
  \node at (3.5,3.5) {\normalsize 1};
  \node at (5.5,3.5) {\normalsize 1};
  \node at (3.5,5.5) {\normalsize 1};
  \node at (5.5,5.5) {\normalsize 1};
  \node at (2.5,2.5) {\textcolor{red}{$\times$}};
  \node at (3.5,2.5) {\textcolor{red}{$\times$}};
  \node at (4.5,2.5) {\textcolor{red}{$\times$}};
  \node at (5.5,2.5) {\textcolor{red}{$\times$}};
  \node at (6.5,2.5) {\textcolor{red}{$\times$}};
  
  \node at (2.5,3.5) {\textcolor{red}{$\times$}};
  \node at (4.5,3.5) {\textcolor{red}{$\times$}};
  \node at (6.5,3.5) {\textcolor{red}{$\times$}};

  \node at (2.5,4.5) {\textcolor{red}{$\times$}};
  \node at (3.5,4.5) {\textcolor{red}{$\times$}};
  \node at (4.5,4.5) {\textcolor{red}{$\times$}};
  \node at (5.5,4.5) {\textcolor{red}{$\times$}};
  \node at (6.5,4.5) {\textcolor{red}{$\times$}};

  \node at (2.5,5.5) {\textcolor{red}{$\times$}};
  \node at (4.5,5.5) {\textcolor{red}{$\times$}};
  \node at (6.5,5.5) {\textcolor{red}{$\times$}};

  \node at (2.5,6.5) {\textcolor{red}{$\times$}};
  \node at (3.5,6.5) {\textcolor{red}{$\times$}};
  \node at (4.5,6.5) {\textcolor{red}{$\times$}};
  \node at (5.5,6.5) {\textcolor{red}{$\times$}};
  \node at (6.5,6.5) {\textcolor{red}{$\times$}};

\end{scope}

\begin{scope}[shift={(-9,0)}]
  \foreach \x in {0,...,8} {
    \draw[gray!70, line width=0.35pt] (\x,0) -- (\x,8);
  }
  \foreach \y in {0,...,8} {
    \draw[gray!70, line width=0.35pt] (0,\y) -- (8,\y);
  }

  \draw[very thick, dashed] (0,0) rectangle (2,2);
  \draw[very thick, dashed] (0,2) rectangle (2,4);
  \draw[very thick, dashed] (0,4) rectangle (2,6);
  \draw[very thick, dashed] (0,6) rectangle (2,7);
  \draw[very thick, dashed] (2,0) rectangle (4,2);
  \draw[very thick, dashed] (4,0) rectangle (6,2);
  \draw[very thick, dashed] (6,0) rectangle (7,2);
  \draw[very thick, dashed] (2,2) rectangle (7,7);
  \node at (3.5,3.5) {\normalsize 1};
  \node at (5.5,3.5) {\normalsize 1};
  \node at (3.5,5.5) {\normalsize 1};
  \node at (5.5,5.5) {\normalsize 1};
  \node at (2.5,2.5) {\textcolor{red}{$\times$}};
  \node at (3.5,2.5) {\textcolor{red}{$\times$}};
  \node at (4.5,2.5) {\textcolor{red}{$\times$}};
  \node at (5.5,2.5) {\textcolor{red}{$\times$}};
  \node at (6.5,2.5) {\textcolor{red}{$\times$}};
  
  \node at (2.5,3.5) {\textcolor{red}{$\times$}};
  \node at (4.5,3.5) {\textcolor{red}{$\times$}};
  \node at (6.5,3.5) {\textcolor{red}{$\times$}};

  \node at (2.5,4.5) {\textcolor{red}{$\times$}};
  \node at (3.5,4.5) {\textcolor{red}{$\times$}};
  \node at (4.5,4.5) {\textcolor{red}{$\times$}};
  \node at (5.5,4.5) {\textcolor{red}{$\times$}};
  \node at (6.5,4.5) {\textcolor{red}{$\times$}};

  \node at (2.5,5.5) {\textcolor{red}{$\times$}};
  \node at (4.5,5.5) {\textcolor{red}{$\times$}};
  \node at (6.5,5.5) {\textcolor{red}{$\times$}};

  \node at (2.5,6.5) {\textcolor{red}{$\times$}};
  \node at (3.5,6.5) {\textcolor{red}{$\times$}};
  \node at (4.5,6.5) {\textcolor{red}{$\times$}};
  \node at (5.5,6.5) {\textcolor{red}{$\times$}};
  \node at (6.5,6.5) {\textcolor{red}{$\times$}};

\end{scope}

\begin{scope}[shift={(0,0)}]
  \foreach \x in {0,...,8} {
    \draw[gray!70, line width=0.35pt] (\x,0) -- (\x,8);
  }
  \foreach \y in {0,...,8} {
    \draw[gray!70, line width=0.35pt] (0,\y) -- (8,\y);
  }

  \draw[very thick, dashed] (0,0) rectangle (2,2);
  \draw[very thick, dashed] (0,2) rectangle (2,4);
  \draw[very thick, dashed] (0,4) rectangle (2,6);
  \draw[very thick, dashed] (0,6) rectangle (2,7);
  \draw[very thick, dashed] (2,0) rectangle (4,2);
  \draw[very thick, dashed] (4,0) rectangle (6,2);
  \draw[very thick, dashed] (6,0) rectangle (7,2);
  \draw[very thick, dashed] (2,2) rectangle (7,7);

  \node at (3.5,3.5) {\normalsize 1};
  \node at (5.5,3.5) {\normalsize 1};
  \node at (3.5,5.5) {\normalsize 1};
  \node at (5.5,5.5) {\normalsize 1};
  \node at (0.5,0.5) {\textcolor{blue} 1};

  \node at (1.5,5.5) {\textcolor{red}{$\times$}};
  \node at (0.5,5.5) {\textcolor{red}{$\times$}};
  \node at (1.5,3.5) {\textcolor{red}{$\times$}};
  \node at (0.5,3.5) {\textcolor{red}{$\times$}};
  \node at (1.5,1.5) {\textcolor{red}{$\times$}};
  \node at (0.5,1.5) {\textcolor{red}{$\times$}};
  \node at (1.5,0.5) {\textcolor{red}{$\times$}};
  \node at (3.5,0.5) {\textcolor{red}{$\times$}};
  \node at (3.5,1.5) {\textcolor{red}{$\times$}};
  \node at (5.5,0.5) {\textcolor{red}{$\times$}};
  \node at (5.5,1.5) {\textcolor{red}{$\times$}};
\end{scope}

\end{tikzpicture}
\caption{$G$ (black) and $G(-1,-1)$ (dashed).}

\label{fig:6}
\end{figure}
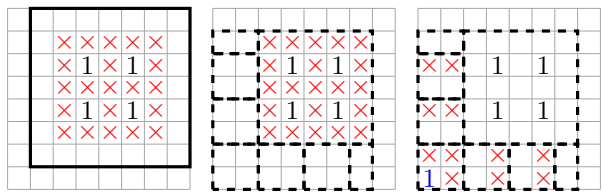

Similarly, by considering $ G(-1, 1) $, $ G(1, -1) $, and $ G(1, 1) $, we obtain that the positions of squares of color 1, and the squares that cannot be assigned color 1, must appear as shown in Figure~\ref{fig:7}.

\begin{figure}[ht]
\centering
\begin{tikzpicture}[scale=0.30]

\foreach \x in {0,...,9} {
  \draw[gray!70, line width=0.35pt] (\x,0) -- (\x,9);
}

\foreach \y in {0,...,9} {
  \draw[gray!70, line width=0.35pt] (0,\y) -- (9,\y);
}

\draw[very thick] (1,1) rectangle (8,8);
\node at (3.5,3.5) {\normalsize 1};
\node at (5.5,3.5) {\normalsize 1};
\node at (3.5,5.5) {\normalsize 1};
\node at (5.5,5.5) {\normalsize 1};
\node at (0.5,0.5) {\textcolor{blue} 1};
\node at (8.5,0.5) {\textcolor{blue} 1};
\node at (8.5,8.5) {\textcolor{blue} 1};
\node at (0.5,8.5) {\textcolor{blue} 1};

\node at (1.5,7.5) {\textcolor{red}{$\times$}};
\node at (0.5,7.5) {\textcolor{red}{$\times$}};
\node at (1.5,5.5) {\textcolor{red}{$\times$}};
\node at (0.5,5.5) {\textcolor{red}{$\times$}};
\node at (1.5,3.5) {\textcolor{red}{$\times$}};
\node at (0.5,3.5) {\textcolor{red}{$\times$}};
\node at (1.5,1.5) {\textcolor{red}{$\times$}};
\node at (0.5,1.5) {\textcolor{red}{$\times$}};

\node at (8.5,7.5) {\textcolor{red}{$\times$}};
\node at (7.5,7.5) {\textcolor{red}{$\times$}};
\node at (8.5,5.5) {\textcolor{red}{$\times$}};
\node at (7.5,5.5) {\textcolor{red}{$\times$}};
\node at (8.5,3.5) {\textcolor{red}{$\times$}};
\node at (7.5,3.5) {\textcolor{red}{$\times$}};
\node at (8.5,1.5) {\textcolor{red}{$\times$}};
\node at (7.5,1.5) {\textcolor{red}{$\times$}};

\node at (1.5,0.5) {\textcolor{red}{$\times$}};
\node at (3.5,0.5) {\textcolor{red}{$\times$}};
\node at (3.5,1.5) {\textcolor{red}{$\times$}};
\node at (5.5,0.5) {\textcolor{red}{$\times$}};
\node at (5.5,1.5) {\textcolor{red}{$\times$}};
\node at (7.5,0.5) {\textcolor{red}{$\times$}};

\node at (1.5,8.5) {\textcolor{red}{$\times$}};
\node at (3.5,7.5) {\textcolor{red}{$\times$}};
\node at (3.5,8.5) {\textcolor{red}{$\times$}};
\node at (5.5,7.5) {\textcolor{red}{$\times$}};
\node at (5.5,8.5) {\textcolor{red}{$\times$}};
\node at (7.5,8.5) {\textcolor{red}{$\times$}};

\node at (2.5,2.5) {\textcolor{red}{$\times$}};
  \node at (3.5,2.5) {\textcolor{red}{$\times$}};
  \node at (4.5,2.5) {\textcolor{red}{$\times$}};
  \node at (5.5,2.5) {\textcolor{red}{$\times$}};
  \node at (6.5,2.5) {\textcolor{red}{$\times$}};
  
  \node at (2.5,3.5) {\textcolor{red}{$\times$}};
  \node at (4.5,3.5) {\textcolor{red}{$\times$}};
  \node at (6.5,3.5) {\textcolor{red}{$\times$}};

  \node at (2.5,4.5) {\textcolor{red}{$\times$}};
  \node at (3.5,4.5) {\textcolor{red}{$\times$}};
  \node at (4.5,4.5) {\textcolor{red}{$\times$}};
  \node at (5.5,4.5) {\textcolor{red}{$\times$}};
  \node at (6.5,4.5) {\textcolor{red}{$\times$}};

  \node at (2.5,5.5) {\textcolor{red}{$\times$}};
  \node at (4.5,5.5) {\textcolor{red}{$\times$}};
  \node at (6.5,5.5) {\textcolor{red}{$\times$}};

  \node at (2.5,6.5) {\textcolor{red}{$\times$}};
  \node at (3.5,6.5) {\textcolor{red}{$\times$}};
  \node at (4.5,6.5) {\textcolor{red}{$\times$}};
  \node at (5.5,6.5) {\textcolor{red}{$\times$}};
  \node at (6.5,6.5) {\textcolor{red}{$\times$}};

\end{tikzpicture}
\caption{$G$.}

\label{fig:7}
\end{figure}

Next, we partition the remaining squares that can be assigned color 1 in $G$ into four parts: three squares in column 1, row 1, column 7, and row 7. By Lemma \ref{lem:4}\,(i), each part contains at most two squares of color 1. By Observation~\ref{obs:2}\,(iv), each part contains at least one square of color 1. Hence, each part contains either one or two squares of color 1. We will show that each part must contain exactly two squares of color 1. Without loss of generality, suppose for contradiction that the part in column 1 contains only one square of color 1. Then, by Observation~\ref{obs:2}\,(iv), the other parts must contain exactly two squares of color 1. By considering $ G(0,1) $, we have $ c(G(0,1)) \leq 10 $, contradicting Observation~\ref{obs:2}\,(iv). Therefore, each part contains exactly two squares of color 1.

Then, we can see that there are only two possible positions for the squares of color~1 in $ G $, as shown in Figure~\ref{fig:8}. 


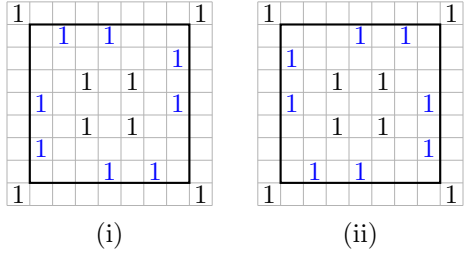
\begin{figure}[ht]
\centering
\begin{tikzpicture}[scale=0.30]

\begin{scope}[shift={(0,0)}]
  \foreach \x in {0,...,9} {
    \draw[gray!60, line width=0.2pt] (\x,0) -- (\x,9);
  }
  \foreach \y in {0,...,9} {
    \draw[gray!60, line width=0.2pt] (0,\y) -- (9,\y);
  }

  \draw[thick] (1,1) rectangle (8,8);

  \node at (3.5,3.5) {\normalsize 1};
  \node at (5.5,3.5) {\normalsize 1};
  \node at (3.5,5.5) {\normalsize 1};
  \node at (5.5,5.5) {\normalsize 1};
  \node at (0.5,0.5) { 1};
  \node at (8.5,0.5) { 1};
  \node at (8.5,8.5) { 1};
  \node at (0.5,8.5) { 1};

  \node at (1.5,2.5) {\textcolor{blue}{1}};
  \node at (1.5,4.5) {\textcolor{blue}{1}};
  \node at (4.5,1.5) {\textcolor{blue}{1}};
  \node at (6.5,1.5) {\textcolor{blue}{1}};
  \node at (7.5,4.5) {\textcolor{blue}{1}};
  \node at (7.5,6.5) {\textcolor{blue}{1}};
  \node at (4.5,7.5) {\textcolor{blue}{1}};
  \node at (2.5,7.5) {\textcolor{blue}{1}};
  \node at (4.5,-1.3) {(i)};
\end{scope}

\begin{scope}[shift={(11,0)}]
  \foreach \x in {0,...,9} {
    \draw[gray!60, line width=0.35pt] (\x,0) -- (\x,9);
  }
  \foreach \y in {0,...,9} {
    \draw[gray!60, line width=0.35pt] (0,\y) -- (9,\y);
  }
  \draw[thick] (1,1) rectangle (8,8);
  \node at (3.5,3.5) {\normalsize 1};
  \node at (5.5,3.5) {\normalsize 1};
  \node at (3.5,5.5) {\normalsize 1};
  \node at (5.5,5.5) {\normalsize 1};
  \node at (0.5,0.5) { 1};
  \node at (8.5,0.5) { 1};
  \node at (8.5,8.5) { 1};
  \node at (0.5,8.5) { 1};
  \node at (1.5,6.5) {\textcolor{blue}{1}};
  \node at (1.5,4.5) {\textcolor{blue}{1}};
  \node at (4.5,1.5) {\textcolor{blue}{1}};
  \node at (2.5,1.5) {\textcolor{blue}{1}};
  \node at (7.5,4.5) {\textcolor{blue}{1}};
  \node at (7.5,2.5) {\textcolor{blue}{1}};
  \node at (4.5,7.5) {\textcolor{blue}{1}};
  \node at (6.5,7.5) {\textcolor{blue}{1}};
  \node at (4.5,-1.3) {(ii)};
\end{scope}

\end{tikzpicture}
\caption{Two possible positions of squares of color~1 in $ G $, referred to as pattern~(i) and pattern~(ii).}
\label{fig:8}
\end{figure}

Without loss of generality, assume that $ G $ has pattern~(i). Consider $ G(1,1) $ (see Figure~\ref{fig:9}). By Lemma \ref{lem:4}\,(i), the square $ (6,7) $ cannot be assigned color~1. By Observation~\ref{obs:2}\,(iv), squares $ (1,1) $, $ (5,7) $, and $ (7,7) $ must be assigned color~1 (see Figure~\ref{fig:9}). Next, consider $ G(3,1) $ (see Figure~\ref{fig:9}). By Lemma \ref{lem:4}\,(i), the square $ (7,7) $ cannot be assigned color~1. Moreover, by Lemma~\ref{lem:4}\,(ii), the squares $ (7,3) $ and $ (7,5) $ cannot be squares of color~1. We then partition the subgraph $ P_2 \boxtimes P_6 $ from $ (6,1) $ to $ (7,6) $ into three parts, where each part is a $ P_2 \boxtimes P_2 $ subgraph. By Observation~\ref{obs:2}\,(i) and Observation~\ref{obs:2}\,(iv), each part must contain exactly one square of color~1. Observe that the square of color 1 in the upper part is either $(7,6)$ or $(6,5)$. We consider the following two cases.

\begin{figure}[ht]
\centering
\begin{tikzpicture}[scale=0.30]

\begin{scope}[shift={(-24,0)}]
  \foreach \x in {0,...,10} {
    \draw[gray!70, line width=0.35pt] (\x,0) -- (\x,8);
  }
  \foreach \y in {0,...,8} {
    \draw[gray!70, line width=0.35pt] (0,\y) -- (10,\y);
  }

  \draw[very thick] (0,0) rectangle (7,7);
  \draw[very thick, blue] (1,1) rectangle (8,8);
  \node at (2.5,2.5) {\normalsize 1};
  \node at (4.5,2.5) {\normalsize 1};
  \node at (2.5,4.5) {\normalsize 1};
  \node at (4.5,4.5) {\normalsize 1};
  \node at (1.5,1.5) {\textcolor{red}{$\times$}};
  \node at (2.5,1.5) {\textcolor{red}{$\times$}};
  \node at (3.5,1.5) {\textcolor{red}{$\times$}};
  \node at (4.5,1.5) {\textcolor{red}{$\times$}};
  \node at (5.5,1.5) {\textcolor{red}{$\times$}};
  
  \node at (1.5,2.5) {\textcolor{red}{$\times$}};
  \node at (3.5,2.5) {\textcolor{red}{$\times$}};
  \node at (5.5,2.5) {\textcolor{red}{$\times$}};

  \node at (1.5,3.5) {\textcolor{red}{$\times$}};
  \node at (2.5,3.5) {\textcolor{red}{$\times$}};
  \node at (3.5,3.5) {\textcolor{red}{$\times$}};
  \node at (4.5,3.5) {\textcolor{red}{$\times$}};
  \node at (5.5,3.5) {\textcolor{red}{$\times$}};

  \node at (1.5,4.5) {\textcolor{red}{$\times$}};
  \node at (3.5,4.5) {\textcolor{red}{$\times$}};
  \node at (5.5,4.5) {\textcolor{red}{$\times$}};

  \node at (1.5,5.5) {\textcolor{red}{$\times$}};
  \node at (2.5,5.5) {\textcolor{red}{$\times$}};
  \node at (3.5,5.5) {\textcolor{red}{$\times$}};
  \node at (4.5,5.5) {\textcolor{red}{$\times$}};
  \node at (5.5,5.5) {\textcolor{red}{$\times$}};

  \node at (0.5,1.5) {\normalsize 1};
  \node at (0.5,3.5) {\normalsize 1};
  \node at (3.5,0.5) {\normalsize 1};
  \node at (5.5,0.5) {\normalsize 1};
  \node at (6.5,3.5) {\normalsize 1};
  \node at (6.5,5.5) {\normalsize 1};
  \node at (1.5,6.5) {\normalsize 1};
  \node at (3.5,6.5) {\normalsize 1};

  \node at (0.5,0.5) {\textcolor{red}{$\times$}};
  \node at (0.5,2.5) {\textcolor{red}{$\times$}};
  \node at (0.5,4.5) {\textcolor{red}{$\times$}};
  \node at (0.5,5.5) {\textcolor{red}{$\times$}};
  \node at (0.5,6.5) {\textcolor{red}{$\times$}};

  \node at (1.5,0.5) {\textcolor{red}{$\times$}};
  \node at (2.5,0.5) {\textcolor{red}{$\times$}};
  \node at (4.5,0.5) {\textcolor{red}{$\times$}};
  \node at (6.5,0.5) {\textcolor{red}{$\times$}};

  \node at (6.5,1.5) {\textcolor{red}{$\times$}};
  \node at (6.5,2.5) {\textcolor{red}{$\times$}};
  \node at (6.5,4.5) {\textcolor{red}{$\times$}};
  \node at (6.5,6.5) {\textcolor{red}{$\times$}};

  \node at (5.5,6.5) {\textcolor{red}{$\times$}};
  \node at (4.5,6.5) {\textcolor{red}{$\times$}};
  \node at (2.5,6.5) {\textcolor{red}{$\times$}};

  \node at (0.5,7.5) {\textcolor{red}{$\times$}};
  \node at (1.5,7.5) {\textcolor{red}{$\times$}};
  \node at (2.5,7.5) {\textcolor{red}{$\times$}};
  \node at (3.5,7.5) {\textcolor{red}{$\times$}};
  \node at (4.5,7.5) {\textcolor{red}{$\times$}};

  \node at (7.5,6.5) {\textcolor{red}{$\times$}};
  \node at (7.5,5.5) {\textcolor{red}{$\times$}};
  \node at (7.5,4.5) {\textcolor{red}{$\times$}};
  \node at (7.5,3.5) {\textcolor{red}{$\times$}};
  \node at (7.5,2.5) {\textcolor{red}{$\times$}};
\end{scope}

\begin{scope}[shift={(-12,0)}]
  \foreach \x in {0,...,10} {
    \draw[gray!70, line width=0.35pt] (\x,0) -- (\x,8);
  }
  \foreach \y in {0,...,8} {
    \draw[gray!70, line width=0.35pt] (0,\y) -- (10,\y);
  }

  \draw[very thick] (0,0) rectangle (7,7);
  \draw[very thick, blue] (1,1) rectangle (8,8);
  \node at (2.5,2.5) {\normalsize 1};
  \node at (4.5,2.5) {\normalsize 1};
  \node at (2.5,4.5) {\normalsize 1};
  \node at (4.5,4.5) {\normalsize 1};
  \node at (5.5,7.5) {\textcolor{blue}{1}};
  \node at (7.5,7.5) {\textcolor{blue}{1}};
  \node at (7.5,1.5) {\textcolor{blue}{1}};
  \node at (1.5,1.5) {\textcolor{red}{$\times$}};
  \node at (2.5,1.5) {\textcolor{red}{$\times$}};
  \node at (3.5,1.5) {\textcolor{red}{$\times$}};
  \node at (4.5,1.5) {\textcolor{red}{$\times$}};
  \node at (5.5,1.5) {\textcolor{red}{$\times$}};
  
  \node at (1.5,2.5) {\textcolor{red}{$\times$}};
  \node at (3.5,2.5) {\textcolor{red}{$\times$}};
  \node at (5.5,2.5) {\textcolor{red}{$\times$}};

  \node at (1.5,3.5) {\textcolor{red}{$\times$}};
  \node at (2.5,3.5) {\textcolor{red}{$\times$}};
  \node at (3.5,3.5) {\textcolor{red}{$\times$}};
  \node at (4.5,3.5) {\textcolor{red}{$\times$}};
  \node at (5.5,3.5) {\textcolor{red}{$\times$}};

  \node at (1.5,4.5) {\textcolor{red}{$\times$}};
  \node at (3.5,4.5) {\textcolor{red}{$\times$}};
  \node at (5.5,4.5) {\textcolor{red}{$\times$}};

  \node at (1.5,5.5) {\textcolor{red}{$\times$}};
  \node at (2.5,5.5) {\textcolor{red}{$\times$}};
  \node at (3.5,5.5) {\textcolor{red}{$\times$}};
  \node at (4.5,5.5) {\textcolor{red}{$\times$}};
  \node at (5.5,5.5) {\textcolor{red}{$\times$}};

  \node at (0.5,1.5) {\normalsize 1};
  \node at (0.5,3.5) {\normalsize 1};
  \node at (3.5,0.5) {\normalsize 1};
  \node at (5.5,0.5) {\normalsize 1};
  \node at (6.5,3.5) {\normalsize 1};
  \node at (6.5,5.5) {\normalsize 1};
  \node at (1.5,6.5) {\normalsize 1};
  \node at (3.5,6.5) {\normalsize 1};

  \node at (0.5,0.5) {\textcolor{red}{$\times$}};
  \node at (0.5,2.5) {\textcolor{red}{$\times$}};
  \node at (0.5,4.5) {\textcolor{red}{$\times$}};
  \node at (0.5,5.5) {\textcolor{red}{$\times$}};
  \node at (0.5,6.5) {\textcolor{red}{$\times$}};

  \node at (1.5,0.5) {\textcolor{red}{$\times$}};
  \node at (2.5,0.5) {\textcolor{red}{$\times$}};
  \node at (4.5,0.5) {\textcolor{red}{$\times$}};
  \node at (6.5,0.5) {\textcolor{red}{$\times$}};

  \node at (6.5,1.5) {\textcolor{red}{$\times$}};
  \node at (6.5,2.5) {\textcolor{red}{$\times$}};
  \node at (6.5,4.5) {\textcolor{red}{$\times$}};
  \node at (6.5,6.5) {\textcolor{red}{$\times$}};

  \node at (5.5,6.5) {\textcolor{red}{$\times$}};
  \node at (4.5,6.5) {\textcolor{red}{$\times$}};
  \node at (2.5,6.5) {\textcolor{red}{$\times$}};

  \node at (0.5,7.5) {\textcolor{red}{$\times$}};
  \node at (1.5,7.5) {\textcolor{red}{$\times$}};
  \node at (2.5,7.5) {\textcolor{red}{$\times$}};
  \node at (3.5,7.5) {\textcolor{red}{$\times$}};
  \node at (4.5,7.5) {\textcolor{red}{$\times$}};

  \node at (7.5,6.5) {\textcolor{red}{$\times$}};
  \node at (7.5,5.5) {\textcolor{red}{$\times$}};
  \node at (7.5,4.5) {\textcolor{red}{$\times$}};
  \node at (7.5,3.5) {\textcolor{red}{$\times$}};
  \node at (7.5,2.5) {\textcolor{red}{$\times$}};

  \node at (6.5,7.5) {\textcolor{blue}{$\times$}};
  \node at (8.5,7.5) {\textcolor{blue}{$\times$}};
  \node at (8.5,6.5) {\textcolor{blue}{$\times$}};
  \node at (8.5,2.5) {\textcolor{blue}{$\times$}};
  \node at (8.5,1.5) {\textcolor{blue}{$\times$}};
  \node at (8.5,0.5) {\textcolor{blue}{$\times$}};
  \node at (7.5,0.5) {\textcolor{blue}{$\times$}};

\end{scope}

\begin{scope}[shift={(0,0)}]
  \foreach \x in {0,...,10} {
    \draw[gray!70, line width=0.35pt] (\x,0) -- (\x,8);
  }
  \foreach \y in {0,...,8} {
    \draw[gray!70, line width=0.35pt] (0,\y) -- (10,\y);
  }

  \draw[very thick] (0,0) rectangle (7,7);
  \draw[very thick, darkgreen] (3,1) rectangle (10,8);
  \draw[very thick, darkgreen] (8,1) rectangle (10,3);
  \draw[very thick, darkgreen] (8,3) rectangle (10,5);
  \draw[very thick, darkgreen] (8,5) rectangle (10,7);
  \node at (2.5,2.5) {\normalsize 1};
  \node at (4.5,2.5) {\normalsize 1};
  \node at (2.5,4.5) {\normalsize 1};
  \node at (4.5,4.5) {\normalsize 1};
  \node at (5.5,7.5) {\textcolor{blue}{1}};
  \node at (7.5,7.5) {\textcolor{blue}{1}};
  \node at (7.5,1.5) {\textcolor{blue}{1}};
  \node at (1.5,1.5) {\textcolor{red}{$\times$}};
  \node at (2.5,1.5) {\textcolor{red}{$\times$}};
  \node at (3.5,1.5) {\textcolor{red}{$\times$}};
  \node at (4.5,1.5) {\textcolor{red}{$\times$}};
  \node at (5.5,1.5) {\textcolor{red}{$\times$}};
  
  \node at (1.5,2.5) {\textcolor{red}{$\times$}};
  \node at (3.5,2.5) {\textcolor{red}{$\times$}};
  \node at (5.5,2.5) {\textcolor{red}{$\times$}};

  \node at (1.5,3.5) {\textcolor{red}{$\times$}};
  \node at (2.5,3.5) {\textcolor{red}{$\times$}};
  \node at (3.5,3.5) {\textcolor{red}{$\times$}};
  \node at (4.5,3.5) {\textcolor{red}{$\times$}};
  \node at (5.5,3.5) {\textcolor{red}{$\times$}};

  \node at (1.5,4.5) {\textcolor{red}{$\times$}};
  \node at (3.5,4.5) {\textcolor{red}{$\times$}};
  \node at (5.5,4.5) {\textcolor{red}{$\times$}};

  \node at (1.5,5.5) {\textcolor{red}{$\times$}};
  \node at (2.5,5.5) {\textcolor{red}{$\times$}};
  \node at (3.5,5.5) {\textcolor{red}{$\times$}};
  \node at (4.5,5.5) {\textcolor{red}{$\times$}};
  \node at (5.5,5.5) {\textcolor{red}{$\times$}};

  \node at (0.5,1.5) {\normalsize 1};
  \node at (0.5,3.5) {\normalsize 1};
  \node at (3.5,0.5) {\normalsize 1};
  \node at (5.5,0.5) {\normalsize 1};
  \node at (6.5,3.5) {\normalsize 1};
  \node at (6.5,5.5) {\normalsize 1};
  \node at (1.5,6.5) {\normalsize 1};
  \node at (3.5,6.5) {\normalsize 1};

  \node at (0.5,0.5) {\textcolor{red}{$\times$}};
  \node at (0.5,2.5) {\textcolor{red}{$\times$}};
  \node at (0.5,4.5) {\textcolor{red}{$\times$}};
  \node at (0.5,5.5) {\textcolor{red}{$\times$}};
  \node at (0.5,6.5) {\textcolor{red}{$\times$}};

  \node at (1.5,0.5) {\textcolor{red}{$\times$}};
  \node at (2.5,0.5) {\textcolor{red}{$\times$}};
  \node at (4.5,0.5) {\textcolor{red}{$\times$}};
  \node at (6.5,0.5) {\textcolor{red}{$\times$}};

  \node at (6.5,1.5) {\textcolor{red}{$\times$}};
  \node at (6.5,2.5) {\textcolor{red}{$\times$}};
  \node at (6.5,4.5) {\textcolor{red}{$\times$}};
  \node at (6.5,6.5) {\textcolor{red}{$\times$}};

  \node at (5.5,6.5) {\textcolor{red}{$\times$}};
  \node at (4.5,6.5) {\textcolor{red}{$\times$}};
  \node at (2.5,6.5) {\textcolor{red}{$\times$}};

  \node at (0.5,7.5) {\textcolor{red}{$\times$}};
  \node at (1.5,7.5) {\textcolor{red}{$\times$}};
  \node at (2.5,7.5) {\textcolor{red}{$\times$}};
  \node at (3.5,7.5) {\textcolor{red}{$\times$}};
  \node at (4.5,7.5) {\textcolor{red}{$\times$}};

  \node at (7.5,6.5) {\textcolor{red}{$\times$}};
  \node at (7.5,5.5) {\textcolor{red}{$\times$}};
  \node at (7.5,4.5) {\textcolor{red}{$\times$}};
  \node at (7.5,3.5) {\textcolor{red}{$\times$}};
  \node at (7.5,2.5) {\textcolor{red}{$\times$}};

  \node at (6.5,7.5) {\textcolor{blue}{$\times$}};
  \node at (8.5,7.5) {\textcolor{blue}{$\times$}};
  \node at (8.5,6.5) {\textcolor{blue}{$\times$}};
  \node at (8.5,2.5) {\textcolor{blue}{$\times$}};
  \node at (8.5,1.5) {\textcolor{blue}{$\times$}};
  \node at (8.5,0.5) {\textcolor{blue}{$\times$}};
  \node at (7.5,0.5) {\textcolor{blue}{$\times$}};

  \node at (9.5,7.5) {\textcolor{darkgreen}{$\times$}};
  \node at (9.5,5.5) {\textcolor{darkgreen}{$\times$}};
  \node at (9.5,3.5) {\textcolor{darkgreen}{$\times$}};

  \fill[yellow, opacity=0.3] (9,6) rectangle (10,7);
  \fill[yellow, opacity=0.3] (8,5) rectangle (9,6);
  
\end{scope}

\end{tikzpicture}
\caption{$G$ (black), $G(1,1)$ (blue) and $G(3,1)$ (green).}

\label{fig:9}
\end{figure}
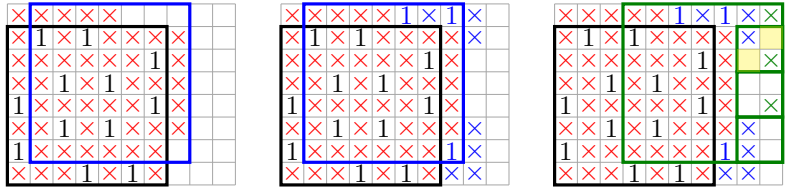

\textbf{Case 1:} Square $ (7,6) $ of $G(3,1)$, which is square $(10,7)$ of $G$, has color 1.

\begin{claim*}
    If $H$ has pattern (i) and square $(10,7)$ of $H$ is of color 1, then $ H(7,-2) $ has pattern~(i) and square $(7,-2)$ of $H$ is of color 1.
\end{claim*}

\begin{proof}
    We know that there are some squares of color 1 outside $H$, as shown in Figure~\ref{fig:9}, since the same argument for $G$ also applies. Observe that the squares that cannot be assigned color~1 must appear as shown in Figure~\ref{fig:10}. Consider $ H(4,0) $, and partition its squares into four parts as illustrated in Figure~\ref{fig:10}. By Observation~\ref{obs:2}\,(i) and~(iv), each part must contain exactly one square of color~1. Since the bottom part must contain a square of color~1, the square $ (6,2) $ in $ H(4,0) $ cannot be assigned color~1, and thus $ (6,3) $ must be assigned color~1. It follows that the squares of color~1 in $ H(4,0) $ must be as shown in Figure~\ref{fig:10}.

    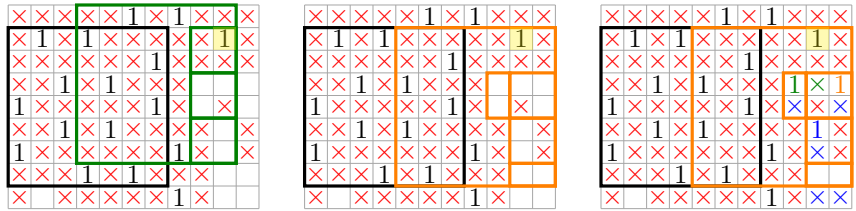
\begin{figure}[ht]
\centering
\begin{tikzpicture}[scale=0.30]

\begin{scope}[shift={(-26,0)}]
  \foreach \x in {0,...,11} {
    \draw[gray!70, line width=0.35pt] (\x,0) -- (\x,9);
  }
  \foreach \y in {0,...,9} {
    \draw[gray!70, line width=0.35pt] (0,\y) -- (11,\y);
  }

\draw[very thick] (0,1) rectangle (7,8);
\draw[very thick, darkgreen] (3,2) rectangle (10,9);
\draw[very thick, darkgreen] (8,2) rectangle (10,4);
\draw[very thick, darkgreen] (8,4) rectangle (10,6);
\draw[very thick, darkgreen] (8,6) rectangle (10,8);

\node at (2.5,3.5) {\normalsize 1};
\node at (4.5,3.5) {\normalsize 1};
\node at (2.5,5.5) {\normalsize 1};
\node at (4.5,5.5) {\normalsize 1};
\node at (5.5,8.5) {{1}};
\node at (7.5,8.5) {{1}};
\node at (7.5,2.5) {{1}};

\node at (1.5,2.5) {\textcolor{red}{$\times$}};
\node at (2.5,2.5) {\textcolor{red}{$\times$}};
\node at (3.5,2.5) {\textcolor{red}{$\times$}};
\node at (4.5,2.5) {\textcolor{red}{$\times$}};
\node at (5.5,2.5) {\textcolor{red}{$\times$}};

\node at (1.5,3.5) {\textcolor{red}{$\times$}};
\node at (3.5,3.5) {\textcolor{red}{$\times$}};
\node at (5.5,3.5) {\textcolor{red}{$\times$}};

\node at (1.5,4.5) {\textcolor{red}{$\times$}};
\node at (2.5,4.5) {\textcolor{red}{$\times$}};
\node at (3.5,4.5) {\textcolor{red}{$\times$}};
\node at (4.5,4.5) {\textcolor{red}{$\times$}};
\node at (5.5,4.5) {\textcolor{red}{$\times$}};

\node at (1.5,5.5) {\textcolor{red}{$\times$}};
\node at (3.5,5.5) {\textcolor{red}{$\times$}};
\node at (5.5,5.5) {\textcolor{red}{$\times$}};

\node at (1.5,6.5) {\textcolor{red}{$\times$}};
\node at (2.5,6.5) {\textcolor{red}{$\times$}};
\node at (3.5,6.5) {\textcolor{red}{$\times$}};
\node at (4.5,6.5) {\textcolor{red}{$\times$}};
\node at (5.5,6.5) {\textcolor{red}{$\times$}};

\node at (0.5,2.5) {\normalsize 1};
\node at (0.5,4.5) {\normalsize 1};
\node at (3.5,1.5) {\normalsize 1};
\node at (5.5,1.5) {\normalsize 1};
\node at (6.5,4.5) {\normalsize 1};
\node at (6.5,6.5) {\normalsize 1};
\node at (1.5,7.5) {\normalsize 1};
\node at (3.5,7.5) {\normalsize 1};

\node at (7.5,0.5) {\normalsize 1};
\node at (8.5,0.5) {\textcolor{red}{$\times$}};
\node at (6.5,0.5) {\textcolor{red}{$\times$}};
\node at (5.5,0.5) {\textcolor{red}{$\times$}};
\node at (4.5,0.5) {\textcolor{red}{$\times$}};
\node at (3.5,0.5) {\textcolor{red}{$\times$}};
\node at (2.5,0.5) {\textcolor{red}{$\times$}};
\node at (0.5,0.5) {\textcolor{red}{$\times$}};

\node at (0.5,1.5) {\textcolor{red}{$\times$}};
\node at (0.5,3.5) {\textcolor{red}{$\times$}};
\node at (0.5,5.5) {\textcolor{red}{$\times$}};
\node at (0.5,6.5) {\textcolor{red}{$\times$}};
\node at (0.5,7.5) {\textcolor{red}{$\times$}};

\node at (1.5,1.5) {\textcolor{red}{$\times$}};
\node at (2.5,1.5) {\textcolor{red}{$\times$}};
\node at (4.5,1.5) {\textcolor{red}{$\times$}};
\node at (6.5,1.5) {\textcolor{red}{$\times$}};

\node at (6.5,2.5) {\textcolor{red}{$\times$}};
\node at (6.5,3.5) {\textcolor{red}{$\times$}};
\node at (6.5,5.5) {\textcolor{red}{$\times$}};
\node at (6.5,7.5) {\textcolor{red}{$\times$}};

\node at (5.5,7.5) {\textcolor{red}{$\times$}};
\node at (4.5,7.5) {\textcolor{red}{$\times$}};
\node at (2.5,7.5) {\textcolor{red}{$\times$}};

\node at (0.5,8.5) {\textcolor{red}{$\times$}};
\node at (1.5,8.5) {\textcolor{red}{$\times$}};
\node at (2.5,8.5) {\textcolor{red}{$\times$}};
\node at (3.5,8.5) {\textcolor{red}{$\times$}};
\node at (4.5,8.5) {\textcolor{red}{$\times$}};

\node at (7.5,7.5) {\textcolor{red}{$\times$}};
\node at (7.5,6.5) {\textcolor{red}{$\times$}};
\node at (7.5,5.5) {\textcolor{red}{$\times$}};
\node at (7.5,4.5) {\textcolor{red}{$\times$}};
\node at (7.5,3.5) {\textcolor{red}{$\times$}};
\node at (10.5,3.5) {\textcolor{red}{$\times$}};
\node at (10.5,2.5) {\textcolor{red}{$\times$}};

\node at (6.5,8.5) {\textcolor{red}{$\times$}};
\node at (8.5,8.5) {\textcolor{red}{$\times$}};
\node at (8.5,7.5) {\textcolor{red}{$\times$}};
\node at (8.5,3.5) {\textcolor{red}{$\times$}};
\node at (8.5,2.5) {\textcolor{red}{$\times$}};
\node at (8.5,1.5) {\textcolor{red}{$\times$}};
\node at (7.5,1.5) {\textcolor{red}{$\times$}};

\node at (9.5,8.5) {\textcolor{red}{$\times$}};
\node at (9.5,6.5) {\textcolor{red}{$\times$}};
\node at (9.5,4.5) {\textcolor{red}{$\times$}};

\node at (9.5,7.5) {\normalsize 1};
\node at (8.5,6.5) {\textcolor{red}{$\times$}};
\node at (10.5,8.5) {\textcolor{red}{$\times$}};
\node at (10.5,7.5) {\textcolor{red}{$\times$}};
\node at (10.5,6.5) {\textcolor{red}{$\times$}};

\fill[yellow, opacity=0.3] (9,7) rectangle (10,8);
\end{scope}

\begin{scope}[shift={(-13,0)}]
  \foreach \x in {0,...,11} {
    \draw[gray!70, line width=0.35pt] (\x,0) -- (\x,9);
  }
  \foreach \y in {0,...,9} {
    \draw[gray!70, line width=0.35pt] (0,\y) -- (11,\y);
  }

\draw[very thick] (0,1) rectangle (7,8);
\draw[very thick, orange] (4,1) rectangle (11,8);
\draw[very thick, orange] (9,1) rectangle (11,2);
\draw[very thick, orange] (9,2) rectangle (11,4);
\draw[very thick, orange] (9,4) rectangle (11,6);
\draw[very thick, orange] (8,4) rectangle (9,6);

\node at (2.5,3.5) {\normalsize 1};
\node at (4.5,3.5) {\normalsize 1};
\node at (2.5,5.5) {\normalsize 1};
\node at (4.5,5.5) {\normalsize 1};
\node at (5.5,8.5) {{1}};
\node at (7.5,8.5) {{1}};
\node at (7.5,2.5) {{1}};

\node at (1.5,2.5) {\textcolor{red}{$\times$}};
\node at (2.5,2.5) {\textcolor{red}{$\times$}};
\node at (3.5,2.5) {\textcolor{red}{$\times$}};
\node at (4.5,2.5) {\textcolor{red}{$\times$}};
\node at (5.5,2.5) {\textcolor{red}{$\times$}};

\node at (1.5,3.5) {\textcolor{red}{$\times$}};
\node at (3.5,3.5) {\textcolor{red}{$\times$}};
\node at (5.5,3.5) {\textcolor{red}{$\times$}};

\node at (1.5,4.5) {\textcolor{red}{$\times$}};
\node at (2.5,4.5) {\textcolor{red}{$\times$}};
\node at (3.5,4.5) {\textcolor{red}{$\times$}};
\node at (4.5,4.5) {\textcolor{red}{$\times$}};
\node at (5.5,4.5) {\textcolor{red}{$\times$}};

\node at (1.5,5.5) {\textcolor{red}{$\times$}};
\node at (3.5,5.5) {\textcolor{red}{$\times$}};
\node at (5.5,5.5) {\textcolor{red}{$\times$}};

\node at (1.5,6.5) {\textcolor{red}{$\times$}};
\node at (2.5,6.5) {\textcolor{red}{$\times$}};
\node at (3.5,6.5) {\textcolor{red}{$\times$}};
\node at (4.5,6.5) {\textcolor{red}{$\times$}};
\node at (5.5,6.5) {\textcolor{red}{$\times$}};

\node at (0.5,2.5) {\normalsize 1};
\node at (0.5,4.5) {\normalsize 1};
\node at (3.5,1.5) {\normalsize 1};
\node at (5.5,1.5) {\normalsize 1};
\node at (6.5,4.5) {\normalsize 1};
\node at (6.5,6.5) {\normalsize 1};
\node at (1.5,7.5) {\normalsize 1};
\node at (3.5,7.5) {\normalsize 1};

\node at (7.5,0.5) {\normalsize 1};
\node at (8.5,0.5) {\textcolor{red}{$\times$}};
\node at (6.5,0.5) {\textcolor{red}{$\times$}};
\node at (5.5,0.5) {\textcolor{red}{$\times$}};
\node at (4.5,0.5) {\textcolor{red}{$\times$}};
\node at (3.5,0.5) {\textcolor{red}{$\times$}};
\node at (2.5,0.5) {\textcolor{red}{$\times$}};
\node at (0.5,0.5) {\textcolor{red}{$\times$}};

\node at (0.5,1.5) {\textcolor{red}{$\times$}};
\node at (0.5,3.5) {\textcolor{red}{$\times$}};
\node at (0.5,5.5) {\textcolor{red}{$\times$}};
\node at (0.5,6.5) {\textcolor{red}{$\times$}};
\node at (0.5,7.5) {\textcolor{red}{$\times$}};

\node at (1.5,1.5) {\textcolor{red}{$\times$}};
\node at (2.5,1.5) {\textcolor{red}{$\times$}};
\node at (4.5,1.5) {\textcolor{red}{$\times$}};
\node at (6.5,1.5) {\textcolor{red}{$\times$}};

\node at (6.5,2.5) {\textcolor{red}{$\times$}};
\node at (6.5,3.5) {\textcolor{red}{$\times$}};
\node at (6.5,5.5) {\textcolor{red}{$\times$}};
\node at (6.5,7.5) {\textcolor{red}{$\times$}};

\node at (5.5,7.5) {\textcolor{red}{$\times$}};
\node at (4.5,7.5) {\textcolor{red}{$\times$}};
\node at (2.5,7.5) {\textcolor{red}{$\times$}};

\node at (0.5,8.5) {\textcolor{red}{$\times$}};
\node at (1.5,8.5) {\textcolor{red}{$\times$}};
\node at (2.5,8.5) {\textcolor{red}{$\times$}};
\node at (3.5,8.5) {\textcolor{red}{$\times$}};
\node at (4.5,8.5) {\textcolor{red}{$\times$}};

\node at (7.5,7.5) {\textcolor{red}{$\times$}};
\node at (7.5,6.5) {\textcolor{red}{$\times$}};
\node at (7.5,5.5) {\textcolor{red}{$\times$}};
\node at (7.5,4.5) {\textcolor{red}{$\times$}};
\node at (7.5,3.5) {\textcolor{red}{$\times$}};

\node at (6.5,8.5) {\textcolor{red}{$\times$}};
\node at (8.5,8.5) {\textcolor{red}{$\times$}};
\node at (8.5,7.5) {\textcolor{red}{$\times$}};
\node at (8.5,3.5) {\textcolor{red}{$\times$}};
\node at (8.5,2.5) {\textcolor{red}{$\times$}};
\node at (8.5,1.5) {\textcolor{red}{$\times$}};
\node at (7.5,1.5) {\textcolor{red}{$\times$}};

\node at (9.5,8.5) {\textcolor{red}{$\times$}};
\node at (9.5,6.5) {\textcolor{red}{$\times$}};
\node at (9.5,4.5) {\textcolor{red}{$\times$}};
\node at (10.5,3.5) {\textcolor{red}{$\times$}};
\node at (10.5,2.5) {\textcolor{red}{$\times$}};

\node at (9.5,7.5) {\normalsize 1};
\node at (8.5,6.5) {\textcolor{red}{$\times$}};
\node at (10.5,8.5) {\textcolor{red}{$\times$}};
\node at (10.5,7.5) {\textcolor{red}{$\times$}};
\node at (10.5,6.5) {\textcolor{red}{$\times$}};

\fill[yellow, opacity=0.3] (9,7) rectangle (10,8);
  
\end{scope}

\begin{scope}[shift={(0,0)}]
  \foreach \x in {0,...,11} {
    \draw[gray!70, line width=0.35pt] (\x,0) -- (\x,9);
  }
  \foreach \y in {0,...,9} {
    \draw[gray!70, line width=0.35pt] (0,\y) -- (11,\y);
  }

\draw[very thick] (0,1) rectangle (7,8);
\draw[very thick, orange] (4,1) rectangle (11,8);
\draw[very thick, orange] (9,1) rectangle (11,2);
\draw[very thick, orange] (9,2) rectangle (11,4);
\draw[very thick, orange] (9,4) rectangle (11,6);
\draw[very thick, orange] (8,4) rectangle (9,6);

\node at (2.5,3.5) {\normalsize 1};
\node at (4.5,3.5) {\normalsize 1};
\node at (2.5,5.5) {\normalsize 1};
\node at (4.5,5.5) {\normalsize 1};
\node at (5.5,8.5) {{1}};
\node at (7.5,8.5) {{1}};
\node at (7.5,2.5) {{1}};

\node at (1.5,2.5) {\textcolor{red}{$\times$}};
\node at (2.5,2.5) {\textcolor{red}{$\times$}};
\node at (3.5,2.5) {\textcolor{red}{$\times$}};
\node at (4.5,2.5) {\textcolor{red}{$\times$}};
\node at (5.5,2.5) {\textcolor{red}{$\times$}};

\node at (1.5,3.5) {\textcolor{red}{$\times$}};
\node at (3.5,3.5) {\textcolor{red}{$\times$}};
\node at (5.5,3.5) {\textcolor{red}{$\times$}};

\node at (1.5,4.5) {\textcolor{red}{$\times$}};
\node at (2.5,4.5) {\textcolor{red}{$\times$}};
\node at (3.5,4.5) {\textcolor{red}{$\times$}};
\node at (4.5,4.5) {\textcolor{red}{$\times$}};
\node at (5.5,4.5) {\textcolor{red}{$\times$}};

\node at (1.5,5.5) {\textcolor{red}{$\times$}};
\node at (3.5,5.5) {\textcolor{red}{$\times$}};
\node at (5.5,5.5) {\textcolor{red}{$\times$}};

\node at (1.5,6.5) {\textcolor{red}{$\times$}};
\node at (2.5,6.5) {\textcolor{red}{$\times$}};
\node at (3.5,6.5) {\textcolor{red}{$\times$}};
\node at (4.5,6.5) {\textcolor{red}{$\times$}};
\node at (5.5,6.5) {\textcolor{red}{$\times$}};

\node at (0.5,2.5) {\normalsize 1};
\node at (0.5,4.5) {\normalsize 1};
\node at (3.5,1.5) {\normalsize 1};
\node at (5.5,1.5) {\normalsize 1};
\node at (6.5,4.5) {\normalsize 1};
\node at (6.5,6.5) {\normalsize 1};
\node at (1.5,7.5) {\normalsize 1};
\node at (3.5,7.5) {\normalsize 1};

\node at (7.5,0.5) {\normalsize 1};
\node at (8.5,0.5) {\textcolor{red}{$\times$}};
\node at (6.5,0.5) {\textcolor{red}{$\times$}};
\node at (5.5,0.5) {\textcolor{red}{$\times$}};
\node at (4.5,0.5) {\textcolor{red}{$\times$}};
\node at (3.5,0.5) {\textcolor{red}{$\times$}};
\node at (2.5,0.5) {\textcolor{red}{$\times$}};
\node at (0.5,0.5) {\textcolor{red}{$\times$}};

\node at (0.5,1.5) {\textcolor{red}{$\times$}};
\node at (0.5,3.5) {\textcolor{red}{$\times$}};
\node at (0.5,5.5) {\textcolor{red}{$\times$}};
\node at (0.5,6.5) {\textcolor{red}{$\times$}};
\node at (0.5,7.5) {\textcolor{red}{$\times$}};

\node at (1.5,1.5) {\textcolor{red}{$\times$}};
\node at (2.5,1.5) {\textcolor{red}{$\times$}};
\node at (4.5,1.5) {\textcolor{red}{$\times$}};
\node at (6.5,1.5) {\textcolor{red}{$\times$}};

\node at (6.5,2.5) {\textcolor{red}{$\times$}};
\node at (6.5,3.5) {\textcolor{red}{$\times$}};
\node at (6.5,5.5) {\textcolor{red}{$\times$}};
\node at (6.5,7.5) {\textcolor{red}{$\times$}};

\node at (5.5,7.5) {\textcolor{red}{$\times$}};
\node at (4.5,7.5) {\textcolor{red}{$\times$}};
\node at (2.5,7.5) {\textcolor{red}{$\times$}};

\node at (0.5,8.5) {\textcolor{red}{$\times$}};
\node at (1.5,8.5) {\textcolor{red}{$\times$}};
\node at (2.5,8.5) {\textcolor{red}{$\times$}};
\node at (3.5,8.5) {\textcolor{red}{$\times$}};
\node at (4.5,8.5) {\textcolor{red}{$\times$}};

\node at (7.5,7.5) {\textcolor{red}{$\times$}};
\node at (7.5,6.5) {\textcolor{red}{$\times$}};
\node at (7.5,5.5) {\textcolor{red}{$\times$}};
\node at (7.5,4.5) {\textcolor{red}{$\times$}};
\node at (7.5,3.5) {\textcolor{red}{$\times$}};

\node at (6.5,8.5) {\textcolor{red}{$\times$}};
\node at (8.5,8.5) {\textcolor{red}{$\times$}};
\node at (8.5,7.5) {\textcolor{red}{$\times$}};
\node at (8.5,3.5) {\textcolor{red}{$\times$}};
\node at (8.5,2.5) {\textcolor{red}{$\times$}};
\node at (8.5,1.5) {\textcolor{red}{$\times$}};
\node at (7.5,1.5) {\textcolor{red}{$\times$}};

\node at (9.5,8.5) {\textcolor{red}{$\times$}};
\node at (9.5,6.5) {\textcolor{red}{$\times$}};
\node at (9.5,4.5) {\textcolor{red}{$\times$}};
\node at (10.5,3.5) {\textcolor{red}{$\times$}};
\node at (10.5,2.5) {\textcolor{red}{$\times$}};

\node at (9.5,7.5) {\normalsize 1};
\node at (8.5,6.5) {\textcolor{red}{$\times$}};
\node at (10.5,8.5) {\textcolor{red}{$\times$}};
\node at (10.5,7.5) {\textcolor{red}{$\times$}};
\node at (10.5,6.5) {\textcolor{red}{$\times$}};

\node at (9.5,2.5) {\textcolor{blue}{$\times$}};
\node at (9.5,0.5) {\textcolor{blue}{$\times$}};
\node at (10.5,0.5) {\textcolor{blue}{$\times$}};
\node at (9.5,3.5) {\textcolor{blue}{1}};
\node at (10.5,4.5) {\textcolor{blue}{$\times$}};
\node at (8.5,4.5) {\textcolor{blue}{$\times$}};
\node at (8.5,5.5) {\textcolor{darkgreen}{1}};
\node at (9.5,5.5) {\textcolor{darkgreen}{$\times$}};
\node at (10.5,5.5) {\textcolor{orange}{1}};

\fill[yellow, opacity=0.3] (9,7) rectangle (10,8);
  
\end{scope}

\end{tikzpicture}
\caption{$H$ (black), $H(3,1)$ (green) and $H(4,0)$ (orange).}
\label{fig:10}
\end{figure}

Next, we consider $ H(2,-3) $ (see Figure~\ref{fig:11}). By Observation~\ref{obs:2}\,(ii) and~(iv), the combined rows 1 and 2 must contain exactly four squares of color~1, located in columns 1, 3, 5, and 7. Due to the distance constraint, we have that $ (5,1) $ and $ (7,1) $ must be assigned color~1 (see Figure~\ref{fig:11}). We then consider $ H(5,-2) $ and partition its squares into four parts as shown in Figure~\ref{fig:11}. By Observation~\ref{obs:2}\,(ii) and~(iii), each part must contain exactly one square of color~1. Since $ (5,3) $ is the only square in its part, it must be assigned color~1. Moreover, since the bottom part contains either $ (6,1) $ or $ (7,1) $ as a square of color~1, we have that $ (7,2) $ cannot be assigned color~1. It follows that $ (7,3) $ must be a square of color~1 and thus $ (7,5) $ must also be assigned color~1.

\begin{figure}[ht]
\centering
\begin{tikzpicture}[scale=0.30]

\begin{scope}[shift={(-28,0)}]
  \foreach \x in {0,...,12} {
    \draw[gray!70, line width=0.35pt] (\x,0) -- (\x,11);
  }
  \foreach \y in {0,...,11} {
    \draw[gray!70, line width=0.35pt] (0,\y) -- (12,\y);
  }

\draw[very thick] (0,3) rectangle (7,10);
\draw[very thick, blue] (2,0) rectangle (9,7);

\node at (2.5,5.5) {\normalsize 1};
\node at (4.5,5.5) {\normalsize 1};
\node at (2.5,7.5) {\normalsize 1};
\node at (4.5,7.5) {\normalsize 1};
\node at (5.5,10.5) {1};
\node at (7.5,10.5) {1};
\node at (7.5,4.5) {1};

\node at (1.5,4.5) {\textcolor{red}{$\times$}};
\node at (2.5,4.5) {\textcolor{red}{$\times$}};
\node at (3.5,4.5) {\textcolor{red}{$\times$}};
\node at (4.5,4.5) {\textcolor{red}{$\times$}};
\node at (5.5,4.5) {\textcolor{red}{$\times$}};

\node at (1.5,5.5) {\textcolor{red}{$\times$}};
\node at (3.5,5.5) {\textcolor{red}{$\times$}};
\node at (5.5,5.5) {\textcolor{red}{$\times$}};

\node at (1.5,6.5) {\textcolor{red}{$\times$}};
\node at (2.5,6.5) {\textcolor{red}{$\times$}};
\node at (3.5,6.5) {\textcolor{red}{$\times$}};
\node at (4.5,6.5) {\textcolor{red}{$\times$}};
\node at (5.5,6.5) {\textcolor{red}{$\times$}};

\node at (1.5,7.5) {\textcolor{red}{$\times$}};
\node at (3.5,7.5) {\textcolor{red}{$\times$}};
\node at (5.5,7.5) {\textcolor{red}{$\times$}};

\node at (1.5,8.5) {\textcolor{red}{$\times$}};
\node at (2.5,8.5) {\textcolor{red}{$\times$}};
\node at (3.5,8.5) {\textcolor{red}{$\times$}};
\node at (4.5,8.5) {\textcolor{red}{$\times$}};
\node at (5.5,8.5) {\textcolor{red}{$\times$}};

\node at (0.5,4.5) {\normalsize 1};
\node at (0.5,6.5) {\normalsize 1};
\node at (3.5,3.5) {\normalsize 1};
\node at (5.5,3.5) {\normalsize 1};
\node at (6.5,6.5) {\normalsize 1};
\node at (6.5,8.5) {\normalsize 1};
\node at (1.5,9.5) {\normalsize 1};
\node at (3.5,9.5) {\normalsize 1};

\node at (7.5,2.5) {\normalsize 1};
\node at (8.5,2.5) {\textcolor{red}{$\times$}};
\node at (6.5,2.5) {\textcolor{red}{$\times$}};
\node at (5.5,2.5) {\textcolor{red}{$\times$}};
\node at (4.5,2.5) {\textcolor{red}{$\times$}};
\node at (3.5,2.5) {\textcolor{red}{$\times$}};
\node at (2.5,2.5) {\textcolor{red}{$\times$}};
\node at (0.5,2.5) {\textcolor{red}{$\times$}};

\node at (0.5,3.5) {\textcolor{red}{$\times$}};
\node at (0.5,5.5) {\textcolor{red}{$\times$}};
\node at (0.5,7.5) {\textcolor{red}{$\times$}};
\node at (0.5,8.5) {\textcolor{red}{$\times$}};
\node at (0.5,9.5) {\textcolor{red}{$\times$}};

\node at (1.5,3.5) {\textcolor{red}{$\times$}};
\node at (2.5,3.5) {\textcolor{red}{$\times$}};
\node at (4.5,3.5) {\textcolor{red}{$\times$}};
\node at (6.5,3.5) {\textcolor{red}{$\times$}};

\node at (6.5,4.5) {\textcolor{red}{$\times$}};
\node at (6.5,5.5) {\textcolor{red}{$\times$}};
\node at (6.5,7.5) {\textcolor{red}{$\times$}};
\node at (6.5,9.5) {\textcolor{red}{$\times$}};

\node at (5.5,9.5) {\textcolor{red}{$\times$}};
\node at (4.5,9.5) {\textcolor{red}{$\times$}};
\node at (2.5,9.5) {\textcolor{red}{$\times$}};

\node at (0.5,10.5) {\textcolor{red}{$\times$}};
\node at (1.5,10.5) {\textcolor{red}{$\times$}};
\node at (2.5,10.5) {\textcolor{red}{$\times$}};
\node at (3.5,10.5) {\textcolor{red}{$\times$}};
\node at (4.5,10.5) {\textcolor{red}{$\times$}};

\node at (7.5,9.5) {\textcolor{red}{$\times$}};
\node at (7.5,8.5) {\textcolor{red}{$\times$}};
\node at (7.5,7.5) {\textcolor{red}{$\times$}};
\node at (7.5,6.5) {\textcolor{red}{$\times$}};
\node at (7.5,5.5) {\textcolor{red}{$\times$}};

\node at (6.5,10.5) {\textcolor{red}{$\times$}};
\node at (8.5,10.5) {\textcolor{red}{$\times$}};
\node at (8.5,9.5) {\textcolor{red}{$\times$}};
\node at (8.5,5.5) {\textcolor{red}{$\times$}};
\node at (8.5,4.5) {\textcolor{red}{$\times$}};
\node at (8.5,3.5) {\textcolor{red}{$\times$}};
\node at (7.5,3.5) {\textcolor{red}{$\times$}};

\node at (9.5,10.5) {\textcolor{red}{$\times$}};
\node at (9.5,8.5) {\textcolor{red}{$\times$}};
\node at (9.5,6.5) {\textcolor{red}{$\times$}};
\node at (10.5,5.5) {\textcolor{red}{$\times$}};
\node at (10.5,4.5) {\textcolor{red}{$\times$}};

\node at (9.5,9.5) {\normalsize 1};
\node at (8.5,8.5) {\textcolor{red}{$\times$}};
\node at (10.5,10.5) {\textcolor{red}{$\times$}};
\node at (10.5,9.5) {\textcolor{red}{$\times$}};
\node at (10.5,8.5) {\textcolor{red}{$\times$}};

\node at (9.5,4.5) {\textcolor{red}{$\times$}};
\node at (9.5,2.5) {\textcolor{red}{$\times$}};
\node at (10.5,2.5) {\textcolor{red}{$\times$}};
\node at (9.5,5.5) {1};
\node at (10.5,6.5) {\textcolor{red}{$\times$}};
\node at (8.5,6.5) {\textcolor{red}{$\times$}};
\node at (8.5,7.5) {1};
\node at (9.5,7.5) {\textcolor{red}{$\times$}};
\node at (10.5,7.5) {1};

\node at (7.5,1.5) {\textcolor{red}{$\times$}};
\node at (6.5,1.5) {\textcolor{red}{$\times$}};
\node at (8.5,1.5) {\textcolor{red}{$\times$}};

\fill[yellow, opacity=0.3] (9,9) rectangle (10,10);

\end{scope}

\begin{scope}[shift={(-14,0)}]
  \foreach \x in {0,...,12} {
    \draw[gray!70, line width=0.35pt] (\x,0) -- (\x,11);
  }
  \foreach \y in {0,...,11} {
    \draw[gray!70, line width=0.35pt] (0,\y) -- (12,\y);
  }

\draw[very thick] (0,3) rectangle (7,10);
\draw[very thick, blue] (2,0) rectangle (9,7);

\node at (2.5,5.5) {\normalsize 1};
\node at (4.5,5.5) {\normalsize 1};
\node at (2.5,7.5) {\normalsize 1};
\node at (4.5,7.5) {\normalsize 1};
\node at (5.5,10.5) {1};
\node at (7.5,10.5) {1};
\node at (7.5,4.5) {1};

\node at (1.5,4.5) {\textcolor{red}{$\times$}};
\node at (2.5,4.5) {\textcolor{red}{$\times$}};
\node at (3.5,4.5) {\textcolor{red}{$\times$}};
\node at (4.5,4.5) {\textcolor{red}{$\times$}};
\node at (5.5,4.5) {\textcolor{red}{$\times$}};

\node at (1.5,5.5) {\textcolor{red}{$\times$}};
\node at (3.5,5.5) {\textcolor{red}{$\times$}};
\node at (5.5,5.5) {\textcolor{red}{$\times$}};

\node at (1.5,6.5) {\textcolor{red}{$\times$}};
\node at (2.5,6.5) {\textcolor{red}{$\times$}};
\node at (3.5,6.5) {\textcolor{red}{$\times$}};
\node at (4.5,6.5) {\textcolor{red}{$\times$}};
\node at (5.5,6.5) {\textcolor{red}{$\times$}};

\node at (1.5,7.5) {\textcolor{red}{$\times$}};
\node at (3.5,7.5) {\textcolor{red}{$\times$}};
\node at (5.5,7.5) {\textcolor{red}{$\times$}};

\node at (1.5,8.5) {\textcolor{red}{$\times$}};
\node at (2.5,8.5) {\textcolor{red}{$\times$}};
\node at (3.5,8.5) {\textcolor{red}{$\times$}};
\node at (4.5,8.5) {\textcolor{red}{$\times$}};
\node at (5.5,8.5) {\textcolor{red}{$\times$}};

\node at (0.5,4.5) {\normalsize 1};
\node at (0.5,6.5) {\normalsize 1};
\node at (3.5,3.5) {\normalsize 1};
\node at (5.5,3.5) {\normalsize 1};
\node at (6.5,6.5) {\normalsize 1};
\node at (6.5,8.5) {\normalsize 1};
\node at (1.5,9.5) {\normalsize 1};
\node at (3.5,9.5) {\normalsize 1};

\node at (7.5,2.5) {\normalsize 1};
\node at (8.5,2.5) {\textcolor{red}{$\times$}};
\node at (6.5,2.5) {\textcolor{red}{$\times$}};
\node at (5.5,2.5) {\textcolor{red}{$\times$}};
\node at (4.5,2.5) {\textcolor{red}{$\times$}};
\node at (3.5,2.5) {\textcolor{red}{$\times$}};
\node at (2.5,2.5) {\textcolor{red}{$\times$}};
\node at (0.5,2.5) {\textcolor{red}{$\times$}};

\node at (0.5,3.5) {\textcolor{red}{$\times$}};
\node at (0.5,5.5) {\textcolor{red}{$\times$}};
\node at (0.5,7.5) {\textcolor{red}{$\times$}};
\node at (0.5,8.5) {\textcolor{red}{$\times$}};
\node at (0.5,9.5) {\textcolor{red}{$\times$}};

\node at (1.5,3.5) {\textcolor{red}{$\times$}};
\node at (2.5,3.5) {\textcolor{red}{$\times$}};
\node at (4.5,3.5) {\textcolor{red}{$\times$}};
\node at (6.5,3.5) {\textcolor{red}{$\times$}};

\node at (6.5,4.5) {\textcolor{red}{$\times$}};
\node at (6.5,5.5) {\textcolor{red}{$\times$}};
\node at (6.5,7.5) {\textcolor{red}{$\times$}};
\node at (6.5,9.5) {\textcolor{red}{$\times$}};

\node at (5.5,9.5) {\textcolor{red}{$\times$}};
\node at (4.5,9.5) {\textcolor{red}{$\times$}};
\node at (2.5,9.5) {\textcolor{red}{$\times$}};

\node at (0.5,10.5) {\textcolor{red}{$\times$}};
\node at (1.5,10.5) {\textcolor{red}{$\times$}};
\node at (2.5,10.5) {\textcolor{red}{$\times$}};
\node at (3.5,10.5) {\textcolor{red}{$\times$}};
\node at (4.5,10.5) {\textcolor{red}{$\times$}};

\node at (7.5,9.5) {\textcolor{red}{$\times$}};
\node at (7.5,8.5) {\textcolor{red}{$\times$}};
\node at (7.5,7.5) {\textcolor{red}{$\times$}};
\node at (7.5,6.5) {\textcolor{red}{$\times$}};
\node at (7.5,5.5) {\textcolor{red}{$\times$}};

\node at (6.5,10.5) {\textcolor{red}{$\times$}};
\node at (8.5,10.5) {\textcolor{red}{$\times$}};
\node at (8.5,9.5) {\textcolor{red}{$\times$}};
\node at (8.5,5.5) {\textcolor{red}{$\times$}};
\node at (8.5,4.5) {\textcolor{red}{$\times$}};
\node at (8.5,3.5) {\textcolor{red}{$\times$}};
\node at (7.5,3.5) {\textcolor{red}{$\times$}};

\node at (9.5,10.5) {\textcolor{red}{$\times$}};
\node at (9.5,8.5) {\textcolor{red}{$\times$}};
\node at (9.5,6.5) {\textcolor{red}{$\times$}};
\node at (10.5,5.5) {\textcolor{red}{$\times$}};
\node at (10.5,4.5) {\textcolor{red}{$\times$}};

\node at (9.5,9.5) {\normalsize 1};
\node at (8.5,8.5) {\textcolor{red}{$\times$}};
\node at (10.5,10.5) {\textcolor{red}{$\times$}};
\node at (10.5,9.5) {\textcolor{red}{$\times$}};
\node at (10.5,8.5) {\textcolor{red}{$\times$}};

\node at (9.5,4.5) {\textcolor{red}{$\times$}};
\node at (9.5,2.5) {\textcolor{red}{$\times$}};
\node at (10.5,2.5) {\textcolor{red}{$\times$}};
\node at (9.5,5.5) {1};
\node at (10.5,6.5) {\textcolor{red}{$\times$}};
\node at (8.5,6.5) {\textcolor{red}{$\times$}};
\node at (8.5,7.5) {1};
\node at (9.5,7.5) {\textcolor{red}{$\times$}};
\node at (10.5,7.5) {1};

\node at (7.5,1.5) {\textcolor{red}{$\times$}};
\node at (6.5,1.5) {\textcolor{red}{$\times$}};
\node at (8.5,1.5) {\textcolor{red}{$\times$}};

\node at (8.5,0.5) {\textcolor{blue}{1}};
\node at (7.5,0.5) {\textcolor{blue}{$\times$}};
\node at (9.5,0.5) {\textcolor{blue}{$\times$}};
\node at (9.5,1.5) {\textcolor{blue}{$\times$}};
\node at (6.5,0.5) {\textcolor{blue}{1}};
\node at (5.5,0.5) {\textcolor{blue}{$\times$}};
\node at (5.5,1.5) {\textcolor{blue}{$\times$}};

\fill[yellow, opacity=0.3] (9,9) rectangle (10,10);

\end{scope}

\begin{scope}[shift={(0,0)}]
  \foreach \x in {0,...,12} {
    \draw[gray!70, line width=0.35pt] (\x,0) -- (\x,11);
  }
  \foreach \y in {0,...,11} {
    \draw[gray!70, line width=0.35pt] (0,\y) -- (12,\y);
  }

\draw[very thick] (0,3) rectangle (7,10);
\draw[very thick, darkgreen] (5,1) rectangle (12,8);
\draw[very thick, darkgreen] (10,1) rectangle (12,2);
\draw[very thick, darkgreen] (10,2) rectangle (12,4);
\draw[very thick, darkgreen] (11,4) rectangle (12,6);
\draw[very thick, darkgreen] (9,3) rectangle (10,4);

\node at (2.5,5.5) {\normalsize 1};
\node at (4.5,5.5) {\normalsize 1};
\node at (2.5,7.5) {\normalsize 1};
\node at (4.5,7.5) {\normalsize 1};
\node at (5.5,10.5) {1};
\node at (7.5,10.5) {1};
\node at (7.5,4.5) {1};

\node at (1.5,4.5) {\textcolor{red}{$\times$}};
\node at (2.5,4.5) {\textcolor{red}{$\times$}};
\node at (3.5,4.5) {\textcolor{red}{$\times$}};
\node at (4.5,4.5) {\textcolor{red}{$\times$}};
\node at (5.5,4.5) {\textcolor{red}{$\times$}};

\node at (1.5,5.5) {\textcolor{red}{$\times$}};
\node at (3.5,5.5) {\textcolor{red}{$\times$}};
\node at (5.5,5.5) {\textcolor{red}{$\times$}};

\node at (1.5,6.5) {\textcolor{red}{$\times$}};
\node at (2.5,6.5) {\textcolor{red}{$\times$}};
\node at (3.5,6.5) {\textcolor{red}{$\times$}};
\node at (4.5,6.5) {\textcolor{red}{$\times$}};
\node at (5.5,6.5) {\textcolor{red}{$\times$}};

\node at (1.5,7.5) {\textcolor{red}{$\times$}};
\node at (3.5,7.5) {\textcolor{red}{$\times$}};
\node at (5.5,7.5) {\textcolor{red}{$\times$}};

\node at (1.5,8.5) {\textcolor{red}{$\times$}};
\node at (2.5,8.5) {\textcolor{red}{$\times$}};
\node at (3.5,8.5) {\textcolor{red}{$\times$}};
\node at (4.5,8.5) {\textcolor{red}{$\times$}};
\node at (5.5,8.5) {\textcolor{red}{$\times$}};

\node at (0.5,4.5) {\normalsize 1};
\node at (0.5,6.5) {\normalsize 1};
\node at (3.5,3.5) {\normalsize 1};
\node at (5.5,3.5) {\normalsize 1};
\node at (6.5,6.5) {\normalsize 1};
\node at (6.5,8.5) {\normalsize 1};
\node at (1.5,9.5) {\normalsize 1};
\node at (3.5,9.5) {\normalsize 1};

\node at (7.5,2.5) {\normalsize 1};
\node at (8.5,2.5) {\textcolor{red}{$\times$}};
\node at (6.5,2.5) {\textcolor{red}{$\times$}};
\node at (5.5,2.5) {\textcolor{red}{$\times$}};
\node at (4.5,2.5) {\textcolor{red}{$\times$}};
\node at (3.5,2.5) {\textcolor{red}{$\times$}};
\node at (2.5,2.5) {\textcolor{red}{$\times$}};
\node at (0.5,2.5) {\textcolor{red}{$\times$}};

\node at (0.5,3.5) {\textcolor{red}{$\times$}};
\node at (0.5,5.5) {\textcolor{red}{$\times$}};
\node at (0.5,7.5) {\textcolor{red}{$\times$}};
\node at (0.5,8.5) {\textcolor{red}{$\times$}};
\node at (0.5,9.5) {\textcolor{red}{$\times$}};

\node at (1.5,3.5) {\textcolor{red}{$\times$}};
\node at (2.5,3.5) {\textcolor{red}{$\times$}};
\node at (4.5,3.5) {\textcolor{red}{$\times$}};
\node at (6.5,3.5) {\textcolor{red}{$\times$}};

\node at (6.5,4.5) {\textcolor{red}{$\times$}};
\node at (6.5,5.5) {\textcolor{red}{$\times$}};
\node at (6.5,7.5) {\textcolor{red}{$\times$}};
\node at (6.5,9.5) {\textcolor{red}{$\times$}};

\node at (5.5,9.5) {\textcolor{red}{$\times$}};
\node at (4.5,9.5) {\textcolor{red}{$\times$}};
\node at (2.5,9.5) {\textcolor{red}{$\times$}};

\node at (0.5,10.5) {\textcolor{red}{$\times$}};
\node at (1.5,10.5) {\textcolor{red}{$\times$}};
\node at (2.5,10.5) {\textcolor{red}{$\times$}};
\node at (3.5,10.5) {\textcolor{red}{$\times$}};
\node at (4.5,10.5) {\textcolor{red}{$\times$}};

\node at (7.5,9.5) {\textcolor{red}{$\times$}};
\node at (7.5,8.5) {\textcolor{red}{$\times$}};
\node at (7.5,7.5) {\textcolor{red}{$\times$}};
\node at (7.5,6.5) {\textcolor{red}{$\times$}};
\node at (7.5,5.5) {\textcolor{red}{$\times$}};

\node at (6.5,10.5) {\textcolor{red}{$\times$}};
\node at (8.5,10.5) {\textcolor{red}{$\times$}};
\node at (8.5,9.5) {\textcolor{red}{$\times$}};
\node at (8.5,5.5) {\textcolor{red}{$\times$}};
\node at (8.5,4.5) {\textcolor{red}{$\times$}};
\node at (8.5,3.5) {\textcolor{red}{$\times$}};
\node at (7.5,3.5) {\textcolor{red}{$\times$}};

\node at (9.5,10.5) {\textcolor{red}{$\times$}};
\node at (9.5,8.5) {\textcolor{red}{$\times$}};
\node at (9.5,6.5) {\textcolor{red}{$\times$}};
\node at (10.5,5.5) {\textcolor{red}{$\times$}};
\node at (10.5,4.5) {\textcolor{red}{$\times$}};

\node at (9.5,9.5) {\normalsize 1};
\node at (8.5,8.5) {\textcolor{red}{$\times$}};
\node at (10.5,10.5) {\textcolor{red}{$\times$}};
\node at (10.5,9.5) {\textcolor{red}{$\times$}};
\node at (10.5,8.5) {\textcolor{red}{$\times$}};

\node at (9.5,4.5) {\textcolor{red}{$\times$}};
\node at (9.5,2.5) {\textcolor{red}{$\times$}};
\node at (10.5,2.5) {\textcolor{red}{$\times$}};
\node at (9.5,5.5) {1};
\node at (10.5,6.5) {\textcolor{red}{$\times$}};
\node at (8.5,6.5) {\textcolor{red}{$\times$}};
\node at (8.5,7.5) {1};
\node at (9.5,7.5) {\textcolor{red}{$\times$}};
\node at (10.5,7.5) {1};

\node at (7.5,1.5) {\textcolor{red}{$\times$}};
\node at (6.5,1.5) {\textcolor{red}{$\times$}};
\node at (8.5,1.5) {\textcolor{red}{$\times$}};
\node at (11.5,8.5) {\textcolor{red}{$\times$}};
\node at (11.5,7.5) {\textcolor{red}{$\times$}};
\node at (11.5,6.5) {\textcolor{red}{$\times$}};

\node at (8.5,0.5) {{1}};
\node at (7.5,0.5) {\textcolor{red}{$\times$}};
\node at (9.5,0.5) {\textcolor{red}{$\times$}};
\node at (9.5,1.5) {\textcolor{red}{$\times$}};
\node at (6.5,0.5) {{1}};
\node at (5.5,0.5) {\textcolor{red}{$\times$}};
\node at (5.5,1.5) {\textcolor{red}{$\times$}};

\node at (9.5,3.5) {\textcolor{blue}{1}};
\node at (10.5,3.5) {\textcolor{blue}{$\times$}};
\node at (11.5,2.5) {\textcolor{blue}{$\times$}};
\node at (11.5,3.5) {\textcolor{darkgreen}{1}};
\node at (11.5,4.5) {\textcolor{darkgreen}{$\times$}};
\node at (11.5,5.5) {\textcolor{orange}{1}};

\fill[yellow, opacity=0.3] (9,9) rectangle (10,10);
  
\end{scope}

\end{tikzpicture}
\caption{$H$ (black), $H(2,-3)$ (blue) and $H(5,-2)$ (green).}

\label{fig:11}
\end{figure}
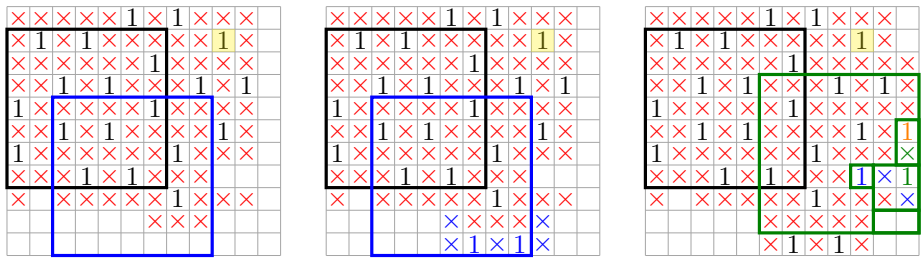

Consider $ H(7,-2) $ (see Figure~\ref{fig:12}). Observe that this subgraph is a $ P_7 \boxtimes P_7 $ graph that contains the pattern
\[
\begin{bmatrix}
1 & x & 1\\
x & x & x\\
1 & x & 1
\end{bmatrix}
\]
in the middle, which is the same as in $ G $. Since we have already shown that there are only two possible positions for squares of color~1 in $ G $, as shown in Figure~\ref{fig:8}, it follows that there are also only two possible positions for squares of color~1 in $ H(7,-2) $. Among these, only pattern (i) matches the observed position of squares of color 1 in $ H(7,-2) $. Hence, the pattern of squares of color~1 in $ H(7,-2) $ must be the same as pattern (i) (see Figure~\ref{fig:12}).
\end{proof}

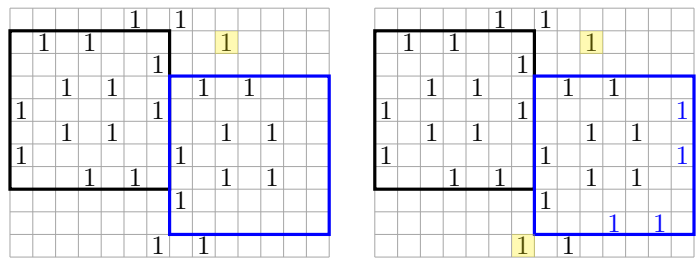
\begin{figure}[ht]
\centering
\begin{tikzpicture}[scale=0.30]
\begin{scope}[shift={(-16,0)}]
  \foreach \x in {0,...,14} {
    \draw[gray!70, line width=0.35pt] (\x,0) -- (\x,11);
  }
  \foreach \y in {0,...,11} {
    \draw[gray!70, line width=0.35pt] (0,\y) -- (14,\y);
  }

\draw[very thick] (0,3) rectangle (7,10);
\draw[very thick, blue] (7,1) rectangle (14,8);

\node at (2.5,5.5) {\normalsize 1};
\node at (4.5,5.5) {\normalsize 1};
\node at (2.5,7.5) {\normalsize 1};
\node at (4.5,7.5) {\normalsize 1};
\node at (5.5,10.5) {1};
\node at (7.5,10.5) {1};
\node at (7.5,4.5) {1};

\node at (9.5,5.5) {1};
\node at (8.5,7.5) {1};
\node at (10.5,7.5) {1};
\node at (0.5,4.5) {\normalsize 1};
\node at (0.5,6.5) {\normalsize 1};
\node at (3.5,3.5) {\normalsize 1};
\node at (5.5,3.5) {\normalsize 1};
\node at (6.5,6.5) {\normalsize 1};
\node at (6.5,8.5) {\normalsize 1};
\node at (1.5,9.5) {\normalsize 1};
\node at (3.5,9.5) {\normalsize 1};

\node at (7.5,2.5) {\normalsize 1};
\node at (9.5,9.5) {\normalsize 1};
\node at (8.5,0.5) {{1}};
\node at (6.5,0.5) {{1}};
\node at (9.5,3.5) {{1}};
\node at (11.5,3.5) {{1}};
\node at (11.5,5.5) {{1}};

\fill[yellow, opacity=0.3] (9,9) rectangle (10,10);

\end{scope}

\begin{scope}[shift={(0,0)}]
  \foreach \x in {0,...,14} {
    \draw[gray!70, line width=0.35pt] (\x,0) -- (\x,11);
  }
  \foreach \y in {0,...,11} {
    \draw[gray!70, line width=0.35pt] (0,\y) -- (14,\y);
  }

\draw[very thick] (0,3) rectangle (7,10);
\draw[very thick, blue] (7,1) rectangle (14,8);

\node at (2.5,5.5) {\normalsize 1};
\node at (4.5,5.5) {\normalsize 1};
\node at (2.5,7.5) {\normalsize 1};
\node at (4.5,7.5) {\normalsize 1};
\node at (5.5,10.5) {1};
\node at (7.5,10.5) {1};
\node at (7.5,4.5) {1};

\node at (0.5,4.5) {\normalsize 1};
\node at (0.5,6.5) {\normalsize 1};
\node at (3.5,3.5) {\normalsize 1};
\node at (5.5,3.5) {\normalsize 1};
\node at (6.5,6.5) {\normalsize 1};
\node at (6.5,8.5) {\normalsize 1};
\node at (1.5,9.5) {\normalsize 1};
\node at (3.5,9.5) {\normalsize 1};

\node at (7.5,2.5) {\normalsize 1};
\node at (9.5,9.5) {\normalsize 1};
\node at (9.5,5.5) {1};

\node at (8.5,7.5) {1};
\node at (10.5,7.5) {1};
\node at (8.5,0.5) {{1}};
\node at (6.5,0.5) {{1}};
\fill[yellow, opacity=0.3] (6,0) rectangle (7,1);

\node at (9.5,3.5) {{1}};
\node at (11.5,3.5) {{1}};
\node at (11.5,5.5) {{1}};

\node at (10.5,1.5) {\textcolor{blue}{1}};
\node at (12.5,1.5) {\textcolor{blue}{1}};
\node at (13.5,4.5) {\textcolor{blue}{1}};
\node at (13.5,6.5) {\textcolor{blue}{1}};

\fill[yellow, opacity=0.3] (9,9) rectangle (10,10);
  
\end{scope}

\end{tikzpicture}
\caption{ $ H $ (black) and $H(7,-2)$ (blue).}

\label{fig:12}
\end{figure}

By this claim, we have that $ G(7,-2) $ has pattern~(i) and square $(7,-2)$ of $G$ is of color~1. Since $ G^{(1)} $ has pattern~(i) and square $(10,7)$ of $ G^{(1)} $, which is square $(7,-2)$ of $G$, is of color~1, we can also apply this claim with $ H = G^{(1)} $. As a result, $ G(-2,-7) $ has pattern~(i) and square $(-2,1)$ of $G$ is of color 1. Next, we apply this claim with $ H = G^{(2)} $ and obtain that $ G(-7,2) $ has pattern~(i) and square $(1,10)$ of $G$ is of color~1. Finally, we apply this claim with $ H = G^{(3)} $. Hence, $ G(2,7) $ has pattern~(i) (see Figure~\ref{fig:13}).

Moreover, we can apply this claim with $H =  G(-2, -7) $. Similarly, we can apply the claim with $ H= G^{(1)}(-7, 2) $, $H = G^{(2)}(2, 7) $, and $H = G^{(3)}(7, -2) $ (see Figure~\ref{fig:13}). Therefore, this pattern of squares of color~1 extends repeatedly throughout $ P_\infty \boxtimes P_\infty $. In this pattern, we can see that it is impossible to assign any additional squares of color~1.

\begin{figure}[ht]
\centering
\begin{tikzpicture}[scale=0.30]
\begin{scope}[shift={(-29,0)}]
  \foreach \x in {0,...,27} {
    \draw[gray!70, line width=0.35pt] (\x,0) -- (\x,27);
  }
  \foreach \y in {0,...,27} {
    \draw[gray!70, line width=0.35pt] (0,\y) -- (27,\y);
  }

\draw[very thick] (10,10) rectangle (17,17);
\draw[very thick, blue] (17,8) rectangle (24,15);
\draw[very thick, blue] (8,3) rectangle (15,10);
\draw[very thick, blue] (3,12) rectangle (10,19);
\draw[very thick, blue] (12,17) rectangle (19,24);

\node at (12.5,12.5) {\normalsize 1};
\node at (14.5,12.5) {\normalsize 1};
\node at (12.5,14.5) {\normalsize 1};
\node at (14.5,14.5) {\normalsize 1};
\node at (15.5,17.5) {1};
\node at (17.5,17.5) {1};
\node at (17.5,11.5) {1};

\node at (10.5,11.5) {\normalsize 1};
\node at (10.5,13.5) {\normalsize 1};
\node at (13.5,10.5) {\normalsize 1};
\node at (15.5,10.5) {\normalsize 1};
\node at (16.5,13.5) {\normalsize 1};
\node at (16.5,15.5) {\normalsize 1};
\node at (11.5,16.5) {\normalsize 1};
\node at (13.5,16.5) {\normalsize 1};

\node at (17.5,9.5) {\normalsize 1};
\node at (19.5,16.5) {\normalsize 1};
\node at (19.5,12.5) {1};

\node at (18.5,14.5) {1};
\node at (20.5,14.5) {1};
\node at (18.5,7.5) {1};
\node at (16.5,7.5) {1};

\node at (19.5,10.5) {1};
\node at (21.5,10.5) {1};
\node at (21.5,12.5) {1};

\node at (20.5,8.5) {1};
\node at (22.5,8.5) {1};
\node at (23.5,11.5) {1};
\node at (23.5,13.5) {1};

\node at (8.5,4.5) {\textcolor{blue}{1}};
\node at (8.5,6.5) {\textcolor{blue}{1}};
\node at (9.5,9.5) {\textcolor{blue}{1}};
\node at (11.5,9.5) {\textcolor{blue}{1}};
\node at (14.5,8.5) {\textcolor{blue}{1}};
\node at (14.5,6.5) {\textcolor{blue}{1}};
\node at (13.5,3.5) {\textcolor{blue}{1}};
\node at (11.5,3.5) {\textcolor{blue}{1}};
\node at (10.5,5.5) {\textcolor{blue}{1}};
\node at (10.5,7.5) {\textcolor{blue}{1}};
\node at (12.5,5.5) {\textcolor{blue}{1}};
\node at (12.5,7.5) {\textcolor{blue}{1}};
\node at (7.5,10.5) {\textcolor{blue}{1}};

\node at (3.5,13.5) {\textcolor{darkgreen}{1}};
\node at (3.5,15.5) {\textcolor{darkgreen}{1}};
\node at (4.5,18.5) {\textcolor{darkgreen}{1}};
\node at (6.5,18.5) {\textcolor{darkgreen}{1}};
\node at (9.5,17.5) {\textcolor{darkgreen}{1}};
\node at (9.5,15.5) {\textcolor{darkgreen}{1}};
\node at (8.5,12.5) {\textcolor{darkgreen}{1}};
\node at (6.5,12.5) {\textcolor{darkgreen}{1}};
\node at (5.5,14.5) {\textcolor{darkgreen}{1}};
\node at (5.5,16.5) {\textcolor{darkgreen}{1}};
\node at (7.5,14.5) {\textcolor{darkgreen}{1}};
\node at (7.5,16.5) {\textcolor{darkgreen}{1}};
\node at (10.5,19.5) {\textcolor{darkgreen}{1}};

\node at (12.5,18.5) {\textcolor{orange}{1}};
\node at (12.5,20.5) {\textcolor{orange}{1}};
\node at (13.5,23.5) {\textcolor{orange}{1}};
\node at (15.5,23.5) {\textcolor{orange}{1}};
\node at (18.5,22.5) {\textcolor{orange}{1}};
\node at (18.5,20.5) {\textcolor{orange}{1}};
\node at (14.5,19.5) {\textcolor{orange}{1}};
\node at (14.5,21.5) {\textcolor{orange}{1}};
\node at (16.5,19.5) {\textcolor{orange}{1}};
\node at (16.5,21.5) {\textcolor{orange}{1}};

\fill[yellow, opacity=0.3] (19,16) rectangle (20,17);
\fill[yellow, opacity=0.3] (16,7) rectangle (17,8);
\fill[blue, opacity=0.3] (7,10) rectangle (8,11);
\fill[green, opacity=0.3] (10,19) rectangle (11,20);

\end{scope}

\begin{scope}[shift={(0,0)}]
  \foreach \x in {0,...,27} {
    \draw[gray!70, line width=0.35pt] (\x,0) -- (\x,27);
  }
  \foreach \y in {0,...,27} {
    \draw[gray!70, line width=0.35pt] (0,\y) -- (27,\y);
  }

\draw[very thick] (10,10) rectangle (17,17);
\draw[very thick] (17,8) rectangle (24,15);
\draw[very thick] (8,3) rectangle (15,10);
\draw[very thick] (3,12) rectangle (10,19);
\draw[very thick] (12,17) rectangle (19,24);

\draw[very thick, blue] (15,1) rectangle (22,8);
\draw[very thick, blue] (1,5) rectangle (8,12);
\draw[very thick, blue] (5,19) rectangle (12,26);
\draw[very thick, blue] (19,15) rectangle (26,22);

\node at (12.5,12.5) {\normalsize 1};
\node at (14.5,12.5) {\normalsize 1};
\node at (12.5,14.5) {\normalsize 1};
\node at (14.5,14.5) {\normalsize 1};
\node at (15.5,17.5) {1};
\node at (17.5,17.5) {1};
\node at (17.5,11.5) {1};

\node at (10.5,11.5) {\normalsize 1};
\node at (10.5,13.5) {\normalsize 1};
\node at (13.5,10.5) {\normalsize 1};
\node at (15.5,10.5) {\normalsize 1};
\node at (16.5,13.5) {\normalsize 1};
\node at (16.5,15.5) {\normalsize 1};
\node at (11.5,16.5) {\normalsize 1};
\node at (13.5,16.5) {\normalsize 1};

\node at (17.5,9.5) {\normalsize 1};
\node at (19.5,16.5) {\normalsize 1};
\node at (19.5,12.5) {1};

\node at (18.5,14.5) {1};
\node at (20.5,14.5) {1};
\node at (18.5,7.5) {1};
\node at (16.5,7.5) {1};

\node at (19.5,10.5) {1};
\node at (21.5,10.5) {1};
\node at (21.5,12.5) {1};

\node at (20.5,8.5) {1};
\node at (22.5,8.5) {1};
\node at (23.5,11.5) {1};
\node at (23.5,13.5) {1};

\node at (8.5,4.5) {1};
\node at (8.5,6.5) {1};
\node at (9.5,9.5) {1};
\node at (11.5,9.5) {1};
\node at (14.5,8.5) {1};
\node at (14.5,6.5) {1};
\node at (13.5,3.5) {1};
\node at (11.5,3.5) {1};
\node at (10.5,5.5) {1};
\node at (10.5,7.5) {1};
\node at (12.5,5.5) {1};
\node at (12.5,7.5) {1};
\node at (7.5,10.5) {1};

\node at (3.5,13.5) {1};
\node at (3.5,15.5) {1};
\node at (4.5,18.5) {1};
\node at (6.5,18.5) {1};
\node at (9.5,17.5) {1};
\node at (9.5,15.5) {1};
\node at (8.5,12.5) {1};
\node at (6.5,12.5) {1};
\node at (5.5,14.5) {1};
\node at (5.5,16.5) {1};
\node at (7.5,14.5) {1};
\node at (7.5,16.5) {1};
\node at (10.5,19.5) {1};

\node at (12.5,18.5) {1};
\node at (12.5,20.5) {1};
\node at (13.5,23.5) {1};
\node at (15.5,23.5) {1};
\node at (18.5,22.5) {1};
\node at (18.5,20.5) {1};
\node at (14.5,19.5) {1};
\node at (14.5,21.5) {1};
\node at (16.5,19.5) {1};
\node at (16.5,21.5) {1};

\node at (8.5,19.5) {\textcolor{blue}{1}};
\node at (5.5,20.5) {\textcolor{blue}{1}};
\node at (5.5,22.5) {\textcolor{blue}{1}};
\node at (6.5,25.5) {\textcolor{blue}{1}};
\node at (8.5,25.5) {\textcolor{blue}{1}};
\node at (11.5,24.5) {\textcolor{blue}{1}};
\node at (11.5,22.5) {\textcolor{blue}{1}};
\node at (7.5,21.5) {\textcolor{blue}{1}};
\node at (7.5,23.5) {\textcolor{blue}{1}};
\node at (9.5,21.5) {\textcolor{blue}{1}};
\node at (9.5,23.5) {\textcolor{blue}{1}};

\node at (24.5,15.5) {\textcolor{blue}{1}};
\node at (22.5,15.5) {\textcolor{blue}{1}};
\node at (19.5,18.5) {\textcolor{blue}{1}};
\node at (20.5,21.5) {\textcolor{blue}{1}};
\node at (22.5,21.5) {\textcolor{blue}{1}};
\node at (25.5,20.5) {\textcolor{blue}{1}};
\node at (25.5,18.5) {\textcolor{blue}{1}};
\node at (21.5,17.5) {\textcolor{blue}{1}};
\node at (21.5,19.5) {\textcolor{blue}{1}};
\node at (23.5,17.5) {\textcolor{blue}{1}};
\node at (23.5,19.5) {\textcolor{blue}{1}};

\node at (20.5,1.5) {\textcolor{blue}{1}};
\node at (18.5,1.5) {\textcolor{blue}{1}};
\node at (15.5,2.5) {\textcolor{blue}{1}};
\node at (15.5,4.5) {\textcolor{blue}{1}};
\node at (21.5,6.5) {\textcolor{blue}{1}};
\node at (21.5,4.5) {\textcolor{blue}{1}};
\node at (17.5,3.5) {\textcolor{blue}{1}};
\node at (17.5,5.5) {\textcolor{blue}{1}};
\node at (19.5,3.5) {\textcolor{blue}{1}};
\node at (19.5,5.5) {\textcolor{blue}{1}};

\node at (6.5,5.5) {\textcolor{blue}{1}};
\node at (4.5,5.5) {\textcolor{blue}{1}};
\node at (1.5,6.5) {\textcolor{blue}{1}};
\node at (1.5,8.5) {\textcolor{blue}{1}};
\node at (2.5,11.5) {\textcolor{blue}{1}};
\node at (4.5,11.5) {\textcolor{blue}{1}};
\node at (7.5,8.5) {\textcolor{blue}{1}};
\node at (3.5,7.5) {\textcolor{blue}{1}};
\node at (3.5,9.5) {\textcolor{blue}{1}};
\node at (5.5,7.5) {\textcolor{blue}{1}};
\node at (5.5,9.5) {\textcolor{blue}{1}};

\node at (14.5,0.5) {\textcolor{blue}{1}};
\node at (0.5,12.5) {\textcolor{blue}{1}};
\node at (12.5,26.5) {\textcolor{blue}{1}};
\node at (26.5,14.5) {\textcolor{blue}{1}};

\fill[blue, opacity=0.3] (14,0) rectangle (15,1);
\fill[blue, opacity=0.3] (0,12) rectangle (1,13);
\fill[blue, opacity=0.3] (12,26) rectangle (13,27);
\fill[blue, opacity=0.3] (26,14) rectangle (27,15);

\fill[yellow, opacity=0.3] (17,9) rectangle (18,10);
\fill[yellow, opacity=0.3] (17,17) rectangle (18,18);
\fill[yellow, opacity=0.3] (9,9) rectangle (10,10);
\fill[yellow, opacity=0.3] (9,17) rectangle (10,18);

\end{scope}

\end{tikzpicture}
\caption{Expansion of the pattern of squares of color 1 in Case~1.}

\label{fig:13}
\end{figure}
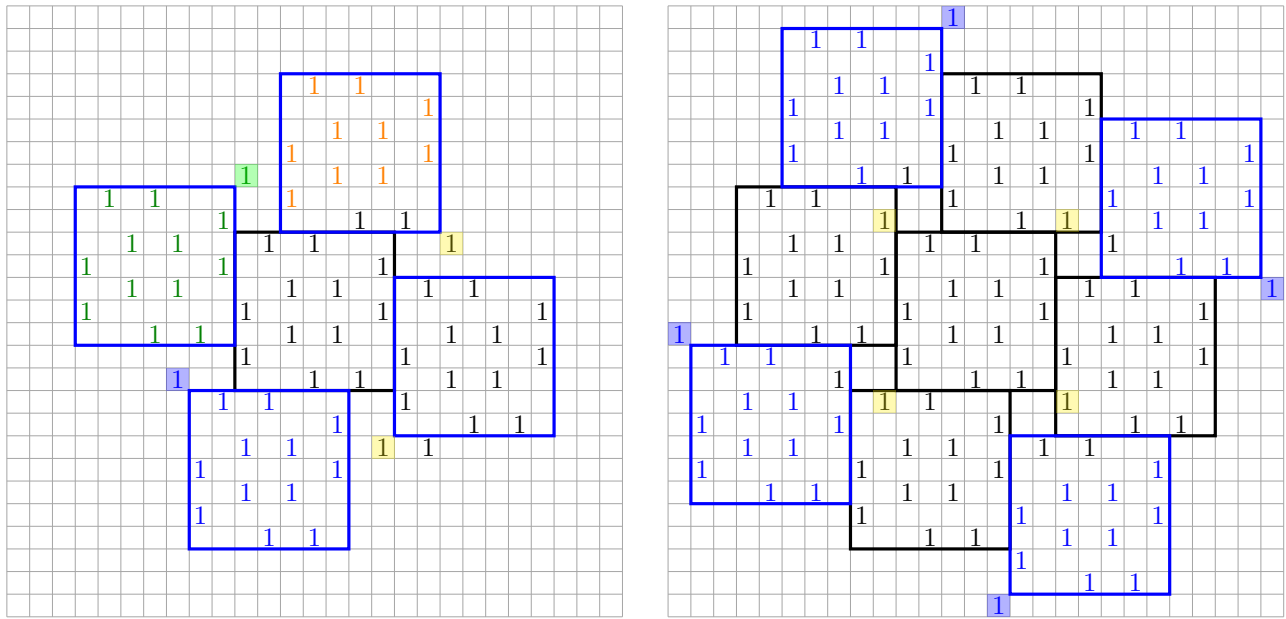

We observe that $ G $ is a critical subgraph. By Observation~\ref{obs:3}\,(ii), let $ k \in \{2, 3, \ldots, 39\} $ be the color that does not appear in $ G $ and there are only two possible positions for color~$ k $, referred to as type-A and type-B. We consider the two cases separately based on the position of color~$ k $.

\textbf{Case 1.1:} $ k $ appears in a type-A position.

Consider $ G(-1,-7) $ (see Figure~\ref{fig:14}). Since $ c(G(-1,-7)) = 11 $, by Observation~\ref{obs:2}\,(iv), color~$ k $ must appear in this subgraph. Due to the distance constraint, the only square that can be assigned color~$ k $ in this subgraph is $ (1,1) $. Hence, $ (1,1) $ is a square of color~$ k $. Similarly, by considering $ G(0,-8) $, we have that $ (7,1) $ in this subgraph must be a square of color~$ k $. We then consider $ G(-1,-14) $. Observe that this is a critical subgraph in which color~$ k $ does not appear, since the only square that can be assigned color~$ k $ is $ (1,1) $ and it is already assigned color~1. By Observation~\ref{obs:3}\,(ii), color~$ k $ must appear in a type-A position. Again, since $ c(G(-2,-21)) = c(G(-1,-22)) = c(G(-2,-28)) = c(G(-1,-29)) = 11 $, by Observation~\ref{obs:2}\,(iv), color~$ k $ must appear as shown in Figure~\ref{fig:14}.

Similarly, by considering $ G(7,-1) $, $ G(7,-8) $, $ G(6,-15) $, and $ G(6,-22) $, color~$ k $ must be assigned as shown in Figure~\ref{fig:14}. We then consider $ G(6,-29) $ and observe that this subgraph is a critical subgraph that does not contain color~$ k $. This contradicts Observation~\ref{obs:3}, since the position of color~$ k $ can be neither type-A nor type-B.

\textbf{Case 1.2:} $ k $ appears in a type-B position.

We proceed similarly to Case~1.1. Consider $ G(1, -7) $. Observe that this subgraph is a critical subgraph that does not contain color~$ k $. By Observation~\ref{obs:3}\,(ii), the position of color~$ k $ must be type-B. Since $ c(G(2,-14)) = c(G(1,-15)) = c(G(2,-21)) = 11 $, by Observation~\ref{obs:2}\,(iv), color~$ k $ must appear as shown in the second diagram of Figure~\ref{fig:14}. Finally, consider $ G(1,-22) $. We observe that $ c(G(1,-22)) = 11 $, and this subgraph cannot contain color~$ k $, since the only square that can be assigned color~$ k $ is $ (1,1) $, which is already assigned color~1. This contradicts Observation~\ref{obs:2}\,(iv).

\begin{figure}[ht]
\centering
\begin{tikzpicture}[scale=0.30]
\begin{scope}[shift={(-18,0)}]
  \foreach \x in {0,...,16} {
    \draw[gray!70, line width=0.35pt] (\x,0) -- (\x,38);
  }
  \foreach \y in {0,...,38} {
    \draw[gray!70, line width=0.35pt] (0,\y) -- (16,\y);
  }

\draw[very thick] (2,30) rectangle (9,37);

\node at (4.5,32.5) {\normalsize 1};
\node at (6.5,32.5) {\normalsize 1};
\node at (4.5,34.5) {\normalsize 1};
\node at (6.5,34.5) {\normalsize 1};
\node at (7.5,37.5) {1};
\node at (9.5,37.5) {1};
\node at (9.5,31.5) {1};

\node at (2.5,31.5) {\normalsize 1};
\node at (2.5,33.5) {\normalsize 1};
\node at (5.5,30.5) {\normalsize 1};
\node at (7.5,30.5) {\normalsize 1};
\node at (8.5,33.5) {\normalsize 1};
\node at (8.5,35.5) {\normalsize 1};
\node at (3.5,36.5) {\normalsize 1};
\node at (5.5,36.5) {\normalsize 1};

\node at (9.5,29.5) {\normalsize 1};
\node at (11.5,36.5) {\normalsize 1};
\node at (11.5,32.5) {1};

\node at (10.5,34.5) {1};
\node at (12.5,34.5) {1};
\node at (10.5,27.5) {1};
\node at (8.5,27.5) {1};

\node at (11.5,30.5) {1};
\node at (13.5,30.5) {1};
\node at (13.5,32.5) {1};

\node at (12.5,28.5) {1};
\node at (14.5,28.5) {1};
\node at (15.5,31.5) {1};
\node at (15.5,33.5) {1};
\node at (8.5,27.5) {1};

\node at (14.5,35.5) {1};
\node at (13.5,37.5) {1};
\node at (15.5,37.5) {1};
\node at (1.5,37.5) {1};
\node at (1.5,35.5) {1};
\node at (0.5,32.5) {1};
\node at (1.5,29.5) {1};
\node at (3.5,29.5) {1};
\node at (6.5,28.5) {1};
\node at (4.5,27.5) {1};
\node at (2.5,27.5) {1};
\node at (0.5,26.5) {1};
\node at (6.5,26.5) {1};
\node at (13.5,26.5) {1};
\node at (2.5,25.5) {1};
\node at (4.5,25.5) {1};
\node at (9.5,25.5) {1};
\node at (11.5,25.5) {1};
\node at (15.5,25.5) {1};
\node at (0.5,24.5) {1};
\node at (7.5,24.5) {1};
\node at (13.5,24.5) {1};
\node at (3.5,23.5) {1};
\node at (5.5,23.5) {1};
\node at (9.5,23.5) {1};
\node at (11.5,23.5) {1};
\node at (1.5,22.5) {1};
\node at (7.5,22.5) {1};
\node at (14.5,22.5) {1};
\node at (4.5,21.5) {1};
\node at (10.5,21.5) {1};
\node at (12.5,21.5) {1};
\node at (0.5,20.5) {1};
\node at (2.5,20.5) {1};
\node at (6.5,20.5) {1};
\node at (8.5,20.5) {1};
\node at (14.5,20.5) {1};
\node at (4.5,19.5) {1};
\node at (11.5,19.5) {1};
\node at (0.5,18.5) {1};
\node at (2.5,18.5) {1};
\node at (7.5,18.5) {1};
\node at (9.5,18.5) {1};
\node at (13.5,18.5) {1};
\node at (15.5,18.5) {1};
\node at (5.5,17.5) {1};
\node at (11.5,17.5) {1};
\node at (1.5,16.5) {1};
\node at (3.5,16.5) {1};
\node at (7.5,16.5) {1};
\node at (9.5,16.5) {1};
\node at (14.5,16.5) {1};
\node at (5.5,15.5) {1};
\node at (12.5,15.5) {1};
\node at (2.5,14.5) {1};
\node at (8.5,14.5) {1};
\node at (10.5,14.5) {1};
\node at (14.5,14.5) {1};
\node at (0.5,13.5) {1};
\node at (4.5,13.5) {1};
\node at (6.5,13.5) {1};
\node at (12.5,13.5) {1};
\node at (2.5,12.5) {1};
\node at (9.5,12.5) {1};
\node at (15.5,12.5) {1};
\node at (0.5,11.5) {1};
\node at (5.5,11.5) {1};
\node at (7.5,11.5) {1};
\node at (11.5,11.5) {1};
\node at (13.5,11.5) {1};
\node at (3.5,10.5) {1};
\node at (9.5,10.5) {1};
\node at (1.5,9.5) {1};
\node at (5.5,9.5) {1};
\node at (7.5,9.5) {1};
\node at (12.5,9.5) {1};
\node at (14.5,9.5) {1};
\node at (3.5,8.5) {1};
\node at (10.5,8.5) {1};
\node at (0.5,7.5) {1};
\node at (6.5,7.5) {1};
\node at (8.5,7.5) {1};
\node at (12.5,7.5) {1};
\node at (14.5,7.5) {1};
\node at (2.5,6.5) {1};
\node at (4.5,6.5) {1};
\node at (10.5,6.5) {1};
\node at (0.5,5.5) {1};
\node at (7.5,5.5) {1};
\node at (13.5,5.5) {1};
\node at (15.5,5.5) {1};
\node at (3.5,4.5) {1};
\node at (5.5,4.5) {1};
\node at (9.5,4.5) {1};
\node at (11.5,4.5) {1};
\node at (1.5,3.5) {1};
\node at (7.5,3.5) {1};
\node at (14.5,3.5) {1};
\node at (3.5,2.5) {1};
\node at (5.5,2.5) {1};
\node at (10.5,2.5) {1};
\node at (12.5,2.5) {1};
\node at (1.5,1.5) {1};
\node at (8.5,1.5) {1};
\node at (14.5,1.5) {1};
\node at (4.5,0.5) {1};
\node at (6.5,0.5) {1};
\node at (10.5,0.5) {1};
\node at (12.5,0.5) {1};

\node at (2.5,37.5) {\textcolor{blue}{$k$}};
\node at (1.5,30.5) {\textcolor{blue}{$k$}};
\node at (8.5,29.5) {\textcolor{blue}{$k$}};
\node at (9.5,36.5) {\textcolor{blue}{$k$}};
\draw[very thick, blue] (1,23) rectangle (8,30);
\node at (1.5,23.5) {\textcolor{blue}{$k$}};
\draw[very thick, blue] (2,22) rectangle (9,29);
\node at (8.5,22.5) {\textcolor{blue}{$k$}};
\draw[very thick, blue] (1,16) rectangle (8,23);
\node at (0.5,16.5) {\textcolor{blue}{$k$}};
\node at (7.5,15.5) {\textcolor{blue}{$k$}};
\draw[very thick, blue] (0,9) rectangle (7,16);
\node at (0.5,9.5) {\textcolor{blue}{$k$}};
\draw[very thick, blue] (1,8) rectangle (8,15);
\node at (7.5,8.5) {\textcolor{blue}{$k$}};
\draw[very thick, blue] (0,2) rectangle (7,9);
\node at (0.5,2.5) {\textcolor{blue}{$k$}};
\draw[very thick, blue] (1,1) rectangle (8,8);
\node at (7.5,1.5) {\textcolor{blue}{$k$}};
\draw[very thick, darkgreen] (9,29) rectangle (16,36);
\node at (15.5,29.5) {\textcolor{darkgreen}{$k$}};
\draw[very thick, darkgreen] (9,22) rectangle (16,29);
\node at (15.5,22.5) {\textcolor{darkgreen}{$k$}};
\draw[very thick, darkgreen] (8,15) rectangle (15,22);
\node at (14.5,15.5) {\textcolor{darkgreen}{$k$}};
\draw[very thick, darkgreen] (8,8) rectangle (15,15);
\node at (14.5,8.5) {\textcolor{darkgreen}{$k$}};
\draw[very thick, darkgreen] (8,1) rectangle (15,8);
\fill[yellow, opacity=0.3] (8,1) rectangle (15,8);
\node at (9,-1.3) {Case 1.1};

\end{scope}

\begin{scope}[shift={(0,0)}]
  \foreach \x in {0,...,16} {
    \draw[gray!70, line width=0.35pt] (\x,0) -- (\x,38);
  }
  \foreach \y in {0,...,38} {
    \draw[gray!70, line width=0.35pt] (0,\y) -- (16,\y);
  }

\draw[very thick] (2,30) rectangle (9,37);

\node at (4.5,32.5) {\normalsize 1};
\node at (6.5,32.5) {\normalsize 1};
\node at (4.5,34.5) {\normalsize 1};
\node at (6.5,34.5) {\normalsize 1};
\node at (7.5,37.5) {1};
\node at (9.5,37.5) {1};
\node at (9.5,31.5) {1};

\node at (2.5,31.5) {\normalsize 1};
\node at (2.5,33.5) {\normalsize 1};
\node at (5.5,30.5) {\normalsize 1};
\node at (7.5,30.5) {\normalsize 1};
\node at (8.5,33.5) {\normalsize 1};
\node at (8.5,35.5) {\normalsize 1};
\node at (3.5,36.5) {\normalsize 1};
\node at (5.5,36.5) {\normalsize 1};

\node at (9.5,29.5) {\normalsize 1};
\node at (11.5,36.5) {\normalsize 1};
\node at (11.5,32.5) {1};

\node at (10.5,34.5) {1};
\node at (12.5,34.5) {1};
\node at (10.5,27.5) {1};
\node at (8.5,27.5) {1};

\node at (11.5,30.5) {1};
\node at (13.5,30.5) {1};
\node at (13.5,32.5) {1};

\node at (12.5,28.5) {1};
\node at (14.5,28.5) {1};
\node at (15.5,31.5) {1};
\node at (15.5,33.5) {1};
\node at (8.5,27.5) {1};

\node at (14.5,35.5) {1};
\node at (13.5,37.5) {1};
\node at (15.5,37.5) {1};
\node at (1.5,37.5) {1};
\node at (1.5,35.5) {1};
\node at (0.5,32.5) {1};
\node at (1.5,29.5) {1};
\node at (3.5,29.5) {1};
\node at (6.5,28.5) {1};
\node at (4.5,27.5) {1};
\node at (2.5,27.5) {1};
\node at (0.5,26.5) {1};
\node at (6.5,26.5) {1};
\node at (13.5,26.5) {1};
\node at (2.5,25.5) {1};
\node at (4.5,25.5) {1};
\node at (9.5,25.5) {1};
\node at (11.5,25.5) {1};
\node at (15.5,25.5) {1};
\node at (0.5,24.5) {1};
\node at (7.5,24.5) {1};
\node at (13.5,24.5) {1};
\node at (3.5,23.5) {1};
\node at (5.5,23.5) {1};
\node at (9.5,23.5) {1};
\node at (11.5,23.5) {1};
\node at (1.5,22.5) {1};
\node at (7.5,22.5) {1};
\node at (14.5,22.5) {1};
\node at (4.5,21.5) {1};
\node at (10.5,21.5) {1};
\node at (12.5,21.5) {1};
\node at (0.5,20.5) {1};
\node at (2.5,20.5) {1};
\node at (6.5,20.5) {1};
\node at (8.5,20.5) {1};
\node at (14.5,20.5) {1};
\node at (4.5,19.5) {1};
\node at (11.5,19.5) {1};
\node at (0.5,18.5) {1};
\node at (2.5,18.5) {1};
\node at (7.5,18.5) {1};
\node at (9.5,18.5) {1};
\node at (13.5,18.5) {1};
\node at (15.5,18.5) {1};
\node at (5.5,17.5) {1};
\node at (11.5,17.5) {1};
\node at (1.5,16.5) {1};
\node at (3.5,16.5) {1};
\node at (7.5,16.5) {1};
\node at (9.5,16.5) {1};
\node at (14.5,16.5) {1};
\node at (5.5,15.5) {1};
\node at (12.5,15.5) {1};
\node at (2.5,14.5) {1};
\node at (8.5,14.5) {1};
\node at (10.5,14.5) {1};
\node at (14.5,14.5) {1};
\node at (0.5,13.5) {1};
\node at (4.5,13.5) {1};
\node at (6.5,13.5) {1};
\node at (12.5,13.5) {1};
\node at (2.5,12.5) {1};
\node at (9.5,12.5) {1};
\node at (15.5,12.5) {1};
\node at (0.5,11.5) {1};
\node at (5.5,11.5) {1};
\node at (7.5,11.5) {1};
\node at (11.5,11.5) {1};
\node at (13.5,11.5) {1};
\node at (3.5,10.5) {1};
\node at (9.5,10.5) {1};
\node at (1.5,9.5) {1};
\node at (5.5,9.5) {1};
\node at (7.5,9.5) {1};
\node at (12.5,9.5) {1};
\node at (14.5,9.5) {1};
\node at (3.5,8.5) {1};
\node at (10.5,8.5) {1};
\node at (0.5,7.5) {1};
\node at (6.5,7.5) {1};
\node at (8.5,7.5) {1};
\node at (12.5,7.5) {1};
\node at (14.5,7.5) {1};
\node at (2.5,6.5) {1};
\node at (4.5,6.5) {1};
\node at (10.5,6.5) {1};
\node at (0.5,5.5) {1};
\node at (7.5,5.5) {1};
\node at (13.5,5.5) {1};
\node at (15.5,5.5) {1};
\node at (3.5,4.5) {1};
\node at (5.5,4.5) {1};
\node at (9.5,4.5) {1};
\node at (11.5,4.5) {1};
\node at (1.5,3.5) {1};
\node at (7.5,3.5) {1};
\node at (14.5,3.5) {1};
\node at (3.5,2.5) {1};
\node at (5.5,2.5) {1};
\node at (10.5,2.5) {1};
\node at (12.5,2.5) {1};
\node at (1.5,1.5) {1};
\node at (8.5,1.5) {1};
\node at (14.5,1.5) {1};
\node at (4.5,0.5) {1};
\node at (6.5,0.5) {1};
\node at (10.5,0.5) {1};
\node at (12.5,0.5) {1};

\node at (8.5,37.5) {\textcolor{blue}{$k$}};
\node at (1.5,36.5) {\textcolor{blue}{$k$}};
\node at (2.5,29.5) {\textcolor{blue}{$k$}};
\node at (9.5,30.5) {\textcolor{blue}{$k$}};
\draw[very thick, blue] (3,23) rectangle (10,30);
\node at (10.5,23.5) {\textcolor{blue}{$k$}};
\node at (3.5,22.5) {\textcolor{blue}{$k$}};
\draw[very thick, blue] (4,16) rectangle (11,23);
\node at (10.5,16.5) {\textcolor{blue}{$k$}};
\draw[very thick, blue] (3,15) rectangle (10,22);
\node at (3.5,15.5) {\textcolor{blue}{$k$}};
\draw[very thick, blue] (4,9) rectangle (11,16);
\node at (10.5,9.5) {\textcolor{blue}{$k$}};
\draw[very thick, blue] (3,8) rectangle (10,15);
\fill[yellow, opacity=0.3] (3,8) rectangle (4,9);

\node at (9,-1.3) {Case 1.2};

\end{scope}

\end{tikzpicture}
\caption{Proof of Cases 1.1 and 1.2.}
\label{fig:14}
\end{figure}
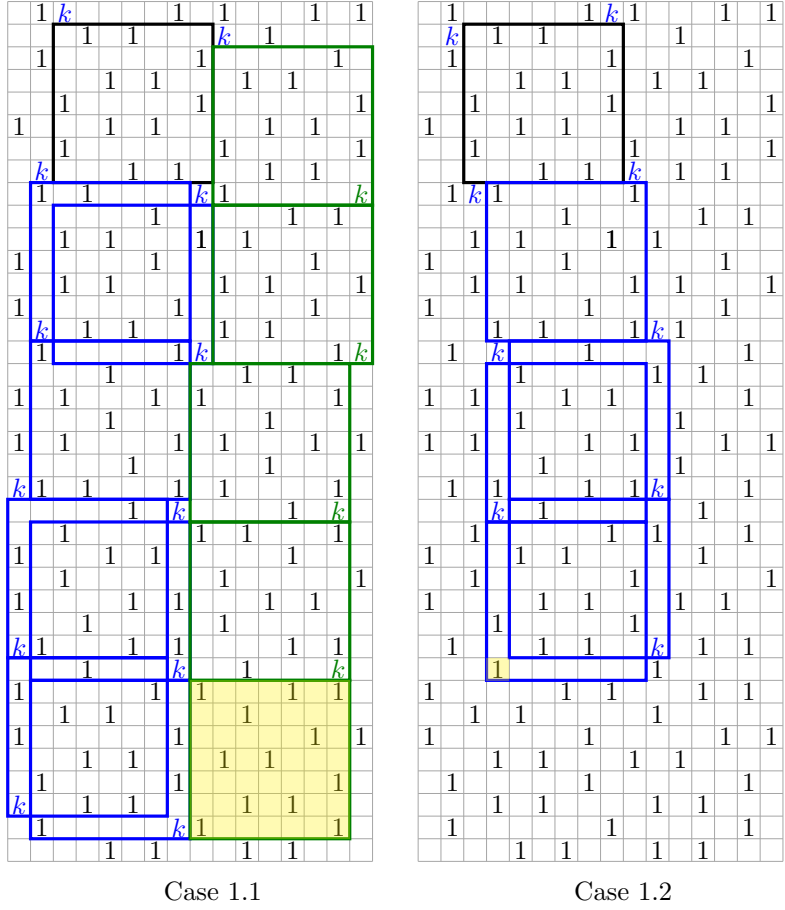

\textbf{Case 2:} Square $ (6,5) $ of $G(3,1)$, which is square $(9,6)$ of $G$, has color 1.

In fact, we can assume that square $(10,7)$ of any $ P_7 \boxtimes P_7 $ subgraph with pattern~(i) cannot be assigned color~1, since the case where it is assigned color~1 is already covered in Case~1.

\begin{claim*}
    If $ H $ has pattern~(i), then $ H(4,1) $ also has pattern~(i).
\end{claim*}

\begin{proof}
We know that there are some squares of color 1 outside $H$, as shown in Figure~\ref{fig:9}, since the same argument for $G$ also applies. Since square $ (10,7) $ of $ H $ is not of color~1, square $(9,6)$ must be of color 1. We have that the squares that cannot be assigned color 1 must be as shown in Figure \ref{fig:15}. Consider $H(3,1)$. Since each part must contains exactly one square of color 1, we have that, in the middle part, square of color 1 must be $(6,3)$. Observe that $H(4,1)$ is $P_7 \boxtimes P_7$ that contains the pattern
\[
\begin{bmatrix}
1 & x & 1\\
x & x & x\\
1 & x & 1
\end{bmatrix}
\]
in the middle, which is the same as in $ G $ (see Figure~\ref{fig:15}). Since we have already shown that there are only two possible positions for squares of color~1 in $ G $, as shown in Figure~\ref{fig:8}, it follows that there are also only two possible positions for squares of color~1 in $ H(4,1) $. Among these, only pattern (i) matches the observed position of squares of color 1 in $ H(4,1) $. Hence, the pattern of squares of color~1 in $ H(4,1) $ must be pattern (i) (see Figure~\ref{fig:15}).
\end{proof}

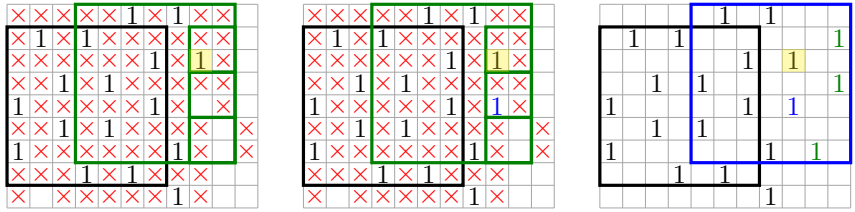
\begin{figure}[ht]
\centering
\begin{tikzpicture}[scale=0.30]

\begin{scope}[shift={(-26,0)}]
  \foreach \x in {0,...,11} {
    \draw[gray!70, line width=0.35pt] (\x,0) -- (\x,9);
  }
  \foreach \y in {0,...,9} {
    \draw[gray!70, line width=0.35pt] (0,\y) -- (11,\y);
  }

\draw[very thick] (0,1) rectangle (7,8);
\draw[very thick, darkgreen] (3,2) rectangle (10,9);
\draw[very thick, darkgreen] (8,2) rectangle (10,4);
\draw[very thick, darkgreen] (8,4) rectangle (10,6);
\draw[very thick, darkgreen] (8,6) rectangle (10,8);

\node at (2.5,3.5) {\normalsize 1};
\node at (4.5,3.5) {\normalsize 1};
\node at (2.5,5.5) {\normalsize 1};
\node at (4.5,5.5) {\normalsize 1};
\node at (5.5,8.5) {{1}};
\node at (7.5,8.5) {{1}};
\node at (7.5,2.5) {{1}};

\node at (1.5,2.5) {\textcolor{red}{$\times$}};
\node at (2.5,2.5) {\textcolor{red}{$\times$}};
\node at (3.5,2.5) {\textcolor{red}{$\times$}};
\node at (4.5,2.5) {\textcolor{red}{$\times$}};
\node at (5.5,2.5) {\textcolor{red}{$\times$}};

\node at (1.5,3.5) {\textcolor{red}{$\times$}};
\node at (3.5,3.5) {\textcolor{red}{$\times$}};
\node at (5.5,3.5) {\textcolor{red}{$\times$}};

\node at (1.5,4.5) {\textcolor{red}{$\times$}};
\node at (2.5,4.5) {\textcolor{red}{$\times$}};
\node at (3.5,4.5) {\textcolor{red}{$\times$}};
\node at (4.5,4.5) {\textcolor{red}{$\times$}};
\node at (5.5,4.5) {\textcolor{red}{$\times$}};

\node at (1.5,5.5) {\textcolor{red}{$\times$}};
\node at (3.5,5.5) {\textcolor{red}{$\times$}};
\node at (5.5,5.5) {\textcolor{red}{$\times$}};

\node at (1.5,6.5) {\textcolor{red}{$\times$}};
\node at (2.5,6.5) {\textcolor{red}{$\times$}};
\node at (3.5,6.5) {\textcolor{red}{$\times$}};
\node at (4.5,6.5) {\textcolor{red}{$\times$}};
\node at (5.5,6.5) {\textcolor{red}{$\times$}};

\node at (0.5,2.5) {\normalsize 1};
\node at (0.5,4.5) {\normalsize 1};
\node at (3.5,1.5) {\normalsize 1};
\node at (5.5,1.5) {\normalsize 1};
\node at (6.5,4.5) {\normalsize 1};
\node at (6.5,6.5) {\normalsize 1};
\node at (1.5,7.5) {\normalsize 1};
\node at (3.5,7.5) {\normalsize 1};

\node at (7.5,0.5) {\normalsize 1};
\node at (8.5,0.5) {\textcolor{red}{$\times$}};
\node at (6.5,0.5) {\textcolor{red}{$\times$}};
\node at (5.5,0.5) {\textcolor{red}{$\times$}};
\node at (4.5,0.5) {\textcolor{red}{$\times$}};
\node at (3.5,0.5) {\textcolor{red}{$\times$}};
\node at (2.5,0.5) {\textcolor{red}{$\times$}};
\node at (0.5,0.5) {\textcolor{red}{$\times$}};

\node at (0.5,1.5) {\textcolor{red}{$\times$}};
\node at (0.5,3.5) {\textcolor{red}{$\times$}};
\node at (0.5,5.5) {\textcolor{red}{$\times$}};
\node at (0.5,6.5) {\textcolor{red}{$\times$}};
\node at (0.5,7.5) {\textcolor{red}{$\times$}};

\node at (1.5,1.5) {\textcolor{red}{$\times$}};
\node at (2.5,1.5) {\textcolor{red}{$\times$}};
\node at (4.5,1.5) {\textcolor{red}{$\times$}};
\node at (6.5,1.5) {\textcolor{red}{$\times$}};

\node at (6.5,2.5) {\textcolor{red}{$\times$}};
\node at (6.5,3.5) {\textcolor{red}{$\times$}};
\node at (6.5,5.5) {\textcolor{red}{$\times$}};
\node at (6.5,7.5) {\textcolor{red}{$\times$}};

\node at (5.5,7.5) {\textcolor{red}{$\times$}};
\node at (4.5,7.5) {\textcolor{red}{$\times$}};
\node at (2.5,7.5) {\textcolor{red}{$\times$}};

\node at (0.5,8.5) {\textcolor{red}{$\times$}};
\node at (1.5,8.5) {\textcolor{red}{$\times$}};
\node at (2.5,8.5) {\textcolor{red}{$\times$}};
\node at (3.5,8.5) {\textcolor{red}{$\times$}};
\node at (4.5,8.5) {\textcolor{red}{$\times$}};

\node at (7.5,7.5) {\textcolor{red}{$\times$}};
\node at (7.5,6.5) {\textcolor{red}{$\times$}};
\node at (7.5,5.5) {\textcolor{red}{$\times$}};
\node at (7.5,4.5) {\textcolor{red}{$\times$}};
\node at (7.5,3.5) {\textcolor{red}{$\times$}};
\node at (10.5,3.5) {\textcolor{red}{$\times$}};
\node at (10.5,2.5) {\textcolor{red}{$\times$}};

\node at (6.5,8.5) {\textcolor{red}{$\times$}};
\node at (8.5,8.5) {\textcolor{red}{$\times$}};
\node at (8.5,7.5) {\textcolor{red}{$\times$}};
\node at (8.5,3.5) {\textcolor{red}{$\times$}};
\node at (8.5,2.5) {\textcolor{red}{$\times$}};
\node at (8.5,1.5) {\textcolor{red}{$\times$}};
\node at (7.5,1.5) {\textcolor{red}{$\times$}};

\node at (9.5,8.5) {\textcolor{red}{$\times$}};
\node at (9.5,6.5) {\textcolor{red}{$\times$}};
\node at (9.5,4.5) {\textcolor{red}{$\times$}};

\node at (8.5,6.5) {\normalsize 1};
\node at (9.5,7.5) {\textcolor{red}{$\times$}};
\node at (8.5,5.5) {\textcolor{red}{$\times$}};
\node at (9.5,5.5) {\textcolor{red}{$\times$}};

\fill[yellow, opacity=0.3] (8,6) rectangle (9,7);
\end{scope}

\begin{scope}[shift={(-13,0)}]
  \foreach \x in {0,...,11} {
    \draw[gray!70, line width=0.35pt] (\x,0) -- (\x,9);
  }
  \foreach \y in {0,...,9} {
    \draw[gray!70, line width=0.35pt] (0,\y) -- (11,\y);
  }

\draw[very thick] (0,1) rectangle (7,8);
\draw[very thick, darkgreen] (3,2) rectangle (10,9);
\draw[very thick, darkgreen] (8,2) rectangle (10,4);
\draw[very thick, darkgreen] (8,4) rectangle (10,6);
\draw[very thick, darkgreen] (8,6) rectangle (10,8);

\node at (2.5,3.5) {\normalsize 1};
\node at (4.5,3.5) {\normalsize 1};
\node at (2.5,5.5) {\normalsize 1};
\node at (4.5,5.5) {\normalsize 1};
\node at (5.5,8.5) {{1}};
\node at (7.5,8.5) {{1}};
\node at (7.5,2.5) {{1}};

\node at (1.5,2.5) {\textcolor{red}{$\times$}};
\node at (2.5,2.5) {\textcolor{red}{$\times$}};
\node at (3.5,2.5) {\textcolor{red}{$\times$}};
\node at (4.5,2.5) {\textcolor{red}{$\times$}};
\node at (5.5,2.5) {\textcolor{red}{$\times$}};

\node at (1.5,3.5) {\textcolor{red}{$\times$}};
\node at (3.5,3.5) {\textcolor{red}{$\times$}};
\node at (5.5,3.5) {\textcolor{red}{$\times$}};

\node at (1.5,4.5) {\textcolor{red}{$\times$}};
\node at (2.5,4.5) {\textcolor{red}{$\times$}};
\node at (3.5,4.5) {\textcolor{red}{$\times$}};
\node at (4.5,4.5) {\textcolor{red}{$\times$}};
\node at (5.5,4.5) {\textcolor{red}{$\times$}};

\node at (1.5,5.5) {\textcolor{red}{$\times$}};
\node at (3.5,5.5) {\textcolor{red}{$\times$}};
\node at (5.5,5.5) {\textcolor{red}{$\times$}};

\node at (1.5,6.5) {\textcolor{red}{$\times$}};
\node at (2.5,6.5) {\textcolor{red}{$\times$}};
\node at (3.5,6.5) {\textcolor{red}{$\times$}};
\node at (4.5,6.5) {\textcolor{red}{$\times$}};
\node at (5.5,6.5) {\textcolor{red}{$\times$}};

\node at (0.5,2.5) {\normalsize 1};
\node at (0.5,4.5) {\normalsize 1};
\node at (3.5,1.5) {\normalsize 1};
\node at (5.5,1.5) {\normalsize 1};
\node at (6.5,4.5) {\normalsize 1};
\node at (6.5,6.5) {\normalsize 1};
\node at (1.5,7.5) {\normalsize 1};
\node at (3.5,7.5) {\normalsize 1};

\node at (7.5,0.5) {\normalsize 1};
\node at (8.5,0.5) {\textcolor{red}{$\times$}};
\node at (6.5,0.5) {\textcolor{red}{$\times$}};
\node at (5.5,0.5) {\textcolor{red}{$\times$}};
\node at (4.5,0.5) {\textcolor{red}{$\times$}};
\node at (3.5,0.5) {\textcolor{red}{$\times$}};
\node at (2.5,0.5) {\textcolor{red}{$\times$}};
\node at (0.5,0.5) {\textcolor{red}{$\times$}};

\node at (0.5,1.5) {\textcolor{red}{$\times$}};
\node at (0.5,3.5) {\textcolor{red}{$\times$}};
\node at (0.5,5.5) {\textcolor{red}{$\times$}};
\node at (0.5,6.5) {\textcolor{red}{$\times$}};
\node at (0.5,7.5) {\textcolor{red}{$\times$}};

\node at (1.5,1.5) {\textcolor{red}{$\times$}};
\node at (2.5,1.5) {\textcolor{red}{$\times$}};
\node at (4.5,1.5) {\textcolor{red}{$\times$}};
\node at (6.5,1.5) {\textcolor{red}{$\times$}};

\node at (6.5,2.5) {\textcolor{red}{$\times$}};
\node at (6.5,3.5) {\textcolor{red}{$\times$}};
\node at (6.5,5.5) {\textcolor{red}{$\times$}};
\node at (6.5,7.5) {\textcolor{red}{$\times$}};

\node at (5.5,7.5) {\textcolor{red}{$\times$}};
\node at (4.5,7.5) {\textcolor{red}{$\times$}};
\node at (2.5,7.5) {\textcolor{red}{$\times$}};

\node at (0.5,8.5) {\textcolor{red}{$\times$}};
\node at (1.5,8.5) {\textcolor{red}{$\times$}};
\node at (2.5,8.5) {\textcolor{red}{$\times$}};
\node at (3.5,8.5) {\textcolor{red}{$\times$}};
\node at (4.5,8.5) {\textcolor{red}{$\times$}};

\node at (7.5,7.5) {\textcolor{red}{$\times$}};
\node at (7.5,6.5) {\textcolor{red}{$\times$}};
\node at (7.5,5.5) {\textcolor{red}{$\times$}};
\node at (7.5,4.5) {\textcolor{red}{$\times$}};
\node at (7.5,3.5) {\textcolor{red}{$\times$}};
\node at (10.5,3.5) {\textcolor{red}{$\times$}};
\node at (10.5,2.5) {\textcolor{red}{$\times$}};

\node at (6.5,8.5) {\textcolor{red}{$\times$}};
\node at (8.5,8.5) {\textcolor{red}{$\times$}};
\node at (8.5,7.5) {\textcolor{red}{$\times$}};
\node at (8.5,3.5) {\textcolor{red}{$\times$}};
\node at (8.5,2.5) {\textcolor{red}{$\times$}};
\node at (8.5,1.5) {\textcolor{red}{$\times$}};
\node at (7.5,1.5) {\textcolor{red}{$\times$}};

\node at (9.5,8.5) {\textcolor{red}{$\times$}};
\node at (9.5,6.5) {\textcolor{red}{$\times$}};
\node at (9.5,4.5) {\textcolor{red}{$\times$}};

\node at (8.5,6.5) {\normalsize 1};
\node at (9.5,7.5) {\textcolor{red}{$\times$}};
\node at (8.5,5.5) {\textcolor{red}{$\times$}};
\node at (9.5,5.5) {\textcolor{red}{$\times$}};
\node at (8.5,4.5) {\textcolor{blue} {1}};

\fill[yellow, opacity=0.3] (8,6) rectangle (9,7);
  
\end{scope}

\begin{scope}[shift={(0,0)}]
  \foreach \x in {0,...,11} {
    \draw[gray!70, line width=0.35pt] (\x,0) -- (\x,9);
  }
  \foreach \y in {0,...,9} {
    \draw[gray!70, line width=0.35pt] (0,\y) -- (11,\y);
  }

\draw[very thick] (0,1) rectangle (7,8);
\draw[very thick, blue] (4,2) rectangle (11,9);

\node at (2.5,3.5) {\normalsize 1};
\node at (4.5,3.5) {\normalsize 1};
\node at (2.5,5.5) {\normalsize 1};
\node at (4.5,5.5) {\normalsize 1};
\node at (5.5,8.5) {{1}};
\node at (7.5,8.5) {{1}};
\node at (7.5,2.5) {{1}};

\node at (0.5,2.5) {\normalsize 1};
\node at (0.5,4.5) {\normalsize 1};
\node at (3.5,1.5) {\normalsize 1};
\node at (5.5,1.5) {\normalsize 1};
\node at (6.5,4.5) {\normalsize 1};
\node at (6.5,6.5) {\normalsize 1};
\node at (1.5,7.5) {\normalsize 1};
\node at (3.5,7.5) {\normalsize 1};

\node at (7.5,0.5) {\normalsize 1};

\node at (8.5,6.5) {\normalsize 1};
\node at (8.5,4.5) {\textcolor{blue} {1}};
\node at (9.5,2.5) {\textcolor{darkgreen} {1}};
\node at (10.5,5.5) {\textcolor{darkgreen} {1}};
\node at (10.5,7.5) {\textcolor{darkgreen} {1}};

\fill[yellow, opacity=0.3] (8,6) rectangle (9,7);
  
\end{scope}

\end{tikzpicture}
\caption{$ H $ (black), $H(3,1)$ (green) and $H(4,1)$ (blue).}
\label{fig:15}
\end{figure}



By this claim, $G(4,1)$ has pattern (i). Since $ G^{(1)} $, $G^{(2)}$ and $G^{(3)}$ have pattern (i) and we have shown that the squares of color~1 cannot appear as in Case~1, we can also apply this claim with $H = G^{(1)}$, $H = G^{(2)}$ and $H = G^{(3)}$. Hence, the positions of squares of color~1 in $G(1,-4)$, $G(-4,-1)$ and $G(-1,4)$ follow Pattern~(i) (see Figure \ref{fig:16}).

Moreover, we can apply this claim recursively with $H =  G(4, 1) $, $H =  G(1, -4) $, $H =  G(-4, -1) $ and $H =  G(-1, 4) $ (see Figure~\ref{fig:16}). Therefore, this pattern of squares of color~1 extends repeatedly throughout $ P_\infty \boxtimes P_\infty $.

\begin{figure}[ht]
\centering
\begin{tikzpicture}[scale=0.30]
\begin{scope}[shift={(-25,0)}]
  \foreach \x in {0,...,23} {
    \draw[gray!70, line width=0.35pt] (\x,0) -- (\x,23);
  }
  \foreach \y in {0,...,23} {
    \draw[gray!70, line width=0.35pt] (0,\y) -- (23,\y);
  }

\draw[very thick] (8,8) rectangle (15,15);
\draw[very thick, blue] (12,9) rectangle (19,16);
\draw[very thick, blue] (9,4) rectangle (16,11);
\draw[very thick, blue] (4,7) rectangle (11,14);
\draw[very thick, blue] (7,12) rectangle (14,19);

\node at (10.5,10.5) {\normalsize 1};
\node at (12.5,10.5) {\normalsize 1};
\node at (10.5,12.5) {\normalsize 1};
\node at (12.5,12.5) {\normalsize 1};
\node at (8.5,9.5) {\normalsize 1};
\node at (8.5,11.5) {\normalsize 1};
\node at (11.5,8.5) {\normalsize 1};
\node at (13.5,8.5) {\normalsize 1};
\node at (14.5,11.5) {\normalsize 1};
\node at (14.5,13.5) {\normalsize 1};
\node at (9.5,14.5) {\normalsize 1};

\node at (11.5,14.5) {\textcolor{blue}{1}};
\node at (16.5,11.5) {\textcolor{blue}{1}};
\node at (16.5,13.5) {\textcolor{blue}{1}};
\node at (15.5,9.5) {\textcolor{blue}{1}};
\node at (17.5,9.5) {\textcolor{blue}{1}};
\node at (18.5,12.5) {\textcolor{blue}{1}};
\node at (18.5,14.5) {\textcolor{blue}{1}};
\node at (13.5,15.5) {\textcolor{blue}{1}};
\node at (15.5,15.5) {\textcolor{blue}{1}};
\node at (9.5,16.5) {\textcolor{blue}{1}};
\node at (11.5,16.5) {\textcolor{blue}{1}};
\node at (7.5,13.5) {\textcolor{blue}{1}};
\node at (7.5,15.5) {\textcolor{blue}{1}};
\node at (13.5,17.5) {\textcolor{blue}{1}};
\node at (8.5,18.5) {\textcolor{blue}{1}};
\node at (10.5,18.5) {\textcolor{blue}{1}};
\node at (6.5,9.5) {\textcolor{blue}{1}};
\node at (6.5,11.5) {\textcolor{blue}{1}};
\node at (4.5,8.5) {\textcolor{blue}{1}};
\node at (4.5,10.5) {\textcolor{blue}{1}};
\node at (7.5,7.5) {\textcolor{blue}{1}};
\node at (9.5,7.5) {\textcolor{blue}{1}};
\node at (5.5,13.5) {\textcolor{blue}{1}};
\node at (11.5,6.5) {\textcolor{blue}{1}};
\node at (13.5,6.5) {\textcolor{blue}{1}};
\node at (9.5,5.5) {\textcolor{blue}{1}};
\node at (12.5,4.5) {\textcolor{blue}{1}};
\node at (14.5,4.5) {\textcolor{blue}{1}};
\node at (15.5,7.5) {\textcolor{blue}{1}};

\end{scope}

\begin{scope}[shift={(0,0)}]
  \foreach \x in {0,...,23} {
    \draw[gray!70, line width=0.35pt] (\x,0) -- (\x,23);
  }
  \foreach \y in {0,...,23} {
    \draw[gray!70, line width=0.35pt] (0,\y) -- (23,\y);
  }

\draw[very thick] (8,8) rectangle (15,15);
\draw[very thick] (12,9) rectangle (19,16);
\draw[very thick] (9,4) rectangle (16,11);
\draw[very thick] (4,7) rectangle (11,14);
\draw[very thick] (7,12) rectangle (14,19);

\draw[very thick, blue] (16,10) rectangle (23,17);
\draw[very thick, blue] (13,5) rectangle (20,12);
\draw[very thick, blue] (11,13) rectangle (18,20);

\draw[very thick, blue] (10,0) rectangle (17,7);
\draw[very thick, blue] (5,3) rectangle (12,10);

\draw[very thick, blue] (0,6) rectangle (7,13);
\draw[very thick, blue] (3,11) rectangle (10,18);
\draw[very thick, blue] (6,16) rectangle (13,23);

\node at (10.5,10.5) {\normalsize 1};
\node at (12.5,10.5) {\normalsize 1};
\node at (10.5,12.5) {\normalsize 1};
\node at (12.5,12.5) {\normalsize 1};
\node at (8.5,9.5) {\normalsize 1};
\node at (8.5,11.5) {\normalsize 1};
\node at (11.5,8.5) {\normalsize 1};
\node at (13.5,8.5) {\normalsize 1};
\node at (14.5,11.5) {\normalsize 1};
\node at (14.5,13.5) {\normalsize 1};
\node at (9.5,14.5) {\normalsize 1};
\node at (11.5,14.5) {\normalsize 1};
\node at (16.5,11.5) {\normalsize 1};
\node at (16.5,13.5) {\normalsize 1};
\node at (15.5,9.5) {\normalsize 1};
\node at (17.5,9.5) {\normalsize 1};
\node at (18.5,12.5) {\normalsize 1};
\node at (18.5,14.5) {\normalsize 1};
\node at (13.5,15.5) {\normalsize 1};
\node at (15.5,15.5) {\normalsize 1};
\node at (9.5,16.5) {\normalsize 1};
\node at (11.5,16.5) {\normalsize 1};
\node at (7.5,13.5) {\normalsize 1};
\node at (7.5,15.5) {\normalsize 1};
\node at (13.5,17.5) {\normalsize 1};
\node at (8.5,18.5) {\normalsize 1};
\node at (10.5,18.5) {\normalsize 1};
\node at (6.5,9.5) {\normalsize 1};
\node at (6.5,11.5) {\normalsize 1};
\node at (4.5,8.5) {\normalsize 1};
\node at (4.5,10.5) {\normalsize 1};
\node at (7.5,7.5) {\normalsize 1};
\node at (9.5,7.5) {\normalsize 1};
\node at (5.5,13.5) {\normalsize 1};
\node at (11.5,6.5) {\normalsize 1};
\node at (13.5,6.5) {\normalsize 1};
\node at (9.5,5.5) {\normalsize 1};
\node at (12.5,4.5) {\normalsize 1};
\node at (14.5,4.5) {\normalsize 1};
\node at (15.5,7.5) {\normalsize 1};

\node at (20.5,12.5) {\textcolor{blue}{1}};
\node at (20.5,14.5) {\textcolor{blue}{1}};
\node at (19.5,10.5) {\textcolor{blue}{1}};
\node at (21.5,10.5) {\textcolor{blue}{1}};
\node at (22.5,13.5) {\textcolor{blue}{1}};
\node at (22.5,15.5) {\textcolor{blue}{1}};
\node at (17.5,16.5) {\textcolor{blue}{1}};
\node at (19.5,16.5) {\textcolor{blue}{1}};
\node at (15.5,17.5) {\textcolor{blue}{1}};
\node at (17.5,18.5) {\textcolor{blue}{1}};
\node at (12.5,19.5) {\textcolor{blue}{1}};
\node at (14.5,19.5) {\textcolor{blue}{1}};
\node at (8.5,20.5) {\textcolor{blue}{1}};
\node at (10.5,20.5) {\textcolor{blue}{1}};
\node at (6.5,17.5) {\textcolor{blue}{1}};
\node at (6.5,19.5) {\textcolor{blue}{1}};
\node at (12.5,21.5) {\textcolor{blue}{1}};
\node at (7.5,22.5) {\textcolor{blue}{1}};
\node at (9.5,22.5) {\textcolor{blue}{1}};
\node at (5.5,15.5) {\textcolor{blue}{1}};
\node at (3.5,12.5) {\textcolor{blue}{1}};
\node at (3.5,14.5) {\textcolor{blue}{1}};
\node at (6.5,11.5) {\normalsize 1};
\node at (8.5,11.5) {\normalsize 1};
\node at (9.5,14.5) {\normalsize 1};
\node at (4.5,17.5) {\textcolor{blue}{1}};
\node at (2.5,8.5) {\textcolor{blue}{1}};
\node at (2.5,10.5) {\textcolor{blue}{1}};
\node at (0.5,7.5) {\textcolor{blue}{1}};
\node at (0.5,9.5) {\textcolor{blue}{1}};
\node at (3.5,6.5) {\textcolor{blue}{1}};
\node at (5.5,6.5) {\textcolor{blue}{1}};
\node at (6.5,9.5) {\normalsize 1};
\node at (1.5,12.5) {\textcolor{blue}{1}};
\node at (7.5,5.5) {\textcolor{blue}{1}};
\node at (5.5,4.5) {\textcolor{blue}{1}};
\node at (8.5,3.5) {\textcolor{blue}{1}};
\node at (10.5,3.5) {\textcolor{blue}{1}};
\node at (11.5,8.5) {\normalsize 1};
\node at (8.5,9.5) {\normalsize 1};
\node at (12.5,2.5) {\textcolor{blue}{1}};
\node at (14.5,2.5) {\textcolor{blue}{1}};
\node at (12.5,4.5) {\normalsize 1};
\node at (10.5,1.5) {\textcolor{blue}{1}};
\node at (13.5,0.5) {\textcolor{blue}{1}};
\node at (15.5,0.5) {\textcolor{blue}{1}};
\node at (16.5,3.5) {\textcolor{blue}{1}};
\node at (16.5,5.5) {\textcolor{blue}{1}};
\node at (19.5,5.5) {\textcolor{blue}{1}};
\node at (19.5,7.5) {\textcolor{blue}{1}};
\node at (17.5,7.5) {\textcolor{blue}{1}};

\end{scope}

\end{tikzpicture}
\caption{Expansion of the pattern of squares of color 1 in Case~2.}
\label{fig:16}
\end{figure}
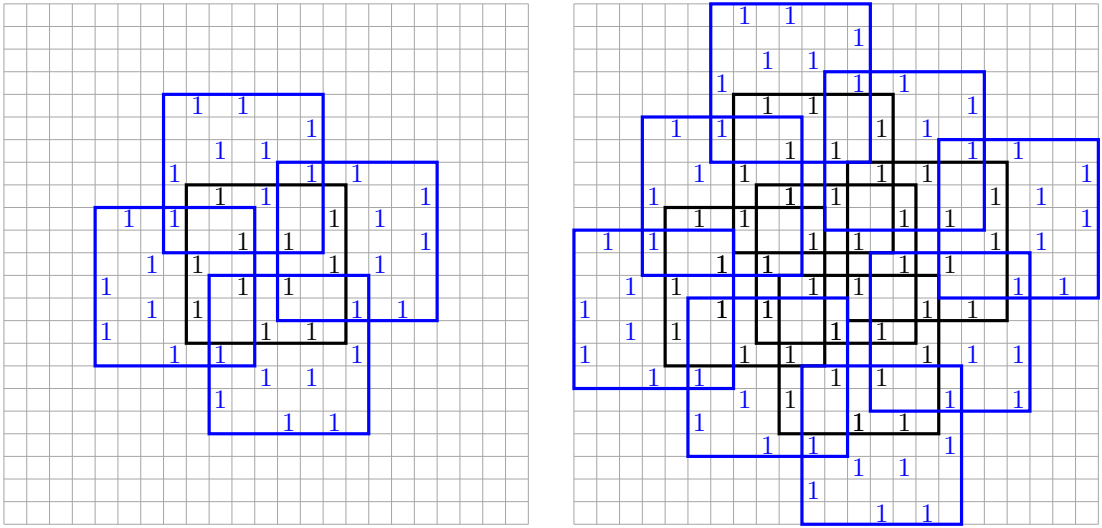

We observe that $ G $ is a critical subgraph. By Observation~\ref{obs:3}\,(ii), let $ k \in \{2, 3, \ldots, 39\} $ be the color that does not appear in $ G $ and there are only two possible positions for color~$ k $, referred to as type-A and type-B. We consider the two cases separately based on the position of color~$ k $.

\textbf{Case 2.1:} $ k $ appears in a type-A position.

We observe that $G(-6, 7)$ follows pattern (i) and is a critical subgraph (see Figure \ref{fig:17}). By Observation~\ref{obs:3}\,(ii), let $ m \in \{2, 3, \ldots, 39\} $ be the color that does not appear in $ G(-6, 7) $ and there are only two possible positions for color~$ m $, referred to as type-A and type-B. We consider the two cases separately based on the position of color~$ m $.

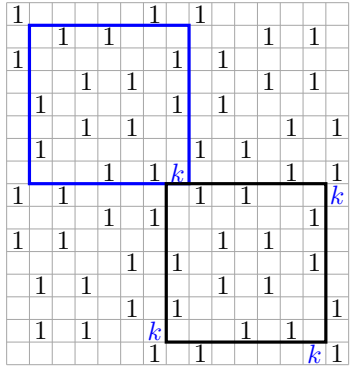
\begin{figure}[ht]
\centering
\begin{tikzpicture}[scale=0.30]

  \foreach \x in {0,...,15} {
    \draw[gray!70, line width=0.35pt] (\x,0) -- (\x,16);
  }
  \foreach \y in {0,...,16} {
    \draw[gray!70, line width=0.35pt] (0,\y) -- (15,\y);
  }

\draw[very thick, blue] (1,8) rectangle (8,15);
\draw[very thick] (7,1) rectangle (14,8);
\node at (3.5,10.5) {\normalsize 1};
\node at (5.5,10.5) {\normalsize 1};
\node at (3.5,12.5) {\normalsize 1};
\node at (5.5,12.5) {\normalsize 1};
\node at (1.5,9.5) {\normalsize 1};
\node at (1.5,11.5) {\normalsize 1};
\node at (4.5,8.5) {\normalsize 1};
\node at (6.5,8.5) {\normalsize 1};
\node at (7.5,11.5) {\normalsize 1};
\node at (7.5,13.5) {\normalsize 1};
\node at (2.5,14.5) {\normalsize 1};
\node at (4.5,14.5) {\normalsize 1};

\node at (9.5,3.5) {\normalsize 1};
\node at (11.5,3.5) {\normalsize 1};
\node at (9.5,5.5) {\normalsize 1};
\node at (11.5,5.5) {\normalsize 1};
\node at (7.5,2.5) {\normalsize 1};
\node at (7.5,4.5) {\normalsize 1};
\node at (10.5,1.5) {\normalsize 1};
\node at (12.5,1.5) {\normalsize 1};
\node at (13.5,4.5) {\normalsize 1};
\node at (13.5,6.5) {\normalsize 1};
\node at (8.5,7.5) {\normalsize 1};
\node at (10.5,7.5) {\normalsize 1};

\node at (9.5,13.5) {1};
\node at (9.5,11.5) {1};
\node at (6.5,15.5) {1};
\node at (8.5,15.5) {1};
\node at (0.5,15.5) {1};
\node at (0.5,13.5) {1};
\node at (8.5,9.5) {1};
\node at (10.5,9.5) {1};
\node at (11.5,12.5) {1};
\node at (11.5,14.5) {1};
\node at (13.5,12.5) {1};
\node at (13.5,14.5) {1};
\node at (12.5,10.5) {1};
\node at (14.5,10.5) {1};
\node at (12.5,8.5) {1};
\node at (14.5,8.5) {1};
\node at (14.5,2.5) {1};
\node at (14.5,0.5) {1};
\node at (0.5,7.5) {1};
\node at (2.5,7.5) {1};
\node at (4.5,6.5) {1};
\node at (6.5,6.5) {1};
\node at (0.5,5.5) {1};
\node at (2.5,5.5) {1};
\node at (5.5,4.5) {1};
\node at (3.5,3.5) {1};
\node at (1.5,3.5) {1};
\node at (5.5,2.5) {1};
\node at (3.5,1.5) {1};
\node at (1.5,1.5) {1};
\node at (6.5,0.5) {1};
\node at (8.5,0.5) {1};
\node at (14.5,7.5) {\textcolor{blue} {$k$}};
\node at (7.5,8.5) {\textcolor{blue} {$k$}};
\node at (6.5,1.5) {\textcolor{blue} {$k$}};
\node at (13.5,0.5) {\textcolor{blue} {$k$}};

\end{tikzpicture}
\caption{$ G $ (black) and $G(-6,7)$ (blue).}
\label{fig:17}
\end{figure}

\textbf{Case 2.1.1:} $ m $ appears in a type-B position (see Figure~\ref{fig:18}).

Since $G(-6,7)$ is a critical subgraph, $m$ appears in a type-B position and $(7,1)$ has color $k$, by Observation \ref{obs:3}\,(iii), $(1,1)$ of $G(-7,7)$ must also has color $k$ (see Figure~\ref{fig:18}). Next, we consider $ G(-8,0) $ (see Figure~\ref{fig:18}). Since $ c(G(-8,0)) = 11 $, by Observation~\ref{obs:2}\,(iv), this subgraph must contain color~$ m $. Due to the distance constraint, the only square that can be assigned color~$ m $ is $ (1,1) $. Hence, $ (1,1) $ must be assigned color~$ m $.

\begin{figure}[ht]
\centering
\begin{tikzpicture}[scale=0.30]
\begin{scope}[shift={(-36,0)}]
  \foreach \x in {0,...,16} {
    \draw[gray!70, line width=0.35pt] (\x,0) -- (\x,16);
  }
  \foreach \y in {0,...,16} {
    \draw[gray!70, line width=0.35pt] (0,\y) -- (16,\y);
  }

\draw[very thick, blue] (2,8) rectangle (9,15);
\draw[very thick] (8,1) rectangle (15,8);

\node at (4.5,10.5) {\normalsize 1};
\node at (6.5,10.5) {\normalsize 1};
\node at (4.5,12.5) {\normalsize 1};
\node at (6.5,12.5) {\normalsize 1};
\node at (2.5,9.5) {\normalsize 1};
\node at (2.5,11.5) {\normalsize 1};
\node at (5.5,8.5) {\normalsize 1};
\node at (7.5,8.5) {\normalsize 1};
\node at (8.5,11.5) {\normalsize 1};
\node at (8.5,13.5) {\normalsize 1};
\node at (3.5,14.5) {\normalsize 1};
\node at (5.5,14.5) {\normalsize 1};

\node at (10.5,3.5) {\normalsize 1};
\node at (12.5,3.5) {\normalsize 1};
\node at (10.5,5.5) {\normalsize 1};
\node at (12.5,5.5) {\normalsize 1};
\node at (8.5,2.5) {\normalsize 1};
\node at (8.5,4.5) {\normalsize 1};
\node at (11.5,1.5) {\normalsize 1};
\node at (13.5,1.5) {\normalsize 1};
\node at (14.5,4.5) {\normalsize 1};
\node at (14.5,6.5) {\normalsize 1};
\node at (9.5,7.5) {\normalsize 1};
\node at (11.5,7.5) {\normalsize 1};

\node at (10.5,13.5) {1};
\node at (10.5,11.5) {1};
\node at (7.5,15.5) {1};
\node at (9.5,15.5) {1};
\node at (1.5,15.5) {1};
\node at (1.5,13.5) {1};
\node at (9.5,9.5) {1};
\node at (11.5,9.5) {1};
\node at (12.5,12.5) {1};
\node at (12.5,14.5) {1};
\node at (14.5,12.5) {1};
\node at (14.5,14.5) {1};
\node at (13.5,10.5) {1};
\node at (15.5,10.5) {1};
\node at (13.5,8.5) {1};
\node at (15.5,8.5) {1};
\node at (15.5,2.5) {1};
\node at (15.5,0.5) {1};
\node at (1.5,7.5) {1};
\node at (3.5,7.5) {1};
\node at (5.5,6.5) {1};
\node at (7.5,6.5) {1};
\node at (1.5,5.5) {1};
\node at (3.5,5.5) {1};
\node at (6.5,4.5) {1};
\node at (4.5,3.5) {1};
\node at (2.5,3.5) {1};
\node at (6.5,2.5) {1};
\node at (4.5,1.5) {1};
\node at (2.5,1.5) {1};
\node at (7.5,0.5) {1};
\node at (9.5,0.5) {1};
\node at (0.5,11.5) {1};
\node at (0.5,9.5) {1};
\node at (0.5,2.5) {1};
\node at (0.5,0.5) {1};

\node at (8.5,15.5) {\textcolor{darkgreen} {$m$}};
\node at (9.5,8.5) {\textcolor{darkgreen} {$m$}};
\node at (2.5,7.5) {\textcolor{darkgreen} {$m$}};
\node at (1.5,14.5) {\textcolor{darkgreen} {$m$}};

\node at (8.5,8.5) {\textcolor{blue} {$k$}};
\node at (7.5,1.5) {\textcolor{blue} {$k$}};
\node at (14.5,0.5) {\textcolor{blue} {$k$}};
\node at (15.5,7.5) {\textcolor{blue} {$k$}};

\end{scope}

\begin{scope}[shift={(-18,0)}]
  \foreach \x in {0,...,16} {
    \draw[gray!70, line width=0.35pt] (\x,0) -- (\x,16);
  }
  \foreach \y in {0,...,16} {
    \draw[gray!70, line width=0.35pt] (0,\y) -- (16,\y);
  }

\draw[very thick, blue] (2,8) rectangle (9,15);
\draw[very thick] (8,1) rectangle (15,8);

\node at (4.5,10.5) {\normalsize 1};
\node at (6.5,10.5) {\normalsize 1};
\node at (4.5,12.5) {\normalsize 1};
\node at (6.5,12.5) {\normalsize 1};
\node at (2.5,9.5) {\normalsize 1};
\node at (2.5,11.5) {\normalsize 1};
\node at (5.5,8.5) {\normalsize 1};
\node at (7.5,8.5) {\normalsize 1};
\node at (8.5,11.5) {\normalsize 1};
\node at (8.5,13.5) {\normalsize 1};
\node at (3.5,14.5) {\normalsize 1};
\node at (5.5,14.5) {\normalsize 1};

\node at (10.5,3.5) {\normalsize 1};
\node at (12.5,3.5) {\normalsize 1};
\node at (10.5,5.5) {\normalsize 1};
\node at (12.5,5.5) {\normalsize 1};
\node at (8.5,2.5) {\normalsize 1};
\node at (8.5,4.5) {\normalsize 1};
\node at (11.5,1.5) {\normalsize 1};
\node at (13.5,1.5) {\normalsize 1};
\node at (14.5,4.5) {\normalsize 1};
\node at (14.5,6.5) {\normalsize 1};
\node at (9.5,7.5) {\normalsize 1};
\node at (11.5,7.5) {\normalsize 1};

\node at (10.5,13.5) {1};
\node at (10.5,11.5) {1};
\node at (7.5,15.5) {1};
\node at (9.5,15.5) {1};
\node at (1.5,15.5) {1};
\node at (1.5,13.5) {1};
\node at (9.5,9.5) {1};
\node at (11.5,9.5) {1};
\node at (12.5,12.5) {1};
\node at (12.5,14.5) {1};
\node at (14.5,12.5) {1};
\node at (14.5,14.5) {1};
\node at (13.5,10.5) {1};
\node at (15.5,10.5) {1};
\node at (13.5,8.5) {1};
\node at (15.5,8.5) {1};
\node at (15.5,2.5) {1};
\node at (15.5,0.5) {1};
\node at (1.5,7.5) {1};
\node at (3.5,7.5) {1};
\node at (5.5,6.5) {1};
\node at (7.5,6.5) {1};
\node at (1.5,5.5) {1};
\node at (3.5,5.5) {1};
\node at (6.5,4.5) {1};
\node at (4.5,3.5) {1};
\node at (2.5,3.5) {1};
\node at (6.5,2.5) {1};
\node at (4.5,1.5) {1};
\node at (2.5,1.5) {1};
\node at (7.5,0.5) {1};
\node at (9.5,0.5) {1};
\node at (0.5,11.5) {1};
\node at (0.5,9.5) {1};
\node at (0.5,2.5) {1};
\node at (0.5,0.5) {1};

\node at (8.5,15.5) {\textcolor{darkgreen} {$m$}};
\node at (9.5,8.5) {\textcolor{darkgreen} {$m$}};
\node at (2.5,7.5) {\textcolor{darkgreen} {$m$}};
\node at (1.5,14.5) {\textcolor{darkgreen} {$m$}};

\node at (8.5,8.5) {\textcolor{blue} {$k$}};
\node at (7.5,1.5) {\textcolor{blue} {$k$}};
\node at (14.5,0.5) {\textcolor{blue} {$k$}};
\node at (15.5,7.5) {\textcolor{blue} {$k$}};
\node at (1.5,8.5) {\textcolor{blue} {$k$}};

\end{scope}

\begin{scope}[shift={(0,0)}]
  \foreach \x in {0,...,16} {
    \draw[gray!70, line width=0.35pt] (\x,0) -- (\x,16);
  }
  \foreach \y in {0,...,16} {
    \draw[gray!70, line width=0.35pt] (0,\y) -- (16,\y);
  }

\draw[very thick, blue] (2,8) rectangle (9,15);
\draw[very thick] (8,1) rectangle (15,8);
\draw[very thick, dashed] (0,1) rectangle (7,8);

\node at (4.5,10.5) {\normalsize 1};
\node at (6.5,10.5) {\normalsize 1};
\node at (4.5,12.5) {\normalsize 1};
\node at (6.5,12.5) {\normalsize 1};
\node at (2.5,9.5) {\normalsize 1};
\node at (2.5,11.5) {\normalsize 1};
\node at (5.5,8.5) {\normalsize 1};
\node at (7.5,8.5) {\normalsize 1};
\node at (8.5,11.5) {\normalsize 1};
\node at (8.5,13.5) {\normalsize 1};
\node at (3.5,14.5) {\normalsize 1};
\node at (5.5,14.5) {\normalsize 1};

\node at (10.5,3.5) {\normalsize 1};
\node at (12.5,3.5) {\normalsize 1};
\node at (10.5,5.5) {\normalsize 1};
\node at (12.5,5.5) {\normalsize 1};
\node at (8.5,2.5) {\normalsize 1};
\node at (8.5,4.5) {\normalsize 1};
\node at (11.5,1.5) {\normalsize 1};
\node at (13.5,1.5) {\normalsize 1};
\node at (14.5,4.5) {\normalsize 1};
\node at (14.5,6.5) {\normalsize 1};
\node at (9.5,7.5) {\normalsize 1};
\node at (11.5,7.5) {\normalsize 1};

\node at (10.5,13.5) {1};
\node at (10.5,11.5) {1};
\node at (7.5,15.5) {1};
\node at (9.5,15.5) {1};
\node at (1.5,15.5) {1};
\node at (1.5,13.5) {1};
\node at (9.5,9.5) {1};
\node at (11.5,9.5) {1};
\node at (12.5,12.5) {1};
\node at (12.5,14.5) {1};
\node at (14.5,12.5) {1};
\node at (14.5,14.5) {1};
\node at (13.5,10.5) {1};
\node at (15.5,10.5) {1};
\node at (13.5,8.5) {1};
\node at (15.5,8.5) {1};
\node at (15.5,2.5) {1};
\node at (15.5,0.5) {1};
\node at (1.5,7.5) {1};
\node at (3.5,7.5) {1};
\node at (5.5,6.5) {1};
\node at (7.5,6.5) {1};
\node at (1.5,5.5) {1};
\node at (3.5,5.5) {1};
\node at (6.5,4.5) {1};
\node at (4.5,3.5) {1};
\node at (2.5,3.5) {1};
\node at (6.5,2.5) {1};
\node at (4.5,1.5) {1};
\node at (2.5,1.5) {1};
\node at (7.5,0.5) {1};
\node at (9.5,0.5) {1};
\node at (0.5,11.5) {1};
\node at (0.5,9.5) {1};
\node at (0.5,2.5) {1};
\node at (0.5,0.5) {1};

\node at (8.5,15.5) {\textcolor{darkgreen} {$m$}};
\node at (9.5,8.5) {\textcolor{darkgreen} {$m$}};
\node at (2.5,7.5) {\textcolor{darkgreen} {$m$}};
\node at (1.5,14.5) {\textcolor{darkgreen} {$m$}};

\node at (8.5,8.5) {\textcolor{blue} {$k$}};
\node at (7.5,1.5) {\textcolor{blue} {$k$}};
\node at (14.5,0.5) {\textcolor{blue} {$k$}};
\node at (15.5,7.5) {\textcolor{blue} {$k$}};
\node at (1.5,8.5) {\textcolor{blue} {$k$}};
\node at (0.5,1.5) {\textcolor{blue} {$k$}};

\end{scope}

\end{tikzpicture}
\caption{$ G $ (black), $G(-6,7)$ (blue) and $G(-8,0)$ (dashed).}
\label{fig:18}
\end{figure}
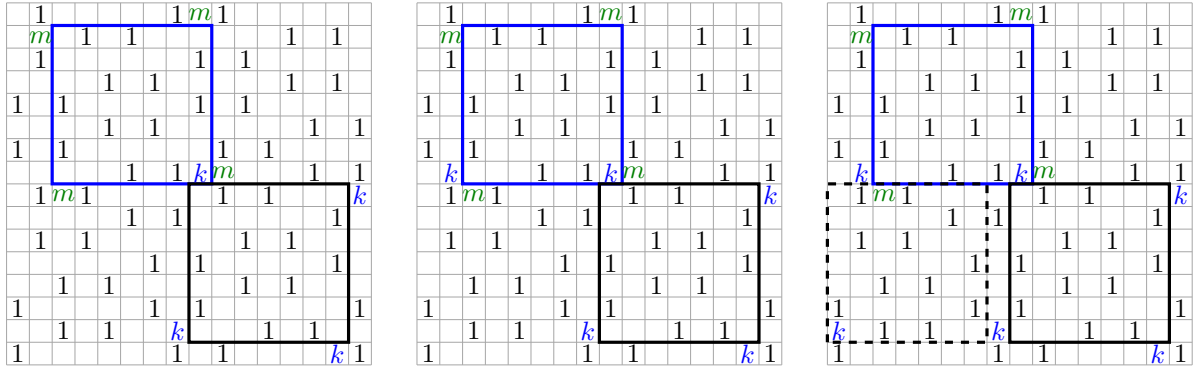

We then consider $ G(-14,1) $ (see Figure~\ref{fig:19}). Since $ c(G(-14,1)) = 11 $, by Observation~\ref{obs:2}\,(iv), this subgraph must contain a square of color~$ m $. Due to the distance constraint, color~$ m $ must be assigned to $ (1,7) $.

Next, we observe that $ G(-7,-6) $ is a critical subgraph. By Observation~\ref{obs:3}\,(ii), let $ n \in \{2, 3, \ldots, 39\} $ be the color that does not appear in $ G(-7,-6) $, and there are only two possible positions for color~$ n $, referred to as type-A and type-B. It follows that $ n \neq k $ since color~$ k $ appears in $ G(-7,-6) $. Moreover, color~$ n $ must appear in a type-A position, as the square in the type-A position has already been assigned color~$ k $.

Similarly, since $ G(-13,1) $ is a critical subgraph, by Observation~\ref{obs:3}\,(ii), let $ p \in \{2, 3, \ldots, 39\} $ be the color that does not appear in $ G(-13,1) $. Then $ p \neq k $, and it must appear in a type-A position (see Figure \ref{fig:19}).

\begin{figure}[ht]
\centering
\begin{tikzpicture}[scale=0.30]
\begin{scope}[shift={(-24,0)}]
  \foreach \x in {0,...,22} {
    \draw[gray!70, line width=0.35pt] (\x,0) -- (\x,22);
  }
  \foreach \y in {0,...,22} {
    \draw[gray!70, line width=0.35pt] (0,\y) -- (22,\y);
  }

\draw[very thick, blue] (8,14) rectangle (15,21);
\draw[very thick] (14,7) rectangle (21,14);
\draw[very thick, dashed] (0,8) rectangle (7,15);

\node at (10.5,16.5) {\normalsize 1};
\node at (12.5,16.5) {\normalsize 1};
\node at (10.5,18.5) {\normalsize 1};
\node at (12.5,18.5) {\normalsize 1};
\node at (8.5,15.5) {\normalsize 1};
\node at (8.5,17.5) {\normalsize 1};
\node at (11.5,14.5) {\normalsize 1};
\node at (13.5,14.5) {\normalsize 1};
\node at (14.5,17.5) {\normalsize 1};
\node at (14.5,19.5) {\normalsize 1};
\node at (9.5,20.5) {\normalsize 1};
\node at (11.5,20.5) {\normalsize 1};

\node at (16.5,9.5) {\normalsize 1};
\node at (18.5,9.5) {\normalsize 1};
\node at (16.5,11.5) {\normalsize 1};
\node at (18.5,11.5) {\normalsize 1};
\node at (14.5,8.5) {\normalsize 1};
\node at (14.5,10.5) {\normalsize 1};
\node at (17.5,7.5) {\normalsize 1};
\node at (19.5,7.5) {\normalsize 1};
\node at (20.5,10.5) {\normalsize 1};
\node at (20.5,12.5) {\normalsize 1};
\node at (15.5,13.5) {\normalsize 1};
\node at (17.5,13.5) {\normalsize 1};

\node at (16.5,19.5) {1};
\node at (16.5,17.5) {1};
\node at (13.5,21.5) {1};
\node at (15.5,21.5) {1};
\node at (7.5,21.5) {1};
\node at (7.5,19.5) {1};
\node at (15.5,15.5) {1};
\node at (17.5,15.5) {1};
\node at (18.5,18.5) {1};
\node at (18.5,20.5) {1};
\node at (20.5,18.5) {1};
\node at (20.5,20.5) {1};
\node at (19.5,16.5) {1};
\node at (21.5,16.5) {1};
\node at (19.5,14.5) {1};
\node at (21.5,14.5) {1};
\node at (21.5,8.5) {1};
\node at (21.5,6.5) {1};
\node at (7.5,13.5) {1};
\node at (9.5,13.5) {1};
\node at (11.5,12.5) {1};
\node at (13.5,12.5) {1};
\node at (7.5,11.5) {1};
\node at (9.5,11.5) {1};
\node at (12.5,10.5) {1};
\node at (10.5,9.5) {1};
\node at (8.5,9.5) {1};
\node at (12.5,8.5) {1};
\node at (10.5,7.5) {1};
\node at (8.5,7.5) {1};
\node at (13.5,6.5) {1};
\node at (15.5,6.5) {1};
\node at (6.5,17.5) {1};
\node at (6.5,15.5) {1};
\node at (6.5,8.5) {1};
\node at (6.5,6.5) {1};
\node at (9.5,5.5) {1};
\node at (11.5,5.5) {1};
\node at (17.5,5.5) {1};
\node at (19.5,5.5) {1};
\node at (7.5,4.5) {1};
\node at (13.5,4.5) {1};
\node at (15.5,4.5) {1};
\node at (9.5,3.5) {1};
\node at (11.5,3.5) {1};
\node at (18.5,3.5) {1};
\node at (20.5,3.5) {1};
\node at (7.5,2.5) {1};
\node at (14.5,2.5) {1};
\node at (16.5,2.5) {1};
\node at (10.5,1.5) {1};
\node at (12.5,1.5) {1};
\node at (18.5,1.5) {1};
\node at (20.5,1.5) {1};
\node at (6.5,0.5) {1};
\node at (8.5,0.5) {1};
\node at (14.5,0.5) {1};
\node at (16.5,0.5) {1};

\node at (5.5,2.5) {1};
\node at (5.5,4.5) {1};
\node at (5.5,10.5) {1};
\node at (5.5,12.5) {1};
\node at (5.5,19.5) {1};
\node at (5.5,21.5) {1};
\node at (4.5,6.5) {1};
\node at (4.5,8.5) {1};
\node at (4.5,14.5) {1};
\node at (4.5,16.5) {1};
\node at (3.5,1.5) {1};
\node at (3.5,3.5) {1};
\node at (3.5,10.5) {1};
\node at (3.5,12.5) {1};
\node at (3.5,18.5) {1};
\node at (3.5,20.5) {1};
\node at (2.5,5.5) {1};
\node at (2.5,7.5) {1};
\node at (2.5,14.5) {1};
\node at (2.5,16.5) {1};
\node at (1.5,1.5) {1};
\node at (1.5,3.5) {1};
\node at (1.5,9.5) {1};
\node at (1.5,11.5) {1};
\node at (1.5,18.5) {1};
\node at (1.5,20.5) {1};
\node at (0.5,5.5) {1};
\node at (0.5,7.5) {1};
\node at (0.5,13.5) {1};
\node at (0.5,15.5) {1};

\node at (14.5,21.5) {\textcolor{darkgreen} {$m$}};
\node at (15.5,14.5) {\textcolor{darkgreen} {$m$}};
\node at (8.5,13.5) {\textcolor{darkgreen} {$m$}};
\node at (7.5,20.5) {\textcolor{darkgreen} {$m$}};

\node at (14.5,14.5) {\textcolor{blue} {$k$}};
\node at (13.5,7.5) {\textcolor{blue} {$k$}};
\node at (20.5,6.5) {\textcolor{blue} {$k$}};
\node at (21.5,13.5) {\textcolor{blue} {$k$}};
\node at (7.5,14.5) {\textcolor{blue} {$k$}};
\node at (6.5,7.5) {\textcolor{blue} {$k$}};
\node at (0.5,14.5) {\textcolor{blue} {$k$}};

\end{scope}

\begin{scope}[shift={(0,0)}]
  \foreach \x in {0,...,22} {
    \draw[gray!70, line width=0.35pt] (\x,0) -- (\x,22);
  }
  \foreach \y in {0,...,22} {
    \draw[gray!70, line width=0.35pt] (0,\y) -- (22,\y);
  }

\draw[very thick, blue] (8,14) rectangle (15,21);
\draw[very thick] (14,7) rectangle (21,14);
\draw[very thick, darkgreen] (7,1) rectangle (14,8);
\draw[very thick, orange] (1,8) rectangle (8,15);

\node at (10.5,16.5) {\normalsize 1};
\node at (12.5,16.5) {\normalsize 1};
\node at (10.5,18.5) {\normalsize 1};
\node at (12.5,18.5) {\normalsize 1};
\node at (8.5,15.5) {\normalsize 1};
\node at (8.5,17.5) {\normalsize 1};
\node at (11.5,14.5) {\normalsize 1};
\node at (13.5,14.5) {\normalsize 1};
\node at (14.5,17.5) {\normalsize 1};
\node at (14.5,19.5) {\normalsize 1};
\node at (9.5,20.5) {\normalsize 1};
\node at (11.5,20.5) {\normalsize 1};

\node at (16.5,9.5) {\normalsize 1};
\node at (18.5,9.5) {\normalsize 1};
\node at (16.5,11.5) {\normalsize 1};
\node at (18.5,11.5) {\normalsize 1};
\node at (14.5,8.5) {\normalsize 1};
\node at (14.5,10.5) {\normalsize 1};
\node at (17.5,7.5) {\normalsize 1};
\node at (19.5,7.5) {\normalsize 1};
\node at (20.5,10.5) {\normalsize 1};
\node at (20.5,12.5) {\normalsize 1};
\node at (15.5,13.5) {\normalsize 1};
\node at (17.5,13.5) {\normalsize 1};

\node at (16.5,19.5) {1};
\node at (16.5,17.5) {1};
\node at (13.5,21.5) {1};
\node at (15.5,21.5) {1};
\node at (7.5,21.5) {1};
\node at (7.5,19.5) {1};
\node at (15.5,15.5) {1};
\node at (17.5,15.5) {1};
\node at (18.5,18.5) {1};
\node at (18.5,20.5) {1};
\node at (20.5,18.5) {1};
\node at (20.5,20.5) {1};
\node at (19.5,16.5) {1};
\node at (21.5,16.5) {1};
\node at (19.5,14.5) {1};
\node at (21.5,14.5) {1};
\node at (21.5,8.5) {1};
\node at (21.5,6.5) {1};
\node at (7.5,13.5) {1};
\node at (9.5,13.5) {1};
\node at (11.5,12.5) {1};
\node at (13.5,12.5) {1};
\node at (7.5,11.5) {1};
\node at (9.5,11.5) {1};
\node at (12.5,10.5) {1};
\node at (10.5,9.5) {1};
\node at (8.5,9.5) {1};
\node at (12.5,8.5) {1};
\node at (10.5,7.5) {1};
\node at (8.5,7.5) {1};
\node at (13.5,6.5) {1};
\node at (15.5,6.5) {1};
\node at (6.5,17.5) {1};
\node at (6.5,15.5) {1};
\node at (6.5,8.5) {1};
\node at (6.5,6.5) {1};
\node at (9.5,5.5) {1};
\node at (11.5,5.5) {1};
\node at (17.5,5.5) {1};
\node at (19.5,5.5) {1};
\node at (7.5,4.5) {1};
\node at (13.5,4.5) {1};
\node at (15.5,4.5) {1};
\node at (9.5,3.5) {1};
\node at (11.5,3.5) {1};
\node at (18.5,3.5) {1};
\node at (20.5,3.5) {1};
\node at (7.5,2.5) {1};
\node at (14.5,2.5) {1};
\node at (16.5,2.5) {1};
\node at (10.5,1.5) {1};
\node at (12.5,1.5) {1};
\node at (18.5,1.5) {1};
\node at (20.5,1.5) {1};
\node at (6.5,0.5) {1};
\node at (8.5,0.5) {1};
\node at (14.5,0.5) {1};
\node at (16.5,0.5) {1};

\node at (5.5,2.5) {1};
\node at (5.5,4.5) {1};
\node at (5.5,10.5) {1};
\node at (5.5,12.5) {1};
\node at (5.5,19.5) {1};
\node at (5.5,21.5) {1};
\node at (4.5,6.5) {1};
\node at (4.5,8.5) {1};
\node at (4.5,14.5) {1};
\node at (4.5,16.5) {1};
\node at (3.5,1.5) {1};
\node at (3.5,3.5) {1};
\node at (3.5,10.5) {1};
\node at (3.5,12.5) {1};
\node at (3.5,18.5) {1};
\node at (3.5,20.5) {1};
\node at (2.5,5.5) {1};
\node at (2.5,7.5) {1};
\node at (2.5,14.5) {1};
\node at (2.5,16.5) {1};
\node at (1.5,1.5) {1};
\node at (1.5,3.5) {1};
\node at (1.5,9.5) {1};
\node at (1.5,11.5) {1};
\node at (1.5,18.5) {1};
\node at (1.5,20.5) {1};
\node at (0.5,5.5) {1};
\node at (0.5,7.5) {1};
\node at (0.5,13.5) {1};
\node at (0.5,15.5) {1};

\node at (14.5,21.5) {\textcolor{darkgreen} {$m$}};
\node at (15.5,14.5) {\textcolor{darkgreen} {$m$}};
\node at (8.5,13.5) {\textcolor{darkgreen} {$m$}};
\node at (7.5,20.5) {\textcolor{darkgreen} {$m$}};

\node at (14.5,14.5) {\textcolor{blue} {$k$}};
\node at (13.5,7.5) {\textcolor{blue} {$k$}};
\node at (20.5,6.5) {\textcolor{blue} {$k$}};
\node at (21.5,13.5) {\textcolor{blue} {$k$}};
\node at (7.5,14.5) {\textcolor{blue} {$k$}};
\node at (6.5,7.5) {\textcolor{blue} {$k$}};
\node at (0.5,14.5) {\textcolor{blue} {$k$}};

\node at (13.5,0.5) {\textcolor{orange} {$n$}};
\node at (14.5,7.5) {\textcolor{orange} {$n$}};
\node at (7.5,8.5) {\textcolor{orange} {$n$}};
\node at (6.5,1.5) {\textcolor{orange} {$n$}};

\node at (7.5,7.5) {\textcolor{purple} {$p$}};
\node at (8.5,14.5) {\textcolor{purple} {$p$}};
\node at (1.5,15.5) {\textcolor{purple} {$p$}};
\node at (0.5,8.5) {\textcolor{purple} {$p$}};

\end{scope}

\end{tikzpicture}
\caption{$ G $ (black), $G(-6,7)$ (blue), $G(-14,1)$ (dashed), $G(-7,-6)$ (green) and $G(-13,1)$ (orange).}
\label{fig:19}
\end{figure}
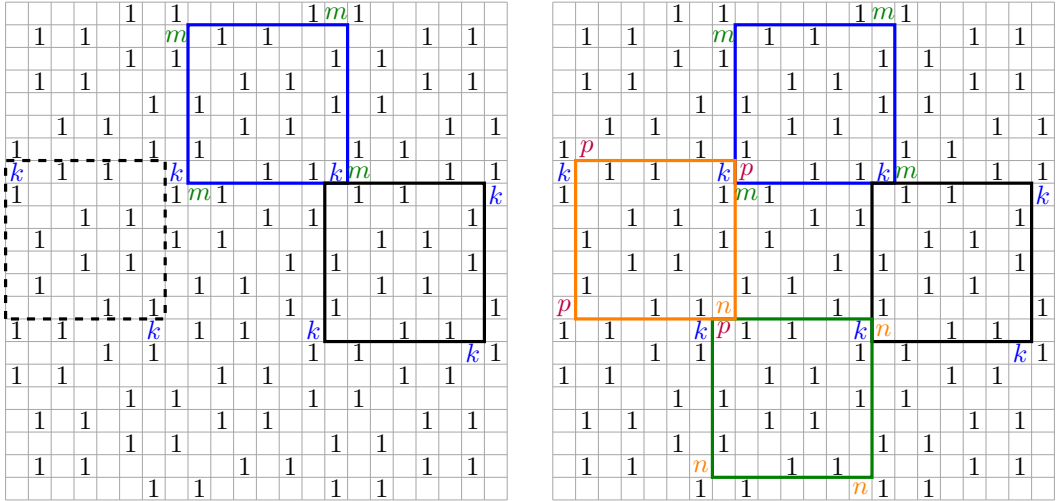

We then consider $ G(-13,1) $ (see Figure~\ref{fig:20}). Since $G(-13,1)$ is a critical subgraph, color $p$ appears in a type-A position and $(7,1)$ has color $n$, by Observation~\ref{obs:3}\,(iii), $(7,7)$ of $G(-13,2)$ must also has color $n$. Next, we consider $ G(-6,2) $. Since $ c(G(-6,2)) = 11 $, by Observation~\ref{obs:2}\,(iv), color~$ n $ must appear in this subgraph. Due to the distance constraint, the only square that can be assigned color~$ n $ is $ (7,7) $. Hence, $ (7,7) $ must be assigned color~$ n $.

Finally, we consider $ G(-1,1) $ (see Figure~\ref{fig:20}). Due to the distance constraint, there is no square in this subgraph that can be assigned color~$ n $. Since $ c(G(-1,1)) = 11 $, by Observation~\ref{obs:2}\,(iv), color~$ n $ must appear in this subgraph, which leads to a contradiction.

\begin{figure}[ht]
\centering
\begin{tikzpicture}[scale=0.30]
\begin{scope}[shift={(-24,0)}]
  \foreach \x in {0,...,22} {
    \draw[gray!70, line width=0.35pt] (\x,0) -- (\x,22);
  }
  \foreach \y in {0,...,22} {
    \draw[gray!70, line width=0.35pt] (0,\y) -- (22,\y);
  }

\draw[very thick,blue] (8,14) rectangle (15,21);
\draw[very thick] (14,7) rectangle (21,14);
\draw[very thick, darkgreen] (7,1) rectangle (14,8);
\draw[very thick, orange] (1,8) rectangle (8,15);
\draw[very thick, dashed] (8,9) rectangle (15,16);

\node at (10.5,16.5) {\normalsize 1};
\node at (12.5,16.5) {\normalsize 1};
\node at (10.5,18.5) {\normalsize 1};
\node at (12.5,18.5) {\normalsize 1};
\node at (8.5,15.5) {\normalsize 1};
\node at (8.5,17.5) {\normalsize 1};
\node at (11.5,14.5) {\normalsize 1};
\node at (13.5,14.5) {\normalsize 1};
\node at (14.5,17.5) {\normalsize 1};
\node at (14.5,19.5) {\normalsize 1};
\node at (9.5,20.5) {\normalsize 1};
\node at (11.5,20.5) {\normalsize 1};

\node at (16.5,9.5) {\normalsize 1};
\node at (18.5,9.5) {\normalsize 1};
\node at (16.5,11.5) {\normalsize 1};
\node at (18.5,11.5) {\normalsize 1};
\node at (14.5,8.5) {\normalsize 1};
\node at (14.5,10.5) {\normalsize 1};
\node at (17.5,7.5) {\normalsize 1};
\node at (19.5,7.5) {\normalsize 1};
\node at (20.5,10.5) {\normalsize 1};
\node at (20.5,12.5) {\normalsize 1};
\node at (15.5,13.5) {\normalsize 1};
\node at (17.5,13.5) {\normalsize 1};

\node at (16.5,19.5) {1};
\node at (16.5,17.5) {1};
\node at (13.5,21.5) {1};
\node at (15.5,21.5) {1};
\node at (7.5,21.5) {1};
\node at (7.5,19.5) {1};
\node at (15.5,15.5) {1};
\node at (17.5,15.5) {1};
\node at (18.5,18.5) {1};
\node at (18.5,20.5) {1};
\node at (20.5,18.5) {1};
\node at (20.5,20.5) {1};
\node at (19.5,16.5) {1};
\node at (21.5,16.5) {1};
\node at (19.5,14.5) {1};
\node at (21.5,14.5) {1};
\node at (21.5,8.5) {1};
\node at (21.5,6.5) {1};
\node at (7.5,13.5) {1};
\node at (9.5,13.5) {1};
\node at (11.5,12.5) {1};
\node at (13.5,12.5) {1};
\node at (7.5,11.5) {1};
\node at (9.5,11.5) {1};
\node at (12.5,10.5) {1};
\node at (10.5,9.5) {1};
\node at (8.5,9.5) {1};
\node at (12.5,8.5) {1};
\node at (10.5,7.5) {1};
\node at (8.5,7.5) {1};
\node at (13.5,6.5) {1};
\node at (15.5,6.5) {1};
\node at (6.5,17.5) {1};
\node at (6.5,15.5) {1};
\node at (6.5,8.5) {1};
\node at (6.5,6.5) {1};
\node at (9.5,5.5) {1};
\node at (11.5,5.5) {1};
\node at (17.5,5.5) {1};
\node at (19.5,5.5) {1};
\node at (7.5,4.5) {1};
\node at (13.5,4.5) {1};
\node at (15.5,4.5) {1};
\node at (9.5,3.5) {1};
\node at (11.5,3.5) {1};
\node at (18.5,3.5) {1};
\node at (20.5,3.5) {1};
\node at (7.5,2.5) {1};
\node at (14.5,2.5) {1};
\node at (16.5,2.5) {1};
\node at (10.5,1.5) {1};
\node at (12.5,1.5) {1};
\node at (18.5,1.5) {1};
\node at (20.5,1.5) {1};
\node at (6.5,0.5) {1};
\node at (8.5,0.5) {1};
\node at (14.5,0.5) {1};
\node at (16.5,0.5) {1};

\node at (5.5,2.5) {1};
\node at (5.5,4.5) {1};
\node at (5.5,10.5) {1};
\node at (5.5,12.5) {1};
\node at (5.5,19.5) {1};
\node at (5.5,21.5) {1};
\node at (4.5,6.5) {1};
\node at (4.5,8.5) {1};
\node at (4.5,14.5) {1};
\node at (4.5,16.5) {1};
\node at (3.5,1.5) {1};
\node at (3.5,3.5) {1};
\node at (3.5,10.5) {1};
\node at (3.5,12.5) {1};
\node at (3.5,18.5) {1};
\node at (3.5,20.5) {1};
\node at (2.5,5.5) {1};
\node at (2.5,7.5) {1};
\node at (2.5,14.5) {1};
\node at (2.5,16.5) {1};
\node at (1.5,1.5) {1};
\node at (1.5,3.5) {1};
\node at (1.5,9.5) {1};
\node at (1.5,11.5) {1};
\node at (1.5,18.5) {1};
\node at (1.5,20.5) {1};
\node at (0.5,5.5) {1};
\node at (0.5,7.5) {1};
\node at (0.5,13.5) {1};
\node at (0.5,15.5) {1};

\node at (14.5,21.5) {\textcolor{darkgreen} {$m$}};
\node at (15.5,14.5) {\textcolor{darkgreen} {$m$}};
\node at (8.5,13.5) {\textcolor{darkgreen} {$m$}};
\node at (7.5,20.5) {\textcolor{darkgreen} {$m$}};

\node at (14.5,14.5) {\textcolor{blue} {$k$}};
\node at (13.5,7.5) {\textcolor{blue} {$k$}};
\node at (20.5,6.5) {\textcolor{blue} {$k$}};
\node at (21.5,13.5) {\textcolor{blue} {$k$}};
\node at (7.5,14.5) {\textcolor{blue} {$k$}};
\node at (6.5,7.5) {\textcolor{blue} {$k$}};
\node at (0.5,14.5) {\textcolor{blue} {$k$}};

\node at (13.5,0.5) {\textcolor{orange} {$n$}};
\node at (14.5,7.5) {\textcolor{orange} {$n$}};
\node at (7.5,8.5) {\textcolor{orange} {$n$}};
\node at (6.5,1.5) {\textcolor{orange} {$n$}};
\node at (7.5,15.5) {\textcolor{orange} {$n$}};
\node at (14.5,15.5) {\textcolor{orange} {$n$}};

\node at (7.5,7.5) {\textcolor{purple} {$p$}};
\node at (8.5,14.5) {\textcolor{purple} {$p$}};
\node at (1.5,15.5) {\textcolor{purple} {$p$}};
\node at (0.5,8.5) {\textcolor{purple} {$p$}};

\end{scope}

\begin{scope}[shift={(0,0)}]
  \foreach \x in {0,...,22} {
    \draw[gray!70, line width=0.35pt] (\x,0) -- (\x,22);
  }
  \foreach \y in {0,...,22} {
    \draw[gray!70, line width=0.35pt] (0,\y) -- (22,\y);
  }

\draw[very thick, blue] (8,14) rectangle (15,21);
\draw[very thick] (14,7) rectangle (21,14);
\draw[very thick, darkgreen] (7,1) rectangle (14,8);
\draw[very thick, orange] (1,8) rectangle (8,15);
\draw[very thick, red] (13,8) rectangle (20,15);

\node at (10.5,16.5) {\normalsize 1};
\node at (12.5,16.5) {\normalsize 1};
\node at (10.5,18.5) {\normalsize 1};
\node at (12.5,18.5) {\normalsize 1};
\node at (8.5,15.5) {\normalsize 1};
\node at (8.5,17.5) {\normalsize 1};
\node at (11.5,14.5) {\normalsize 1};
\node at (13.5,14.5) {\normalsize 1};
\node at (14.5,17.5) {\normalsize 1};
\node at (14.5,19.5) {\normalsize 1};
\node at (9.5,20.5) {\normalsize 1};
\node at (11.5,20.5) {\normalsize 1};

\node at (16.5,9.5) {\normalsize 1};
\node at (18.5,9.5) {\normalsize 1};
\node at (16.5,11.5) {\normalsize 1};
\node at (18.5,11.5) {\normalsize 1};
\node at (14.5,8.5) {\normalsize 1};
\node at (14.5,10.5) {\normalsize 1};
\node at (17.5,7.5) {\normalsize 1};
\node at (19.5,7.5) {\normalsize 1};
\node at (20.5,10.5) {\normalsize 1};
\node at (20.5,12.5) {\normalsize 1};
\node at (15.5,13.5) {\normalsize 1};
\node at (17.5,13.5) {\normalsize 1};

\node at (16.5,19.5) {1};
\node at (16.5,17.5) {1};
\node at (13.5,21.5) {1};
\node at (15.5,21.5) {1};
\node at (7.5,21.5) {1};
\node at (7.5,19.5) {1};
\node at (15.5,15.5) {1};
\node at (17.5,15.5) {1};
\node at (18.5,18.5) {1};
\node at (18.5,20.5) {1};
\node at (20.5,18.5) {1};
\node at (20.5,20.5) {1};
\node at (19.5,16.5) {1};
\node at (21.5,16.5) {1};
\node at (19.5,14.5) {1};
\node at (21.5,14.5) {1};
\node at (21.5,8.5) {1};
\node at (21.5,6.5) {1};
\node at (7.5,13.5) {1};
\node at (9.5,13.5) {1};
\node at (11.5,12.5) {1};
\node at (13.5,12.5) {1};
\node at (7.5,11.5) {1};
\node at (9.5,11.5) {1};
\node at (12.5,10.5) {1};
\node at (10.5,9.5) {1};
\node at (8.5,9.5) {1};
\node at (12.5,8.5) {1};
\node at (10.5,7.5) {1};
\node at (8.5,7.5) {1};
\node at (13.5,6.5) {1};
\node at (15.5,6.5) {1};
\node at (6.5,17.5) {1};
\node at (6.5,15.5) {1};
\node at (6.5,8.5) {1};
\node at (6.5,6.5) {1};
\node at (9.5,5.5) {1};
\node at (11.5,5.5) {1};
\node at (17.5,5.5) {1};
\node at (19.5,5.5) {1};
\node at (7.5,4.5) {1};
\node at (13.5,4.5) {1};
\node at (15.5,4.5) {1};
\node at (9.5,3.5) {1};
\node at (11.5,3.5) {1};
\node at (18.5,3.5) {1};
\node at (20.5,3.5) {1};
\node at (7.5,2.5) {1};
\node at (14.5,2.5) {1};
\node at (16.5,2.5) {1};
\node at (10.5,1.5) {1};
\node at (12.5,1.5) {1};
\node at (18.5,1.5) {1};
\node at (20.5,1.5) {1};
\node at (6.5,0.5) {1};
\node at (8.5,0.5) {1};
\node at (14.5,0.5) {1};
\node at (16.5,0.5) {1};

\node at (5.5,2.5) {1};
\node at (5.5,4.5) {1};
\node at (5.5,10.5) {1};
\node at (5.5,12.5) {1};
\node at (5.5,19.5) {1};
\node at (5.5,21.5) {1};
\node at (4.5,6.5) {1};
\node at (4.5,8.5) {1};
\node at (4.5,14.5) {1};
\node at (4.5,16.5) {1};
\node at (3.5,1.5) {1};
\node at (3.5,3.5) {1};
\node at (3.5,10.5) {1};
\node at (3.5,12.5) {1};
\node at (3.5,18.5) {1};
\node at (3.5,20.5) {1};
\node at (2.5,5.5) {1};
\node at (2.5,7.5) {1};
\node at (2.5,14.5) {1};
\node at (2.5,16.5) {1};
\node at (1.5,1.5) {1};
\node at (1.5,3.5) {1};
\node at (1.5,9.5) {1};
\node at (1.5,11.5) {1};
\node at (1.5,18.5) {1};
\node at (1.5,20.5) {1};
\node at (0.5,5.5) {1};
\node at (0.5,7.5) {1};
\node at (0.5,13.5) {1};
\node at (0.5,15.5) {1};

\node at (14.5,21.5) {\textcolor{darkgreen} {$m$}};
\node at (15.5,14.5) {\textcolor{darkgreen} {$m$}};
\node at (8.5,13.5) {\textcolor{darkgreen} {$m$}};
\node at (7.5,20.5) {\textcolor{darkgreen} {$m$}};

\node at (14.5,14.5) {\textcolor{blue} {$k$}};
\node at (13.5,7.5) {\textcolor{blue} {$k$}};
\node at (20.5,6.5) {\textcolor{blue} {$k$}};
\node at (21.5,13.5) {\textcolor{blue} {$k$}};
\node at (7.5,14.5) {\textcolor{blue} {$k$}};
\node at (6.5,7.5) {\textcolor{blue} {$k$}};
\node at (0.5,14.5) {\textcolor{blue} {$k$}};

\node at (13.5,0.5) {\textcolor{orange} {$n$}};
\node at (14.5,7.5) {\textcolor{orange} {$n$}};
\node at (7.5,8.5) {\textcolor{orange} {$n$}};
\node at (6.5,1.5) {\textcolor{orange} {$n$}};
\node at (7.5,15.5) {\textcolor{orange} {$n$}};
\node at (14.5,15.5) {\textcolor{orange} {$n$}};

\node at (7.5,7.5) {\textcolor{purple} {$p$}};
\node at (8.5,14.5) {\textcolor{purple} {$p$}};
\node at (1.5,15.5) {\textcolor{purple} {$p$}};
\node at (0.5,8.5) {\textcolor{purple} {$p$}};

\end{scope}

\end{tikzpicture}
\caption{$ G $ (black), $G(-6,7)$ (blue), $G(-7,-6)$ (green), $G(-13,1)$ (orange), $G(-6,2)$ (dashed) and $G(-1,1)$ (red).}
\label{fig:20}
\end{figure}
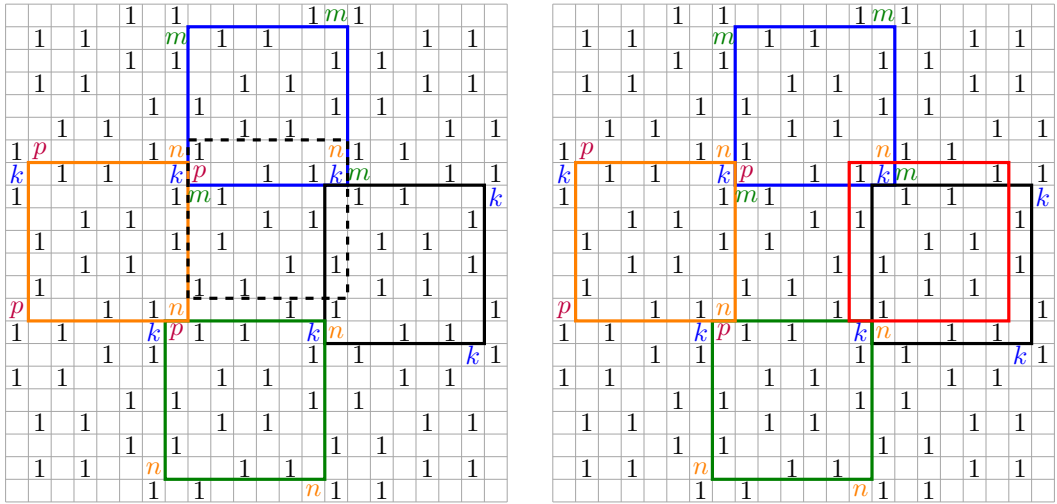

\textbf{Case 2.1.2:} $ m $ appears in a type-A position (see Figure \ref{fig:21}).

Since $G(-6,7)$ is a critical subgraph, $m$ appears in a type-A position and $(7,1)$ has color $k$, by Observation \ref{obs:3}\,(iii), we have that $(7,7)$ has color $k$ (see Figure \ref{fig:21}). Next, we consider $G(1,8)$. Since $c(G(1,8)) = 11$, this subgraph must contain color $k$. Due to distance constraint, either $(7,6)$ or $(7,7)$ must has color $k$ (see Figure \ref{fig:21}).

\begin{figure}[ht]
\centering
\begin{tikzpicture}[scale=0.30]
\begin{scope}[shift={(-34,0)}]
  \foreach \x in {0,...,15} {
    \draw[gray!70, line width=0.35pt] (\x,0) -- (\x,16);
  }
  \foreach \y in {0,...,16} {
    \draw[gray!70, line width=0.35pt] (0,\y) -- (15,\y);
  }

\draw[very thick, blue] (1,8) rectangle (8,15);
\draw[very thick] (7,1) rectangle (14,8);

\node at (3.5,10.5) {\normalsize 1};
\node at (5.5,10.5) {\normalsize 1};
\node at (3.5,12.5) {\normalsize 1};
\node at (5.5,12.5) {\normalsize 1};
\node at (1.5,9.5) {\normalsize 1};
\node at (1.5,11.5) {\normalsize 1};
\node at (4.5,8.5) {\normalsize 1};
\node at (6.5,8.5) {\normalsize 1};
\node at (7.5,11.5) {\normalsize 1};
\node at (7.5,13.5) {\normalsize 1};
\node at (2.5,14.5) {\normalsize 1};
\node at (4.5,14.5) {\normalsize 1};

\node at (9.5,3.5) {\normalsize 1};
\node at (11.5,3.5) {\normalsize 1};
\node at (9.5,5.5) {\normalsize 1};
\node at (11.5,5.5) {\normalsize 1};
\node at (7.5,2.5) {\normalsize 1};
\node at (7.5,4.5) {\normalsize 1};
\node at (10.5,1.5) {\normalsize 1};
\node at (12.5,1.5) {\normalsize 1};
\node at (13.5,4.5) {\normalsize 1};
\node at (13.5,6.5) {\normalsize 1};
\node at (8.5,7.5) {\normalsize 1};
\node at (10.5,7.5) {\normalsize 1};

\node at (9.5,13.5) {1};
\node at (9.5,11.5) {1};
\node at (6.5,15.5) {1};
\node at (8.5,15.5) {1};
\node at (0.5,15.5) {1};
\node at (0.5,13.5) {1};
\node at (8.5,9.5) {1};
\node at (10.5,9.5) {1};
\node at (11.5,12.5) {1};
\node at (11.5,14.5) {1};
\node at (13.5,12.5) {1};
\node at (13.5,14.5) {1};
\node at (12.5,10.5) {1};
\node at (14.5,10.5) {1};
\node at (12.5,8.5) {1};
\node at (14.5,8.5) {1};
\node at (14.5,2.5) {1};
\node at (14.5,0.5) {1};
\node at (0.5,7.5) {1};
\node at (2.5,7.5) {1};
\node at (4.5,6.5) {1};
\node at (6.5,6.5) {1};
\node at (0.5,5.5) {1};
\node at (2.5,5.5) {1};
\node at (5.5,4.5) {1};
\node at (3.5,3.5) {1};
\node at (1.5,3.5) {1};
\node at (5.5,2.5) {1};
\node at (3.5,1.5) {1};
\node at (1.5,1.5) {1};
\node at (6.5,0.5) {1};
\node at (8.5,0.5) {1};

\node at (1.5,15.5) {\textcolor{darkgreen} {$m$}};
\node at (8.5,14.5) {\textcolor{darkgreen} {$m$}};
\node at (7.5,7.5) {\textcolor{darkgreen} {$m$}};
\node at (0.5,8.5) {\textcolor{darkgreen} {$m$}};

\node at (7.5,8.5) {\textcolor{blue} {$k$}};
\node at (6.5,1.5) {\textcolor{blue} {$k$}};
\node at (13.5,0.5) {\textcolor{blue} {$k$}};
\node at (14.5,7.5) {\textcolor{blue} {$k$}};

\end{scope}

\begin{scope}[shift={(-17,0)}]
  \foreach \x in {0,...,15} {
    \draw[gray!70, line width=0.35pt] (\x,0) -- (\x,16);
  }
  \foreach \y in {0,...,16} {
    \draw[gray!70, line width=0.35pt] (0,\y) -- (15,\y);
  }

\draw[very thick, blue] (1,8) rectangle (8,15);
\draw[very thick] (7,1) rectangle (14,8);

\node at (3.5,10.5) {\normalsize 1};
\node at (5.5,10.5) {\normalsize 1};
\node at (3.5,12.5) {\normalsize 1};
\node at (5.5,12.5) {\normalsize 1};
\node at (1.5,9.5) {\normalsize 1};
\node at (1.5,11.5) {\normalsize 1};
\node at (4.5,8.5) {\normalsize 1};
\node at (6.5,8.5) {\normalsize 1};
\node at (7.5,11.5) {\normalsize 1};
\node at (7.5,13.5) {\normalsize 1};
\node at (2.5,14.5) {\normalsize 1};
\node at (4.5,14.5) {\normalsize 1};

\node at (9.5,3.5) {\normalsize 1};
\node at (11.5,3.5) {\normalsize 1};
\node at (9.5,5.5) {\normalsize 1};
\node at (11.5,5.5) {\normalsize 1};
\node at (7.5,2.5) {\normalsize 1};
\node at (7.5,4.5) {\normalsize 1};
\node at (10.5,1.5) {\normalsize 1};
\node at (12.5,1.5) {\normalsize 1};
\node at (13.5,4.5) {\normalsize 1};
\node at (13.5,6.5) {\normalsize 1};
\node at (8.5,7.5) {\normalsize 1};
\node at (10.5,7.5) {\normalsize 1};

\node at (9.5,13.5) {1};
\node at (9.5,11.5) {1};
\node at (6.5,15.5) {1};
\node at (8.5,15.5) {1};
\node at (0.5,15.5) {1};
\node at (0.5,13.5) {1};
\node at (8.5,9.5) {1};
\node at (10.5,9.5) {1};
\node at (11.5,12.5) {1};
\node at (11.5,14.5) {1};
\node at (13.5,12.5) {1};
\node at (13.5,14.5) {1};
\node at (12.5,10.5) {1};
\node at (14.5,10.5) {1};
\node at (12.5,8.5) {1};
\node at (14.5,8.5) {1};
\node at (14.5,2.5) {1};
\node at (14.5,0.5) {1};
\node at (0.5,7.5) {1};
\node at (2.5,7.5) {1};
\node at (4.5,6.5) {1};
\node at (6.5,6.5) {1};
\node at (0.5,5.5) {1};
\node at (2.5,5.5) {1};
\node at (5.5,4.5) {1};
\node at (3.5,3.5) {1};
\node at (1.5,3.5) {1};
\node at (5.5,2.5) {1};
\node at (3.5,1.5) {1};
\node at (1.5,1.5) {1};
\node at (6.5,0.5) {1};
\node at (8.5,0.5) {1};

\node at (1.5,15.5) {\textcolor{darkgreen} {$m$}};
\node at (8.5,14.5) {\textcolor{darkgreen} {$m$}};
\node at (7.5,7.5) {\textcolor{darkgreen} {$m$}};
\node at (0.5,8.5) {\textcolor{darkgreen} {$m$}};

\node at (7.5,8.5) {\textcolor{blue} {$k$}};
\node at (6.5,1.5) {\textcolor{blue} {$k$}};
\node at (13.5,0.5) {\textcolor{blue} {$k$}};
\node at (14.5,7.5) {\textcolor{blue} {$k$}};
\node at (7.5,15.5) {\textcolor{blue} {$k$}};

\end{scope}

\begin{scope}[shift={(0,0)}]
  \foreach \x in {0,...,15} {
    \draw[gray!70, line width=0.35pt] (\x,0) -- (\x,16);
  }
  \foreach \y in {0,...,16} {
    \draw[gray!70, line width=0.35pt] (0,\y) -- (15,\y);
  }

\draw[very thick, dashed] (8,9) rectangle (15,16);
\draw[very thick, blue] (1,8) rectangle (8,15);
\draw[very thick] (7,1) rectangle (14,8);

\node at (3.5,10.5) {\normalsize 1};
\node at (5.5,10.5) {\normalsize 1};
\node at (3.5,12.5) {\normalsize 1};
\node at (5.5,12.5) {\normalsize 1};
\node at (1.5,9.5) {\normalsize 1};
\node at (1.5,11.5) {\normalsize 1};
\node at (4.5,8.5) {\normalsize 1};
\node at (6.5,8.5) {\normalsize 1};
\node at (7.5,11.5) {\normalsize 1};
\node at (7.5,13.5) {\normalsize 1};
\node at (2.5,14.5) {\normalsize 1};
\node at (4.5,14.5) {\normalsize 1};

\node at (9.5,3.5) {\normalsize 1};
\node at (11.5,3.5) {\normalsize 1};
\node at (9.5,5.5) {\normalsize 1};
\node at (11.5,5.5) {\normalsize 1};
\node at (7.5,2.5) {\normalsize 1};
\node at (7.5,4.5) {\normalsize 1};
\node at (10.5,1.5) {\normalsize 1};
\node at (12.5,1.5) {\normalsize 1};
\node at (13.5,4.5) {\normalsize 1};
\node at (13.5,6.5) {\normalsize 1};
\node at (8.5,7.5) {\normalsize 1};
\node at (10.5,7.5) {\normalsize 1};

\node at (9.5,13.5) {1};
\node at (9.5,11.5) {1};
\node at (6.5,15.5) {1};
\node at (8.5,15.5) {1};
\node at (0.5,15.5) {1};
\node at (0.5,13.5) {1};
\node at (8.5,9.5) {1};
\node at (10.5,9.5) {1};
\node at (11.5,12.5) {1};
\node at (11.5,14.5) {1};
\node at (13.5,12.5) {1};
\node at (13.5,14.5) {1};
\node at (12.5,10.5) {1};
\node at (14.5,10.5) {1};
\node at (12.5,8.5) {1};
\node at (14.5,8.5) {1};
\node at (14.5,2.5) {1};
\node at (14.5,0.5) {1};
\node at (0.5,7.5) {1};
\node at (2.5,7.5) {1};
\node at (4.5,6.5) {1};
\node at (6.5,6.5) {1};
\node at (0.5,5.5) {1};
\node at (2.5,5.5) {1};
\node at (5.5,4.5) {1};
\node at (3.5,3.5) {1};
\node at (1.5,3.5) {1};
\node at (5.5,2.5) {1};
\node at (3.5,1.5) {1};
\node at (1.5,1.5) {1};
\node at (6.5,0.5) {1};
\node at (8.5,0.5) {1};

\node at (1.5,15.5) {\textcolor{darkgreen} {$m$}};
\node at (8.5,14.5) {\textcolor{darkgreen} {$m$}};
\node at (7.5,7.5) {\textcolor{darkgreen} {$m$}};
\node at (0.5,8.5) {\textcolor{darkgreen} {$m$}};

\node at (7.5,8.5) {\textcolor{blue} {$k$}};
\node at (6.5,1.5) {\textcolor{blue} {$k$}};
\node at (13.5,0.5) {\textcolor{blue} {$k$}};
\node at (14.5,7.5) {\textcolor{blue} {$k$}};
\node at (7.5,15.5) {\textcolor{blue} {$k$}};
\fill[yellow, opacity=0.3] (14,14) rectangle (15,16);
\end{scope}

\end{tikzpicture}
\caption{$ G $ (black), $G(-6,7)$ (blue) and $G(1,8)$ (dashed).}
\label{fig:21}
\end{figure}
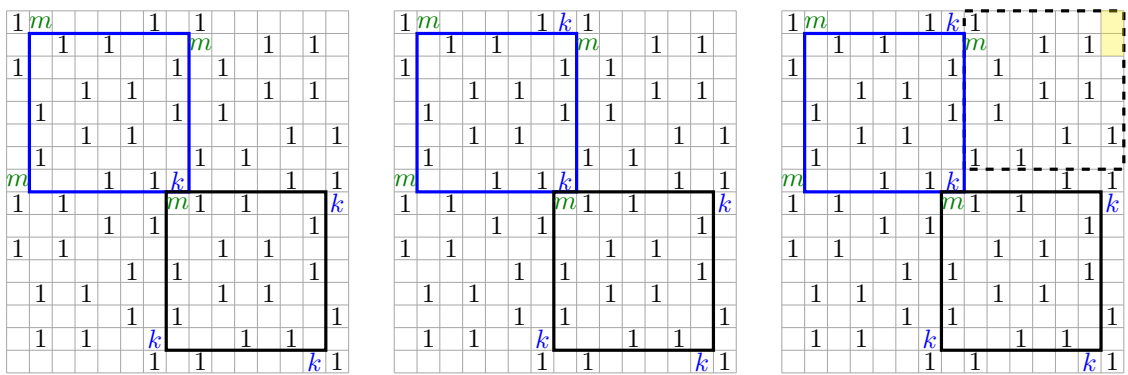

Assume that $(7,7)$ has color $k$. Consider $G(6,7)$ (see Figure \ref{fig:22}). Since $c(G(6,7)) = 11$, this subgraph must contains color $k$. Due to distance constraint, this subgraph cannot have a square of color $k$ which is contradiction. Hence, $(7,6)$ must have color $k$.

Finally, we consider $G(7,6)$ (see Figure~\ref{fig:22}). Observe that this subgraph is a critical subgraph. By Observation~\ref{obs:3}\,(ii), let $ n \in \{2,3,\ldots,39\} $ be the color that does not appear in $ G(7,6) $. Then $ n \neq k $, and it must appear in a type-B position (see Figure~\ref{fig:22}). If we rotate this configuration by $90^\circ$ counterclockwise, it falls into Case~2.1.1. Therefore, this case also leads to a contradiction.

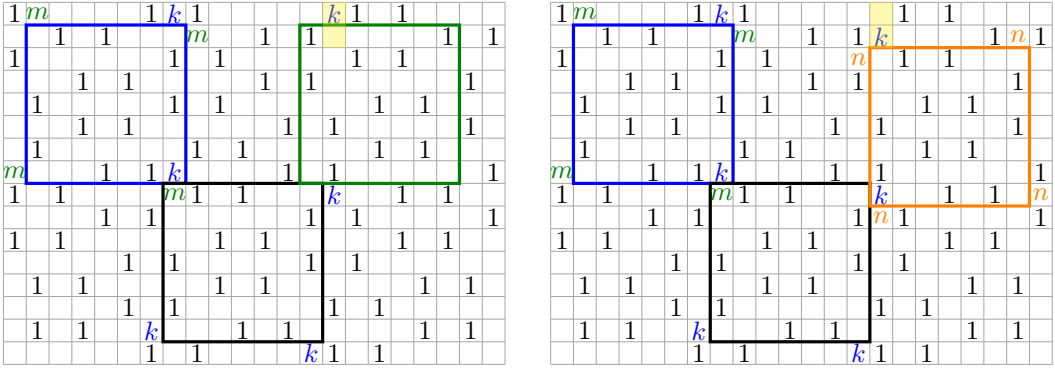
\begin{figure}[ht]
\centering
\begin{tikzpicture}[scale=0.30]
\begin{scope}[shift={(-24,0)}]
  \foreach \x in {0,...,22} {
    \draw[gray!70, line width=0.35pt] (\x,0) -- (\x,16);
  }
  \foreach \y in {0,...,16} {
    \draw[gray!70, line width=0.35pt] (0,\y) -- (22,\y);
  }

\draw[very thick, blue] (1,8) rectangle (8,15);
\draw[very thick] (7,1) rectangle (14,8);
\draw[very thick, darkgreen] (13,8) rectangle (20,15);

\node at (3.5,10.5) {\normalsize 1};
\node at (5.5,10.5) {\normalsize 1};
\node at (3.5,12.5) {\normalsize 1};
\node at (5.5,12.5) {\normalsize 1};
\node at (1.5,9.5) {\normalsize 1};
\node at (1.5,11.5) {\normalsize 1};
\node at (4.5,8.5) {\normalsize 1};
\node at (6.5,8.5) {\normalsize 1};
\node at (7.5,11.5) {\normalsize 1};
\node at (7.5,13.5) {\normalsize 1};
\node at (2.5,14.5) {\normalsize 1};
\node at (4.5,14.5) {\normalsize 1};

\node at (9.5,3.5) {\normalsize 1};
\node at (11.5,3.5) {\normalsize 1};
\node at (9.5,5.5) {\normalsize 1};
\node at (11.5,5.5) {\normalsize 1};
\node at (7.5,2.5) {\normalsize 1};
\node at (7.5,4.5) {\normalsize 1};
\node at (10.5,1.5) {\normalsize 1};
\node at (12.5,1.5) {\normalsize 1};
\node at (13.5,4.5) {\normalsize 1};
\node at (13.5,6.5) {\normalsize 1};
\node at (8.5,7.5) {\normalsize 1};
\node at (10.5,7.5) {\normalsize 1};

\node at (9.5,13.5) {1};
\node at (9.5,11.5) {1};
\node at (6.5,15.5) {1};
\node at (8.5,15.5) {1};
\node at (0.5,15.5) {1};
\node at (0.5,13.5) {1};
\node at (8.5,9.5) {1};
\node at (10.5,9.5) {1};
\node at (11.5,12.5) {1};
\node at (11.5,14.5) {1};
\node at (13.5,12.5) {1};
\node at (13.5,14.5) {1};
\node at (12.5,10.5) {1};
\node at (14.5,10.5) {1};
\node at (12.5,8.5) {1};
\node at (14.5,8.5) {1};
\node at (14.5,2.5) {1};
\node at (14.5,0.5) {1};
\node at (0.5,7.5) {1};
\node at (2.5,7.5) {1};
\node at (4.5,6.5) {1};
\node at (6.5,6.5) {1};
\node at (0.5,5.5) {1};
\node at (2.5,5.5) {1};
\node at (5.5,4.5) {1};
\node at (3.5,3.5) {1};
\node at (1.5,3.5) {1};
\node at (5.5,2.5) {1};
\node at (3.5,1.5) {1};
\node at (1.5,1.5) {1};
\node at (6.5,0.5) {1};
\node at (8.5,0.5) {1};

\node at (15.5,4.5) {1};
\node at (15.5,6.5) {1};
\node at (15.5,13.5) {1};
\node at (15.5,15.5) {1};
\node at (16.5,0.5) {1};
\node at (16.5,2.5) {1};
\node at (16.5,9.5) {1};
\node at (16.5,11.5) {1};
\node at (17.5,5.5) {1};
\node at (17.5,7.5) {1};
\node at (17.5,13.5) {1};
\node at (17.5,15.5) {1};
\node at (18.5,1.5) {1};
\node at (18.5,3.5) {1};
\node at (18.5,9.5) {1};
\node at (18.5,11.5) {1};
\node at (19.5,5.5) {1};
\node at (19.5,7.5) {1};
\node at (19.5,14.5) {1};
\node at (20.5,1.5) {1};
\node at (20.5,3.5) {1};
\node at (20.5,10.5) {1};
\node at (20.5,12.5) {1};
\node at (21.5,6.5) {1};
\node at (21.5,8.5) {1};
\node at (21.5,14.5) {1};

\node at (1.5,15.5) {\textcolor{darkgreen} {$m$}};
\node at (8.5,14.5) {\textcolor{darkgreen} {$m$}};
\node at (7.5,7.5) {\textcolor{darkgreen} {$m$}};
\node at (0.5,8.5) {\textcolor{darkgreen} {$m$}};

\node at (7.5,8.5) {\textcolor{blue} {$k$}};
\node at (6.5,1.5) {\textcolor{blue} {$k$}};
\node at (13.5,0.5) {\textcolor{blue} {$k$}};
\node at (14.5,7.5) {\textcolor{blue} {$k$}};
\node at (7.5,15.5) {\textcolor{blue} {$k$}};
\node at (14.5,15.5) {\textcolor{blue} {$k$}};

\fill[yellow, opacity=0.3] (14,14) rectangle (15,16);
\end{scope}

\begin{scope}[shift={(0,0)}]
  \foreach \x in {0,...,22} {
    \draw[gray!70, line width=0.35pt] (\x,0) -- (\x,16);
  }
  \foreach \y in {0,...,16} {
    \draw[gray!70, line width=0.35pt] (0,\y) -- (22,\y);
  }

\draw[very thick, blue] (1,8) rectangle (8,15);
\draw[very thick] (7,1) rectangle (14,8);
\draw[very thick, orange] (14,7) rectangle (21,14);

\node at (3.5,10.5) {\normalsize 1};
\node at (5.5,10.5) {\normalsize 1};
\node at (3.5,12.5) {\normalsize 1};
\node at (5.5,12.5) {\normalsize 1};
\node at (1.5,9.5) {\normalsize 1};
\node at (1.5,11.5) {\normalsize 1};
\node at (4.5,8.5) {\normalsize 1};
\node at (6.5,8.5) {\normalsize 1};
\node at (7.5,11.5) {\normalsize 1};
\node at (7.5,13.5) {\normalsize 1};
\node at (2.5,14.5) {\normalsize 1};
\node at (4.5,14.5) {\normalsize 1};

\node at (9.5,3.5) {\normalsize 1};
\node at (11.5,3.5) {\normalsize 1};
\node at (9.5,5.5) {\normalsize 1};
\node at (11.5,5.5) {\normalsize 1};
\node at (7.5,2.5) {\normalsize 1};
\node at (7.5,4.5) {\normalsize 1};
\node at (10.5,1.5) {\normalsize 1};
\node at (12.5,1.5) {\normalsize 1};
\node at (13.5,4.5) {\normalsize 1};
\node at (13.5,6.5) {\normalsize 1};
\node at (8.5,7.5) {\normalsize 1};
\node at (10.5,7.5) {\normalsize 1};

\node at (9.5,13.5) {1};
\node at (9.5,11.5) {1};
\node at (6.5,15.5) {1};
\node at (8.5,15.5) {1};
\node at (0.5,15.5) {1};
\node at (0.5,13.5) {1};
\node at (8.5,9.5) {1};
\node at (10.5,9.5) {1};
\node at (11.5,12.5) {1};
\node at (11.5,14.5) {1};
\node at (13.5,12.5) {1};
\node at (13.5,14.5) {1};
\node at (12.5,10.5) {1};
\node at (14.5,10.5) {1};
\node at (12.5,8.5) {1};
\node at (14.5,8.5) {1};
\node at (14.5,2.5) {1};
\node at (14.5,0.5) {1};
\node at (0.5,7.5) {1};
\node at (2.5,7.5) {1};
\node at (4.5,6.5) {1};
\node at (6.5,6.5) {1};
\node at (0.5,5.5) {1};
\node at (2.5,5.5) {1};
\node at (5.5,4.5) {1};
\node at (3.5,3.5) {1};
\node at (1.5,3.5) {1};
\node at (5.5,2.5) {1};
\node at (3.5,1.5) {1};
\node at (1.5,1.5) {1};
\node at (6.5,0.5) {1};
\node at (8.5,0.5) {1};

\node at (15.5,4.5) {1};
\node at (15.5,6.5) {1};
\node at (15.5,13.5) {1};
\node at (15.5,15.5) {1};
\node at (16.5,0.5) {1};
\node at (16.5,2.5) {1};
\node at (16.5,9.5) {1};
\node at (16.5,11.5) {1};
\node at (17.5,5.5) {1};
\node at (17.5,7.5) {1};
\node at (17.5,13.5) {1};
\node at (17.5,15.5) {1};
\node at (18.5,1.5) {1};
\node at (18.5,3.5) {1};
\node at (18.5,9.5) {1};
\node at (18.5,11.5) {1};
\node at (19.5,5.5) {1};
\node at (19.5,7.5) {1};
\node at (19.5,14.5) {1};
\node at (20.5,1.5) {1};
\node at (20.5,3.5) {1};
\node at (20.5,10.5) {1};
\node at (20.5,12.5) {1};
\node at (21.5,6.5) {1};
\node at (21.5,8.5) {1};
\node at (21.5,14.5) {1};

\node at (1.5,15.5) {\textcolor{darkgreen} {$m$}};
\node at (8.5,14.5) {\textcolor{darkgreen} {$m$}};
\node at (7.5,7.5) {\textcolor{darkgreen} {$m$}};
\node at (0.5,8.5) {\textcolor{darkgreen} {$m$}};

\node at (7.5,8.5) {\textcolor{blue} {$k$}};
\node at (6.5,1.5) {\textcolor{blue} {$k$}};
\node at (13.5,0.5) {\textcolor{blue} {$k$}};
\node at (14.5,7.5) {\textcolor{blue} {$k$}};
\node at (7.5,15.5) {\textcolor{blue} {$k$}};
\node at (14.5,14.5) {\textcolor{blue} {$k$}};

\node at (20.5,14.5) {\textcolor{orange} {$n$}};
\node at (21.5,7.5) {\textcolor{orange} {$n$}};
\node at (14.5,6.5) {\textcolor{orange} {$n$}};
\node at (13.5,13.5) {\textcolor{orange} {$n$}};
\fill[yellow, opacity=0.3] (14,14) rectangle (15,16);
\end{scope}

\end{tikzpicture}
\caption{$ G $ (black), $G(-6,7)$ (blue),  $G(6,7)$ (green) and $G(7,6)$ (orange).}
\label{fig:22}
\end{figure}

\textbf{Case 2.2:} $ k $ appears in a type-B position.

In fact, we can assume that the color that does not appear in every critical subgraph $P_7 \boxtimes P_7$ with pattern~(i) must appear in a type-B position, since the type-A case is already covered in Case 2.1. Let $m$ be the color assigned to $(7,7)$ of $G$ (see Figure \ref{fig:23}). We will show that $(1,5)$ of $G(-1,0)$ must also be assigned color $m$. Let $n$ be the color that assigned to $(7,1)$. By Observation \ref{obs:3}\,(iii), we have that $(1,1)$ of $G(-1,0)$ must also be assigned color $n$.

Consider $G(-1,4)$ (see Figure \ref{fig:23}). By Observation~\ref{obs:3}\,(ii), since this subgraph is a critical subgraph, let $p \in \{2,3,\ldots,39\}$ be a color that does not appear in this subgraph, and it must appear in a type-B position. Then, we consider $G(-7,-6)$. By Observation~\ref{obs:3}\,(ii), since this subgraph is a critical subgraph, let $q \in \{2,3,\ldots,39\}$ be a color that does not appear in this subgraph, and it must appear in a type-B position. Next, we consider $G(0,-5)$. By Observation~\ref{obs:2}\,(iv), since $c(G(0,-5)) = 11$, this subgraph must contain color~$q$. By the distance constraint, the only possible square that can be assigned color~$q$ is $(7,7)$, and thus this square must be assigned color~$q$.

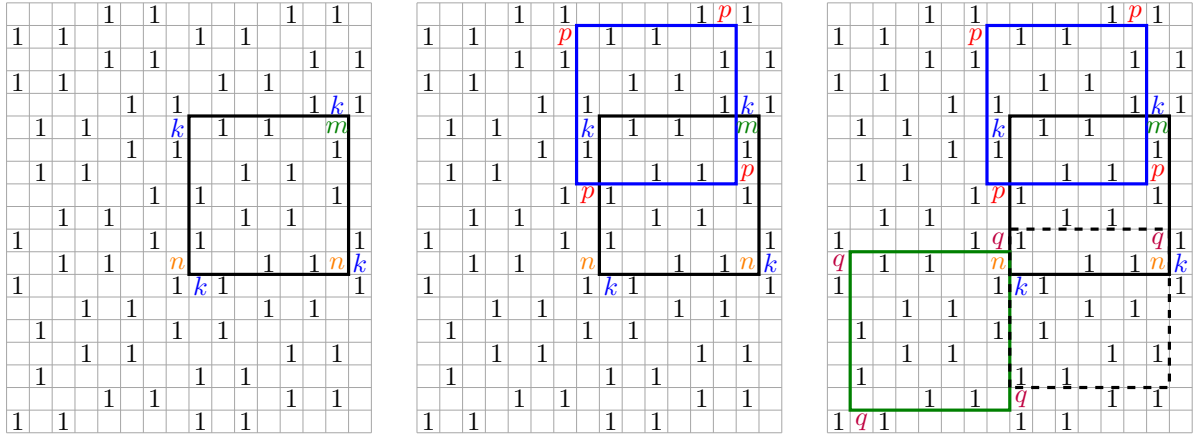
\begin{figure}[ht]
\centering
\begin{tikzpicture}[scale=0.30]

\begin{scope}[shift={(-36,0)}]
    \foreach \x in {0,...,16} {
    \draw[gray!70, line width=0.35pt] (\x,0) -- (\x,19);
  }
  \foreach \y in {0,...,19} {
    \draw[gray!70, line width=0.35pt] (0,\y) -- (16,\y);
  }
\draw[very thick] (8,7) rectangle (15,14);

\node at (4.5,16.5) {\normalsize 1};
\node at (6.5,16.5) {\normalsize 1};
\node at (4.5,18.5) {\normalsize 1};
\node at (6.5,18.5) {\normalsize 1};
\node at (2.5,15.5) {\normalsize 1};
\node at (2.5,17.5) {\normalsize 1};
\node at (5.5,14.5) {\normalsize 1};
\node at (7.5,14.5) {\normalsize 1};
\node at (8.5,17.5) {\normalsize 1};

\node at (10.5,9.5) {\normalsize 1};
\node at (12.5,9.5) {\normalsize 1};
\node at (10.5,11.5) {\normalsize 1};
\node at (12.5,11.5) {\normalsize 1};
\node at (8.5,8.5) {\normalsize 1};
\node at (8.5,10.5) {\normalsize 1};
\node at (11.5,7.5) {\normalsize 1};
\node at (13.5,7.5) {\normalsize 1};
\node at (14.5,10.5) {\normalsize 1};
\node at (14.5,12.5) {\normalsize 1};
\node at (9.5,13.5) {\normalsize 1};
\node at (11.5,13.5) {\normalsize 1};

\node at (10.5,17.5) {1};
\node at (9.5,15.5) {1};
\node at (11.5,15.5) {1};
\node at (12.5,18.5) {1};
\node at (14.5,18.5) {1};
\node at (13.5,16.5) {1};
\node at (15.5,16.5) {1};
\node at (13.5,14.5) {1};
\node at (15.5,14.5) {1};
\node at (15.5,8.5) {1};
\node at (15.5,6.5) {1};
\node at (1.5,13.5) {1};
\node at (3.5,13.5) {1};
\node at (5.5,12.5) {1};
\node at (7.5,12.5) {1};
\node at (1.5,11.5) {1};
\node at (3.5,11.5) {1};
\node at (6.5,10.5) {1};
\node at (4.5,9.5) {1};
\node at (2.5,9.5) {1};
\node at (6.5,8.5) {1};
\node at (4.5,7.5) {1};
\node at (2.5,7.5) {1};
\node at (7.5,6.5) {1};
\node at (9.5,6.5) {1};
\node at (0.5,17.5) {1};
\node at (0.5,15.5) {1};
\node at (0.5,8.5) {1};
\node at (0.5,6.5) {1};
\node at (3.5,5.5) {1};
\node at (5.5,5.5) {1};
\node at (11.5,5.5) {1};
\node at (13.5,5.5) {1};
\node at (1.5,4.5) {1};
\node at (7.5,4.5) {1};
\node at (9.5,4.5) {1};
\node at (3.5,3.5) {1};
\node at (5.5,3.5) {1};
\node at (12.5,3.5) {1};
\node at (14.5,3.5) {1};
\node at (1.5,2.5) {1};
\node at (8.5,2.5) {1};
\node at (10.5,2.5) {1};
\node at (4.5,1.5) {1};
\node at (6.5,1.5) {1};
\node at (12.5,1.5) {1};
\node at (14.5,1.5) {1};
\node at (0.5,0.5) {1};
\node at (2.5,0.5) {1};
\node at (8.5,0.5) {1};
\node at (10.5,0.5) {1};

\node at (8.5,6.5) {\textcolor{blue} {$k$}};
\node at (7.5,13.5) {\textcolor{blue} {$k$}};
\node at (14.5,14.5) {\textcolor{blue} {$k$}};
\node at (15.5,7.5) {\textcolor{blue} {$k$}};

\node at (14.5,13.5) {\textcolor{darkgreen} {$m$}};
\node at (14.5,7.5) {\textcolor{orange} {$n$}};
\node at (7.5,7.5) {\textcolor{orange} {$n$}};
\end{scope}

\begin{scope}[shift={(-18,0)}]
    \foreach \x in {0,...,16} {
    \draw[gray!70, line width=0.35pt] (\x,0) -- (\x,19);
  }
  \foreach \y in {0,...,19} {
    \draw[gray!70, line width=0.35pt] (0,\y) -- (16,\y);
  }
\draw[very thick] (8,7) rectangle (15,14);
\draw[very thick, blue] (7,11) rectangle (14,18);

\node at (4.5,16.5) {\normalsize 1};
\node at (6.5,16.5) {\normalsize 1};
\node at (4.5,18.5) {\normalsize 1};
\node at (6.5,18.5) {\normalsize 1};
\node at (2.5,15.5) {\normalsize 1};
\node at (2.5,17.5) {\normalsize 1};
\node at (5.5,14.5) {\normalsize 1};
\node at (7.5,14.5) {\normalsize 1};
\node at (8.5,17.5) {\normalsize 1};

\node at (10.5,9.5) {\normalsize 1};
\node at (12.5,9.5) {\normalsize 1};
\node at (10.5,11.5) {\normalsize 1};
\node at (12.5,11.5) {\normalsize 1};
\node at (8.5,8.5) {\normalsize 1};
\node at (8.5,10.5) {\normalsize 1};
\node at (11.5,7.5) {\normalsize 1};
\node at (13.5,7.5) {\normalsize 1};
\node at (14.5,10.5) {\normalsize 1};
\node at (14.5,12.5) {\normalsize 1};
\node at (9.5,13.5) {\normalsize 1};
\node at (11.5,13.5) {\normalsize 1};

\node at (10.5,17.5) {1};
\node at (9.5,15.5) {1};
\node at (11.5,15.5) {1};
\node at (12.5,18.5) {1};
\node at (14.5,18.5) {1};
\node at (13.5,16.5) {1};
\node at (15.5,16.5) {1};
\node at (13.5,14.5) {1};
\node at (15.5,14.5) {1};
\node at (15.5,8.5) {1};
\node at (15.5,6.5) {1};
\node at (1.5,13.5) {1};
\node at (3.5,13.5) {1};
\node at (5.5,12.5) {1};
\node at (7.5,12.5) {1};
\node at (1.5,11.5) {1};
\node at (3.5,11.5) {1};
\node at (6.5,10.5) {1};
\node at (4.5,9.5) {1};
\node at (2.5,9.5) {1};
\node at (6.5,8.5) {1};
\node at (4.5,7.5) {1};
\node at (2.5,7.5) {1};
\node at (7.5,6.5) {1};
\node at (9.5,6.5) {1};
\node at (0.5,17.5) {1};
\node at (0.5,15.5) {1};
\node at (0.5,8.5) {1};
\node at (0.5,6.5) {1};
\node at (3.5,5.5) {1};
\node at (5.5,5.5) {1};
\node at (11.5,5.5) {1};
\node at (13.5,5.5) {1};
\node at (1.5,4.5) {1};
\node at (7.5,4.5) {1};
\node at (9.5,4.5) {1};
\node at (3.5,3.5) {1};
\node at (5.5,3.5) {1};
\node at (12.5,3.5) {1};
\node at (14.5,3.5) {1};
\node at (1.5,2.5) {1};
\node at (8.5,2.5) {1};
\node at (10.5,2.5) {1};
\node at (4.5,1.5) {1};
\node at (6.5,1.5) {1};
\node at (12.5,1.5) {1};
\node at (14.5,1.5) {1};
\node at (0.5,0.5) {1};
\node at (2.5,0.5) {1};
\node at (8.5,0.5) {1};
\node at (10.5,0.5) {1};

\node at (8.5,6.5) {\textcolor{blue} {$k$}};
\node at (7.5,13.5) {\textcolor{blue} {$k$}};
\node at (14.5,14.5) {\textcolor{blue} {$k$}};
\node at (15.5,7.5) {\textcolor{blue} {$k$}};

\node at (14.5,13.5) {\textcolor{darkgreen} {$m$}};
\node at (14.5,7.5) {\textcolor{orange} {$n$}};
\node at (7.5,7.5) {\textcolor{orange} {$n$}};

\node at (7.5,10.5) {\textcolor{red} {$p$}};
\node at (14.5,11.5) {\textcolor{red} {$p$}};
\node at (13.5,18.5) {\textcolor{red} {$p$}};
\node at (6.5,17.5) {\textcolor{red} {$p$}};

\end{scope}

\begin{scope}[shift={(0,0)}]
    \foreach \x in {0,...,16} {
    \draw[gray!70, line width=0.35pt] (\x,0) -- (\x,19);
  }
  \foreach \y in {0,...,19} {
    \draw[gray!70, line width=0.35pt] (0,\y) -- (16,\y);
  }
\draw[very thick] (8,7) rectangle (15,14);
\draw[very thick, darkgreen] (1,1) rectangle (8,8);
\draw[very thick, blue] (7,11) rectangle (14,18);
\draw[very thick, dashed] (8,2) rectangle (15,9);

\node at (4.5,16.5) {\normalsize 1};
\node at (6.5,16.5) {\normalsize 1};
\node at (4.5,18.5) {\normalsize 1};
\node at (6.5,18.5) {\normalsize 1};
\node at (2.5,15.5) {\normalsize 1};
\node at (2.5,17.5) {\normalsize 1};
\node at (5.5,14.5) {\normalsize 1};
\node at (7.5,14.5) {\normalsize 1};
\node at (8.5,17.5) {\normalsize 1};

\node at (10.5,9.5) {\normalsize 1};
\node at (12.5,9.5) {\normalsize 1};
\node at (10.5,11.5) {\normalsize 1};
\node at (12.5,11.5) {\normalsize 1};
\node at (8.5,8.5) {\normalsize 1};
\node at (8.5,10.5) {\normalsize 1};
\node at (11.5,7.5) {\normalsize 1};
\node at (13.5,7.5) {\normalsize 1};
\node at (14.5,10.5) {\normalsize 1};
\node at (14.5,12.5) {\normalsize 1};
\node at (9.5,13.5) {\normalsize 1};
\node at (11.5,13.5) {\normalsize 1};

\node at (10.5,17.5) {1};
\node at (9.5,15.5) {1};
\node at (11.5,15.5) {1};
\node at (12.5,18.5) {1};
\node at (14.5,18.5) {1};
\node at (13.5,16.5) {1};
\node at (15.5,16.5) {1};
\node at (13.5,14.5) {1};
\node at (15.5,14.5) {1};
\node at (15.5,8.5) {1};
\node at (15.5,6.5) {1};
\node at (1.5,13.5) {1};
\node at (3.5,13.5) {1};
\node at (5.5,12.5) {1};
\node at (7.5,12.5) {1};
\node at (1.5,11.5) {1};
\node at (3.5,11.5) {1};
\node at (6.5,10.5) {1};
\node at (4.5,9.5) {1};
\node at (2.5,9.5) {1};
\node at (6.5,8.5) {1};
\node at (4.5,7.5) {1};
\node at (2.5,7.5) {1};
\node at (7.5,6.5) {1};
\node at (9.5,6.5) {1};
\node at (0.5,17.5) {1};
\node at (0.5,15.5) {1};
\node at (0.5,8.5) {1};
\node at (0.5,6.5) {1};
\node at (3.5,5.5) {1};
\node at (5.5,5.5) {1};
\node at (11.5,5.5) {1};
\node at (13.5,5.5) {1};
\node at (1.5,4.5) {1};
\node at (7.5,4.5) {1};
\node at (9.5,4.5) {1};
\node at (3.5,3.5) {1};
\node at (5.5,3.5) {1};
\node at (12.5,3.5) {1};
\node at (14.5,3.5) {1};
\node at (1.5,2.5) {1};
\node at (8.5,2.5) {1};
\node at (10.5,2.5) {1};
\node at (4.5,1.5) {1};
\node at (6.5,1.5) {1};
\node at (12.5,1.5) {1};
\node at (14.5,1.5) {1};
\node at (0.5,0.5) {1};
\node at (2.5,0.5) {1};
\node at (8.5,0.5) {1};
\node at (10.5,0.5) {1};

\node at (8.5,6.5) {\textcolor{blue} {$k$}};
\node at (7.5,13.5) {\textcolor{blue} {$k$}};
\node at (14.5,14.5) {\textcolor{blue} {$k$}};
\node at (15.5,7.5) {\textcolor{blue} {$k$}};

\node at (14.5,13.5) {\textcolor{darkgreen} {$m$}};
\node at (14.5,7.5) {\textcolor{orange} {$n$}};
\node at (7.5,7.5) {\textcolor{orange} {$n$}};

\node at (7.5,10.5) {\textcolor{red} {$p$}};
\node at (14.5,11.5) {\textcolor{red} {$p$}};
\node at (13.5,18.5) {\textcolor{red} {$p$}};
\node at (6.5,17.5) {\textcolor{red} {$p$}};

\node at (7.5,8.5) {\textcolor{purple} {$q$}};
\node at (0.5,7.5) {\textcolor{purple} {$q$}};
\node at (1.5,0.5) {\textcolor{purple} {$q$}};
\node at (8.5,1.5) {\textcolor{purple} {$q$}};
\node at (14.5,8.5) {\textcolor{purple} {$q$}};
\end{scope}

\end{tikzpicture}
\caption{$ G $ (black), $G(-1,4)$ (blue), $G(-7,-6)$ (green) and $G(0,-5)$ (dashed).}
\label{fig:23}
\end{figure}

Following this, we consider $G(6,10)$ (see Figure \ref{fig:24}). By Observation~\ref{obs:3}\,(ii), since this subgraph is a critical subgraph, let $r \in \{2,3,\ldots,39\}$ be a color that does not appear in this subgraph, and it must appear in a type-B position. Then, we consider $G(-1,9)$. By Observation \ref{obs:2}\,(iv), since $c(G(-1,9)) = 11$, this subgraph must contain color $r$. By the distance constraint, the square that assigned color $r$ must be $(1,1)$. Similarly, by considering $G(7,3)$, we have that $(7,1)$ of this subgraph must be assigned color $r$. We then consider $G(5,2)$ (see Figure \ref{fig:24}). By Observation \ref{obs:2}\,(iv), since $c(G(5,2)) = 11$, this subgraph must contain color $r$. By the distance constraint, the square that assigned color $r$ must be either $(1,1)$ or $(1,2)$.

Assume that the square assigned color $r$ is $(1,1)$. Then, we consider $G(6,-2)$. By Observation~\ref{obs:2}\,(iv), since $c(G(6,-2)) = 11$, this subgraph must contain color $r$. However, by the distance constraint, no square in this subgraph can be assigned color $r$, leading to a contradiction. Hence, the square assigned color $r$ is $(1,2)$. After that, we consider $G(-1,1)$ (see Figure \ref{fig:24}). By Observation~\ref{obs:2}\,(iv), since $c(G(-1,1)) = 11$, this subgraph must contain color $r$. The only possible squares that can be assigned color $r$ are $(1,1)$ and $(1,2)$. However, square $(1,1)$ has already been assigned color $q$, and clearly $q \neq r$. Therefore, square $(1,2)$ must be assigned color $r$.  

\begin{figure}[ht]
\centering
\begin{tikzpicture}[scale=0.30]

\begin{scope}[shift={(-36,0)}]
    \foreach \x in {0,...,15} {
    \draw[gray!70, line width=0.35pt] (\x,0) -- (\x,20);
  }
  \foreach \y in {0,...,20} {
    \draw[gray!70, line width=0.35pt] (0,\y) -- (15,\y);
  }
\draw[very thick] (1,2) rectangle (8,9);
\draw[very thick, orange] (7,12) rectangle (14,19);
\draw[very thick, red] (0,11) rectangle (7,18);
\draw[very thick, teal] (8,5) rectangle (15,12);

\node at (4.5,17.5) {\normalsize 1};
\node at (6.5,17.5) {\normalsize 1};
\node at (4.5,19.5) {\normalsize 1};
\node at (6.5,19.5) {\normalsize 1};
\node at (2.5,16.5) {\normalsize 1};
\node at (2.5,18.5) {\normalsize 1};
\node at (5.5,15.5) {\normalsize 1};
\node at (7.5,15.5) {\normalsize 1};
\node at (8.5,18.5) {\normalsize 1};

\node at (10.5,10.5) {\normalsize 1};
\node at (12.5,10.5) {\normalsize 1};
\node at (10.5,12.5) {\normalsize 1};
\node at (12.5,12.5) {\normalsize 1};
\node at (8.5,9.5) {\normalsize 1};
\node at (8.5,11.5) {\normalsize 1};
\node at (11.5,8.5) {\normalsize 1};
\node at (13.5,8.5) {\normalsize 1};
\node at (14.5,11.5) {\normalsize 1};
\node at (14.5,13.5) {\normalsize 1};
\node at (9.5,14.5) {\normalsize 1};
\node at (11.5,14.5) {\normalsize 1};

\node at (10.5,18.5) {1};
\node at (9.5,16.5) {1};
\node at (11.5,16.5) {1};
\node at (12.5,19.5) {1};
\node at (14.5,19.5) {1};
\node at (13.5,17.5) {1};
\node at (13.5,15.5) {1};
\node at (1.5,14.5) {1};
\node at (3.5,14.5) {1};
\node at (5.5,13.5) {1};
\node at (7.5,13.5) {1};
\node at (1.5,12.5) {1};
\node at (3.5,12.5) {1};
\node at (6.5,11.5) {1};
\node at (4.5,10.5) {1};
\node at (2.5,10.5) {1};
\node at (6.5,9.5) {1};
\node at (4.5,8.5) {1};
\node at (2.5,8.5) {1};
\node at (7.5,7.5) {1};
\node at (9.5,7.5) {1};
\node at (0.5,18.5) {1};
\node at (0.5,16.5) {1};
\node at (0.5,9.5) {1};
\node at (0.5,7.5) {1};
\node at (3.5,6.5) {1};
\node at (5.5,6.5) {1};
\node at (11.5,6.5) {1};
\node at (13.5,6.5) {1};
\node at (1.5,5.5) {1};
\node at (7.5,5.5) {1};
\node at (9.5,5.5) {1};
\node at (3.5,4.5) {1};
\node at (5.5,4.5) {1};
\node at (12.5,4.5) {1};
\node at (14.5,4.5) {1};
\node at (1.5,3.5) {1};
\node at (8.5,3.5) {1};
\node at (10.5,3.5) {1};
\node at (4.5,2.5) {1};
\node at (6.5,2.5) {1};
\node at (12.5,2.5) {1};
\node at (14.5,2.5) {1};
\node at (0.5,1.5) {1};
\node at (2.5,1.5) {1};
\node at (8.5,1.5) {1};
\node at (10.5,1.5) {1};
\node at (4.5,0.5) {1};
\node at (6.5,0.5) {1};
\node at (13.5,0.5) {1};

\node at (1.5,1.5) {\textcolor{blue} {$k$}};
\node at (0.5,8.5) {\textcolor{blue} {$k$}};
\node at (7.5,9.5) {\textcolor{blue} {$k$}};
\node at (8.5,2.5) {\textcolor{blue} {$k$}};

\node at (7.5,8.5) {\textcolor{darkgreen} {$m$}};
\node at (7.5,2.5) {\textcolor{orange} {$n$}};
\node at (0.5,2.5) {\textcolor{orange} {$n$}};

\node at (0.5,5.5) {\textcolor{red} {$p$}};
\node at (7.5,6.5) {\textcolor{red} {$p$}};
\node at (6.5,13.5) {\textcolor{red} {$p$}};

\node at (0.5,3.5) {\textcolor{purple} {$q$}};
\node at (7.5,3.5) {\textcolor{purple} {$q$}};

\node at (6.5,18.5) {\textcolor{darkcyan} {$r$}};
\node at (13.5,19.5) {\textcolor{darkcyan} {$r$}};
\node at (14.5,12.5) {\textcolor{darkcyan} {$r$}};
\node at (7.5,11.5) {\textcolor{darkcyan} {$r$}};
\node at (0.5,11.5) {\textcolor{darkcyan} {$r$}};
\node at (14.5,5.5) {\textcolor{darkcyan} {$r$}};

\end{scope}

\begin{scope}[shift={(-18,0)}]
    \foreach \x in {0,...,15} {
    \draw[gray!70, line width=0.35pt] (\x,0) -- (\x,20);
  }
  \foreach \y in {0,...,20} {
    \draw[gray!70, line width=0.35pt] (0,\y) -- (15,\y);
  }
\draw[very thick] (1,2) rectangle (8,9);
\draw[very thick, orange] (7,12) rectangle (14,19);
\draw[very thick, blue] (6,4) rectangle (13,11);
\draw[very thick, darkgreen] (7,0) rectangle (14,7);

\node at (4.5,17.5) {\normalsize 1};
\node at (6.5,17.5) {\normalsize 1};
\node at (4.5,19.5) {\normalsize 1};
\node at (6.5,19.5) {\normalsize 1};
\node at (2.5,16.5) {\normalsize 1};
\node at (2.5,18.5) {\normalsize 1};
\node at (5.5,15.5) {\normalsize 1};
\node at (7.5,15.5) {\normalsize 1};
\node at (8.5,18.5) {\normalsize 1};

\node at (10.5,10.5) {\normalsize 1};
\node at (12.5,10.5) {\normalsize 1};
\node at (10.5,12.5) {\normalsize 1};
\node at (12.5,12.5) {\normalsize 1};
\node at (8.5,9.5) {\normalsize 1};
\node at (8.5,11.5) {\normalsize 1};
\node at (11.5,8.5) {\normalsize 1};
\node at (13.5,8.5) {\normalsize 1};
\node at (14.5,11.5) {\normalsize 1};
\node at (14.5,13.5) {\normalsize 1};
\node at (9.5,14.5) {\normalsize 1};
\node at (11.5,14.5) {\normalsize 1};

\node at (10.5,18.5) {1};
\node at (9.5,16.5) {1};
\node at (11.5,16.5) {1};
\node at (12.5,19.5) {1};
\node at (14.5,19.5) {1};
\node at (13.5,17.5) {1};
\node at (13.5,15.5) {1};
\node at (1.5,14.5) {1};
\node at (3.5,14.5) {1};
\node at (5.5,13.5) {1};
\node at (7.5,13.5) {1};
\node at (1.5,12.5) {1};
\node at (3.5,12.5) {1};
\node at (6.5,11.5) {1};
\node at (4.5,10.5) {1};
\node at (2.5,10.5) {1};
\node at (6.5,9.5) {1};
\node at (4.5,8.5) {1};
\node at (2.5,8.5) {1};
\node at (7.5,7.5) {1};
\node at (9.5,7.5) {1};
\node at (0.5,18.5) {1};
\node at (0.5,16.5) {1};
\node at (0.5,9.5) {1};
\node at (0.5,7.5) {1};
\node at (3.5,6.5) {1};
\node at (5.5,6.5) {1};
\node at (11.5,6.5) {1};
\node at (13.5,6.5) {1};
\node at (1.5,5.5) {1};
\node at (7.5,5.5) {1};
\node at (9.5,5.5) {1};
\node at (3.5,4.5) {1};
\node at (5.5,4.5) {1};
\node at (12.5,4.5) {1};
\node at (14.5,4.5) {1};
\node at (1.5,3.5) {1};
\node at (8.5,3.5) {1};
\node at (10.5,3.5) {1};
\node at (4.5,2.5) {1};
\node at (6.5,2.5) {1};
\node at (12.5,2.5) {1};
\node at (14.5,2.5) {1};
\node at (0.5,1.5) {1};
\node at (2.5,1.5) {1};
\node at (8.5,1.5) {1};
\node at (10.5,1.5) {1};
\node at (4.5,0.5) {1};
\node at (6.5,0.5) {1};
\node at (13.5,0.5) {1};

\node at (1.5,1.5) {\textcolor{blue} {$k$}};
\node at (0.5,8.5) {\textcolor{blue} {$k$}};
\node at (7.5,9.5) {\textcolor{blue} {$k$}};
\node at (8.5,2.5) {\textcolor{blue} {$k$}};

\node at (7.5,8.5) {\textcolor{darkgreen} {$m$}};
\node at (7.5,2.5) {\textcolor{orange} {$n$}};
\node at (0.5,2.5) {\textcolor{orange} {$n$}};

\node at (0.5,5.5) {\textcolor{red} {$p$}};
\node at (7.5,6.5) {\textcolor{red} {$p$}};
\node at (6.5,13.5) {\textcolor{red} {$p$}};

\node at (0.5,3.5) {\textcolor{purple} {$q$}};
\node at (7.5,3.5) {\textcolor{purple} {$q$}};

\node at (6.5,18.5) {\textcolor{darkcyan} {$r$}};
\node at (13.5,19.5) {\textcolor{darkcyan} {$r$}};
\node at (14.5,12.5) {\textcolor{darkcyan} {$r$}};
\node at (7.5,11.5) {\textcolor{darkcyan} {$r$}};
\node at (0.5,11.5) {\textcolor{darkcyan} {$r$}};
\node at (14.5,5.5) {\textcolor{darkcyan} {$r$}};
\node at (6.5,4.5) {\textcolor{darkcyan} {$r$}};

\fill[yellow, opacity=0.3] (6,4) rectangle (8,5);

\end{scope}

\begin{scope}[shift={(0,0)}]
    \foreach \x in {0,...,15} {
    \draw[gray!70, line width=0.35pt] (\x,0) -- (\x,20);
  }
  \foreach \y in {0,...,20} {
    \draw[gray!70, line width=0.35pt] (0,\y) -- (15,\y);
  }
\draw[very thick] (1,2) rectangle (8,9);
\draw[very thick, orange] (7,12) rectangle (14,19);
\draw[very thick, dashed] (0,3) rectangle (7,10);

\node at (4.5,17.5) {\normalsize 1};
\node at (6.5,17.5) {\normalsize 1};
\node at (4.5,19.5) {\normalsize 1};
\node at (6.5,19.5) {\normalsize 1};
\node at (2.5,16.5) {\normalsize 1};
\node at (2.5,18.5) {\normalsize 1};
\node at (5.5,15.5) {\normalsize 1};
\node at (7.5,15.5) {\normalsize 1};
\node at (8.5,18.5) {\normalsize 1};

\node at (10.5,10.5) {\normalsize 1};
\node at (12.5,10.5) {\normalsize 1};
\node at (10.5,12.5) {\normalsize 1};
\node at (12.5,12.5) {\normalsize 1};
\node at (8.5,9.5) {\normalsize 1};
\node at (8.5,11.5) {\normalsize 1};
\node at (11.5,8.5) {\normalsize 1};
\node at (13.5,8.5) {\normalsize 1};
\node at (14.5,11.5) {\normalsize 1};
\node at (14.5,13.5) {\normalsize 1};
\node at (9.5,14.5) {\normalsize 1};
\node at (11.5,14.5) {\normalsize 1};

\node at (10.5,18.5) {1};
\node at (9.5,16.5) {1};
\node at (11.5,16.5) {1};
\node at (12.5,19.5) {1};
\node at (14.5,19.5) {1};
\node at (13.5,17.5) {1};
\node at (13.5,15.5) {1};
\node at (1.5,14.5) {1};
\node at (3.5,14.5) {1};
\node at (5.5,13.5) {1};
\node at (7.5,13.5) {1};
\node at (1.5,12.5) {1};
\node at (3.5,12.5) {1};
\node at (6.5,11.5) {1};
\node at (4.5,10.5) {1};
\node at (2.5,10.5) {1};
\node at (6.5,9.5) {1};
\node at (4.5,8.5) {1};
\node at (2.5,8.5) {1};
\node at (7.5,7.5) {1};
\node at (9.5,7.5) {1};
\node at (0.5,18.5) {1};
\node at (0.5,16.5) {1};
\node at (0.5,9.5) {1};
\node at (0.5,7.5) {1};
\node at (3.5,6.5) {1};
\node at (5.5,6.5) {1};
\node at (11.5,6.5) {1};
\node at (13.5,6.5) {1};
\node at (1.5,5.5) {1};
\node at (7.5,5.5) {1};
\node at (9.5,5.5) {1};
\node at (3.5,4.5) {1};
\node at (5.5,4.5) {1};
\node at (12.5,4.5) {1};
\node at (14.5,4.5) {1};
\node at (1.5,3.5) {1};
\node at (8.5,3.5) {1};
\node at (10.5,3.5) {1};
\node at (4.5,2.5) {1};
\node at (6.5,2.5) {1};
\node at (12.5,2.5) {1};
\node at (14.5,2.5) {1};
\node at (0.5,1.5) {1};
\node at (2.5,1.5) {1};
\node at (8.5,1.5) {1};
\node at (10.5,1.5) {1};
\node at (4.5,0.5) {1};
\node at (6.5,0.5) {1};
\node at (13.5,0.5) {1};

\node at (1.5,1.5) {\textcolor{blue} {$k$}};
\node at (0.5,8.5) {\textcolor{blue} {$k$}};
\node at (7.5,9.5) {\textcolor{blue} {$k$}};
\node at (8.5,2.5) {\textcolor{blue} {$k$}};

\node at (7.5,8.5) {\textcolor{darkgreen} {$m$}};
\node at (7.5,2.5) {\textcolor{orange} {$n$}};
\node at (0.5,2.5) {\textcolor{orange} {$n$}};

\node at (0.5,5.5) {\textcolor{red} {$p$}};
\node at (7.5,6.5) {\textcolor{red} {$p$}};
\node at (6.5,13.5) {\textcolor{red} {$p$}};

\node at (0.5,3.5) {\textcolor{purple} {$q$}};
\node at (7.5,3.5) {\textcolor{purple} {$q$}};

\node at (6.5,18.5) {\textcolor{darkcyan} {$r$}};
\node at (13.5,19.5) {\textcolor{darkcyan} {$r$}};
\node at (14.5,12.5) {\textcolor{darkcyan} {$r$}};
\node at (7.5,11.5) {\textcolor{darkcyan} {$r$}};
\node at (0.5,11.5) {\textcolor{darkcyan} {$r$}};
\node at (14.5,5.5) {\textcolor{darkcyan} {$r$}};
\node at (7.5,4.5) {\textcolor{darkcyan} {$r$}};
\node at (0.5,4.5) {\textcolor{darkcyan} {$r$}};

\fill[yellow, opacity=0.3] (6,4) rectangle (8,5);
\end{scope}

\end{tikzpicture}
\caption{$ G $ (black), $G(6,10)$ (orange),  $G(-1,9)$ (red), $G(7,3)$ (teal), $G(5,2)$ (blue), $G(6,-2)$ (green) and $G(-1,1)$ (dashed).}
\label{fig:24}
\end{figure}
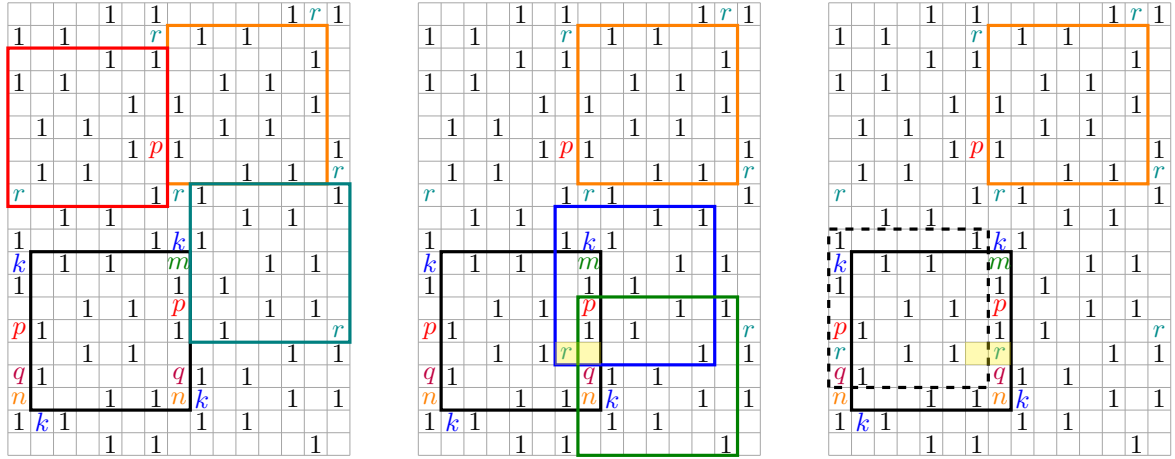

Next, we consider $G(-1,0)$ (see Figure \ref{fig:25}). By Observation~\ref{obs:2}\,(iv), since $c(G(-1,0)) = 11$, this subgraph must contain color $m$. By the distance constraint, the only possible squares that can be assigned color $m$ are those in column~1. However, squares $(1,1)$, $(1,2)$, $(1,3)$, $(1,4)$, $(1,6)$, and $(1,7)$ have already been assigned other colors, which cannot be $m$. Therefore, square $(1,5)$ must be assigned color $m$. Then, we consider $G(-1,5)$. By Observation \ref{obs:2}\,(iv), since $c(G(-1,5)) = 11$, this subgraph must contain color $m$. By the distance constraint, color $m$ must be assigned to square $(1,7)$.

After that, we consider $G(1,8)$ (see Figure \ref{fig:25}). By Observation~\ref{obs:2}\,(iv), since $c(G(1,8)) = 11$, this subgraph must contain color $m$. By the distance constraint, the only possible squares that can be assigned color $m$ are $(6,7)$, $(7,6)$ and $(7,7)$.

Assume that either $(7,6)$ or $(7,7)$ is assigned color $m$. Consider $G(0,12)$ (see Figure~\ref{fig:25}). By Observation~\ref{obs:2}\,(iv), since $c(G(0,12)) = 11$, this subgraph must contain color $m$. However, by the distance constraint, no square in this subgraph can be assigned color $m$, leading to a contradiction. Hence, square $(6,7)$ must be assigned color $m$. Next, consider $G(6,7)$. By Observation~\ref{obs:2}\,(iv), since $c(G(6,7)) = 11$, this subgraph must contain color $m$. Again, by the distance constraint, this subgraph cannot contain color $m$, which is also a contradiction.
\end{proof}

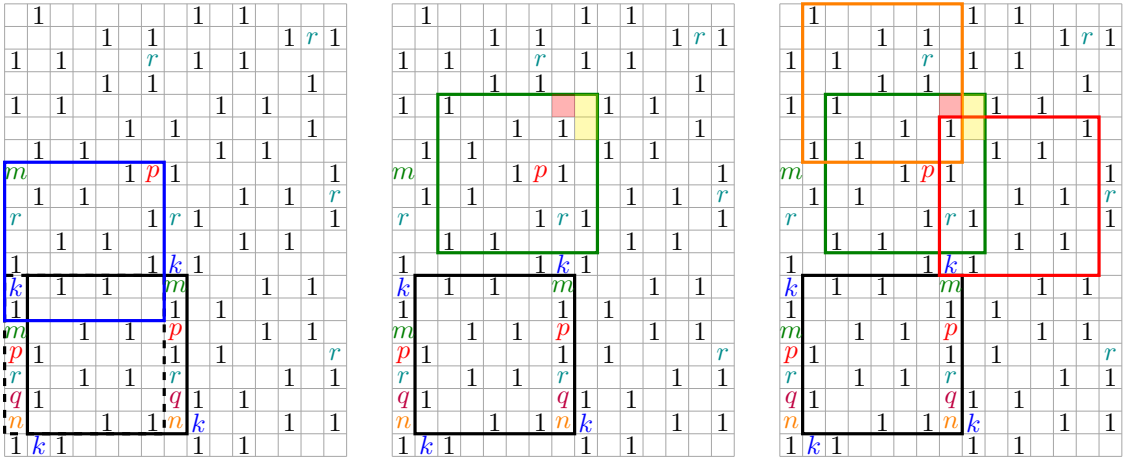
\begin{figure}[ht]
\centering
\begin{tikzpicture}[scale=0.30]

\begin{scope}[shift={(-34,0)}]
    \foreach \x in {0,...,15} {
    \draw[gray!70, line width=0.35pt] (\x,0) -- (\x,20);
  }
  \foreach \y in {0,...,20} {
    \draw[gray!70, line width=0.35pt] (0,\y) -- (15,\y);
  }
\draw[very thick] (1,1) rectangle (8,8);
\draw[very thick, dashed] (0,1) rectangle (7,8);
\draw[very thick, blue] (0,6) rectangle (7,13);

\node at (4.5,16.5) {\normalsize 1};
\node at (6.5,16.5) {\normalsize 1};
\node at (4.5,18.5) {\normalsize 1};
\node at (6.5,18.5) {\normalsize 1};
\node at (2.5,15.5) {\normalsize 1};
\node at (2.5,17.5) {\normalsize 1};
\node at (5.5,14.5) {\normalsize 1};
\node at (7.5,14.5) {\normalsize 1};
\node at (8.5,17.5) {\normalsize 1};

\node at (10.5,9.5) {\normalsize 1};
\node at (12.5,9.5) {\normalsize 1};
\node at (10.5,11.5) {\normalsize 1};
\node at (12.5,11.5) {\normalsize 1};
\node at (8.5,8.5) {\normalsize 1};
\node at (8.5,10.5) {\normalsize 1};
\node at (11.5,7.5) {\normalsize 1};
\node at (13.5,7.5) {\normalsize 1};
\node at (14.5,10.5) {\normalsize 1};
\node at (14.5,12.5) {\normalsize 1};
\node at (9.5,13.5) {\normalsize 1};
\node at (11.5,13.5) {\normalsize 1};

\node at (10.5,17.5) {1};
\node at (9.5,15.5) {1};
\node at (11.5,15.5) {1};
\node at (12.5,18.5) {1};
\node at (14.5,18.5) {1};
\node at (13.5,16.5) {1};
\node at (13.5,14.5) {1};
\node at (1.5,13.5) {1};
\node at (3.5,13.5) {1};
\node at (5.5,12.5) {1};
\node at (7.5,12.5) {1};
\node at (1.5,11.5) {1};
\node at (3.5,11.5) {1};
\node at (6.5,10.5) {1};
\node at (4.5,9.5) {1};
\node at (2.5,9.5) {1};
\node at (6.5,8.5) {1};
\node at (4.5,7.5) {1};
\node at (2.5,7.5) {1};
\node at (7.5,6.5) {1};
\node at (9.5,6.5) {1};
\node at (0.5,17.5) {1};
\node at (0.5,15.5) {1};
\node at (0.5,8.5) {1};
\node at (0.5,6.5) {1};
\node at (3.5,5.5) {1};
\node at (5.5,5.5) {1};
\node at (11.5,5.5) {1};
\node at (13.5,5.5) {1};
\node at (1.5,4.5) {1};
\node at (7.5,4.5) {1};
\node at (9.5,4.5) {1};
\node at (3.5,3.5) {1};
\node at (5.5,3.5) {1};
\node at (12.5,3.5) {1};
\node at (14.5,3.5) {1};
\node at (1.5,2.5) {1};
\node at (8.5,2.5) {1};
\node at (10.5,2.5) {1};
\node at (4.5,1.5) {1};
\node at (6.5,1.5) {1};
\node at (12.5,1.5) {1};
\node at (14.5,1.5) {1};
\node at (0.5,0.5) {1};
\node at (2.5,0.5) {1};
\node at (8.5,0.5) {1};
\node at (10.5,0.5) {1};
\node at (1.5,19.5) {1};
\node at (8.5,19.5) {1};
\node at (10.5,19.5) {1};

\node at (1.5,0.5) {\textcolor{blue} {$k$}};
\node at (0.5,7.5) {\textcolor{blue} {$k$}};
\node at (7.5,8.5) {\textcolor{blue} {$k$}};
\node at (8.5,1.5) {\textcolor{blue} {$k$}};

\node at (7.5,7.5) {\textcolor{darkgreen} {$m$}};
\node at (7.5,1.5) {\textcolor{orange} {$n$}};
\node at (0.5,1.5) {\textcolor{orange} {$n$}};

\node at (0.5,4.5) {\textcolor{red} {$p$}};
\node at (7.5,5.5) {\textcolor{red} {$p$}};
\node at (6.5,12.5) {\textcolor{red} {$p$}};

\node at (0.5,2.5) {\textcolor{purple} {$q$}};
\node at (7.5,2.5) {\textcolor{purple} {$q$}};

\node at (6.5,17.5) {\textcolor{darkcyan} {$r$}};
\node at (13.5,18.5) {\textcolor{darkcyan} {$r$}};
\node at (14.5,11.5) {\textcolor{darkcyan} {$r$}};
\node at (7.5,10.5) {\textcolor{darkcyan} {$r$}};
\node at (0.5,10.5) {\textcolor{darkcyan} {$r$}};
\node at (14.5,4.5) {\textcolor{darkcyan} {$r$}};
\node at (7.5,3.5) {\textcolor{darkcyan} {$r$}};
\node at (0.5,3.5) {\textcolor{darkcyan} {$r$}};
\node at (0.5,5.5) {\textcolor{darkgreen} {$m$}};
\node at (0.5,12.5) {\textcolor{darkgreen} {$m$}};
\end{scope}

\begin{scope}[shift={(-17,0)}]
    \foreach \x in {0,...,15} {
    \draw[gray!70, line width=0.35pt] (\x,0) -- (\x,20);
  }
  \foreach \y in {0,...,20} {
    \draw[gray!70, line width=0.35pt] (0,\y) -- (15,\y);
  }
\draw[very thick] (1,1) rectangle (8,8);
\draw[very thick, darkgreen] (2,9) rectangle (9,16);

\node at (4.5,16.5) {\normalsize 1};
\node at (6.5,16.5) {\normalsize 1};
\node at (4.5,18.5) {\normalsize 1};
\node at (6.5,18.5) {\normalsize 1};
\node at (2.5,15.5) {\normalsize 1};
\node at (2.5,17.5) {\normalsize 1};
\node at (5.5,14.5) {\normalsize 1};
\node at (7.5,14.5) {\normalsize 1};
\node at (8.5,17.5) {\normalsize 1};

\node at (10.5,9.5) {\normalsize 1};
\node at (12.5,9.5) {\normalsize 1};
\node at (10.5,11.5) {\normalsize 1};
\node at (12.5,11.5) {\normalsize 1};
\node at (8.5,8.5) {\normalsize 1};
\node at (8.5,10.5) {\normalsize 1};
\node at (11.5,7.5) {\normalsize 1};
\node at (13.5,7.5) {\normalsize 1};
\node at (14.5,10.5) {\normalsize 1};
\node at (14.5,12.5) {\normalsize 1};
\node at (9.5,13.5) {\normalsize 1};
\node at (11.5,13.5) {\normalsize 1};

\node at (10.5,17.5) {1};
\node at (9.5,15.5) {1};
\node at (11.5,15.5) {1};
\node at (12.5,18.5) {1};
\node at (14.5,18.5) {1};
\node at (13.5,16.5) {1};
\node at (13.5,14.5) {1};
\node at (1.5,13.5) {1};
\node at (3.5,13.5) {1};
\node at (5.5,12.5) {1};
\node at (7.5,12.5) {1};
\node at (1.5,11.5) {1};
\node at (3.5,11.5) {1};
\node at (6.5,10.5) {1};
\node at (4.5,9.5) {1};
\node at (2.5,9.5) {1};
\node at (6.5,8.5) {1};
\node at (4.5,7.5) {1};
\node at (2.5,7.5) {1};
\node at (7.5,6.5) {1};
\node at (9.5,6.5) {1};
\node at (0.5,17.5) {1};
\node at (0.5,15.5) {1};
\node at (0.5,8.5) {1};
\node at (0.5,6.5) {1};
\node at (3.5,5.5) {1};
\node at (5.5,5.5) {1};
\node at (11.5,5.5) {1};
\node at (13.5,5.5) {1};
\node at (1.5,4.5) {1};
\node at (7.5,4.5) {1};
\node at (9.5,4.5) {1};
\node at (3.5,3.5) {1};
\node at (5.5,3.5) {1};
\node at (12.5,3.5) {1};
\node at (14.5,3.5) {1};
\node at (1.5,2.5) {1};
\node at (8.5,2.5) {1};
\node at (10.5,2.5) {1};
\node at (4.5,1.5) {1};
\node at (6.5,1.5) {1};
\node at (12.5,1.5) {1};
\node at (14.5,1.5) {1};
\node at (0.5,0.5) {1};
\node at (2.5,0.5) {1};
\node at (8.5,0.5) {1};
\node at (10.5,0.5) {1};
\node at (1.5,19.5) {1};
\node at (8.5,19.5) {1};
\node at (10.5,19.5) {1};

\node at (1.5,0.5) {\textcolor{blue} {$k$}};
\node at (0.5,7.5) {\textcolor{blue} {$k$}};
\node at (7.5,8.5) {\textcolor{blue} {$k$}};
\node at (8.5,1.5) {\textcolor{blue} {$k$}};

\node at (7.5,7.5) {\textcolor{darkgreen} {$m$}};
\node at (7.5,1.5) {\textcolor{orange} {$n$}};
\node at (0.5,1.5) {\textcolor{orange} {$n$}};

\node at (0.5,4.5) {\textcolor{red} {$p$}};
\node at (7.5,5.5) {\textcolor{red} {$p$}};
\node at (6.5,12.5) {\textcolor{red} {$p$}};

\node at (0.5,2.5) {\textcolor{purple} {$q$}};
\node at (7.5,2.5) {\textcolor{purple} {$q$}};

\node at (6.5,17.5) {\textcolor{darkcyan} {$r$}};
\node at (13.5,18.5) {\textcolor{darkcyan} {$r$}};
\node at (14.5,11.5) {\textcolor{darkcyan} {$r$}};
\node at (7.5,10.5) {\textcolor{darkcyan} {$r$}};
\node at (0.5,10.5) {\textcolor{darkcyan} {$r$}};
\node at (14.5,4.5) {\textcolor{darkcyan} {$r$}};
\node at (7.5,3.5) {\textcolor{darkcyan} {$r$}};
\node at (0.5,3.5) {\textcolor{darkcyan} {$r$}};
\node at (0.5,5.5) {\textcolor{darkgreen} {$m$}};
\node at (0.5,12.5) {\textcolor{darkgreen} {$m$}};
\fill[yellow, opacity=0.3] (8,14) rectangle (9,16);
\fill[red, opacity=0.3] (7,15) rectangle (8,16);
\end{scope}

\begin{scope}[shift={(0,0)}]
    \foreach \x in {0,...,15} {
    \draw[gray!70, line width=0.35pt] (\x,0) -- (\x,20);
  }
  \foreach \y in {0,...,20} {
    \draw[gray!70, line width=0.35pt] (0,\y) -- (15,\y);
  }
\draw[very thick] (1,1) rectangle (8,8);
\draw[very thick, darkgreen] (2,9) rectangle (9,16);
\draw[very thick, orange] (1,13) rectangle (8,20);
\draw[very thick, red] (7,8) rectangle (14,15);

\node at (4.5,16.5) {\normalsize 1};
\node at (6.5,16.5) {\normalsize 1};
\node at (4.5,18.5) {\normalsize 1};
\node at (6.5,18.5) {\normalsize 1};
\node at (2.5,15.5) {\normalsize 1};
\node at (2.5,17.5) {\normalsize 1};
\node at (5.5,14.5) {\normalsize 1};
\node at (7.5,14.5) {\normalsize 1};
\node at (8.5,17.5) {\normalsize 1};

\node at (10.5,9.5) {\normalsize 1};
\node at (12.5,9.5) {\normalsize 1};
\node at (10.5,11.5) {\normalsize 1};
\node at (12.5,11.5) {\normalsize 1};
\node at (8.5,8.5) {\normalsize 1};
\node at (8.5,10.5) {\normalsize 1};
\node at (11.5,7.5) {\normalsize 1};
\node at (13.5,7.5) {\normalsize 1};
\node at (14.5,10.5) {\normalsize 1};
\node at (14.5,12.5) {\normalsize 1};
\node at (9.5,13.5) {\normalsize 1};
\node at (11.5,13.5) {\normalsize 1};

\node at (10.5,17.5) {1};
\node at (9.5,15.5) {1};
\node at (11.5,15.5) {1};
\node at (12.5,18.5) {1};
\node at (14.5,18.5) {1};
\node at (13.5,16.5) {1};
\node at (13.5,14.5) {1};
\node at (1.5,13.5) {1};
\node at (3.5,13.5) {1};
\node at (5.5,12.5) {1};
\node at (7.5,12.5) {1};
\node at (1.5,11.5) {1};
\node at (3.5,11.5) {1};
\node at (6.5,10.5) {1};
\node at (4.5,9.5) {1};
\node at (2.5,9.5) {1};
\node at (6.5,8.5) {1};
\node at (4.5,7.5) {1};
\node at (2.5,7.5) {1};
\node at (7.5,6.5) {1};
\node at (9.5,6.5) {1};
\node at (0.5,17.5) {1};
\node at (0.5,15.5) {1};
\node at (0.5,8.5) {1};
\node at (0.5,6.5) {1};
\node at (3.5,5.5) {1};
\node at (5.5,5.5) {1};
\node at (11.5,5.5) {1};
\node at (13.5,5.5) {1};
\node at (1.5,4.5) {1};
\node at (7.5,4.5) {1};
\node at (9.5,4.5) {1};
\node at (3.5,3.5) {1};
\node at (5.5,3.5) {1};
\node at (12.5,3.5) {1};
\node at (14.5,3.5) {1};
\node at (1.5,2.5) {1};
\node at (8.5,2.5) {1};
\node at (10.5,2.5) {1};
\node at (4.5,1.5) {1};
\node at (6.5,1.5) {1};
\node at (12.5,1.5) {1};
\node at (14.5,1.5) {1};
\node at (0.5,0.5) {1};
\node at (2.5,0.5) {1};
\node at (8.5,0.5) {1};
\node at (10.5,0.5) {1};
\node at (1.5,19.5) {1};
\node at (8.5,19.5) {1};
\node at (10.5,19.5) {1};

\node at (1.5,0.5) {\textcolor{blue} {$k$}};
\node at (0.5,7.5) {\textcolor{blue} {$k$}};
\node at (7.5,8.5) {\textcolor{blue} {$k$}};
\node at (8.5,1.5) {\textcolor{blue} {$k$}};

\node at (7.5,7.5) {\textcolor{darkgreen} {$m$}};
\node at (7.5,1.5) {\textcolor{orange} {$n$}};
\node at (0.5,1.5) {\textcolor{orange} {$n$}};

\node at (0.5,4.5) {\textcolor{red} {$p$}};
\node at (7.5,5.5) {\textcolor{red} {$p$}};
\node at (6.5,12.5) {\textcolor{red} {$p$}};

\node at (0.5,2.5) {\textcolor{purple} {$q$}};
\node at (7.5,2.5) {\textcolor{purple} {$q$}};

\node at (6.5,17.5) {\textcolor{darkcyan} {$r$}};
\node at (13.5,18.5) {\textcolor{darkcyan} {$r$}};
\node at (14.5,11.5) {\textcolor{darkcyan} {$r$}};
\node at (7.5,10.5) {\textcolor{darkcyan} {$r$}};
\node at (0.5,10.5) {\textcolor{darkcyan} {$r$}};
\node at (14.5,4.5) {\textcolor{darkcyan} {$r$}};
\node at (7.5,3.5) {\textcolor{darkcyan} {$r$}};
\node at (0.5,3.5) {\textcolor{darkcyan} {$r$}};
\node at (0.5,5.5) {\textcolor{darkgreen} {$m$}};
\node at (0.5,12.5) {\textcolor{darkgreen} {$m$}};
\fill[yellow, opacity=0.3] (8,14) rectangle (9,16);
\fill[red, opacity=0.3] (7,15) rectangle (8,16);
\end{scope}

\end{tikzpicture}
\caption{$ G $ (black), $G(-1,0)$ (dashed), $G(-1,5)$ (blue), $G(1,8)$ (green), $G(0,12)$ (orange) and $G(6,7)$ (red).}
\label{fig:25}
\end{figure}

\subsection{Proof of Lemma \ref{lem:4}\,(iv)}
\begin{proof}
    Suppose for contradiction that the pattern
    \[
    \begin{bmatrix}
    x & x & x & 1\\
    1 & x & x & x\\
    x & x & x & x\\
    1 & x & 1 & x
    \end{bmatrix}
    \]
    appears in the coloring. Let $ G $ be a $ P_7 \boxtimes P_7 $ subgraph in which the pattern appears, where the squares at $(2,2)$, $(2,4)$, $(4,2)$, and $(5,5)$ are of color~1. Observe that the squares that cannot be assigned color~1 are as shown in Figure~\ref{fig:26}. We now partition the remaining squares into seven parts, as shown in Figure~\ref{fig:26}. By Observation~\ref{obs:2}\,(i) and (iv), each part must contain exactly one square of color~1. In particular, the part at $(7,7)$ must be a square of color~1. By the distance constraint, squares $(6,7)$ and $(7,6)$ cannot be assigned color 1. As a result, squares $(5,7)$ and $(7,5)$ must be of color~1. Observe that this subgraph contains the pattern
\[
\begin{bmatrix}
1 & x & 1\\
x & x & x\\
1 & x & 1
\end{bmatrix}
\]
which contradicts Lemma~\ref{lem:4}\,(iii).
\end{proof}

    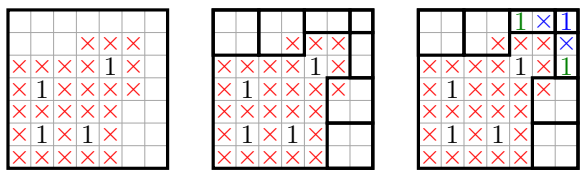
\begin{figure}[ht]
\centering
\begin{tikzpicture}[scale=0.30]

\begin{scope}[shift={(-18,0)}]
  \foreach \x in {0,...,7} {
    \draw[gray!70, line width=0.35pt] (\x,0) -- (\x,7);
  }
  \foreach \y in {0,...,7} {
    \draw[gray!70, line width=0.35pt] (0,\y) -- (7,\y);
  }
  \draw[very thick] (0,0) rectangle (7,7);

  \node at (1.5,1.5) {\normalsize 1};
  \node at (3.5,1.5) {\normalsize 1};
  \node at (1.5,3.5) {\normalsize 1};
  \node at (4.5,4.5) {\normalsize 1};
\node at (0.5,0.5) {\textcolor{red}{$\times$}};
\node at (1.5,0.5) {\textcolor{red}{$\times$}};
\node at (2.5,0.5) {\textcolor{red}{$\times$}};
\node at (3.5,0.5) {\textcolor{red}{$\times$}};
\node at (4.5,0.5) {\textcolor{red}{$\times$}};
\node at (0.5,1.5) {\textcolor{red}{$\times$}};
\node at (2.5,1.5) {\textcolor{red}{$\times$}};
\node at (4.5,1.5) {\textcolor{red}{$\times$}};
\node at (0.5,2.5) {\textcolor{red}{$\times$}};
\node at (1.5,2.5) {\textcolor{red}{$\times$}};
\node at (2.5,2.5) {\textcolor{red}{$\times$}};
\node at (3.5,2.5) {\textcolor{red}{$\times$}};
\node at (4.5,2.5) {\textcolor{red}{$\times$}};
\node at (0.5,3.5) {\textcolor{red}{$\times$}};
\node at (2.5,3.5) {\textcolor{red}{$\times$}};
\node at (3.5,3.5) {\textcolor{red}{$\times$}};
\node at (4.5,3.5) {\textcolor{red}{$\times$}};
\node at (5.5,3.5) {\textcolor{red}{$\times$}};
\node at (0.5,4.5) {\textcolor{red}{$\times$}};
\node at (1.5,4.5) {\textcolor{red}{$\times$}};
\node at (2.5,4.5) {\textcolor{red}{$\times$}};
\node at (3.5,4.5) {\textcolor{red}{$\times$}};
\node at (5.5,4.5) {\textcolor{red}{$\times$}};
\node at (3.5,5.5) {\textcolor{red}{$\times$}};
\node at (4.5,5.5) {\textcolor{red}{$\times$}};
\node at (5.5,5.5) {\textcolor{red}{$\times$}};
  
\end{scope}

\begin{scope}[shift={(-9,0)}]
\foreach \x in {0,...,7} {
    \draw[gray!70, line width=0.35pt] (\x,0) -- (\x,7);
  }
  \foreach \y in {0,...,7} {
    \draw[gray!70, line width=0.35pt] (0,\y) -- (7,\y);
  }
  \draw[very thick] (0,0) rectangle (7,7);
  \draw[very thick] (0,5) rectangle (2,7);
  \draw[very thick] (2,5) rectangle (4,7);
  \draw[very thick] (4,6) rectangle (6,7);
  \draw[very thick] (6,6) rectangle (7,7);
  \draw[very thick] (6,4) rectangle (7,6);
  \draw[very thick] (5,2) rectangle (7,4);
  \draw[very thick] (5,0) rectangle (7,2);

  \node at (1.5,1.5) {\normalsize 1};
  \node at (3.5,1.5) {\normalsize 1};
  \node at (1.5,3.5) {\normalsize 1};
  \node at (4.5,4.5) {\normalsize 1};
\node at (0.5,0.5) {\textcolor{red}{$\times$}};
\node at (1.5,0.5) {\textcolor{red}{$\times$}};
\node at (2.5,0.5) {\textcolor{red}{$\times$}};
\node at (3.5,0.5) {\textcolor{red}{$\times$}};
\node at (4.5,0.5) {\textcolor{red}{$\times$}};
\node at (0.5,1.5) {\textcolor{red}{$\times$}};
\node at (2.5,1.5) {\textcolor{red}{$\times$}};
\node at (4.5,1.5) {\textcolor{red}{$\times$}};
\node at (0.5,2.5) {\textcolor{red}{$\times$}};
\node at (1.5,2.5) {\textcolor{red}{$\times$}};
\node at (2.5,2.5) {\textcolor{red}{$\times$}};
\node at (3.5,2.5) {\textcolor{red}{$\times$}};
\node at (4.5,2.5) {\textcolor{red}{$\times$}};
\node at (0.5,3.5) {\textcolor{red}{$\times$}};
\node at (2.5,3.5) {\textcolor{red}{$\times$}};
\node at (3.5,3.5) {\textcolor{red}{$\times$}};
\node at (4.5,3.5) {\textcolor{red}{$\times$}};
\node at (5.5,3.5) {\textcolor{red}{$\times$}};
\node at (0.5,4.5) {\textcolor{red}{$\times$}};
\node at (1.5,4.5) {\textcolor{red}{$\times$}};
\node at (2.5,4.5) {\textcolor{red}{$\times$}};
\node at (3.5,4.5) {\textcolor{red}{$\times$}};
\node at (5.5,4.5) {\textcolor{red}{$\times$}};
\node at (3.5,5.5) {\textcolor{red}{$\times$}};
\node at (4.5,5.5) {\textcolor{red}{$\times$}};
\node at (5.5,5.5) {\textcolor{red}{$\times$}};
\end{scope}

\begin{scope}[shift={(0,0)}]
  \foreach \x in {0,...,7} {
    \draw[gray!70, line width=0.35pt] (\x,0) -- (\x,7);
  }
  \foreach \y in {0,...,7} {
    \draw[gray!70, line width=0.35pt] (0,\y) -- (7,\y);
  }
  \draw[very thick] (0,0) rectangle (7,7);
  \draw[very thick] (0,5) rectangle (2,7);
  \draw[very thick] (2,5) rectangle (4,7);
  \draw[very thick] (4,6) rectangle (6,7);
  \draw[very thick] (6,6) rectangle (7,7);
  \draw[very thick] (6,4) rectangle (7,6);
  \draw[very thick] (5,2) rectangle (7,4);
  \draw[very thick] (5,0) rectangle (7,2);

  \node at (1.5,1.5) {\normalsize 1};
  \node at (3.5,1.5) {\normalsize 1};
  \node at (1.5,3.5) {\normalsize 1};
  \node at (4.5,4.5) {\normalsize 1};
\node at (0.5,0.5) {\textcolor{red}{$\times$}};
\node at (1.5,0.5) {\textcolor{red}{$\times$}};
\node at (2.5,0.5) {\textcolor{red}{$\times$}};
\node at (3.5,0.5) {\textcolor{red}{$\times$}};
\node at (4.5,0.5) {\textcolor{red}{$\times$}};
\node at (0.5,1.5) {\textcolor{red}{$\times$}};
\node at (2.5,1.5) {\textcolor{red}{$\times$}};
\node at (4.5,1.5) {\textcolor{red}{$\times$}};
\node at (0.5,2.5) {\textcolor{red}{$\times$}};
\node at (1.5,2.5) {\textcolor{red}{$\times$}};
\node at (2.5,2.5) {\textcolor{red}{$\times$}};
\node at (3.5,2.5) {\textcolor{red}{$\times$}};
\node at (4.5,2.5) {\textcolor{red}{$\times$}};
\node at (0.5,3.5) {\textcolor{red}{$\times$}};
\node at (2.5,3.5) {\textcolor{red}{$\times$}};
\node at (3.5,3.5) {\textcolor{red}{$\times$}};
\node at (4.5,3.5) {\textcolor{red}{$\times$}};
\node at (5.5,3.5) {\textcolor{red}{$\times$}};
\node at (0.5,4.5) {\textcolor{red}{$\times$}};
\node at (1.5,4.5) {\textcolor{red}{$\times$}};
\node at (2.5,4.5) {\textcolor{red}{$\times$}};
\node at (3.5,4.5) {\textcolor{red}{$\times$}};
\node at (5.5,4.5) {\textcolor{red}{$\times$}};
\node at (3.5,5.5) {\textcolor{red}{$\times$}};
\node at (4.5,5.5) {\textcolor{red}{$\times$}};
\node at (5.5,5.5) {\textcolor{red}{$\times$}};
\node at (6.5,6.5) {\textcolor{blue} {1}};
\node at (5.5,6.5) {\textcolor{blue} {$\times$}};
\node at (6.5,5.5) {\textcolor{blue} {$\times$}};
\node at (4.5,6.5) {\textcolor{darkgreen} {1}};
\node at (6.5,4.5) {\textcolor{darkgreen} {1}};
\end{scope}

\end{tikzpicture}
\caption{$G$ in the proof of Lemma~4\,(iv).}

\label{fig:26}
\end{figure}

\subsection{Proof of Lemma \ref{lem:4}\,(v)}
\begin{proof}
    Suppose for contradiction that the pattern 
    \[
    \begin{bmatrix}
    1 & x & x\\
    x & x & x\\
    1 & x & 1
    \end{bmatrix}
    \]
    appears in the coloring. Let $ G $ be a $ P_7 \boxtimes P_7 $ subgraph in which the pattern appears, where the squares at $(2,2)$, $(2,4)$, and $(4,2)$ are of color~1. By the distance constraint, Lemma~\ref{lem:4}\,(ii), (iii) and (iv), the squares that cannot be assigned color~1 are shown in Figure~\ref{fig:27}. We now partition the remaining squares into seven parts, as shown in Figure~\ref{fig:27}. By Observation~\ref{obs:2}\,(i), each part contains at most one square of color~1. Hence, $c(G) \leq 3+7 = 10$, which contradicts Observation~\ref{obs:2}\,(iv).

\begin{figure}[ht]
\centering
\begin{tikzpicture}[scale=0.30]

\begin{scope}[shift={(-18,0)}]
  \foreach \x in {0,...,7} {
    \draw[gray!70, line width=0.35pt] (\x,0) -- (\x,7);
  }
  \foreach \y in {0,...,7} {
    \draw[gray!70, line width=0.35pt] (0,\y) -- (7,\y);
  }
  \draw[very thick] (0,0) rectangle (7,7);

  \node at (1.5,1.5) {\normalsize 1};
  \node at (3.5,1.5) {\normalsize 1};
  \node at (1.5,3.5) {\normalsize 1};
\node at (0.5,0.5) {\textcolor{red}{$\times$}};
\node at (1.5,0.5) {\textcolor{red}{$\times$}};
\node at (2.5,0.5) {\textcolor{red}{$\times$}};
\node at (3.5,0.5) {\textcolor{red}{$\times$}};
\node at (4.5,0.5) {\textcolor{red}{$\times$}};
\node at (0.5,1.5) {\textcolor{red}{$\times$}};
\node at (2.5,1.5) {\textcolor{red}{$\times$}};
\node at (4.5,1.5) {\textcolor{red}{$\times$}};
\node at (0.5,2.5) {\textcolor{red}{$\times$}};
\node at (1.5,2.5) {\textcolor{red}{$\times$}};
\node at (2.5,2.5) {\textcolor{red}{$\times$}};
\node at (3.5,2.5) {\textcolor{red}{$\times$}};
\node at (4.5,2.5) {\textcolor{red}{$\times$}};
\node at (0.5,3.5) {\textcolor{red}{$\times$}};
\node at (2.5,3.5) {\textcolor{red}{$\times$}};
\node at (0.5,4.5) {\textcolor{red}{$\times$}};
\node at (1.5,4.5) {\textcolor{red}{$\times$}};
\node at (2.5,4.5) {\textcolor{red}{$\times$}};
  
\end{scope}

\begin{scope}[shift={(-9,0)}]
\foreach \x in {0,...,7} {
    \draw[gray!70, line width=0.35pt] (\x,0) -- (\x,7);
  }
  \foreach \y in {0,...,7} {
    \draw[gray!70, line width=0.35pt] (0,\y) -- (7,\y);
  }
  \draw[very thick] (0,0) rectangle (7,7);

  \node at (1.5,1.5) {\normalsize 1};
  \node at (3.5,1.5) {\normalsize 1};
  \node at (1.5,3.5) {\normalsize 1};
\node at (0.5,0.5) {\textcolor{red}{$\times$}};
\node at (1.5,0.5) {\textcolor{red}{$\times$}};
\node at (2.5,0.5) {\textcolor{red}{$\times$}};
\node at (3.5,0.5) {\textcolor{red}{$\times$}};
\node at (4.5,0.5) {\textcolor{red}{$\times$}};
\node at (0.5,1.5) {\textcolor{red}{$\times$}};
\node at (2.5,1.5) {\textcolor{red}{$\times$}};
\node at (4.5,1.5) {\textcolor{red}{$\times$}};
\node at (0.5,2.5) {\textcolor{red}{$\times$}};
\node at (1.5,2.5) {\textcolor{red}{$\times$}};
\node at (2.5,2.5) {\textcolor{red}{$\times$}};
\node at (3.5,2.5) {\textcolor{red}{$\times$}};
\node at (4.5,2.5) {\textcolor{red}{$\times$}};
\node at (0.5,3.5) {\textcolor{red}{$\times$}};
\node at (2.5,3.5) {\textcolor{red}{$\times$}};
\node at (0.5,4.5) {\textcolor{red}{$\times$}};
\node at (1.5,4.5) {\textcolor{red}{$\times$}};
\node at (2.5,4.5) {\textcolor{red}{$\times$}};

\node at (3.5,3.5) {\textcolor{darkgreen}{$\times$}};
\node at (3.5,4.5) {\textcolor{blue}{$\times$}};
\node at (4.5,3.5) {\textcolor{blue}{$\times$}};
\node at (4.5,4.5) {\textcolor{orange}{$\times$}};
\end{scope}

\begin{scope}[shift={(0,0)}]
  \foreach \x in {0,...,7} {
    \draw[gray!70, line width=0.35pt] (\x,0) -- (\x,7);
  }
  \foreach \y in {0,...,7} {
    \draw[gray!70, line width=0.35pt] (0,\y) -- (7,\y);
  }
  \draw[very thick] (0,0) rectangle (7,7);
  \draw[very thick] (0,5) rectangle (2,7);
  \draw[very thick] (2,5) rectangle (4,7);
  \draw[very thick] (4,5) rectangle (6,7);
  \draw[very thick] (6,5) rectangle (7,7);
  \draw[very thick] (5,4) rectangle (7,5);
  \draw[very thick] (5,2) rectangle (7,4);
  \draw[very thick] (5,0) rectangle (7,2);

  \node at (1.5,1.5) {\normalsize 1};
  \node at (3.5,1.5) {\normalsize 1};
  \node at (1.5,3.5) {\normalsize 1};
\node at (0.5,0.5) {\textcolor{red}{$\times$}};
\node at (1.5,0.5) {\textcolor{red}{$\times$}};
\node at (2.5,0.5) {\textcolor{red}{$\times$}};
\node at (3.5,0.5) {\textcolor{red}{$\times$}};
\node at (4.5,0.5) {\textcolor{red}{$\times$}};
\node at (0.5,1.5) {\textcolor{red}{$\times$}};
\node at (2.5,1.5) {\textcolor{red}{$\times$}};
\node at (4.5,1.5) {\textcolor{red}{$\times$}};
\node at (0.5,2.5) {\textcolor{red}{$\times$}};
\node at (1.5,2.5) {\textcolor{red}{$\times$}};
\node at (2.5,2.5) {\textcolor{red}{$\times$}};
\node at (3.5,2.5) {\textcolor{red}{$\times$}};
\node at (4.5,2.5) {\textcolor{red}{$\times$}};
\node at (0.5,3.5) {\textcolor{red}{$\times$}};
\node at (2.5,3.5) {\textcolor{red}{$\times$}};
\node at (0.5,4.5) {\textcolor{red}{$\times$}};
\node at (1.5,4.5) {\textcolor{red}{$\times$}};
\node at (2.5,4.5) {\textcolor{red}{$\times$}};

\node at (3.5,3.5) {\textcolor{darkgreen}{$\times$}};
\node at (3.5,4.5) {\textcolor{blue}{$\times$}};
\node at (4.5,3.5) {\textcolor{blue}{$\times$}};
\node at (4.5,4.5) {\textcolor{orange}{$\times$}};
\end{scope}

\end{tikzpicture}

\caption{$G$ in the proof of Lemma~4\,(v).}
\label{fig:27}
\end{figure}
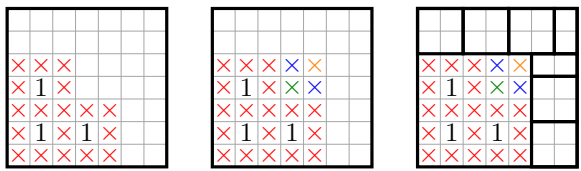

    Moreover, we can show that each L-shaped subgraph contains at most two squares of color 1. Assume for contradiction that there exists an L-shaped subgraph that contains three squares of color 1. Then the subgraph has the pattern  
\[
    \begin{bmatrix}
    1 & x & x\\
    x & x & x\\
    1 & x & 1
    \end{bmatrix}
    \]
which is a contradiction.
\end{proof}

\subsection{Proof of Lemma \ref{lem:4}\,(vi)}
\begin{proof}
Suppose for contradiction that the pattern
\[
    \begin{bmatrix}
    1 & x & 1 & x\\
    x & x & x & x\\
    x & 1 & x & 1
    \end{bmatrix}
    \]
appears in the coloring. Let $ G $ be a $ P_7 \boxtimes P_7 $ subgraph in which the pattern appears, where the squares at $(4,2)$, $(6,2)$, $(3,4)$, and $(5,4)$ are of color~1. By the distance constraint and Lemma~\ref{lem:4}\,(i), (ii), and (v), the squares that cannot be assigned color~1 are shown in Figure~\ref{fig:29}. We now partition the remaining squares into six parts, as shown in Figure~\ref{fig:29}. By Observations~\ref{obs:2}\,(i), (iv), and Lemma \ref{lem:4}\,(i), the $P_1 \boxtimes P_5$ part must contain exactly two squares of color~1, and the other five parts must each contain exactly one square of color~1. In particular, the part at $(7,5)$ must be a square of color~1.  

Next, consider the part consisting of squares $(2,6)$ and $(2,7)$. Assume that square $(2,6)$ is of color~1. As a result, by the distance constraint, the $P_1 \boxtimes P_5$ part can contain at most one square of color~1, which is a contradiction. Hence, square $(2,7)$ must be of color~1. By the distance constraint, squares $(1,3)$ and $(1,5)$ must be of color~1. By Lemma~\ref{lem:4}\,(i), square $(1,1)$ cannot be assigned color~1. Therefore, square $(2,1)$ must be of color~1.  

Following this, we consider $G(-1,-1)$ (see Figure~\ref{fig:29}). By the distance constraint and Lemma~\ref{lem:4}\,(v), the squares that cannot be assigned color~1 are shown in Figure~\ref{fig:29}. We partition the remaining squares into three parts, as shown in Figure~\ref{fig:29}. By Observation~\ref{obs:2}\,(i), each part contains at most one square of color~1. Consequently, we obtain $c(G(-1,-1)) \leq 7+3 = 10$, which contradicts Observation~\ref{obs:2}\,(iv).
\end{proof}

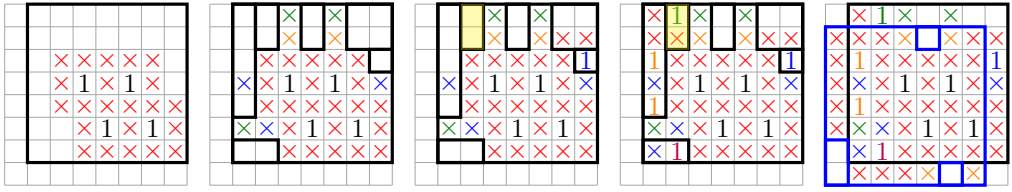
\begin{figure}[ht]
\centering
\begin{tikzpicture}[scale=0.30]

\begin{scope}[shift={(-36,0)}]
  \foreach \x in {0,...,8} {
    \draw[gray!70, line width=0.35pt] (\x,0) -- (\x,8);
  }
  \foreach \y in {0,...,8} {
    \draw[gray!70, line width=0.35pt] (0,\y) -- (8,\y);
  }
  \draw[very thick] (1,1) rectangle (8,8);

\node at (4.5,2.5) {\normalsize 1};
\node at (6.5,2.5) {\normalsize 1};
\node at (3.5,4.5) {\normalsize 1};
\node at (5.5,4.5) {\normalsize 1};

\node at (3.5,1.5) {\textcolor{red}{$\times$}};
\node at (4.5,1.5) {\textcolor{red}{$\times$}};
\node at (5.5,1.5) {\textcolor{red}{$\times$}};
\node at (6.5,1.5) {\textcolor{red}{$\times$}};
\node at (7.5,1.5) {\textcolor{red}{$\times$}};

\node at (3.5,2.5) {\textcolor{red}{$\times$}};
\node at (5.5,2.5) {\textcolor{red}{$\times$}};
\node at (7.5,2.5) {\textcolor{red}{$\times$}};

\node at (2.5,3.5) {\textcolor{red}{$\times$}};
\node at (3.5,3.5) {\textcolor{red}{$\times$}};
\node at (4.5,3.5) {\textcolor{red}{$\times$}};
\node at (5.5,3.5) {\textcolor{red}{$\times$}};
\node at (6.5,3.5) {\textcolor{red}{$\times$}};
\node at (7.5,3.5) {\textcolor{red}{$\times$}};

\node at (2.5,4.5) {\textcolor{red}{$\times$}};
\node at (4.5,4.5) {\textcolor{red}{$\times$}};
\node at (6.5,4.5) {\textcolor{red}{$\times$}};

\node at (2.5,5.5) {\textcolor{red}{$\times$}};
\node at (3.5,5.5) {\textcolor{red}{$\times$}};
\node at (4.5,5.5) {\textcolor{red}{$\times$}};
\node at (5.5,5.5) {\textcolor{red}{$\times$}};
\node at (6.5,5.5) {\textcolor{red}{$\times$}};
  
\end{scope}

\begin{scope}[shift={(-27,0)}]
  \foreach \x in {0,...,8} {
    \draw[gray!70, line width=0.35pt] (\x,0) -- (\x,8);
  }
  \foreach \y in {0,...,8} {
    \draw[gray!70, line width=0.35pt] (0,\y) -- (8,\y);
  }
  \draw[very thick] (1,1) rectangle (8,8);
  \draw[very thick] (1,1) rectangle (3,2);
  \draw[very thick] (1,3) rectangle (2,8);
  \draw[very thick] (2,6) rectangle (3,8);
  \draw[very thick] (4,6) rectangle (5,8);
  \draw[very thick] (6,6) rectangle (8,8);
  \draw[very thick] (7,5) rectangle (8,6);

\node at (4.5,2.5) {\normalsize 1};
\node at (6.5,2.5) {\normalsize 1};
\node at (3.5,4.5) {\normalsize 1};
\node at (5.5,4.5) {\normalsize 1};

\node at (3.5,1.5) {\textcolor{red}{$\times$}};
\node at (4.5,1.5) {\textcolor{red}{$\times$}};
\node at (5.5,1.5) {\textcolor{red}{$\times$}};
\node at (6.5,1.5) {\textcolor{red}{$\times$}};
\node at (7.5,1.5) {\textcolor{red}{$\times$}};

\node at (3.5,2.5) {\textcolor{red}{$\times$}};
\node at (5.5,2.5) {\textcolor{red}{$\times$}};
\node at (7.5,2.5) {\textcolor{red}{$\times$}};

\node at (2.5,3.5) {\textcolor{red}{$\times$}};
\node at (3.5,3.5) {\textcolor{red}{$\times$}};
\node at (4.5,3.5) {\textcolor{red}{$\times$}};
\node at (5.5,3.5) {\textcolor{red}{$\times$}};
\node at (6.5,3.5) {\textcolor{red}{$\times$}};
\node at (7.5,3.5) {\textcolor{red}{$\times$}};

\node at (2.5,4.5) {\textcolor{red}{$\times$}};
\node at (4.5,4.5) {\textcolor{red}{$\times$}};
\node at (6.5,4.5) {\textcolor{red}{$\times$}};

\node at (2.5,5.5) {\textcolor{red}{$\times$}};
\node at (3.5,5.5) {\textcolor{red}{$\times$}};
\node at (4.5,5.5) {\textcolor{red}{$\times$}};
\node at (5.5,5.5) {\textcolor{red}{$\times$}};
\node at (6.5,5.5) {\textcolor{red}{$\times$}};

\node at (2.5,2.5) {\textcolor{blue}{$\times$}};
\node at (1.5,4.5) {\textcolor{blue}{$\times$}};
\node at (7.5,4.5) {\textcolor{blue}{$\times$}};

\node at (1.5,2.5) {\textcolor{darkgreen}{$\times$}};
\node at (5.5,7.5) {\textcolor{darkgreen}{$\times$}};
\node at (3.5,7.5) {\textcolor{darkgreen}{$\times$}};

\node at (3.5,6.5) {\textcolor{orange}{$\times$}};
\node at (5.5,6.5) {\textcolor{orange}{$\times$}};
  
\end{scope}

\begin{scope}[shift={(-18,0)}]
  \foreach \x in {0,...,8} {
    \draw[gray!70, line width=0.35pt] (\x,0) -- (\x,8);
  }
  \foreach \y in {0,...,8} {
    \draw[gray!70, line width=0.35pt] (0,\y) -- (8,\y);
  }
  \draw[very thick] (1,1) rectangle (8,8);
  \draw[very thick] (1,1) rectangle (3,2);
  \draw[very thick] (1,3) rectangle (2,8);
  \draw[very thick] (2,6) rectangle (3,8);
  \draw[very thick] (4,6) rectangle (5,8);
  \draw[very thick] (6,6) rectangle (8,8);
  \draw[very thick] (7,5) rectangle (8,6);

\node at (4.5,2.5) {\normalsize 1};
\node at (6.5,2.5) {\normalsize 1};
\node at (3.5,4.5) {\normalsize 1};
\node at (5.5,4.5) {\normalsize 1};

\node at (3.5,1.5) {\textcolor{red}{$\times$}};
\node at (4.5,1.5) {\textcolor{red}{$\times$}};
\node at (5.5,1.5) {\textcolor{red}{$\times$}};
\node at (6.5,1.5) {\textcolor{red}{$\times$}};
\node at (7.5,1.5) {\textcolor{red}{$\times$}};

\node at (3.5,2.5) {\textcolor{red}{$\times$}};
\node at (5.5,2.5) {\textcolor{red}{$\times$}};
\node at (7.5,2.5) {\textcolor{red}{$\times$}};

\node at (2.5,3.5) {\textcolor{red}{$\times$}};
\node at (3.5,3.5) {\textcolor{red}{$\times$}};
\node at (4.5,3.5) {\textcolor{red}{$\times$}};
\node at (5.5,3.5) {\textcolor{red}{$\times$}};
\node at (6.5,3.5) {\textcolor{red}{$\times$}};
\node at (7.5,3.5) {\textcolor{red}{$\times$}};

\node at (2.5,4.5) {\textcolor{red}{$\times$}};
\node at (4.5,4.5) {\textcolor{red}{$\times$}};
\node at (6.5,4.5) {\textcolor{red}{$\times$}};

\node at (2.5,5.5) {\textcolor{red}{$\times$}};
\node at (3.5,5.5) {\textcolor{red}{$\times$}};
\node at (4.5,5.5) {\textcolor{red}{$\times$}};
\node at (5.5,5.5) {\textcolor{red}{$\times$}};
\node at (6.5,5.5) {\textcolor{red}{$\times$}};

\node at (2.5,2.5) {\textcolor{blue}{$\times$}};
\node at (1.5,4.5) {\textcolor{blue}{$\times$}};
\node at (7.5,4.5) {\textcolor{blue}{$\times$}};

\node at (1.5,2.5) {\textcolor{darkgreen}{$\times$}};
\node at (5.5,7.5) {\textcolor{darkgreen}{$\times$}};
\node at (3.5,7.5) {\textcolor{darkgreen}{$\times$}};

\node at (3.5,6.5) {\textcolor{orange}{$\times$}};
\node at (5.5,6.5) {\textcolor{orange}{$\times$}};

\node at (7.5,5.5) {\textcolor{blue}{1}};
\node at (7.5,6.5) {\textcolor{red}{$\times$}};
\node at (6.5,6.5) {\textcolor{red}{$\times$}};

\fill[yellow, opacity=0.3] (2,6) rectangle (3,8);
  
\end{scope}

\begin{scope}[shift={(-9,0)}]
\foreach \x in {0,...,8} {
    \draw[gray!70, line width=0.35pt] (\x,0) -- (\x,8);
  }
  \foreach \y in {0,...,8} {
    \draw[gray!70, line width=0.35pt] (0,\y) -- (8,\y);
  }
  \draw[very thick] (1,1) rectangle (8,8);
  \draw[very thick] (1,1) rectangle (3,2);
  \draw[very thick] (1,3) rectangle (2,8);
  \draw[very thick] (2,6) rectangle (3,8);
  \draw[very thick] (4,6) rectangle (5,8);
  \draw[very thick] (6,6) rectangle (8,8);
  \draw[very thick] (7,5) rectangle (8,6);

\node at (4.5,2.5) {\normalsize 1};
\node at (6.5,2.5) {\normalsize 1};
\node at (3.5,4.5) {\normalsize 1};
\node at (5.5,4.5) {\normalsize 1};

\node at (3.5,1.5) {\textcolor{red}{$\times$}};
\node at (4.5,1.5) {\textcolor{red}{$\times$}};
\node at (5.5,1.5) {\textcolor{red}{$\times$}};
\node at (6.5,1.5) {\textcolor{red}{$\times$}};
\node at (7.5,1.5) {\textcolor{red}{$\times$}};

\node at (3.5,2.5) {\textcolor{red}{$\times$}};
\node at (5.5,2.5) {\textcolor{red}{$\times$}};
\node at (7.5,2.5) {\textcolor{red}{$\times$}};

\node at (2.5,3.5) {\textcolor{red}{$\times$}};
\node at (3.5,3.5) {\textcolor{red}{$\times$}};
\node at (4.5,3.5) {\textcolor{red}{$\times$}};
\node at (5.5,3.5) {\textcolor{red}{$\times$}};
\node at (6.5,3.5) {\textcolor{red}{$\times$}};
\node at (7.5,3.5) {\textcolor{red}{$\times$}};

\node at (2.5,4.5) {\textcolor{red}{$\times$}};
\node at (4.5,4.5) {\textcolor{red}{$\times$}};
\node at (6.5,4.5) {\textcolor{red}{$\times$}};

\node at (2.5,5.5) {\textcolor{red}{$\times$}};
\node at (3.5,5.5) {\textcolor{red}{$\times$}};
\node at (4.5,5.5) {\textcolor{red}{$\times$}};
\node at (5.5,5.5) {\textcolor{red}{$\times$}};
\node at (6.5,5.5) {\textcolor{red}{$\times$}};

\node at (2.5,2.5) {\textcolor{blue}{$\times$}};
\node at (1.5,4.5) {\textcolor{blue}{$\times$}};
\node at (7.5,4.5) {\textcolor{blue}{$\times$}};

\node at (1.5,2.5) {\textcolor{darkgreen}{$\times$}};
\node at (5.5,7.5) {\textcolor{darkgreen}{$\times$}};
\node at (3.5,7.5) {\textcolor{darkgreen}{$\times$}};

\node at (3.5,6.5) {\textcolor{orange}{$\times$}};
\node at (5.5,6.5) {\textcolor{orange}{$\times$}};

\node at (7.5,5.5) {\textcolor{blue}{1}};
\node at (7.5,6.5) {\textcolor{red}{$\times$}};
\node at (6.5,6.5) {\textcolor{red}{$\times$}};

\node at (2.5,7.5) {\textcolor{darkgreen}{1}};
\node at (2.5,6.5) {\textcolor{red}{$\times$}};
\node at (1.5,6.5) {\textcolor{red}{$\times$}};
\node at (1.5,7.5) {\textcolor{red}{$\times$}};
\fill[yellow, opacity=0.3] (2,6) rectangle (3,8);

\node at (1.5,5.5) {\textcolor{orange}{1}};
\node at (1.5,3.5) {\textcolor{orange}{1}};
\node at (1.5,1.5) {\textcolor{blue}{$\times$}};
\node at (2.5,1.5) {\textcolor{purple}{1}};
\end{scope}

\begin{scope}[shift={(0,0)}]
\foreach \x in {0,...,8} {
    \draw[gray!70, line width=0.35pt] (\x,0) -- (\x,8);
  }
  \foreach \y in {0,...,8} {
    \draw[gray!70, line width=0.35pt] (0,\y) -- (8,\y);
  }
  \draw[very thick] (1,1) rectangle (8,8);
  \draw[very thick, blue] (0,0) rectangle (7,7);
  \draw[very thick, blue] (0,0) rectangle (1,2);
  \draw[very thick, blue] (5,0) rectangle (6,1);
  \draw[very thick, blue] (4,6) rectangle (5,7);

\node at (4.5,2.5) {\normalsize 1};
\node at (6.5,2.5) {\normalsize 1};
\node at (3.5,4.5) {\normalsize 1};
\node at (5.5,4.5) {\normalsize 1};

\node at (3.5,1.5) {\textcolor{red}{$\times$}};
\node at (4.5,1.5) {\textcolor{red}{$\times$}};
\node at (5.5,1.5) {\textcolor{red}{$\times$}};
\node at (6.5,1.5) {\textcolor{red}{$\times$}};
\node at (7.5,1.5) {\textcolor{red}{$\times$}};

\node at (3.5,2.5) {\textcolor{red}{$\times$}};
\node at (5.5,2.5) {\textcolor{red}{$\times$}};
\node at (7.5,2.5) {\textcolor{red}{$\times$}};

\node at (2.5,3.5) {\textcolor{red}{$\times$}};
\node at (3.5,3.5) {\textcolor{red}{$\times$}};
\node at (4.5,3.5) {\textcolor{red}{$\times$}};
\node at (5.5,3.5) {\textcolor{red}{$\times$}};
\node at (6.5,3.5) {\textcolor{red}{$\times$}};
\node at (7.5,3.5) {\textcolor{red}{$\times$}};

\node at (2.5,4.5) {\textcolor{red}{$\times$}};
\node at (4.5,4.5) {\textcolor{red}{$\times$}};
\node at (6.5,4.5) {\textcolor{red}{$\times$}};

\node at (2.5,5.5) {\textcolor{red}{$\times$}};
\node at (3.5,5.5) {\textcolor{red}{$\times$}};
\node at (4.5,5.5) {\textcolor{red}{$\times$}};
\node at (5.5,5.5) {\textcolor{red}{$\times$}};
\node at (6.5,5.5) {\textcolor{red}{$\times$}};

\node at (2.5,2.5) {\textcolor{blue}{$\times$}};
\node at (1.5,4.5) {\textcolor{blue}{$\times$}};
\node at (7.5,4.5) {\textcolor{blue}{$\times$}};

\node at (1.5,2.5) {\textcolor{darkgreen}{$\times$}};
\node at (5.5,7.5) {\textcolor{darkgreen}{$\times$}};
\node at (3.5,7.5) {\textcolor{darkgreen}{$\times$}};

\node at (3.5,6.5) {\textcolor{orange}{$\times$}};
\node at (5.5,6.5) {\textcolor{orange}{$\times$}};

\node at (7.5,5.5) {\textcolor{blue}{1}};
\node at (7.5,6.5) {\textcolor{red}{$\times$}};
\node at (6.5,6.5) {\textcolor{red}{$\times$}};

\node at (2.5,7.5) {\textcolor{darkgreen}{1}};
\node at (2.5,6.5) {\textcolor{red}{$\times$}};
\node at (1.5,6.5) {\textcolor{red}{$\times$}};
\node at (1.5,7.5) {\textcolor{red}{$\times$}};

\node at (1.5,5.5) {\textcolor{orange}{1}};
\node at (1.5,3.5) {\textcolor{orange}{1}};
\node at (1.5,1.5) {\textcolor{blue}{$\times$}};
\node at (2.5,1.5) {\textcolor{purple}{1}};

\node at (0.5,2.5) {\textcolor{red}{$\times$}};
\node at (0.5,3.5) {\textcolor{red}{$\times$}};
\node at (0.5,4.5) {\textcolor{red}{$\times$}};
\node at (0.5,5.5) {\textcolor{red}{$\times$}};
\node at (0.5,6.5) {\textcolor{red}{$\times$}};

\node at (1.5,0.5) {\textcolor{red}{$\times$}};
\node at (2.5,0.5) {\textcolor{red}{$\times$}};
\node at (3.5,0.5) {\textcolor{red}{$\times$}};
\node at (4.5,0.5) {\textcolor{orange}{$\times$}};
\node at (6.5,0.5) {\textcolor{orange}{$\times$}};
\end{scope}

\end{tikzpicture}

\caption{$G$ (black) and $G(-1,-1)$ (blue) in the proof of Lemma~4\,(vi).}
\label{fig:29}
\end{figure}

\subsection{Proof of Lemma \ref{lem:4}\,(vii)}
\begin{proof}
    Suppose for contradiction that the pattern
    \[
    \begin{bmatrix}
    x & 1 & x\\
    x & x & x\\
    1 & x & 1
    \end{bmatrix}
    \]
    appears in the coloring. Let $ G $ be a $ P_7 \boxtimes P_7 $ subgraph in which the pattern appears, where the squares at $(2,2)$, $(4,2)$ and $(3,4)$ are of color~1. By the distance constraint and Lemma~\ref{lem:4}\,(ii) and (vi), the squares that cannot be assigned color~1 are shown in Figure~\ref{fig:30}. We now partition the remaining squares into seven parts, as shown in Figure~\ref{fig:30}. By Observations~\ref{obs:2}\,(i), (iv), and Lemma \ref{lem:4}\,(v), the L-shaped part must contain exactly two squares of color~1, and the other six parts must each contain exactly one square of color~1. In particular, the part at $(1,5)$ must be a square of color~1.

Assume that $(3,6)$ is a square of color~1. As a result, by the distance constraint, square $(1,7)$ must be of color~1. We see that this subgraph contains the pattern
\[
\begin{bmatrix}
1 & x & 1 & x\\
x & x & x & x\\
x & 1 & x & 1
\end{bmatrix}
\]
which contradicts Lemma \ref{lem:4}\,(vi).

Hence, square $(3,6)$ cannot be assigned color~1. Since either $(4,6)$ or $(4,7)$ is of color~1, the distance constraint implies that $(5,6)$ and $(5,7)$ cannot be assigned color~1. We then consider the L-shaped part. Square $(5,5)$ must be of color~1 otherwise this part would contain at most one square of color~1 which is a contradiction. By the distance constraint, square $(4,7)$ must be of color~1. By Lemma~\ref{lem:4}\,(ii), square $(7,7)$ cannot be assigned color~1. Then, square $(6,7)$ must be of color 1. By the distance constraint, square $(7,5)$ must be of color~1. Again, this subgraph contains the pattern
\[
\begin{bmatrix}
1 & x & 1 & x\\
x & x & x & x\\
x & 1 & x & 1
\end{bmatrix}
\]
which contradicts Lemma \ref{lem:4}\,(vi).
\end{proof}

    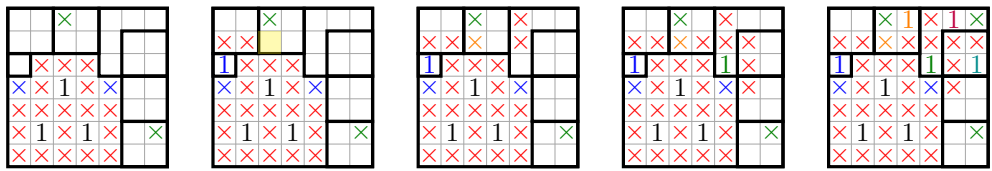
\begin{figure}[ht]
\centering
\begin{tikzpicture}[scale=0.30]

\begin{scope}[shift={(-36,0)}]
  \foreach \x in {0,...,7} {
    \draw[gray!70, line width=0.35pt] (\x,0) -- (\x,7);
  }
  \foreach \y in {0,...,7} {
    \draw[gray!70, line width=0.35pt] (0,\y) -- (7,\y);
  }
  \draw[very thick] (0,0) rectangle (7,7);
  \draw[very thick] (0,4) rectangle (1,5);
  \draw[very thick] (0,5) rectangle (2,7);
  \draw[very thick] (2,5) rectangle (4,7);
  \draw[very thick] (4,4) rectangle (7,7);
  \draw[very thick] (5,4) rectangle (7,6);
  \draw[very thick] (5,2) rectangle (7,4);
  \draw[very thick] (5,0) rectangle (7,2);

  \node at (1.5,1.5) {\normalsize 1};
  \node at (3.5,1.5) {\normalsize 1};
  \node at (2.5,3.5) {\normalsize 1};
\node at (0.5,0.5) {\textcolor{red}{$\times$}};
\node at (1.5,0.5) {\textcolor{red}{$\times$}};
\node at (2.5,0.5) {\textcolor{red}{$\times$}};
\node at (3.5,0.5) {\textcolor{red}{$\times$}};
\node at (4.5,0.5) {\textcolor{red}{$\times$}};
\node at (0.5,1.5) {\textcolor{red}{$\times$}};
\node at (2.5,1.5) {\textcolor{red}{$\times$}};
\node at (4.5,1.5) {\textcolor{red}{$\times$}};
\node at (0.5,2.5) {\textcolor{red}{$\times$}};
\node at (1.5,2.5) {\textcolor{red}{$\times$}};
\node at (2.5,2.5) {\textcolor{red}{$\times$}};
\node at (3.5,2.5) {\textcolor{red}{$\times$}};
\node at (4.5,2.5) {\textcolor{red}{$\times$}};
\node at (1.5,3.5) {\textcolor{red}{$\times$}};
\node at (3.5,3.5) {\textcolor{red}{$\times$}};
\node at (1.5,4.5) {\textcolor{red}{$\times$}};
\node at (2.5,4.5) {\textcolor{red}{$\times$}};
\node at (3.5,4.5) {\textcolor{red}{$\times$}};

\node at (0.5,3.5) {\textcolor{blue}{$\times$}};
\node at (4.5,3.5) {\textcolor{blue}{$\times$}};
\node at (2.5,6.5) {\textcolor{darkgreen}{$\times$}};
\node at (6.5,1.5) {\textcolor{darkgreen}{$\times$}};
\end{scope}

\begin{scope}[shift={(-27,0)}]
  \foreach \x in {0,...,7} {
    \draw[gray!70, line width=0.35pt] (\x,0) -- (\x,7);
  }
  \foreach \y in {0,...,7} {
    \draw[gray!70, line width=0.35pt] (0,\y) -- (7,\y);
  }
  \draw[very thick] (0,0) rectangle (7,7);
  \draw[very thick] (0,4) rectangle (1,5);
  \draw[very thick] (0,5) rectangle (2,7);
  \draw[very thick] (2,5) rectangle (4,7);
  \draw[very thick] (4,4) rectangle (7,7);
  \draw[very thick] (5,4) rectangle (7,6);
  \draw[very thick] (5,2) rectangle (7,4);
  \draw[very thick] (5,0) rectangle (7,2);

  \node at (1.5,1.5) {\normalsize 1};
  \node at (3.5,1.5) {\normalsize 1};
  \node at (2.5,3.5) {\normalsize 1};
\node at (0.5,0.5) {\textcolor{red}{$\times$}};
\node at (1.5,0.5) {\textcolor{red}{$\times$}};
\node at (2.5,0.5) {\textcolor{red}{$\times$}};
\node at (3.5,0.5) {\textcolor{red}{$\times$}};
\node at (4.5,0.5) {\textcolor{red}{$\times$}};
\node at (0.5,1.5) {\textcolor{red}{$\times$}};
\node at (2.5,1.5) {\textcolor{red}{$\times$}};
\node at (4.5,1.5) {\textcolor{red}{$\times$}};
\node at (0.5,2.5) {\textcolor{red}{$\times$}};
\node at (1.5,2.5) {\textcolor{red}{$\times$}};
\node at (2.5,2.5) {\textcolor{red}{$\times$}};
\node at (3.5,2.5) {\textcolor{red}{$\times$}};
\node at (4.5,2.5) {\textcolor{red}{$\times$}};
\node at (1.5,3.5) {\textcolor{red}{$\times$}};
\node at (3.5,3.5) {\textcolor{red}{$\times$}};
\node at (1.5,4.5) {\textcolor{red}{$\times$}};
\node at (2.5,4.5) {\textcolor{red}{$\times$}};
\node at (3.5,4.5) {\textcolor{red}{$\times$}};

\node at (0.5,3.5) {\textcolor{blue}{$\times$}};
\node at (4.5,3.5) {\textcolor{blue}{$\times$}};
\node at (2.5,6.5) {\textcolor{darkgreen}{$\times$}};
\node at (6.5,1.5) {\textcolor{darkgreen}{$\times$}};

\node at (0.5,4.5) {\textcolor{blue}{1}};
\node at (0.5,5.5) {\textcolor{red}{$\times$}};
\node at (1.5,5.5) {\textcolor{red}{$\times$}};
\fill[yellow, opacity=0.3] (2,5) rectangle (3,6);
\end{scope}

\begin{scope}[shift={(-18,0)}]
  \foreach \x in {0,...,7} {
    \draw[gray!70, line width=0.35pt] (\x,0) -- (\x,7);
  }
  \foreach \y in {0,...,7} {
    \draw[gray!70, line width=0.35pt] (0,\y) -- (7,\y);
  }
  \draw[very thick] (0,0) rectangle (7,7);
  \draw[very thick] (0,4) rectangle (1,5);
  \draw[very thick] (0,5) rectangle (2,7);
  \draw[very thick] (2,5) rectangle (4,7);
  \draw[very thick] (4,4) rectangle (7,7);
  \draw[very thick] (5,4) rectangle (7,6);
  \draw[very thick] (5,2) rectangle (7,4);
  \draw[very thick] (5,0) rectangle (7,2);

  \node at (1.5,1.5) {\normalsize 1};
  \node at (3.5,1.5) {\normalsize 1};
  \node at (2.5,3.5) {\normalsize 1};
\node at (0.5,0.5) {\textcolor{red}{$\times$}};
\node at (1.5,0.5) {\textcolor{red}{$\times$}};
\node at (2.5,0.5) {\textcolor{red}{$\times$}};
\node at (3.5,0.5) {\textcolor{red}{$\times$}};
\node at (4.5,0.5) {\textcolor{red}{$\times$}};
\node at (0.5,1.5) {\textcolor{red}{$\times$}};
\node at (2.5,1.5) {\textcolor{red}{$\times$}};
\node at (4.5,1.5) {\textcolor{red}{$\times$}};
\node at (0.5,2.5) {\textcolor{red}{$\times$}};
\node at (1.5,2.5) {\textcolor{red}{$\times$}};
\node at (2.5,2.5) {\textcolor{red}{$\times$}};
\node at (3.5,2.5) {\textcolor{red}{$\times$}};
\node at (4.5,2.5) {\textcolor{red}{$\times$}};
\node at (1.5,3.5) {\textcolor{red}{$\times$}};
\node at (3.5,3.5) {\textcolor{red}{$\times$}};
\node at (1.5,4.5) {\textcolor{red}{$\times$}};
\node at (2.5,4.5) {\textcolor{red}{$\times$}};
\node at (3.5,4.5) {\textcolor{red}{$\times$}};

\node at (0.5,3.5) {\textcolor{blue}{$\times$}};
\node at (4.5,3.5) {\textcolor{blue}{$\times$}};
\node at (2.5,6.5) {\textcolor{darkgreen}{$\times$}};
\node at (6.5,1.5) {\textcolor{darkgreen}{$\times$}};

\node at (0.5,4.5) {\textcolor{blue}{1}};
\node at (0.5,5.5) {\textcolor{red}{$\times$}};
\node at (1.5,5.5) {\textcolor{red}{$\times$}};
\node at (2.5,5.5) {\textcolor{orange}{$\times$}};
\node at (4.5,5.5) {\textcolor{red}{$\times$}};
\node at (4.5,6.5) {\textcolor{red}{$\times$}};
\end{scope}

\begin{scope}[shift={(-9,0)}]
  \foreach \x in {0,...,7} {
    \draw[gray!70, line width=0.35pt] (\x,0) -- (\x,7);
  }
  \foreach \y in {0,...,7} {
    \draw[gray!70, line width=0.35pt] (0,\y) -- (7,\y);
  }
  \draw[very thick] (0,0) rectangle (7,7);
  \draw[very thick] (0,4) rectangle (1,5);
  \draw[very thick] (0,5) rectangle (2,7);
  \draw[very thick] (2,5) rectangle (4,7);
  \draw[very thick] (4,4) rectangle (7,7);
  \draw[very thick] (5,4) rectangle (7,6);
  \draw[very thick] (5,2) rectangle (7,4);
  \draw[very thick] (5,0) rectangle (7,2);

  \node at (1.5,1.5) {\normalsize 1};
  \node at (3.5,1.5) {\normalsize 1};
  \node at (2.5,3.5) {\normalsize 1};
\node at (0.5,0.5) {\textcolor{red}{$\times$}};
\node at (1.5,0.5) {\textcolor{red}{$\times$}};
\node at (2.5,0.5) {\textcolor{red}{$\times$}};
\node at (3.5,0.5) {\textcolor{red}{$\times$}};
\node at (4.5,0.5) {\textcolor{red}{$\times$}};
\node at (0.5,1.5) {\textcolor{red}{$\times$}};
\node at (2.5,1.5) {\textcolor{red}{$\times$}};
\node at (4.5,1.5) {\textcolor{red}{$\times$}};
\node at (0.5,2.5) {\textcolor{red}{$\times$}};
\node at (1.5,2.5) {\textcolor{red}{$\times$}};
\node at (2.5,2.5) {\textcolor{red}{$\times$}};
\node at (3.5,2.5) {\textcolor{red}{$\times$}};
\node at (4.5,2.5) {\textcolor{red}{$\times$}};
\node at (1.5,3.5) {\textcolor{red}{$\times$}};
\node at (3.5,3.5) {\textcolor{red}{$\times$}};
\node at (1.5,4.5) {\textcolor{red}{$\times$}};
\node at (2.5,4.5) {\textcolor{red}{$\times$}};
\node at (3.5,4.5) {\textcolor{red}{$\times$}};

\node at (0.5,3.5) {\textcolor{blue}{$\times$}};
\node at (4.5,3.5) {\textcolor{blue}{$\times$}};
\node at (2.5,6.5) {\textcolor{darkgreen}{$\times$}};
\node at (6.5,1.5) {\textcolor{darkgreen}{$\times$}};

\node at (0.5,4.5) {\textcolor{blue}{1}};
\node at (0.5,5.5) {\textcolor{red}{$\times$}};
\node at (1.5,5.5) {\textcolor{red}{$\times$}};
\node at (2.5,5.5) {\textcolor{orange}{$\times$}};
\node at (4.5,5.5) {\textcolor{red}{$\times$}};
\node at (4.5,6.5) {\textcolor{red}{$\times$}};

\node at (4.5,4.5) {\textcolor{darkgreen}{1}};
\node at (3.5,5.5) {\textcolor{red}{$\times$}};
\node at (5.5,5.5) {\textcolor{red}{$\times$}};
\node at (5.5,4.5) {\textcolor{red}{$\times$}};
\node at (5.5,3.5) {\textcolor{red}{$\times$}};
\end{scope}

\begin{scope}[shift={(0,0)}]
  \foreach \x in {0,...,7} {
    \draw[gray!70, line width=0.35pt] (\x,0) -- (\x,7);
  }
  \foreach \y in {0,...,7} {
    \draw[gray!70, line width=0.35pt] (0,\y) -- (7,\y);
  }
  \draw[very thick] (0,0) rectangle (7,7);
  \draw[very thick] (0,4) rectangle (1,5);
  \draw[very thick] (0,5) rectangle (2,7);
  \draw[very thick] (2,5) rectangle (4,7);
  \draw[very thick] (4,4) rectangle (7,7);
  \draw[very thick] (5,4) rectangle (7,6);
  \draw[very thick] (5,2) rectangle (7,4);
  \draw[very thick] (5,0) rectangle (7,2);

  \node at (1.5,1.5) {\normalsize 1};
  \node at (3.5,1.5) {\normalsize 1};
  \node at (2.5,3.5) {\normalsize 1};
\node at (0.5,0.5) {\textcolor{red}{$\times$}};
\node at (1.5,0.5) {\textcolor{red}{$\times$}};
\node at (2.5,0.5) {\textcolor{red}{$\times$}};
\node at (3.5,0.5) {\textcolor{red}{$\times$}};
\node at (4.5,0.5) {\textcolor{red}{$\times$}};
\node at (0.5,1.5) {\textcolor{red}{$\times$}};
\node at (2.5,1.5) {\textcolor{red}{$\times$}};
\node at (4.5,1.5) {\textcolor{red}{$\times$}};
\node at (0.5,2.5) {\textcolor{red}{$\times$}};
\node at (1.5,2.5) {\textcolor{red}{$\times$}};
\node at (2.5,2.5) {\textcolor{red}{$\times$}};
\node at (3.5,2.5) {\textcolor{red}{$\times$}};
\node at (4.5,2.5) {\textcolor{red}{$\times$}};
\node at (1.5,3.5) {\textcolor{red}{$\times$}};
\node at (3.5,3.5) {\textcolor{red}{$\times$}};
\node at (1.5,4.5) {\textcolor{red}{$\times$}};
\node at (2.5,4.5) {\textcolor{red}{$\times$}};
\node at (3.5,4.5) {\textcolor{red}{$\times$}};

\node at (0.5,3.5) {\textcolor{blue}{$\times$}};
\node at (4.5,3.5) {\textcolor{blue}{$\times$}};
\node at (2.5,6.5) {\textcolor{darkgreen}{$\times$}};
\node at (6.5,1.5) {\textcolor{darkgreen}{$\times$}};

\node at (0.5,4.5) {\textcolor{blue}{1}};
\node at (0.5,5.5) {\textcolor{red}{$\times$}};
\node at (1.5,5.5) {\textcolor{red}{$\times$}};
\node at (2.5,5.5) {\textcolor{orange}{$\times$}};
\node at (4.5,5.5) {\textcolor{red}{$\times$}};
\node at (4.5,6.5) {\textcolor{red}{$\times$}};

\node at (4.5,4.5) {\textcolor{darkgreen}{1}};
\node at (3.5,5.5) {\textcolor{red}{$\times$}};
\node at (5.5,5.5) {\textcolor{red}{$\times$}};
\node at (5.5,4.5) {\textcolor{red}{$\times$}};
\node at (5.5,3.5) {\textcolor{red}{$\times$}};

\node at (3.5,6.5) {\textcolor{orange}{1}};
\node at (6.5,6.5) {\textcolor{darkgreen}{$\times$}};
\node at (5.5,6.5) {\textcolor{purple}{1}};
\node at (6.5,5.5) {\textcolor{red}{$\times$}};
\node at (6.5,4.5) {\textcolor{darkcyan}{1}};
\end{scope}

\end{tikzpicture}

\caption{$G$ in the proof of Lemma~4\,(vii).}
\label{fig:30}
\end{figure}

\subsection{Proof of Lemma \ref{lem:4}\,(viii)}
\begin{proof}
    Suppose for contradiction that the pattern
    \[
    \begin{bmatrix}
    1 & x & 1
    \end{bmatrix}
    \]
    appears in the coloring. Let $ G $ be a $ P_7 \boxtimes P_7 $ subgraph in which the pattern appears, where the squares at $(3,4)$ and $(5,4)$ are of color~1. By the distance constraint and Lemma~\ref{lem:4}\,(ii), (v) and (vii), the squares that cannot be assigned color~1 are shown in Figure~\ref{fig:31}. We partition the remaining squares into eight parts as shown in Figure~\ref{fig:31}. By Lemma~\ref{lem:4}\,(vii), every $ P_1 \boxtimes P_3 $ part contains at most one square of color~1. By Observation~\ref{obs:2}\,(i), each of the other six parts also contains at most one square of color~1. Hence, $c(G) \leq 2 + 2 + 6 = 10$, contradicting Observation~\ref{obs:2}\,(iv).

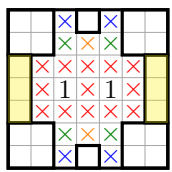
\begin{figure}[ht]
\centering
\begin{tikzpicture}[scale=0.30]

\begin{scope}[shift={(0,0)}]
  \foreach \x in {0,...,7} {
    \draw[gray!70, line width=0.35pt] (\x,0) -- (\x,7);
  }
  \foreach \y in {0,...,7} {
    \draw[gray!70, line width=0.35pt] (0,\y) -- (7,\y);
  }
  \draw[very thick] (0,0) rectangle (7,7);
  \draw[very thick] (0,0) rectangle (2,2);
  \draw[very thick] (0,2) rectangle (1,5);
  \draw[very thick] (0,5) rectangle (2,7);

  \draw[very thick] (3,6) rectangle (4,7);
  \draw[very thick] (5,5) rectangle (7,7);
  \draw[very thick] (6,2) rectangle (7,5);
  \draw[very thick] (5,0) rectangle (7,2);
  \draw[very thick] (3,0) rectangle (4,1);

  \node at (2.5,3.5) {\normalsize 1};
\node at (4.5,3.5) {\normalsize 1};

\node at (1.5,2.5) {\textcolor{red}{$\times$}};
\node at (2.5,2.5) {\textcolor{red}{$\times$}};
\node at (3.5,2.5) {\textcolor{red}{$\times$}};
\node at (4.5,2.5) {\textcolor{red}{$\times$}};
\node at (5.5,2.5) {\textcolor{red}{$\times$}};

\node at (1.5,3.5) {\textcolor{red}{$\times$}};
\node at (3.5,3.5) {\textcolor{red}{$\times$}};
\node at (5.5,3.5) {\textcolor{red}{$\times$}};

\node at (1.5,4.5) {\textcolor{red}{$\times$}};
\node at (2.5,4.5) {\textcolor{red}{$\times$}};
\node at (3.5,4.5) {\textcolor{red}{$\times$}};
\node at (4.5,4.5) {\textcolor{red}{$\times$}};
\node at (5.5,4.5) {\textcolor{red}{$\times$}};

\node at (2.5,0.5) {\textcolor{blue}{$\times$}};
\node at (4.5,0.5) {\textcolor{blue}{$\times$}};
\node at (2.5,6.5) {\textcolor{blue}{$\times$}};
\node at (4.5,6.5) {\textcolor{blue}{$\times$}};

\node at (2.5,1.5) {\textcolor{darkgreen}{$\times$}};
\node at (4.5,1.5) {\textcolor{darkgreen}{$\times$}};
\node at (2.5,5.5) {\textcolor{darkgreen}{$\times$}};
\node at (4.5,5.5) {\textcolor{darkgreen}{$\times$}};

\node at (3.5,5.5) {\textcolor{orange}{$\times$}};
\node at (3.5,1.5) {\textcolor{orange}{$\times$}};
\fill[yellow, opacity=0.3] (0,2) rectangle (1,5);
\fill[yellow, opacity=0.3] (6,2) rectangle (7,5);
\end{scope}

\end{tikzpicture}

\caption{$G$ in the proof of Lemma~4\,(viii).}
\label{fig:31}
\end{figure}

    Moreover, we can show that each $P_3 \boxtimes P_1$ subgraph contains at most one squares of color 1. Assume for contradiction that there exists a $P_3 \boxtimes P_1$  subgraph that contains two squares of color 1. Then the subgraph has the pattern  
\[
    \begin{bmatrix}
    1 & x & 1
    \end{bmatrix}
    \]
which is a contradiction.
\end{proof}

\subsection{Proof of Lemma \ref{lem:4}\,(ix)}
\begin{proof}
Suppose for contradiction that the pattern
\[
    \begin{bmatrix}
    x & x & 1\\
    1 & x & x
    \end{bmatrix}
    \]
appears in the coloring. Let $ G $ be a $ P_7 \boxtimes P_7 $ subgraph in which the pattern appears, where the squares at $(3,4)$ and $(5,5)$ are of color~1. By the distance constraint and Lemma~\ref{lem:4}\,(ii) and (viii), the squares that cannot be assigned color~1 are shown in Figure~\ref{fig:32}. We partition the remaining squares into eight parts as shown in Figure~\ref{fig:32}. By Observation~\ref{obs:2}\,(i) and (iv), each parts contains exactly one square of color~1. Hence, squares $(1,3)$ and $(1,5)$ must be of color 1 (see Figure \ref{fig:32}). This contradicts Lemma \ref{lem:4}\,(viii).
\end{proof}

\begin{figure}[ht]
\centering
\begin{tikzpicture}[scale=0.30]

\begin{scope}[shift={(-9,0)}]
  \foreach \x in {0,...,7} {
    \draw[gray!70, line width=0.35pt] (\x,0) -- (\x,7);
  }
  \foreach \y in {0,...,7} {
    \draw[gray!70, line width=0.35pt] (0,\y) -- (7,\y);
  }
  \draw[very thick] (0,0) rectangle (7,7);
  \draw[very thick] (0,0) rectangle (2,2);
  \draw[very thick] (0,2) rectangle (1,3);
  \draw[very thick] (0,4) rectangle (1,5);
  \draw[very thick] (0,5) rectangle (2,7);

  \draw[very thick] (3,6) rectangle (4,7);
  \draw[very thick] (5,5) rectangle (7,7);
  \draw[very thick] (5,2) rectangle (7,4);
  \draw[very thick] (5,0) rectangle (7,2);
  \draw[very thick] (3,0) rectangle (5,2);

  \node at (2.5,3.5) {\normalsize 1};
\node at (4.5,4.5) {\normalsize 1};

\node at (1.5,2.5) {\textcolor{red}{$\times$}};
\node at (2.5,2.5) {\textcolor{red}{$\times$}};
\node at (3.5,2.5) {\textcolor{red}{$\times$}};
\node at (4.5,2.5) {\textcolor{darkgreen}{$\times$}};

\node at (1.5,3.5) {\textcolor{red}{$\times$}};
\node at (3.5,3.5) {\textcolor{red}{$\times$}};
\node at (5.5,3.5) {\textcolor{red}{$\times$}};

\node at (1.5,4.5) {\textcolor{red}{$\times$}};
\node at (2.5,4.5) {\textcolor{red}{$\times$}};
\node at (3.5,4.5) {\textcolor{red}{$\times$}};
\node at (4.5,3.5) {\textcolor{red}{$\times$}};
\node at (5.5,4.5) {\textcolor{red}{$\times$}};

\node at (2.5,0.5) {\textcolor{blue}{$\times$}};
\node at (2.5,6.5) {\textcolor{blue}{$\times$}};
\node at (4.5,6.5) {\textcolor{darkgreen}{$\times$}};

\node at (2.5,1.5) {\textcolor{darkgreen}{$\times$}};
\node at (0.5,3.5) {\textcolor{darkgreen}{$\times$}};
\node at (6.5,4.5) {\textcolor{darkgreen}{$\times$}};
\node at (4.5,1.5) {\textcolor{blue}{$\times$}};
\node at (2.5,5.5) {\textcolor{darkgreen}{$\times$}};
\node at (4.5,5.5) {\textcolor{red}{$\times$}};
\node at (5.5,5.5) {\textcolor{red}{$\times$}};

\node at (3.5,5.5) {\textcolor{red}{$\times$}};
\end{scope}

\begin{scope}[shift={(0,0)}]
  \foreach \x in {0,...,7} {
    \draw[gray!70, line width=0.35pt] (\x,0) -- (\x,7);
  }
  \foreach \y in {0,...,7} {
    \draw[gray!70, line width=0.35pt] (0,\y) -- (7,\y);
  }
  \draw[very thick] (0,0) rectangle (7,7);
  \draw[very thick] (0,0) rectangle (2,2);
  \draw[very thick] (0,2) rectangle (1,3);
  \draw[very thick] (0,4) rectangle (1,5);
  \draw[very thick] (0,5) rectangle (2,7);

  \draw[very thick] (3,6) rectangle (4,7);
  \draw[very thick] (5,5) rectangle (7,7);
  \draw[very thick] (5,2) rectangle (7,4);
  \draw[very thick] (5,0) rectangle (7,2);
  \draw[very thick] (3,0) rectangle (5,2);

  \node at (2.5,3.5) {\normalsize 1};
\node at (4.5,4.5) {\normalsize 1};

\node at (1.5,2.5) {\textcolor{red}{$\times$}};
\node at (2.5,2.5) {\textcolor{red}{$\times$}};
\node at (3.5,2.5) {\textcolor{red}{$\times$}};
\node at (4.5,2.5) {\textcolor{darkgreen}{$\times$}};

\node at (1.5,3.5) {\textcolor{red}{$\times$}};
\node at (3.5,3.5) {\textcolor{red}{$\times$}};
\node at (5.5,3.5) {\textcolor{red}{$\times$}};

\node at (1.5,4.5) {\textcolor{red}{$\times$}};
\node at (2.5,4.5) {\textcolor{red}{$\times$}};
\node at (3.5,4.5) {\textcolor{red}{$\times$}};
\node at (4.5,3.5) {\textcolor{red}{$\times$}};
\node at (5.5,4.5) {\textcolor{red}{$\times$}};

\node at (2.5,0.5) {\textcolor{blue}{$\times$}};
\node at (2.5,6.5) {\textcolor{blue}{$\times$}};
\node at (4.5,6.5) {\textcolor{darkgreen}{$\times$}};

\node at (2.5,1.5) {\textcolor{darkgreen}{$\times$}};
\node at (0.5,3.5) {\textcolor{darkgreen}{$\times$}};
\node at (6.5,4.5) {\textcolor{darkgreen}{$\times$}};
\node at (4.5,1.5) {\textcolor{blue}{$\times$}};
\node at (2.5,5.5) {\textcolor{darkgreen}{$\times$}};
\node at (4.5,5.5) {\textcolor{red}{$\times$}};
\node at (5.5,5.5) {\textcolor{red}{$\times$}};

\node at (3.5,5.5) {\textcolor{red}{$\times$}};

\node at (0.5,2.5) {\textcolor{blue}{1}};
\node at (0.5,4.5) {\textcolor{blue}{1}};
\end{scope}
\end{tikzpicture}

\caption{$G$ in the proof of Lemma~4\,(ix).}
\label{fig:32}
\end{figure}
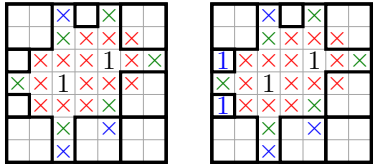

\section{Concluding remarks}\label{sec:6}

We have proved that $\chi_S(P_\infty \boxtimes P_\infty) = 40$ for $S = (1,6,6,\ldots)$. However, the exact values of $\chi_S(P_\infty \boxtimes P_\infty)$ for $S = (1,k,k,\ldots)$ with even $k \geq 8$ are still unknown. The best known bounds of $\chi_S(P_\infty \boxtimes P_\infty)$ in these cases are
\[
\frac{3k^2}{4}+\frac{3k}{2}+3 \leq \chi_S(P_\infty \boxtimes P_\infty) \leq \left\lceil \frac{3k^2+7k+8}{4} \right\rceil,
\]
as shown in~\cite{tiyajamorn2024packing}. We conjecture that the exact value of $\chi_S(P_\infty \boxtimes P_\infty)$ in these cases is equal to the upper bound.

\begin{conjecture}
    Let $k \geq 8$ be even and $S = (1,k,k,\ldots)$. Then
    \[
    \chi_S(P_\infty \boxtimes P_\infty) = \left\lceil \frac{3k^2+7k+8}{4} \right\rceil.
    \]
\end{conjecture}

It may be possible to generalize the approach in this paper to $k\geq 8$ but one has to carefully choose the order of forbidden patterns of squares of color 1 to consider.

More generally, for sequences of the form $S = (i,k,k,\ldots)$, the problem of determining $\chi_S (P_\infty \boxtimes P_\infty)$ is open for the case when $i$ satisfies $i \leq \lfloor\frac{k-3}{2} \rfloor$ and $i+1\nmid k+1$ (see~\cite{tiyajamorn2024packing}).

\bibliographystyle{siam}
\bibliography{s-packing}

\end{document}